\documentclass[a4paper,12pt,twoside]{amsart}
\usepackage{mathtools, stmaryrd}
\usepackage{xparse} \DeclarePairedDelimiterX{\Iintv}[1]{\llbracket}{\rrbracket}{\iintvargs{#1}}
\NewDocumentCommand{\iintvargs}{>{\SplitArgument{1}{,}}m}
{\iintvargsaux#1}
\NewDocumentCommand{\iintvargsaux}{mm} {#1\mkern1.5mu..\mkern1.5mu#2}
\usepackage{abraces}
\usepackage[T1]{fontenc}
\usepackage{vmargin}
\usepackage{amssymb}
\usepackage{csquotes}
\usepackage{url}
\usepackage{datetime}
\usepackage{amsmath}
\numberwithin{equation}{subsection}
\usepackage{amsthm}
\usepackage{mathrsfs}
\usepackage[dvips]{graphicx}
\usepackage{vmargin}
\usepackage{color}
\definecolor{airforceblue}{rgb}{0.36, 0.54, 0.66}
\definecolor{bured}{rgb}{0.8, 0.0, 0.0}
\definecolor{burntsienna}{rgb}{0.91, 0.45, 0.32}
\definecolor{burntsienna}{rgb}{0.91, 0.45, 0.32}
\definecolor{burntsienna}{rgb}{0.91, 0.45, 0.32}
\definecolor{chromeyellow}{rgb}{1.0, 0.65, 0.0}
\definecolor{cobalt}{rgb}{0.0, 0.28, 0.67}
\definecolor{coralred}{rgb}{1.0, 0.25, 0.25}
\definecolor{coralred}{rgb}{1.0, 0.25, 0.25}
\definecolor{coralred}{rgb}{1.0, 0.25, 0.25}
\definecolor{ferrarired}{rgb}{1.0, 0.11, 0.0}
\definecolor{forestgreen}{rgb}{0.0, 0.27, 0.13}
\definecolor{cream}{rgb}{1.0, 0.99, 0.82}
\definecolor{cream}{rgb}{1.0, 0.99, 0.82}
\definecolor{cream}{rgb}{1.0, 0.99, 0.82}
\definecolor{cream}{rgb}{1.0, 0.99, 0.82}
\definecolor{candy}{rgb}{0.64, 0.0, 0.0}
\definecolor{darkblue}{rgb}{0.0, 0.0, 0.55}
\definecolor{crimson}{rgb}{0.86, 0.08, 0.24}
\definecolor{carmine}{rgb}{0.94, 0.19, 0.22}
\definecolor{ashgrey}{rgb}{0.7, 0.75, 0.71}
\definecolor{tomato}{rgb}{1.0, 0.39, 0.28}
\definecolor{awesome}{rgb}{1.0, 0.13, 0.32}
\definecolor{cadmiumorange}{rgb}{0.93, 0.53, 0.18}
\definecolor{darkseagreen}{rgb}{0.56, 0.74, 0.56}
\definecolor{darkspringgreen}{rgb}{0.09, 0.45, 0.27}
\definecolor{antiquefuchsia}{rgb}{0.57, 0.36, 0.51}
\definecolor{amaranth}{rgb}{0.9, 0.17, 0.31}
\usepackage{amscd}
\usepackage{wasysym}
\usepackage{float}
\usepackage{stmaryrd} 
\usepackage{comment}
\DeclareGraphicsExtensions{.pdf, .png, .jpg, .mps}
\usepackage[colorlinks=true,pdfpagemode=FullScreen]{hyperref}
\definecolor{myc}{cmyk}{0.0009,0.8,0.8,0.00}
\newtheorem{theorem}{Theorem}[section]
\newtheorem{definition}{Definition}[section]
\newtheorem{lem}[theorem]{Lemma}
\newtheorem{pro}[theorem]{Proposition}
\newtheorem{defi}[theorem]{Definition}
\newtheorem{cor}[theorem]{Corollary}
\newtheorem{rem}[theorem]{Remark}
\newtheorem{claim}[theorem]{Claim}

\newtheorem{conjecture}[theorem]{Conjecture}
\newtheorem{question}[theorem]{Question}

\newtheorem*{nb}{\footnotesize {N.B}}

\newtheorem*{exs}{Examples}
\def\res#1{\text{\tt Res}\left(#1\right)}

\def\p{\partial}
\def\no{\noindent}
\def\io{{\infty}}
\def\inde{\operatorname{index}}
\def\sech{\operatorname{sech}}
\def\coth{\operatorname{coth}}
\def\cosec{\operatorname{csc}}
\def\cosech{\operatorname{csch}}
\def\cotan{\operatorname{cotan}}
\def\arcth{\operatorname{arcth}}

\def\range{\operatorname{ran}}
\def\diag{\operatorname{diag}}

\def\re{\operatorname{Re}}
\def\card{\operatorname{Card}}
\def\spectrum{\operatorname{Spectrum}}
\def\im{\operatorname{Im}}
\def\Id{\operatorname{Id}}
\def\Lg{\operatorname{Log}}
\def\dive{\operatorname{div}}
\def\moo{C^{\io}}
\def\mooc{C^{\io}_{\textit c}}

\def\N{\mathbb N}
\def\Z{\mathbb Z}
\def\Q{\mathbb Q}
\def\R{\mathbb R}
\def\C{\mathbb C}

\def\poscal#1#2{\langle#1,#2\rangle}

\def\poi#1#2{\left\{#1,#2\right\}}
\def\norm#1{\Vert#1\Vert}

\def\val#1{\vert#1\vert}
\def\Val#1{\left\vert#1\right\vert}
\def\valjp#1{\langle#1\rangle}

\def\l2{L^2(\R^{n})}
\def\L2{L^2(\R^{2n})}
\def\supp{\operatorname{supp}}

\def\sign{\operatorname{sign}}

\def\op#1{{\text{Op}(#1)}}
\def\ops#1#2{{\text{Op}_{#1}(#2)}}
\def\OPS#1#2{{\text{Op}_{#1}\bigl(#2\bigr)}}
\def\opw#1{{\text{\rm Op}_{\text{\rm w}}(#1)}}
\def\OPW#1{{\text{\rm Op}_{\text{\rm w}}\bigl(#1\bigr)}}
\def\OPWM#1{{\text{\rm Op}_{\text{\rm w}}\Bigl(#1\Bigr)}}

\def\mett#1#2#3#4{M_{#1,#2,#3}^{\scriptscriptstyle\{#4\}}}
\def\Mett#1#2#3#4{\mathcal M_{#1,#2,#3}^{\scriptscriptstyle\{#4\}}}
\def\RZ{\R^{2n}}

\def\trace{\operatorname{trace}}
\def\sign{\operatorname{sign}}

\def\hs{{\hskip15pt}}
\def\vs{\vskip.3cm}
\def\indic#1{\mathbf 1_{#1}}
\let \dis=\displaystyle
\let\no=\noindent
\let \dis=\displaystyle

\let\no=\noindent
\def\mat22#1#2#3#4{\begin{pmatrix}#1&#2\\ #3&#4\end{pmatrix}}
\def\matdu#1#2{\begin{pmatrix}#1\\ #2\end{pmatrix}}
\def\rank{\operatorname{rank}}
\def\wt#1{\widetilde{#1}}

\def\XXint#1#2#3{{\setbox0=\hbox{$#1{#2#3}{\int}$}
     \vcenter{\hbox{$#2#3$}}\kern-.5\wd0}}

\def\beq{\begin{equation}}
\def\eeq{\end{equation}}
\def\did{n}

\DeclareMathOperator{\sinc}{sinc}
\DeclareMathOperator{\shc}{shc}
\newcommand{\ai}{\operatorname{\tt A\!i}} 
\makeindex
\begin{document}
\baselineskip=1.15\normalbaselineskip
\title[Integrals of the Wigner Distribution]{Integrating the Wigner Distribution
\\ on subsets of the phase space, a Survey}
\author[Nicolas Lerner]{Nicolas Lerner}
\address{\noindent \textsc{N. Lerner, Institut de Math\'ematiques de Jussieu,
Sorbonne Universit\'e (formerly Paris VI),
Campus Pierre et Marie Curie,
4 Place Jussieu,
75252 Paris cedex 05,
France}}
\email{nicolas.lerner@imj-prg.fr, nicolas.lerner@sorbonne-universite.fr}
\numberwithin{equation}{subsection}
\begin{abstract}
We review several properties of integrals of the Wigner distribution on subsets of the phase space. Along our way, we provide  a theoretical proof of the invalidity of Flandrin's conjecture, a fact already proven via numerical arguments in our joint  paper \cite{DDL}
with B.~Delourme and T.~Duyckaerts. We use also the  J.G.~Wood \& A.J.~Bracken paper \cite{MR2131219},
for which we offer a mathematical perspective.
We review thoroughly the case of subsets of the plane  whose boundary is a conic curve and show that Mehler's formula can be helpful in the analysis of these cases, including for the higher dimensional case investigated in the paper
\cite{MR2761287} by
E.~Lieb and Y.~Ostrover.
Using the Feichtinger algebra,
we show that, generically in the Baire sense, the Wigner distribution of a pulse in $L^2(\R^n)$
does not belong to $L^1(\RZ)$, providing as a byproduct a large class of examples of subsets  of the phase space $\RZ$ on which the integral of the Wigner distribution is infinite.
We study as well the case of convex polygons of the plane, with a rather weak estimate depending on the number of vertices, but independent of the area of the polygon.
\end{abstract}
\maketitle
\centerline{\color{magenta}\boxed{\text{\bf \today,\ \currenttime}}}
{\small\tableofcontents}
\section*{Foreword}
As indicated by the title of this article, this paper is a survey of properties of integrals of the Wigner distribution on subsets of the phase space. Since it is quite lengthy, we wish in this foreword to describe the content of this paper, browsing through the table of contents, expecting  that the reader will find some organization with the way this article is written. In particular, we shall point here what is original in our survey (to the best of our knowledge) and what was well-known beforehand. There is no doubt that the fifty-five articles quoted in the references list are a small part of the literature on the topic and could be probably extended tenfold: we expect nevertheless that our choice of references will be enough to cover the most important contributions.
\par
Section 1 is {\bf{Preliminaries and Definitions}} and is very classical. We have used 
J.~Leray's book \cite{MR644633}
and other Lecture Notes of this author at the {\sl Coll\`ege de France} such as \cite{MR0501198},
L.~H\"ormander's
four-volume treatise, {\sl The Analysis of Linear Partial Differential operators} and in particular Volume III,
as well as {K.~Gr\"{o}chenig's }\cite{MR1843717} {\sl Foundations of time-frequency analysis},
along with G.B.~Folland's \cite{MR983366}, 
A.~Unterberger's \cite{MR552965} and N.~Lerner's \cite{MR2599384}.
Some details are given, in particular on positive quantizations, but that section is far from being self-contained, which is probably unavoidable: the link of properties of the Wigner distribution and of the Weyl quantization of classical Hamiltonians is easy to obtain but turns out to be an important piece of information for our purpose.
\par
Section 2 is stressing the link {\bf{Quantization of radial functions - Mehler's formula}} and is also very classical: here also the link aforementioned is easy to get but gives some simplifications  in the formulas providing the quantization of radial Hamiltonians: in one dimension for the configuration space  (phase space $\R^{2}$), we are reduced to check simple integrals related to the Laguerre polynomials, following P.~Flandrin's method in his 1988 article \cite{conjecture}.
\par
Section 3 is dealing with {\bf{Conics with eccentricity $\mathbf {<1}$.}} 
The result for the disc in $\R^{2}$ is due to P.~Flandrin and the result for the Euclidean ball in $\RZ$ to E.~Lieb \& Y.~Ostrover in \cite{MR2761287}. Using Mehler's formula simplifies a little bit the presentation, but leaves open the case of anisotropic ellipsoids for which we formulate a conjecture.
\par
Section 4 is dealing with {\bf{Epigraphs of Parabolas.}} 
The results obtained in that section follow easily from Section 3 but nevertheless the precise diagonalization proven there   seems to be new.
We formulate also a conjecture on anisotropic paraboloids which is  closely related to the conjecture in Section 3.
\par
Section 5 is concerned with {\bf{Conics with eccentricity $\mathbf{>1}$.}} 
Many of the results in that section are contained in the paper
\cite{MR2131219}
by J.G. Wood and A.J. Bracken; however since the latter article contains some formal calculations, using for instance test functions which do not belong to $L^{2}(\R)$, we have made a mathematically sound presentation. As certainly the most important contribution of this work, we provide a ``theoretical''
disproof of Flandrin's conjecture on integrals of the Wigner distribution on convex subsets of the phase space: we find in particular some $a>0$ and  some function $u\in L^{2}(\R)$ with norm 1 such that
$$
\iint_{[0,a]^{2}} \mathcal W(u,u)(x,\xi) dx d\xi>1,
$$where $\mathcal W(u,u)$ is the  Wigner distribution of $u$.
This fact was already proven in 
our joint  paper \cite{DDL}
with B.~Delourme and T.~Duyckaerts, using a rigorous numerical argument.
\par
Section 6 is entitled {\bf{Unboundedness is Baire generic}} and most of its content
is included in Chapter 12 of K.~Gr\"ochenig's book \cite{MR1843717},
\emph{Foundations of time-frequency analysis}.
Using the Feichtinger algebra,
we show that, generically in the Baire sense, the Wigner distribution of a pulse in $L^2(\R^n)$
does not belong to $L^1(\RZ)$, providing as a byproduct a large class of examples of subsets  of the phase space $\RZ$ on which the integral of the Wigner distribution is infinite.
We raise a couple of questions, in particular whether we can find a pulse $u\in L^{2}(\R^{n})$
such that 
\begin{equation*}
E_{+}(u)=\{(x,\xi)\in \RZ, \mathcal W(u,u)(x,\xi)>0  \}\quad\text{is connected.}
\end{equation*}
\par
Section 7 is  {\bf{Convex polygons in the plane}}: we study there the sets defined by the intersection of $N$ half-spaces in the plane $\R^{2}$ and the integrals of the Wigner distribution on these sets. We start with convex cones ($N=2$) for which a complete result is known and we go on with triangles ($N=3$) for which we find an upper bound: the integral of $\mathcal W(u,u)$ on a triangle of $\R^{2}$ for a normalized pulse in $L^{2}(\R)$ is bounded above by a universal constant. We show also that 
 the integral of $\mathcal W(u,u)$ on a convex polygon with $N$ sides of $\R^{2}$ for a normalized pulse in $L^{2}(\R)$ is bounded above by a universal constant $\times \sqrt N$.
 We raise a couple of questions: in particular it seems possible that the behaviour of convex subsets of the  plane is such that there exists a constant $\alpha>1$ such that
 \begin{multline*}
 \forall C\text{ convex subset of the plane $\R^{2}$},\quad 
 \forall u\in L^{2}(\R) \text{ with $\norm{u}_{L^{2}(\R)}=1$,} 
 \\\text{we have\quad }
 \iint_{C}\mathcal W(u,u)(x,\xi) dx d\xi\le \alpha.
\end{multline*}
That would be a weak version of Flandrin's conjecture: the original Flandrin's conjecture was the above statement with $\alpha=1$, which is untrue, but that does not rule out the existence of a number $\alpha>1$ such that the above estimate holds true.
\par
Section 8 is entitled  {\bf{Open questions and Conjectures}}: we review in that section the various conjectures that we meet along the text of the article, estimating the importance and difficulty of the various questions. Section 9 is an Appendix containing only classical material, hopefully helping the reader by improving the self-containedness of this paper.
\section{Preliminaries \& Definitions}
\subsection{The Wigner Distribution}
\index{Wigner distribution}
\index{{~\bf Notations}!$\Omega(u,v)$}
Let $u,v$ be given functions in $L^2(\R^n)$.
The function $\Omega$, defined on $\R^n\times \R^n$ by
\begin{equation}\label{funcome}
\R^{n}\times \R^{n}\ni(z,x)\mapsto u(x+\frac z2) \bar v(x-\frac z2) =\Omega(u,v)(x,z),
\end{equation}
belongs to $L^{2}(\RZ)$ from the identity
\begin{equation}\label{newide}
\int_{\RZ}\val{\Omega(u,v)(x,z)}^{2} dxdz=\norm{u}^{2}_{L^{2}(\R^{n})}\norm{v}^{2}_{L^{2}(\R^{n})}.
\end{equation}
We have also
\begin{equation}\label{113uhb}
\sup_{x\in \R^n}\int_{\R^n}\val{\Omega(x,z)} dz\le 2^n \norm{u}_{L^{2}(\R^{n})}\norm{v}_{L^{2}(\R^{n})}.
\end{equation}
We may then give the following definition.
\begin{defi}
 Let $u,v$ be given functions in $L^2(\R^n)$. We define the joint Wigner distribution
 $\mathcal W(u,v)$ 
as the partial Fourier transform\footnote{\label{foot1}For $f\in \mathscr S(\R^{N})$,
we define its Fourier transform by 
$
\hat f(\xi)=\int_{\R^{N}} e^{-2i\pi x\cdot \xi} f(x) dx
$ and we obtain the inversion formula
$ f(x)=\int_{\R^{N}} e^{2i\pi x\cdot \xi} \hat f(\xi) d\xi$. Both formulas can be extended to tempered distributions:
for $T\in \mathscr S'(\R^{N})$, we define the tempered distribution  $\hat T$ by 
\begin{equation}\label{res444}
\poscal{\hat T}{\phi}_{ \mathscr S'(\R^{N}), \mathscr S(\R^{N})}=
\poscal{T}{\hat \phi}_{ \mathscr S'(\R^{N}), \mathscr S(\R^{N})}.
\end{equation}
Note also that with this normalization, it is natural to introduce the operators $ D_{x}^\alpha$ defined for $\alpha\in \N^N$
 by 
\begin{equation}\label{deuxipi}
D_{x}^\alpha u=D_{x_{1}}^{\alpha_{1}}\dots D_{x_{N}}^{\alpha_{n}} u,\
 D_{x_{j}}u=\frac{\partial u}{2i\pi\p x_{j}},\text{ so that }\widehat{D_{x}^\alpha u}=\xi^\alpha \hat u(\xi), \quad \text{with $\xi^\alpha=\xi_{1}^{\alpha_{1}}\dots \xi_{N}^{\alpha_{N}}$}.
\end{equation}
\label{fn:first}
} with respect to $z$ of the function $\Omega$ defined in
\eqref{funcome}.
We have for  $(x,\xi)\in\R^n_{x}\times \R^n_{\xi}$, using \eqref{113uhb},
\index{{~\bf Notations}!$\mathcal W(u,v)$}
\begin{equation}\label{wigner}
\mathcal W(u,v)(x,\xi)=\int_{\R^n} e^{-2i\pi z\cdot \xi} u(x+\frac z2) \bar v(x-\frac z2) dz.
\end{equation}
The Wigner distribution of $u$ is defined as $\mathcal W(u,u)$.
\end{defi}
\begin{nb}
 By inverse Fourier transformation we get, in a weak sense,
 \index{reconstruction formula}
  \begin{equation}\label{623new}
u(x_{1})\otimes \bar v(x_{2})=\int \mathcal W(u,v)(\frac{x_{1}+x_{2}}2, \xi)\ e^{2i\pi(x_{1}-x_{2})\cdot \xi}d\xi.
\end{equation}
\end{nb}
\begin{lem}
 Let $u,v$ be given functions in $L^2(\R^n)$.  The function 
 $\mathcal W(u,v)$  belongs to $L^2(\RZ)$ and we have
 \begin{equation}\label{norm}
\norm{\mathcal  W(u,v)}_{L^{2}(\RZ)}=\norm{u}_{L^{2}(\R^{n})}\norm{v}_{L^{2}(\R^{n})}.
\end{equation} 
We have  also
\begin{equation}\label{complex}
\overline{\mathcal W (u,v)(x,\xi)}=\mathcal W (v,u)(x,\xi),
\end{equation}
so that 
 $\mathcal W(u,u)$ is real-valued.
\end{lem}
\begin{proof}
Note that the function $\mathcal W(u,v)$ is in $L^{2}(\RZ)$ and satisfies
\eqref{norm}
from \eqref{newide} and the definition of $\mathcal W$ as the partial Fourier transform of $\Omega$.
Property \eqref{complex} is immediate and entails that $\mathcal W(u,u)$ is real-valued.
\end{proof}
\begin{rem}\label{rem13}
\rm
We note also that the real-valued function $\mathcal W (u,u)$ can take negative values,
choosing for instance $u_{1}(x)=x e^{-\pi x^{2}}$ on the real line, we get
$$
\mathcal W (u_{1},u_{1})(x,\xi)=2^{1/2} e^{-2\pi(x^{2}+\xi^{2})}\bigl(x^{2}+\xi^{2}-\frac1{4\pi}\bigr).
$$
In fact the real-valued function $\mathcal W (u,u)$ will take negative values unless $u$ is a Gaussian function, thanks to a Theorem due to E. Lieb (see \cite{MR1039210}
and the books \cite{MR1843717} and \cite{MR1614527}).
As a matter of fact, this range of $\mathcal W(u,u)$ intersecting $\R_{-}$ for most ``pulses'' $u$ in $L^2(\R^n)$
makes rather weird
the qualification of $\mathcal W(u,u)$ as a ``quasi-probability''
(anyhow the emphasis must be on {\it quasi}, not on {\it probability}).
\end{rem}
\begin{rem}\rm
We have also by Fourier inversion formula, say for $u\in \mathscr S(\R^{n})$,
\begin{equation}\label{flat}
u(x+\frac z2) \bar u(x-\frac z2) =\Omega(x,z)=\int \mathcal W(u,u)(x,\xi) e^{2i\pi z\cdot \xi} d\xi,
\end{equation}
so that, with $z=2x=y$,
we get the Reconstruction Formula,
\begin{equation}\label{rec}
u(y)\bar u(0)=\int \mathcal W(u,u)(\frac y2,\xi) e^{2i\pi y\cdot \xi} d\xi,
\end{equation}
as well as
\begin{equation}\label{hg55ss}
\val{u(x)}^{2}=\int \mathcal W(u,u)(x,\xi) d\xi,
\quad
\val{\hat u(\xi)}^{2}=\int \mathcal W(u,u)(x,\xi) dx,
\end{equation}
the former formula following from \eqref{flat} and the latter from 
\begin{multline}
\int \mathcal W(u,u)(x,\xi) dx=\iint e^{-2i\pi z\xi} u(x+\frac z2)\bar u(x-\frac z2) dzdx
\\=\iint e^{-2i\pi \xi(x_{1}-x_{2})} u(x_{1})\bar u(x_{2}) dx_{1}dx_{2}=\val{\hat u(\xi)}^{2}.
\end{multline}
\end{rem}
\begin{lem}\label{lem.even}
 Let $u$ be a function in $L^2(\R^n)$ which is even or odd. Then $\mathcal W(u,u)$ is an even function.
\end{lem}
\begin{proof}
Using the notation 
\begin{equation}\label{check5}
\check u(x)= u(-x),
\end{equation}
we check
\begin{align*}
\mathcal W(u,v)&(-x,-\xi)=\int_{\R^n} e^{2i\pi z\cdot \xi} u(-x+\frac z2) \bar v(-x-\frac z2) dz
\\&=\int_{\R^n} e^{2i\pi z\cdot \xi} \check u(x-\frac z2) \check{\bar v}(x+\frac z2) dz
=
\int_{\R^n} e^{-2i\pi z\cdot \xi} \check u(x+\frac z2) \bar{\check{ v}}(x-\frac z2) dz\\
&=\mathcal W(\check u, \check v)(x,\xi),  
\end{align*}
so that if $\check u=\pm u$, we get 
$\mathcal W(u,u)(-x,-\xi)=\mathcal W(u,u)(x,\xi)$.
\end{proof}
\begin{nb}\rm
 This lemma is a very particular case of the symplectic covariance property displayed below in 
 \eqref{segal+}.
\end{nb}
\begin{nb}\rm
In 
Part 1 of Volume IV in the collected works
\cite{MR1635991} of Eugene P. Wigner, we find 
the first occurrence of  what will be called later on the \emph{Wigner distribution} along with a physicist point of view.
 \end{nb}
It turns out that most of the properties of the Wigner distribution (in particular Lemma \ref{lem.even}) are inherited
from its links with the Weyl quantization introduced by  H. Weyl in 1926 in the first edition of \cite{MR0450450} and our next remarks are devised to stress that link.
\subsection{Weyl quantization, Composition formulas, Positive quantizations}
\subsubsection{Weyl quantization}\label{sec.weyl}
\index{Weyl quantization}
The main goal of Hermann Weyl in his seminal paper \cite{MR0450450} was to give a simple formula, also providing symplectic covariance, ensuring that  real-valued Hamiltonians $a(x,\xi)$ get quantized by formally self-adjoint operators.  The standard way of dealing with differential operators does not achieve that goal since for instance the standard quantization of the Hamiltonian
$x\xi$ (indeed real-valued) is the operator $xD_{x}$, which is not symmetric ($D_{x}$ is defined in \eqref{deuxipi}); H.~Weyl's choice in that case was
$$
x\xi\text{ \it should be quantized by  the operator\quad} \frac{1}{2}(xD_{x}+D_{x}x),
\quad\text{\it (indeed symmetric),}$$
and more generally, say for 
$a\in \mathscr S(\RZ), u\in \mathscr S(\R^n)$, the quantization of the Hamiltonian $a(x,\xi)$, denoted by $\opw{a}$, 
should be 
given by the formula
\index{{~\bf Notations}!$\opw{a}$}
\begin{equation}
(\opw{a} u)(x)=\iint e^{2i\pi (x-y)\cdot \xi} a\bigl(\frac{x+y}2, \xi\bigr) u(y) dy d\xi.
\end{equation}
For $v\in\mathscr S(\R^n)$, we may consider
\begin{multline*}
\poscal{\opw{a} u}{v}_{L^2(\R^n)}
=\iiint a(x,\xi)e^{-2i\pi z\cdot \xi} 
u(x+\frac z2)\bar v(x-\frac z2) dzdxd\xi 
\\
=\iint_{\R^n\times \R^n} a(x,\xi) \mathcal W(u,v)(x,\xi) dx d\xi,
\end{multline*}
and the latter formula allows us to give the following definition.
\begin{defi}
Let $a\in \mathscr S'(\RZ)$. We define the Weyl quantization $\opw{a}$ of the Hamiltonian $a$, by the formula
\index{Weyl quantization}
\begin{equation}\label{weylq}
(\opw{a}u)(x)=\iint e^{2i\pi (x-y)\cdot \xi} a\bigl(\frac{x+y}2, \xi\bigr) u(y) dy d\xi,
\end{equation}
to be understood weakly as 
\begin{equation}\label{eza654}
\poscal{\opw{a}u}{\bar v}_{\mathscr S'(\R^{n}), \mathscr S(\R^{n})}=\poscal{a}{\mathcal W(u,v)}_{\mathscr S'(\RZ), \mathscr S(\RZ)}.
\end{equation}
\end{defi}
We note that the sesquilinear  mapping 
$$\mathscr S(\R^{n})\times \mathscr S(\R^{n})\ni(u,v)\mapsto \mathcal W(u,v) \in \mathscr S(\RZ),$$
is continuous so that the above bracket of duality $
\poscal{a}{\mathcal W(u,v)}_{\mathscr S'(\R^{2n}), \mathscr S(\R^{2n})}
$
makes sense.
We note as well that a temperate distribution $a\in\mathscr S'(\RZ)$ gets quantized
by a continuous  operator $\opw a$ from 
$\mathscr S(\R^{n})$ 
into
$\mathscr S'(\R^{n})$.
This very general framework  is not really useful since we want to compose our operators $\opw{a}\opw{b}$.
A first step in this direction is to look for sufficient conditions
ensuring that the operator $\opw{a}$ is bounded on $L^{2}(\R^{n})$. 
Moreover, for $a\in \mathscr S'(\RZ)$ and $b$ a polynomial in $\C[x,\xi]$,
we have the composition formula,
\index{{~\bf Notations}!$a\sharp b$}
\index{composition formula}
\begin{align}
\opw{a}\opw{b}&=\opw{a\sharp b},\label{gfcd44}
\\
(a\sharp b)(x,\xi)&= \sum_{k\ge 0}\frac1{(4i\pi)^{k}}\sum_{\val \alpha+\val \beta=k}
\frac{(-1)^{\val \beta}}{\alpha!\beta!}(\p_{\xi}^{\alpha}\p_{x}^{\beta}a)(x,\xi)
(\p_{x}^{\alpha}\p_{\xi}^{\beta}b)(x,\xi),\label{gfcd44+}
\end{align}
which involves here a finite sum.
This follows from (2.1.26) in \cite{MR2599384}
where several generalizations can be found (see in particular in that reference the integral formula (2.1.18) which can be given a meaning for quite general classes of symbols).
As a consequence of \eqref{gfcd44+}, we get that 
\index{Poisson bracket}
\index{{~\bf Notations}!$\poi{a}{b}$}
\begin{align}
&(a\sharp b)=\sum_{k\ge 0}\omega_{k}(a,b), \quad \omega_{0}(a,b)=ab, \quad \omega_{1}(a,b)=\frac1{4i\pi}
\poi{a}{b},
\label{gfl227}\\
&\poi{a}{b}=\p_{\xi}a\cdot \p_{x}b-\p_{x}a\p_{\xi}b,
\end{align}
where $\poi{a}{b}$ is called the \emph{Poisson bracket} of $a$ and $b$.
\begin{pro}\label{pro1717}
Let $a$ be a tempered distribution on $\RZ$. Then we have 
 \begin{equation}\label{norm01}
\norm{\opw{a}}_{\mathcal B(L^{2}(\R^{n}))}\le \min\bigl(2^{n}\norm{a}_{L^{1}(\RZ)}, \norm{\hat a}_{L^{1}(\RZ)}\bigr).
\end{equation}
\end{pro}
\begin{proof}
In fact we have from \eqref{eza654}, $u,v\in \mathscr S(\R^{n})$, 
$$
\poscal{\opw{a}u}{ v}_{L^{2}(\R^{n})}=\iiint a(x,\xi) u(2x-y)\bar v(y) e^{-4i\pi(x-y)\cdot \xi} 2^{n}dy dxd\xi,
$$
so that defining for $(x,\xi)\in \RZ$ the operator $\sigma_{x,\xi}$ by
\index{phase symmetry}
\index{Grossman-Royer operator}
\index{{~\bf Notations}!$\sigma_{x,\xi}$}
\begin{equation}\label{phase}
(\sigma_{x,\xi} u)(y)=u(2x-y)e^{-4i\pi(x-y)\cdot \xi}, 
\end{equation}
we see that the operator $\sigma_{x,\xi}$ (called \emph{phase symmetry}, also known as the \emph{Grossman-Royer operator}) is unitary and self-adjoint\footnote{\label{foot2}Indeed  we have\begin{multline}\label{unit22}
(\sigma_{x,\xi}^2 u)(y)=
(\sigma_{x,\xi}u)(2x-y)e^{-4i\pi(x-y)\cdot \xi}
\\
= u(2x-(2x-y)) e^{-4i\pi(x-(2x-y))\cdot \xi}e^{-4i\pi(x-y)\cdot \xi}
= u(y), \ \text{so that}\ \sigma_{x,\xi}^2=\Id.
\end{multline}
We have 
$
\poscal{\sigma_{x,\xi}^*u}{v}_{L^2(\R^n)}=\poscal{u}{\sigma_{x,\xi}v}_{L^2(\R^n)}
$
$=\overline{\mathcal W(v,u)(x,\xi)}={W(u,v)(x,\xi)}
={\poscal{\sigma_{x,\xi}u}{v}_{L^2(\R^n)}},
$
proving that $\sigma_{x,\xi}^*=\sigma_{x,\xi}$.
} and 
\begin{equation}\label{12hh}
\opw{a}=2^{n}\iint a(x,\xi) \sigma_{x,\xi} dxd\xi,
\end{equation}
proving the first estimate of the proposition. 
As a consequence of \eqref{12hh},
we obtain that 
\begin{equation}\label{}
\left(\opw{a}\right)^{*}=\opw{\bar a},\quad\text{so that for $a$ real-valued, $(\opw{a})^{*}=\opw{a}$.}
\end{equation}
\index{ambiguity function}
\index{{~\bf Notations}!$\mathcal A(u,v)$}
To prove the second estimate, we introduce the so-called ambiguity function $\mathcal A(u,v)$ as the inverse Fourier transform of the Wigner function $\mathcal W(u,v)$, so that for $u,v$ in the Schwartz class, we have
$$
(\mathcal A(u,v))(\eta, y)=\iint \mathcal W(u,v)(x,\xi) e^{2i\pi(x\cdot \eta+\xi\cdot y)} dx d\xi,
$$
i.e.
\begin{equation}\label{1210}
(\mathcal A(u,v))(\eta, y)=\int u(x+\frac{y}2)
\bar v(x-\frac{y}2) e^{2i\pi x\cdot \eta} dx,
\end{equation}
which reads as well as
\begin{multline}\label{1211++}
(\mathcal A(u,v))(\eta, y)=\int u(\frac{y}2+\frac z2)
\
\overline{ \check v}(\frac{y}2-\frac z2) \ e^{2i\pi z\cdot \frac\eta 2} dz\ 2^{-n}
\\=\mathcal W(u, \check v)(\frac y2,- \frac \eta 2) 2^{-n}.
\end{multline}
\begin{nb}\rm
The ambiguity function is called the Fourier-Wigner transform in G.B.~Folland's book
\cite{MR983366}.
\end{nb}
\begin{rem}\label{rem.334455}\rm
 With $\Omega(u,v)$ defined by \eqref{funcome}, we have
  \begin{equation}\label{65rtfg} \mathcal W(u,v)=\mathcal F_{2}\bigl(\Omega(u,v)\bigr),
\end{equation}
 where $\mathcal F_{2}$ stands for 
 the Fourier transformation with respect to the second variable. Taking the  Fourier transform with respect to the second variable  in the previous formula gives, with $\mathcal F_j$ (resp. $\mathcal F$)
 standing for the Fourier transform with respect to the $j^{\text{th}}$ variable
 (resp. all variables),
  $$
\mathcal F_{2}{\mathcal W}=\mathcal C_{2}{\Omega},\quad 
\mathcal F{\mathcal W}=\mathcal F_1{\mathcal C_{2}{\Omega}},
\quad
\mathcal A=\mathcal C\mathcal F{\mathcal W}=\mathcal F_{1}{\mathcal C_{1}{\Omega}},
$$
 where $\mathcal C$
 (resp. $\mathcal C_{1}$ or $\mathcal C_{2}$)
  stands for the ``check'' operator $\mathcal C$ in $\R^{n}\times \R^{n}$
  given by \eqref{check5}
 (resp. with respect to the first or second variable),
 the latter formula being
\eqref{1210}.
\end{rem}
Applying Plancherel formula on \eqref{eza654},
we get
\begin{equation}\label{ambigu}
\poscal{\opw{a}u}{v}_{L^{2}(\R^{n})}=\poscal{\hat a}{\mathcal A(u,v)}_{\mathscr S'(\RZ), \mathscr S(\RZ)}.
\end{equation}
We note that a consequence of \eqref{gfcd44+} is that for a linear form $L(x,\xi)$,
we have
$$
L\sharp L=L^{2},\quad \text{and more generally}\quad L^{\sharp N}=L^{N}.
$$
As a result, considering for $(y,\eta)\in \RZ$, the linear form $L_{\eta, y}$ defined by
\begin{equation}
L_{\eta, y}(x,\xi)=x\cdot \eta+\xi\cdot y,
\end{equation}
we see that 
\begin{equation}\label{}
\mathcal A(u,v)(\eta, y)=\poscal{
\opw{e^{2i\pi (x\cdot \eta+\xi\cdot y)}
}u}{v}_{L^{2}(\R^{n})},
\end{equation}
and thus we get Hermann Weyl's original formula 
\begin{equation}\label{}
\opw{a}=\iint \widehat a(\eta, y) e^{i\opw{L_{\eta,y}}} dyd\eta,
\end{equation}
which implies the second estimate in the proposition.
\end{proof} 
\index{distribution kernel}
\begin{pro}\label{pro1716}
Let $a\in \mathscr S'(\RZ)$. The distribution kernel $k_{a}(x,y)$ of the operator $\opw{a}$ is
\begin{equation}\label{reteey}
k_{a}(x,y)=\widehat a^{[2]}(\frac{x+y}2, y-x),
\end{equation}
where $a^{[2]}$ stands for the Fourier transform of $a$ with respect to the second variable.
Let $k\in \mathscr S'(\RZ)$ be the distribution kernel of a continuous operator $A$ from $\mathscr S(\R^n)$ into   $\mathscr S'(\R^n)$. Then the Weyl 
symbol $a$
of $A$ is 
\begin{equation}\label{}
a(x,\xi)= \int e^{-2\pi i t\cdot \xi}
k(x+\frac t2,x-\frac t2) dt,
\end{equation}
where the integral sign means  that we take the Fourier transform with respect to $t$
of the distribution $k(x+\frac t2,x-\frac t2)$ on $\RZ$
(see \eqref{res444} in  footnote \ref{foot1}  for the definition of the Fourier transformation on tempered distributions). 
\end{pro}
\begin{proof}
 With $u,v\in \mathscr S(\R^n)$, we have defined $\opw{a}$ via Formula \eqref{eza654} 
 and using Remark \ref{rem.334455}, we get
 \begin{align*}
 \poscal{\opw{a}u}{\bar v}_{\mathscr S'(\R^{n}), \mathscr S(\R^{n})}&=\poscal{a(x, \xi)}{
\widehat{\Omega}^{[2]}(x,\xi)
 }_{\mathscr S'(\RZ), \mathscr S(\RZ)}
\\&=
\poscal{\widehat a^{[2]}(t, z)}{
u(t+\frac z2) \bar v(t-\frac z2) 
 }_{\mathscr S'(\RZ), \mathscr S(\RZ)}
 \\&=
\poscal{\widehat a^{[2]}(\frac{x+y}2, y-x)}{
u(y) \bar v(x) 
 }_{\mathscr S'(\RZ), \mathscr S(\RZ)},
\end{align*}
proving \eqref{reteey}. As a consequence, we find that 
$
k_{a}(x+\frac t2,x-\frac t2)={\widehat a^{[2]}}(x,-t),
$
and by Fourier inversion, this entails
\begin{equation}\label{symker}
a(x,\xi)=\mathcal F\text{ourier}_{t}\bigl(k_{a}(x+\frac t2,x-\frac t2)\bigr)(\xi)=\int e^{-2\pi i t\cdot \xi}
k_{a}(x+\frac t2,x-\frac t2) dt,
\end{equation}
where the integral sign means  that we perform a Fourier transformation with respect to the variable $t$.
 \end{proof}
 \index{Segal formula}
A particular case of Segal's formula (see e.g. Theorem 2.1.2 in \cite{MR2599384})
is  with $F$ standing for the Fourier transformation on $\R^{n}$,
\begin{equation}\label{fs55}
F^{*}\opw{a}F=\opw{a(\xi,-x)}.
\end{equation}
\index{{~\bf Notations}!$[X,Y]$}
\subsubsection{The symplectic group}
We define the canonical symplectic form $\sigma$
on $\R^{n}\times \R^{n}$ by
\begin{equation}\label{sympfor}
\poscal{\sigma X}{Y}=\bigl[X,Y\bigr]=\xi\cdot y-\eta\cdot x, \quad \text{with $X=(x,\xi), Y=(y,\eta)$.}
\end{equation}
\index{symplectic group}
\index{{~\bf Notations}!$Sp(n,\R)$}
\index{generators of the symplectic group}
\no
The symplectic group\footnote{This is obviously a group
since for $S_{1}, S_{2}\in  Sp(n,\R)$, the last equation in \eqref{55vhqs} implies that $\val{\det S}=1$ and 
$
[S_{1}S_{2}^{-1}X, S_{1}S_{2}^{-1}Y]=[S_{2}^{-1}X, S_{2}^{-1}Y]=[X,Y],
$
since 
$
[S_{2}^{-1}X, S_{2}^{-1}Y]=[S_{2}S_{2}^{-1}X, S_{2}S_{2}^{-1}Y]=[X,Y].
$
We shall prove below that the determinant of a symplectic mapping is actually 1.
} 
$\text{\sl Sp}(n,\R)$ is the subgroup of $S\in \text{\sl Gl}(2n,\R)$
such that 
\begin{equation}\label{55vhqs}
\forall X, Y\in \RZ, \quad [S X, S Y]=[X,Y], \qquad\text{i.e.\quad} S^{*}\sigma S=\sigma,
\end{equation}
where $S^{*}$ is the transpose and 
\begin{equation}\label{ffqq99}
\sigma=\mat22{0}{I_{n}}{-I_{n}}{0}.
\end{equation}
It is easy to prove directly from \eqref{55vhqs} that
$
Sp(1,\R)=Sl(2, \R).
$
\index{generators of the symplectic group}
\index{{~\bf Notations}!$\Xi_{A,B,C}$}
\begin{theorem}\label{4.thm.gensym}
Let $n$ be an integer $\ge 1$.
The group
$Sp(n,\R)$ is included in $Sl(2n, \R)$
and generated by the following mappings
\begin{align}
\mat22{I_{n}}{0}{A}{I_{n}}&,\ \text{where $A$ is a $n\times n$ symmetric matrix,}
\label{syty01}\\
\mat22{B^{-1}}{0}{0}{B^*}&,\quad B\in Gl(n,\R),
\label{syty02}\\
\mat22{I_{n}}{-C}{0}{I_{n}}&,\ \text{where $C$ is a $n\times n$ symmetric matrix.}
\label{syty03}
\end{align} 
For $A, B, C$ as above, the mapping
\begin{multline}\label{gensym}
\Xi_{A,B,C}=\mat22{B^{-1}}{-B^{-1}C}{AB^{-1}}{B^*-AB^{-1}C}
\\=
\mat22{I_{n}}{0}{A}{I_{n}}
\mat22{B^{-1}}{0}{0}{B^*}
\mat22{I_{n}}{-C}{0}{I_{n}}.
\end{multline}
belongs to $Sp(n,\R)$.
\index{generating function}
Moreover, we define on $\R^{n}\times \R^{n}$ the {\it generating function} $S$
of the symplectic mapping
$\Xi_{A,B,C}$
by the identity
\begin{equation}\label{4.genfun}
S(x,\eta)=\frac12\bigl(\poscal{Ax}{x}+2\poscal{Bx}{\eta}
+\poscal{C\eta}{\eta}
\bigr)
\
\text{so that}
\hs
\Xi \Bigl(\frac{\p S}{\p\eta}\oplus \eta\Bigr)=
x\oplus\frac{\p S}{\p x}.
\end{equation}
For a symplectic mapping $\Xi$, to be of the form \eqref{gensym}
is equivalent
to the assumption that the mapping
$x\mapsto \pi_{\R^{n}\times\{0\}}\Xi(x\oplus 0)$
is invertible from $\R^{n}$ to $\R^{n}$;
moreover, if this mapping is not invertible,
the symplectic mapping $\Xi$ is the product of two mappings of the type
$\Xi_{A,B,C}$.
\end{theorem}
\begin{proof}
The expression of $\Xi$ above
as well as \eqref{4.genfun}
follow from a simple direct computation
left to the reader.
The inclusion of the symplectic group in the special linear group
follows from the statement
on the generators.
We consider now $\Xi$ in $Sp(n,\R)$: we have
\begin{equation}\label{notsym}
\Xi=\mat22{P}{Q}{R}{S},\quad\text{where
$P,
Q,
R,
S,$
are $n\times n$
matrices.
}
\end{equation}

The equation
\begin{equation}\label{4.eqsygr}
{\Xi^{*}} \sigma\Xi=\sigma
\end{equation} is satisfied
with 
$
\sigma=\mat22{0}{I_{n}}{-I_{n}}{0}, 
$
which means
\begin{equation}\label{4.eqsyma}
{P^{*}} R={({P^{*}} R)}^{*},\quad
{Q^{*}} S={({Q^{*} S})}^{*},\quad
{P^{*}} S-{R^{*}} Q=I_{n}.
\end{equation}
We can note also that the mapping $\Xi\mapsto \Xi^{*}$ is an isomorphism
of $Sp(n,\R)$
since 
$\Xi\in Sp(n,\R)$ means
$$
\Xi^{*} \sigma \Xi=\sigma\Longrightarrow
\Xi^{-1} \sigma^{-1} (\Xi^{*})^{-1}=\sigma^{-1}
\Longrightarrow
\Xi^{-1} (-\sigma^{-1}) (\Xi^{*})^{-1}=(-\sigma^{-1}),
$$
and since
$
(-\sigma^{-1})=\mat22{0}{I_{n}}{-I_{n}}{0}, 
$
we get that 
$\Xi^{*}\in Sp(n,\R)$.
  As a result, 
  \begin{equation}\label{fc45y9}
  \Xi=\mat22{P}{Q}{R}{S} \in Sp(n,\R),
\end{equation}
 is also equivalent to
  \begin{equation}\label{4.eqsym'}
 P Q^{*}={(P Q^{*})}^{*},\quad
  R S^{*}={(R S^{*})}^{*},\quad
  P S^{*}-Q R^{*}=I_{n}.
 \end{equation}
Let us assume that
 the mapping $P$ is invertible,
 which is the assumption in the last statement of the theorem.
 We define then the mappings
 $A, B, C$ by
 \begin{gather*}
 A= R P^{-1},\quad
 B= P^{-1},\quad C=-P^{-1}Q,
 \\
  \quad\text{\footnotesize so that we have}\quad
  A^*= {P^*}^{-1}R^* P P^{-1}
  ={P^*}^{-1}P^*R P^{-1}=R P^{-1}=A,
  \\
  \quad\text{\footnotesize as well as}\
  C^*= -Q^*{P^*}^{-1}=-P^{-1}P Q^*{P^*}^{-1}
  =-P^{-1}QP^*{P^*}^{-1}=-P^{-1} Q=C,
 \\
 \text{\footnotesize and}\quad
 P=B^{-1},\quad R= A B^{-1},\quad Q= -B^{-1} C,
 \\
 S={P^*}^{-1}(I_{n}+ R^* Q)= B^*(I_{n}-{B^*}^{-1} A^*B^{-1}C)
 =B^*-A B^{-1}C.
 \end{gather*}
We have thus proven that any symplectic matrix
$\Xi$ as above such that 
$P$ is invertible
is indeed given
by the product appearing in Theorem 
\ref{4.thm.gensym}.
\par
Let us now consider the case where a symplectic mapping $\Xi$ (given by \eqref{fc45y9})
is such that $\det P=0$;
writing
$\R^{n}=\ker P\oplus N$
we have that $P$ is an isomorphism from $N$ onto $\range P$.
Let $B_{1}\in Gl(n,\R)$ such that $B_{1}P$ is the identity on $N$\footnote{This is indeed possible: choosing a supplement space $M$ for $P(N)$, we have
$$
\R^{n}=\underbracket[0.3pt]{\ker P}_{\text{dim } p}\oplus \underbracket[0.3pt]{N}_{\text{dim } n-p}=\underbracket[0.3pt]{P(N)}_{\text{dim } n-p}\oplus \underbracket[0.3pt]{M}_{\text{dim } p},
$$
and we can define $B_{1}$ on $P(N)$ by  $B_{1}(Px)=x$
(without ambiguity since for $x_{1}, x_{2}\in N$ with $Px_{1}=Px_{2}$ we get 
$x_{1}-x_{2}\in \ker P\cap N=\{0\}$) and ${B_{1}}_{\vert M}:M\rightarrow \ker P$
can be chosen as an isomorphism, so that 
$
B_{1}(P(N))+B_{1}(M)=N+\ker P,\ \text{which implies }\rank B_{1}=n.
$
}.
We have
\begin{equation}\label{4.k225sa}
\mat22{B_{1}}{0}{0}{{B_{1}^*}^{-1}}\mat22{P}{Q}{R}{S}=
\mat22{B_{1}P}{B_{1}Q}{{B_{1}^*}^{-1}R}{{B_{1}^*}^{-1}S}.
\end{equation}
If $p=\dim(\ker P)$, we have for the $n\times n$ matrix $B_{1}P$ the following block decomposition
\begin{equation}\label{4.k226sa}
B_{1}P=\mat22{0_{p, p}}{0_{p, n-p}}{0_{n-p,p}}{I_{n-p}},
\end{equation}
where $0_{r,s}$ stands for a $r\times s$ matrix with only 0 as an entry.
On the other hand, we know from
\eqref{4.eqsyma}
that the mapping
$$
(B_{1}P)^* {B_{1}^*}^{-1}R=P^* R
$$
is symmetric.
Writing
$
{B_{1}^*}^{-1}R=\mat22{\tilde R_{p, p}}{\tilde R_{p, n-p}}{\tilde R_{n-p,p}}{\tilde R_{n-p, n-p}},
$
where $\tilde R_{r,s}$ stands for a $r\times s$ matrix,
this gives the symmetry of
$$
\mat22{0_{p, p}}{0_{p, n-p}}{0_{n-p,p}}{I_{n-p}}
\mat22{\tilde R_{p, p}}{\tilde R_{p, n-p}}{\tilde R_{n-p,p}}{\tilde R_{n-p, n-p}}
=
\mat22{0_{p, p}}{0_{p, n-p}}{\tilde R_{n-p,p}}{\tilde R_{n-p, n-p}},
$$
implying that
$\tilde R_{n-p,p}=0$.
The symplectic  matrix \eqref{4.k225sa} is thus equal to
$$
\mat22{\mat22{0_{p, p}}{0_{p, n-p}}{0_{n-p,p}}{I_{n-p}}}{B_{1}Q}{
\mat22{\tilde R_{p, p}}{\tilde R_{p, n-p}}{ 0_{n-p,p}}{\tilde R_{n-p, n-p}}
}{{B_{1}^*}^{-1}S},\quad\text{where $B_{1}Q$ and  ${B_{1}^*}^{-1}S$
are $n\times n$ blocks.}
$$
The invertibility of \eqref{4.k225sa} implies that  $\tilde R_{p, p}$ is invertible.
We consider  now the $n\times n$ symmetric matrix
$$
C=\mat22{I_{p, p}}{0_{p, n-p}}{0_{n-p,p}}{0_{n-p, n-p}},
$$
and the symplectic mapping
\begin{equation}\label{4.decomp}
\mat22{I_{n}}{C}{0}{I_{n}}
\mat22{B_{1}}{0}{0}{{B_{1}^*}^{-1}}\mat22{P}{Q}{R}{S}=\mat22{I_{n}}{C}{0}{I_{n}}
\mat22{B_{1}P}{B_{1}Q}{{B_{1}^*}^{-1}R}{{B_{1}^*}^{-1}S},
\end{equation}
which is a symplectic mapping $\mat22{P'}{Q'}{R'}{S'}$
with
\begin{multline*}
P'=B_{1}P+C{B_{1}^*}^{-1}R\\=
\mat22{0_{p, p}}{0_{p, n-p}}{0_{n-p,p}}{I_{n-p}}+
\mat22{I_{p, p}}{0_{p, n-p}}{0_{n-p,p}}{0_{n-p, n-p}}
\mat22{\tilde R_{p, p}}{\tilde R_{p, n-p}}{0_{n-p,p}}{\tilde R_{n-p, n-p}}
\\
=\mat22{\tilde R_{p, p}}{\tilde R_{p, n-p}}{0_{n-p,p}}{\tilde I_{n-p}}
,
\end{multline*}
which is an invertible mapping.
From the equation \eqref{4.decomp}
and the first part of our discussion, we get that
$$\mat22{P'}{Q'}{R'}{S'}=\mat22{I_{n}}{0}{A'}{I_n}
\mat22{B'^{-1}}{0}{0}{B'^*}
\mat22{I_{n}}{-C'}{0}{I_{n}},$$
with $A', C'$ symmetric and $B'$ invertible
and
$$
\Xi=
\mat22{{B_{1}}^{-1}}{0}{0}{{B_{1}^*}}\
\mat22{I_{n}}{-C}{0}{I_{n}}
\mat22{I_{n}}{0}{A'}{I_{n}}
\mat22{B'^{-1}}{0}{0}{B'^*}
\mat22{I_{n}}{-C'}{0}{I_{n}},
$$
proving that
the $\Xi_{A,B,C}$
generate the symplectic group
and more precisely
that every $\Xi$ in the symplectic group is the product of at most two mappings of  type 
$\Xi_{A,B,C}$.
The proof of
Theorem
\ref{4.thm.gensym}
is complete.
\end{proof}
\begin{cor}
 We have $Sp(n,\R)\subset Sl(2n, \R)$.
\end{cor}
\begin{proof}
 Indeed the symplectic mappings \eqref{syty01}, \eqref{syty02} and  \eqref{syty03}
 do have determinants equal to 1 and since Theorem \ref{4.thm.gensym} implies that they generate the symplectic group, this proves
 the sought result.
\end{proof}
\begin{rem}\rm
 Of course for $n\ge 2$, $Sp(n,\R)$ is a proper subgroup of $Sl(2n, \R)$. Indeed
 the following matrix
 $$
  \mathtt M=
  \begin{pmatrix}
\begin{matrix}
1&0\\
0&1
\end{matrix}
&\begin{matrix}
0&0\\
1&0
\end{matrix}\\
\begin{matrix}
0&0\\
0&0
\end{matrix}&
\begin{matrix}
1&0\\
0&1
\end{matrix}
\end{pmatrix}
 $$
 has determinant 1 but fails to be symplectic: using Notation \eqref{notsym}, we see that the first and the third equation are satisfied, which is not the case for  the second one.
\end{rem}
\begin{nb}\rm
 Since the matrix $-I_{2n}$ belongs to $Sp(n,\R)$ (\eqref{55vhqs} holds trivially), we find that $S\in Sp(n,\R)$ is equivalent to 
 $-S\in Sp(n,\R)$.
 \end{nb}
\index{generators of the symplectic group}
\begin{claim}
The symplectic group is also generated by the mappings
\begin{align}
&\text{(i)\quad } (x,\xi)\mapsto (B^{-1}x, B^{*}\xi), \quad B\in \text{\sl Gl}(n,\R),\label{sym1}\\
&\text{(ii)\quad } (x,\xi)\mapsto (\xi,-x), \label{sym2}\\
&\text{(iii)\quad } (x,\xi)\mapsto (x, \xi+Ax), \quad A\in \text{\sl Sym}(n,\R).\label{sym3}
\end{align}
Another set of generators of the symplectic group is given by the mappings
\begin{align}
&\text{(j)\quad } (x,\xi)\mapsto (B^{-1}x, B^{*}\xi), \quad B\in \text{\sl Gl}(n,\R),\label{sym01}\\
&\text{(jj)\quad } (x,\xi)\mapsto (\xi,-x), \label{sym02}\\
&\text{(jjj)\quad } (x,\xi)\mapsto (x-C\xi,\xi), \quad C\in \text{\sl Sym}(n,\R).\label{sym03}
\end{align}
 \end{claim}
 Indeed, we have for $C^{*}=C$ a real symmetric $n\times n$ matrix
 $$
 \underbrace{\mat22{0}{-I_{n}}{I_{n}}{0}}_{\sigma^{-1}}\mat22{I_{n}}{-C}{0}{I_{n}}\underbrace{\mat22{0}{I_{n}}{-I_{n}}{0}}_{\sigma}
 =\mat22{I_{n}}{0}{C}{I_{n}}.
 $$
\begin{rem}\label{rem.549877}\rm
The symplectic matrix
\begin{equation}\label{noti22}
\mat22{0}{I_{n}}{-I_{n}}{0}=2^{-1/2}
\mat22{I_{n}}{I_{n}}{-I_{n}}{I_{n}}2^{-1/2}
\mat22{I_{n}}{I_{n}}{-I_{n}}{I_{n}}=\Xi_{-I_{n}, 2^{1/2}I_{n}, -I_{n}}^{2},
\end{equation}
is not of the form 
$\Xi_{A,B,C}$ but is the square of such a matrix.
It is also the case of all the mappings
$
(x_{k},\xi_{k})\mapsto (\xi_{k},-x_{k})$
with the other coordinates fixed.
Similarly the symplectic  matrix
\begin{equation}\label{}
\mat22{0}{-I_{n}}{I_{n}}{I_{n}}=
\mat22{I_{n}}{-I_{n}}{0}{I_{n}}
\mat22{I_{n}}{0}{I_{n}}{I_{n}},
\end{equation}
is not of the form 
$\Xi_{A,B,C}$ but is the product
$\Xi_{0,I,I}\Xi_{I,I,0}$.
\end{rem}
\subsubsection{The metaplectic group}
\index{{~\bf Notations}!$M_{A,B,C}$}
\begin{pro}\label{pro.kj77qq}
Let $A,B,C$ be as in Theorem \ref{4.thm.gensym},
and let  $S$ be the generating function of $\Xi_{A,B,C}$ (cf. \eqref{4.genfun}).
We define the operator $M_{A,B,C}$ on $\mathscr S(\R^{n})$ by
\begin{equation}\label{4.genmet}
(M_{A,B,C}v)(x)=\int_{\R^{n}}e^{2i\pi S(x,\eta)}\hat v(\eta) d\eta
(\det B)^{1/2},
\end{equation}
where $(\det B)^{1/2}$ is a square root of $\det B$.
This operator is an automorphism of 
$\mathscr S'(\R^{n})$ and of $\mathscr S(\R^{n})$
which is unitary on $L^2(\R^{n})$,
and such that, for all $a\in \mathscr S'(\RZ)$, 
\begin{equation}\label{1239}
M_{A,B,C}^{*} \opw{a}M_{A,B,C}=\opw{a\circ\Xi_{A,B,C}},
\end{equation}
where $\Xi_{A,B,C}$ is defined   in Theorem \ref{4.thm.gensym}.
\end{pro}
\begin{nb}
 We have for $A,B,C$ as above, 
 \begin{align}
 (M_{A, I, 0} v)(x)&= e^{i\pi \poscal{Ax}{x}} v(x),\label{mety01}\\
 (M_{0,B,0} v)(x)&=(\det B)^{1/2} v(Bx),\label{mety02}\\
   (M_{0, I, C} v)(x)&= \bigl(e^{i\pi \poscal{CD_{x}}{D_{x}}} v\bigr)(x),\label{mety03}
\end{align}
three operators which are obviously automorphisms of $\mathscr S(\R^{n})$ and of 
$\mathscr S'(\R^{n})$ as well as unitary operators in $L^{2}(\R^{n})$.
\end{nb}
\begin{proof}
 Formula \eqref{1239} is easily checked  for each operator \eqref{mety01},
 \eqref{mety02}
 and \eqref{mety03}. Since $\Xi_{A,B,C}=\Xi_{A,I,0}\ \Xi_{0,B,0}\ \Xi_{0,I,C}$ and 
 \begin{equation}\label{54gvrs}
 M_{A,B,C}= M_{A,I,0} M_{0,B,0} M_{0,I,C},
\end{equation}
we get \eqref{1239} and the proposition.
\end{proof}
\index{{~\bf Notations}!$m(B)$}
\begin{rem}\rm
We  define
\begin{multline}\label{maslov}
m(B)=\frac{\arg(\det B)}{\pi}
\\=\begin{cases}
\frac{k2\pi}{\pi}=2k\in \{0,2\}\mod 4\quad \text{for } \det B >0,\\
\frac{k2\pi+\pi}{\pi}=2k+1\in \{1,3\}\mod 4\quad \text{for } \det B <0,
\end{cases}
\end{multline}
so that 
\begin{equation}\label{}
\det B=\val{\det B} e^{i\pi m(B)}, \quad (\det B)^{1/2}\in
\val{\det B}^{1/2} \{e^{i\frac{\pi}2 m(B)},
e^{i\frac{\pi}2 (m(B)+2)}
\}.\footnote{This is a synthetic way to write
$$
(\det B)^{1/2}\in \{(\pm 1)\val{\det B}^{1/2}\}\text{ if $\det B>0$},\quad
(\det B)^{1/2}\in \{(\pm i)\val{\det B}^{1/2}\}\text{ if $\det B<0$}.
$$ }
\end{equation}
We shall consider $m(B)$ as an element of $\Z/4\Z$, so that the function $m(B)\mapsto e^{i\frac{\pi}2 m(B)}$ is well-defined. For $A,B,C$ as in Proposition \ref{pro.kj77qq}, we may define
\index{{~\bf Notations}!$\mett{A}{B}{C}{m}$}
\begin{equation}\label{1256hh}
\bigl(\mett{A}{B}{C}{m(B)} v\bigr)(x)=e^{\frac{i\pi m(B)}2}\val{\det B}^{1/2}
\int_{\R^{n}}e^{i\pi(Ax^{2}+2Bx\cdot \eta+C\eta^{2})}\hat v(\eta) d\eta,\footnote{We can of course define 
$\mett{A}{B}{C}{m}$ for any $m$, but to stay in the metaplectic group (cf. Definition \ref{def.41fdhh}), we have to make sure that $m\in\{m(B), m(B)+2\}$ modulo 4.
}
\end{equation}
but
most of the time, we shall omit the super-script $m(B)$ when we do not want to distinguish between the two roots of $\det B$.\footnote{\label{footsign}
We note in particular that we have
$
\mett{0}{I_{n}}{0}{0}=\Id_{L^{2}(\R^{n})},\ 
\mett{0}{I_{n}}{0}{2}=-\Id_{L^{2}(\R^{n})},
$
and also with the notation \eqref{phase},
$
\mett{0}{-I_{n}}{0}{n}=e^{\frac{i\pi n}2}\sigma_{0},\
\mett{0}{-I_{n}}{0}{n+2}=-e^{\frac{i\pi n}2}\sigma_{0}.
$
More generally, we have 
$$
 \text{for $\det B>0$, }\mett{A}{B}{C}{0}=-\mett{A}{B}{C}{2},
 \qquad
 \text{for $\det B<0$, }\mett{A}{B}{C}{1}=-\mett{A}{B}{C}{3}.
$$}
We note also that for $B\in Gl(n,\R)$, we have
\begin{equation}\label{}
m(B^{*})=m(B)=m(B^{-1}),
\end{equation}
since $\det B=\det B^{*}$ and 
$
\det (B^{-1})=(\det B)^{-1}
$
so that 
$$\arg(\det B)=\arg(\det B^{-1}).
$$
 Moreover we have for $B\in Gl(n,\R)$,
 $$
 \det(-B)=(-1)^{n}\det B, \quad 
\arg(\det(-B))=
\begin{cases}
\arg(\det B)&\text{if $n$ is even,}
 \\
\arg(\det B)+\pi&\text{if $n$ is odd,}
\end{cases}
 $$
 so that 
 \begin{equation}\label{1254}
m(-B)=n+m(B).
\end{equation}
\end{rem}
\begin{exs}\rm
Let us start with a one-dimensional example: in Remark \ref{rem.549877}, we have seen in particular that 
$$
\mat22{0}{1}{-1}{0}=\left\{2^{-1/2}\mat22{1}{1}{-1}{1}\right\}^{2}, \quad
2^{-1/2}\mat22{1}{1}{-1}{1} =\Xi_{-1,2^{1/2}, -1},
$$
where we have used 
\eqref{gensym} to get the second equation.
We have also with the notations of Theorem \ref{4.thm.gensym},
$$
(M_{-1, 2^{1/2},-1}v)(x)=\int_{\R}e^{2i\pi\frac12(-x^{2}+2^{3/2}x\eta-\eta^{2})}\hat v(\eta) d\eta 2^{1/4},
$$
so that the kernel $k_{1}(x,y)$ of the operator $M_{-1, 2^{1/2},-1}$ is 
$$
k_{1}(x,y)=2^{1/4}\int e^{i\pi(
-x^{2}+2^{3/2}x\eta-\eta^{2}
)} e^{-2i\pi y\eta} d\eta\underbracket[0.2pt]{=}_{\text{use  \eqref{foimga}}}
2^{1/4}e^{-i\pi/4} e^{i\pi(x^{2}+y^{2})}e^{-2^{3/2} i\pi xy},
$$
so that the kernel $k_{2}$
of the operator $(M_{-1, 2^{1/2},-1})^{2}$ is (using again  \eqref{foimga}),
\begin{multline*}
k_{2}(x,y)=\int k_{1}(x,z) k_{1}(z,y) dz\\=2^{1/2} e^{-i\pi/2} e^{i\pi(x^{2}+y^{2})}
\int e^{2i\pi z^{2}} e^{-2i\pi z2^{1/2}(x+y)}dz
=e^{-i\pi/4} e^{-2i\pi xy},
\end{multline*}
so that 
\begin{equation}\label{}
(M_{-1, 2^{1/2},-1})^{2}=e^{-i\pi/4}\mathcal F_{1}, \end{equation}
with $\mathcal F_{1}$ standing for the  1d Fourier transformation.
We get similarly that in $n$ dimensions, 
\begin{equation}\label{kjh123}
(M_{-I_{n}, 2^{1/2}I_{n},-I_{n}})^{2}=e^{-i\pi n/4}\mathcal F, \end{equation}
with $\mathcal F$ standing for the  Fourier transformation.
Similar expressions can be obtained for $\mathcal F_{k}$, the Fourier transformation with respect to the variable $x_{k}$ in $n$ dimensions, $k\in \llbracket1,n\rrbracket$ with 
\begin{equation}\label{}
(M_{A_{k}, B_{k},C_{k}})^{2}=
e^{-i\pi/4}\mathcal F_{k},
\end{equation}
where $B_{k}$ is the $n\times n$ diagonal matrix with diagonal entries equal to $1$ except for the $k$th equal to $2^{1/2}$, the $n\times n$ diagonal
matrices $A_{k}=C_{k}$ with diagonal entries equal to 0, except for the $k$th equal to $-1$.
\end{exs}
\index{metaplectic group}
\begin{defi}\label{def.41fdhh}
 The metaplectic group $Mp(n)$
 is defined  as the subgroup of the group of unitary operators on $L^{2}(\R^{n})$ generated by
 \begin{align}
&M_{A,I,0},\text{\footnotesize where $A$ is a $n\times n$ symmetric matrix, cf. \eqref{mety01}},\label{hhjj01}\\
 &M_{0,B,0},\text{\footnotesize
 with $B\in Gl(n,\R)$, with $(\det B)^{\frac12}=\val{\det B}^{\frac12} e^{\frac{i\pi m(B)}2}$, cf. \eqref{maslov},
 \eqref{mety02},
 }\label{hhjj02}\\
 &M_{0,I,C},\text{\footnotesize where $C$ is a $n\times n$ symmetric matrix, cf. \eqref{mety03}}.
 \label{hhjj03}
\end{align}
\end{defi}
\begin{claim}
 If $M$ belongs to $Mp(n)$, then $-M$  belongs to $Mp(n)$.
\end{claim}
\begin{proof}
 According to Footnote \ref{footsign} on page \pageref{footsign}, we have $\mett{0}{I_{n}}{0}{2}=-\mett{0}{I_{n}}{0}{0}=-\Id_{L^{2}(\R^{n})}$
 so that $-\Id_{L^{2}(\R^{n})}$ belongs to $Mp(n)$, proving the claim.
\end{proof}
\index{generators of the metaplectic group}
\begin{pro}
The metaplectic group $Mp(n)$
 is  generated by
 \begin{align}
&M_{A,I,0},\text{\footnotesize where $A$ is a $n\times n$ symmetric matrix, cf. \eqref{mety01}},
\label{778801}\\
 &M_{0,B,0},\text{\footnotesize
 with $B\in Gl(n,\R)$, with $(\det B)^{\frac12}=\val{\det B}^{\frac12} e^{\frac{i\pi m(B)}2}$, cf. \eqref{maslov},
 \eqref{mety02},
 }\label{778802}\\
 &e^{-\frac{i\pi n}4}\mathcal F, \text{\footnotesize where $\mathcal F$ is the Fourier transformation.}
 \label{778803}
\end{align}
 \end{pro}
\begin{proof}
 We check for $C$ symmetric $n\times n$ matrix,
 $$
 \bigl(M_{C,I,0}^{\scriptscriptstyle\{0\}}(e^{-i\pi n/4}\mathcal F v)\bigr)(\eta)=e^{-i\pi n/4} e^{i\pi C\eta^{2}}\hat v(\eta), 
 $$
 so that 
 $$
 e^{i\pi n/4}\bigl(\mathcal F^{-1}(e^{-i\pi n/4} e^{i\pi C\eta^{2}}\hat v(\eta))\bigr)(x)=
 \int e^{2i\pi x\eta}e^{i\pi C\eta^{2}}\hat v(\eta) d\eta=(\mett{0}{I}{C}{0}v)(x),\qquad
 $$
 yielding
 $
 e^{i\pi n/4}\mathcal F^{-1}\mett{0}{I}{C}{0}e^{-i\pi n/4}\mathcal F=\mett{0}{I}{C}{0},
 $
 so that the group generated by \eqref{778801}, \eqref{778802}, \eqref{778803}
 contains 
 \eqref{hhjj01},  \eqref{hhjj02},  \eqref{hhjj03} and thus contains $Mp(n)$.
 Moreover \eqref{kjh123} shows that \eqref{778803} is included  in $Mp(n)$
 so that the group generated by \eqref{778801}, \eqref{778802}, \eqref{778803} is included in $Mp(n)$, proving the proposition.
 \end{proof}
 \begin{rem}\rm
 According to \eqref{azqs22} in our Appendix and to Footnote \ref{footsign} on page \pageref{footsign}, we find
 \begin{equation}\label{}
(e^{-i\pi n/4} \mathcal F)^{*}=e^{i\pi n/4} \mathcal F\sigma_{0}=e^{-i\pi n/4} \mathcal Fe^{i\pi n/2} \sigma_{0}
=e^{-i\pi n/4} \mathcal F M_{0,-I_{n}, 0}^{\scriptscriptstyle\{n\}}.
\end{equation}
As a consequence, 
$
e^{-i\pi n/4}\mathcal F, e^{-i\pi n/2}\sigma_{0}, e^{i\pi n/2}\sigma_{0}
$
belong to the metaplectic group.
\end{rem}
 \begin{lem}\label{lem.54dftz}
 For $Y\in \RZ$, we define the linear form $L_{Y}$ on $\RZ$ by
 \begin{equation}\label{}
L_{Y}(X)=\poscal{\sigma Y}{X}=[Y,X].
\end{equation}
For any $M\in Mp(n)$ there exists a unique $\chi\in Sp(n,\R)$ such that 
\begin{equation}\label{hg54rr}\forall Y\in \RZ, \quad 
M^{*}\opw{L_{Y}}M=\opw{L_{\chi^{-1} Y}}.
\end{equation}
\end{lem}
 \begin{proof}
 Indeed, thanks to \eqref{1239} and Definition \ref{def.41fdhh}, we can find $\chi\in Sp(n,\R)$ such that 
\begin{gather*}
M^{*}\opw{L_{Y}}M=\opw{L_{Y}\circ \chi}=\opw{L_{\chi^{-1} Y}}, \quad \text{since}\\
(L_{Y}\circ \chi)(X)=\poscal{\sigma Y}{\chi X}=\poscal{\chi^{*}\sigma \chi \chi^{-1}Y}{X}=\poscal{\sigma\chi^{-1}Y}{X}=L_{\chi^{-1}Y}(X).
\end{gather*}
Moreover, if $\chi_{1}, \chi_{2}\in Sp(n,\R)$ are such that for all $Y\in\RZ$,
$$
0=\opw{L_{\chi_{2}^{-1} Y}-L_{\chi_{1}^{-1} Y}}=\opw{L_{(\chi_{2}^{-1} -\chi_{1}^{-1})Y }},
$$
we get $L_{(\chi_{2}^{-1} -\chi_{1}^{-1})Y }=0$, implying 
$
\forall Y\in \RZ, \ (\chi_{2}^{-1} -\chi_{1}^{-1})Y =0,  
$
i.e.
$\chi_{1}=\chi_{2}$.\end{proof}
We can thus define a mapping 
\begin{equation}\label{hqma28}
\Psi:Mp(n)\rightarrow Sp(n,\R), \quad\text{with $\Psi(M)=\chi$ satisfying
\eqref{hg54rr}.} 
\end{equation}
In particular we have from \eqref{1239} in Proposition \ref{pro.kj77qq} and \eqref{kjh123} that
\begin{equation}\label{1270}
\Psi(M_{A,B,C})=\Xi_{A,B,C}, \qquad \Psi\bigl(e^{-\frac{i\pi n}4}\mathcal F\bigr)=\sigma=\mat22{0}{I_{n}}{-I_{n}}{0}.
\end{equation}
\begin{theorem}\label{thm.groups}
 The mapping $\Psi$ defined in \eqref{hqma28} is a surjective homomorphism of groups with kernel $\{\pm 
 \Id_{L^{2}(\R^{n})}
 \}$.
\end{theorem}
\begin{proof}
This mapping is an homomorphism of groups: if $M_{1}, M_{2}$ belong to $Mp(n)$, we have with $\chi_{j}=\Psi(M_{j})$, 
\begin{multline*}
(M_{1}M_{2})^{*}\opw{L_{Y}}M_{1}M_{2}=M_{2}^{*}\opw{L_{\chi_{1}^{-1}Y}}M_{2}
\\=\opw{L_{\chi_{2}^{-1}\chi_{1}^{-1}Y}}
=\opw{L_{(\chi_{1}\circ \chi_{2})^{-1}Y}},
\end{multline*}
proving that 
$
\Psi(M_{1}M_{2})=\Psi(M_{1})\Psi(M_{2}).
$
Moreover the homomorphism $\Psi$ is onto, thanks to \eqref{1239} and Theorem \ref{4.thm.gensym}.
The kernel of $\Psi$ is made with $M\in Mp(n)$ such that for all $Y\in\RZ$,
$$
M^{*}\opw{L_{Y}}M=\opw{L_{Y}}, \quad \text{i.e.}\quad \bigl[\opw{L_{Y}}, M\bigr]=0,
$$  
so that, thanks to \eqref{gfcd44+}, \eqref{gfl227}, if $\mu(x,\xi)$ is the Weyl symbol of $M$
($M$ is an endomorphism of $\mathscr S'(\R^{n})$ and thus has a distribution kernel as well as a Weyl symbol
 via Formula \eqref{symker}), we get for all $(y,\eta)\in \RZ$,
$$
0=\poi{\eta\cdot x-y\cdot \xi}{\mu(x,\xi)}\quad\text{so that $d\mu=0$,}
$$
and $\mu$ is a constant so that $M=c\Id_{L^{2}(\R^{n})}$, necessarily with $\val c=1$ (since $M$ is unitary). 
Applying Theorem \ref{thm.phaseb} gives $c\in\{\pm 1\}$, concluding the proof.\end{proof}
\begin{nb}\rm
 The proof of Theorem \ref{thm.phaseb} is relegated in our Appendix, and requires some effort. 
\end{nb}
\begin{cor} For $\chi\in Sp(n,\R)$, 
 the fiber $\Psi^{-1}\{\chi\}$ contains exactly two metaplectic transformations and more precisely
 $$
 \Psi^{-1}\{\chi\}=\{M,-M\},
 $$
 where $M$ is a metaplectic transformation.
\end{cor}
\begin{proof}
 This corollary is an immediate consequence of Theorem \ref{thm.groups}.
\end{proof}
\index{symplectic covariance}
\begin{theorem}[Symplectic covariance of the Weyl calculus]
 Let $a$ be in $\mathscr S'(\RZ)$ and let $\chi$ be in $Sp(n,\R)$. Then for a metaplectic operator $M$ such that $\Psi(M)=\chi$, we have  
 \index{Segal formula}
 \begin{equation}\label{segal}
M^{*}\opw{a} M=\opw{a\circ \chi}.
\end{equation}
For $u,v\in \mathscr S(\R^{n}),$
we have 
\begin{equation}\label{segal+}
\mathcal W\left(M u, Mv\right)=\mathcal W(u,v)\circ \chi^{-1},
\end{equation}
where $\mathcal W$ is the Wigner distribution given in \eqref{wigner}.
\end{theorem}
\begin{proof}
 The first property follows  immediately from \eqref{1239} and Definition \ref{def.41fdhh}
 whereas \eqref{segal+} is a consequence of \eqref{eza654} and \eqref{segal}.
\end{proof}
We note also that for $Y=(y,\eta)\in \RZ$, the symmetry $S_{Y}$ is defined by
$S_{Y}(X)=2Y-X$ and is quantized by the phase symmetry $\sigma_{Y}$ as defined by \eqref{phase}
with the formula
\begin{equation}\label{segalsym}
\opw{a\circ S_{Y}}=\sigma_{Y}^{*}\opw{a}\sigma_{Y}=\sigma_{Y}
\opw{a}\sigma_{Y}.
\end{equation}
\index{phase translation}
\index{{~\bf Notations}!$\tau_{y,\eta}$}
Similarly, 
the translation $T_{Y}$ is defined on the phase space by $T_{Y}(X)=X+Y$ and is quantized by the {\it phase translation} $\tau_{Y}$,
\begin{equation}\label{phtrans}
(\tau_{(y,\eta)} u)(x)=u(x-y) e^{2i\pi(x-\frac y2)\cdot\eta},
\end{equation}
and we have
\begin{equation}\label{1224}
\opw{a\circ T_{Y}}=\tau_{Y}^{*}\opw{a}\tau_{Y}=\tau_{-Y}
\opw{a}\tau_{Y}.
\end{equation}
\begin{rem}\rm
 Property \eqref{segal+} can be extended to the affine symplectic group and we have  with the phase translations defined in \eqref{phtrans},
 \begin{equation}\label{}\forall (X,Y)\in\RZ\times \RZ, \qquad
\mathcal W\left(\tau_{Y} u, \tau_{Y}v\right)(X)=\mathcal W(u,v)(X-Y).
\end{equation}
We shall define the \emph{affine group $Mp_{a}(n)$} as the group of unitary transformations of $L^{2}(\R^{n})$
generated by  transformations \eqref{mety01}, \eqref{mety02}, \eqref{mety03}
  and phase translations given by \eqref{phtrans}.
\end{rem}
\begin{nb}\rm
More information on the metaplectic group 
is given in J. Leray's book \cite{MR644633}, the same author's articles
\cite{MR614311}, \cite{MR0501198},
as well as A. Weil's
paper \cite{MR165033}
(see also V.S. Buslaev's article
\cite{MR0297258},
Chapter 9 in  K.~Gr\"{o}chenig's book \cite{MR1843717},
H. Reiter's Lecture Notes \cite{MR1011671}). 
 \end{nb}
Theorem 1 in E. Lieb's classical article \cite{MR1039210} gives a more precise version  of 
\eqref{1225},  \eqref{1226} and \eqref{lieb44} below.
\begin{theorem}\label{thm776655}
 Let $u, v$ be in $L^2(\R^n
)$. Then $\mathcal W(u,v)$ is a uniformly continuous function belonging to $L^2(\RZ)\cap L^\io(\RZ)$ and using the definitions \eqref{phtrans}, \eqref{phase} for the phase translations and phase symmetry, we have
\begin{align}\label{1224+}
&\mathcal W(u,v)(X)=2^n\poscal{\sigma_{X}u}{v}_{L^2(\R^n)}
=2^n\poscal{\tau_{X}^* u}{\tau_{X}  \check v}_{L^2(\R^n)}
\\&
\hskip 190pt
=
2^n\poscal{\sigma_{0}\tau_{-2X} u}{ v}_{L^2(\R^n)},
\notag\\
&\norm{\mathcal W(u,v)}_{L^2(\RZ)}=\norm{u}_{L^2(\R^n)}\norm{v}_{L^2(\R^n)},
\label{1225}
\\
&\forall p\in [1,+\io], \quad\norm{\mathcal W(u,v)}_{L^\io(\RZ)}\le 2^n\norm{u}_{L^p(\R^n)}\norm{v}_{L^{p'}(\R^n)}.
\label{1226}
\end{align}
More generally, for $q\ge 2$ and $r\in [q',q]$, we have\footnote{\label{footmore}We use the standard notation: for $p\in[1,+\io]$ we define $p'$ by the equality $\frac1{p}+\frac1{p'}=1.$
} 
\begin{equation}\label{lieb44}
\norm{\mathcal W(u,v)}_{L^q(\RZ)}\le 
2^{\frac{n(q-2)}{q}} 
\norm{u}_{L^r(\R^n)}\norm{v}_{L^{r'}(\R^n)}.
\end{equation}
Moreover, we have 
\begin{equation}\label{1228}
\lim_{\substack{\RZ\ni X, \val X\rightarrow+\io}}\Bigl[
{\mathcal W(u,v)(X)}\Bigr]=0.
\end{equation}
\end{theorem}
\begin{proof} 
We have  with $\check v(x)=v(-x)=(\sigma_{0}v)(x)$,
\begin{align*}
&\mathcal W(u,v)(x,\xi)=2^n\int
u(x+z ) \bar v(x-z)
e^{-4i\pi z\xi}dz
\\&=
2^n\int
u(z -(-x)) e^{2i\pi(z-\frac{-x}2)(-\xi)}
\bar{\check v} (z-x)
 e^{-2i\pi(z-\frac{x}2)\xi}
e^{-4i\pi z\xi
+2i\pi(z-\frac{-x}2)\xi
+2i\pi(z-\frac{x}2)\xi}dz
\\&=
2^n\int(\tau_{(-x,-\xi)}  u)(z)
\overline{(\tau_{(x,\xi)}  \check v)(z)} dz
=2^n\poscal{\tau_{(x,\xi)}^* u}{\tau_{(x,\xi)}  \check v}_{L^2(\R^n)},
\end{align*}
or for short
\begin{equation}\label{}
\mathcal W(u,v)(X)=2^n\poscal{\tau_{X}^* u}{\tau_{X}  \check v}_{L^2(\R^n)}.
\end{equation}
As a consequence we find from \eqref{12hh} that
$$
\poscal{\opw{a}u}{v}=\int a(X)2^n \poscal{\sigma_{0} \tau_{2X}^* u}{v} dX,
$$
and since 
$
(\sigma_{x,\xi} u)(y)=u(2x-y)e^{-4i\pi(x-y)\cdot \xi},
$
we can verify directly that 
\begin{equation}\label{symm+}
\sigma_{0}\tau_{-2X}=\sigma_{X}.
\end{equation}
Indeed, composing the translation  of vector $-2X$ in $\RZ$ with the symmetry with respect to 0, 
we have 
$$
Y\mapsto Y-2X\mapsto 2X-Y=Y', \qquad \frac12(Y+Y')=X,
$$
that is the symmetry with respect to $X$.
Quantifying this equality, we use
$$
(\tau_{(-2x,-2\xi)} u)(z)=u(z+2x) e^{2i\pi (z-\frac{-2x}{2})(-2\xi)}
=u(z+2x) e^{-4i\pi (z+x)\xi},
$$
so that we obtain
$$
\sigma_{0}(\tau_{(-2x,-2\xi)} u)(z)
=u(-z+2x) e^{-4i\pi (-z+x)\xi}=(\sigma_{x,\xi}u)(z),
$$
which proves \eqref{symm+} and thus \eqref{1224+}.
Formula \eqref{1225} is already proven in \eqref{norm} and \eqref{1226}
 follows from \eqref{1224+}, H\"older's
inequality and the fact that $\tau_{X}$ is an endomorphism of  $L^{p}(\R^{n})$ 
with norm 1 (cf. the expression \eqref{phtrans}).
To prove \eqref{lieb44} we note that 
from the expression \eqref{65rtfg}, the Hausdorff-Young's inequality implies
\begin{equation}\label{65ggff}
\norm{\mathcal W(u,v)}_{L^{q}\otimes L^{q}}\le 
\norm{\Omega(u,v)}_{L^{q}\otimes L^{q'}}\le \norm{\val u^{q'}\ast \val v^{q'}}_{L^{q/q'}}^{1/q'}
2^{n\frac{q-2}{q}},
\end{equation}
and since 
Young's inequality\footnote{
For $p,q,r\in [1,+\io]$ with $\frac{1}{p'}+\frac{1}{q'}=\frac{1}{r'}$, we have 
\begin{equation}
\norm{f\ast g}_{L^{r}}\le \norm{f}_{L^{p}}\norm{g}_{L^{q}}.
\end{equation}
} gives
$$
\norm{\val u^{q'}\ast \val v^{q'}}_{L^{q/q'}}
\le \norm{\val{u}^{q'}}_{L^{a/q'}}\norm{\val{v}^{q'}}_{L^{b/q'}}, 
$$
$a,b\ge q'$ with
$$
1-\frac{q'}{q}=1-\frac{q'}{a}+1-\frac{q'}{b},\quad
\text{i.e.\quad }q'\bigl(\frac1a+\frac1b\bigr)=1+\frac{q'}q,\
\text{that is \ }\frac1a+\frac1b=1,
$$
so that 
$$
\norm{\val u^{q'}\ast \val v^{q'}}_{L^{q/q'}}
\le \norm{{u}}_{L^{a}}^{q'}\norm{{v}}_{L^{b}}^{q'}, 
$$
in such a way that \eqref{65ggff} yields
$$
\norm{\mathcal W(u,v)}_{L^{q}\otimes L^{q}}\le
2^{n\frac{q-2}{q}}  \norm{{u}}_{L^{a}}\norm{{v}}_{L^{b}}, \quad a, b\ge q',\quad \frac1a+\frac1b=1,
$$
which is 
 \eqref{lieb44} .
 We are left with the proof of uniform continuity of $\mathcal W(u,v)$.
 We have for $X,Y\in \RZ$,
 $$
 \mathcal W(u,v)(Y)-\mathcal W(u,v)(X)
 =2^n\poscal{(\sigma_{Y}-\sigma_{X})u}{v}_{L^2(\R^n)},
 $$
 and since $\sigma_{Y}^2=\Id$  (see the footnote \ref{foot2}  on page \pageref{foot2}), we find
 \begin{multline*}
 \mathcal W(u,v)(Y)-\mathcal W(u,v)(X)
  =2^n\poscal{(\sigma_{Y}\sigma_{X}-\Id)\sigma_{X}u}{v}_{L^2(\R^n)}
  \\=2^n\poscal{\sigma_{X}u}{(\sigma_{X}\sigma_{Y}-\Id)v}_{L^2(\R^n)}.
\end{multline*} 
 According to Formula (2.1.16)
 in \cite{MR2599384}, we have 
 \begin{equation}\label{r}
 \sigma_{X}\sigma_{Y}=\tau_{2X-2Y} e^{4i\pi[Y,X]},
\end{equation}
and this implies 
\begin{equation}\label{rreezz}
\val{\mathcal W(u,v)(Y)-\mathcal W(u,v)(X)}
  \le 2^n\norm{u}_{L^2(\R^n)}
  \norm{\tau_{2(X-Y)} v}_{L^2(\R^n)}.
\end{equation}
We have  from \eqref{segalsym},
\begin{multline*}
\tau_{z,\zeta} v(x) -v(x)= v(x-z) e^{2i\pi(x-\frac z2) \zeta}-v(x)
\\=\bigl(v(x-z)-v(x)\bigr)e^{2i\pi(x-\frac z2) \zeta}
+v(x)\bigl(e^{2i\pi(x-\frac z2) \zeta}-1\bigr),
\end{multline*}
and thus
$$
\norm{\tau_{Z}v-v}_{L^2(\R^n)}
\le \left(\int \val{v(x-z) -v(x)}^2 dx\right)^{1/2}
+\left(\int \val{v(x)}^2\val{e^{2i\pi (x-\frac z2)\zeta}-1}^2 dx\right)^{1/2}.
$$
We have the classical result, due to the density in $L^2$ of continuous compactly supported functions,
$$
\lim_{\R^n\ni z\rightarrow 0}\int \val{v(x-z) -v(x)}^2 dx=0,
$$
and moreover the Lebesgue Dominated Convergence Theorem  implies
$$
\lim_{(z,\zeta)\rightarrow (0,0)}\int \underbrace{\val{v(x)}^2}_{\in L^1(\R^n)}
\underbrace{\val{e^{2i\pi (x-\frac z2)\zeta}-1}^2 }_{\le 4}dx=0,
$$
so that 
$\lim_{\RZ\ni Z\rightarrow 0}
\norm{\tau_{Z}v-v}_{L^2(\R^n)}=0
$.
As a consequence \eqref{rreezz} implies the uniform continuity of $\mathcal W(u,v)$.
Moreover, we have, for  $\phi,\psi\in \mathscr S(\R^n)$,
$$
\mathcal W(u,v)= \mathcal W(u-\phi,v)+\mathcal W(\phi,v-\psi)
+\mathcal W(\phi,\psi),
$$
so that 
\begin{align*}
\val{\mathcal W(u,v)(x,\xi)} &\le 
\int \val{(u-\phi)(x+\frac z2)} \val{v(x-\frac z2)} dz
\\&\hskip58pt+\iint \val{(v-\psi)(x-\frac z2)} \val{\phi(x+\frac z2)} dz
+\val{\mathcal W(\phi,\psi)(x,\xi)} 
\\
&\le
2^n\norm{u-\phi}_{L^2(\R^n)} \norm{v}_{L^2(\R^n)} 
+2^n\norm{v-\psi}_{L^2(\R^n)} \norm{\phi}_{L^2(\R^n)} 
\\&\hskip237pt+\val{\mathcal W(\phi,\psi)(x,\xi)}.
\end{align*}
We choose now sequences $(\phi_{k}), (\psi_{k})$ of $\mathscr S(\R^n)$
converging respectively in $L^2(\R^n)$ towards $u,v$.
We obtain for all $k\in \N$,
\begin{multline}\label{ineq44}
\val{\mathcal W(u,v)(x,\xi)}\le 
2^n\norm{u-\phi_{k}}_{L^2(\R^n)} \norm{v}_{L^2(\R^n)} 
+2^n\norm{v-\psi_{k}}_{L^2(\R^n)} \norm{\phi_{k}}_{L^2(\R^n)} 
\\+\val{\mathcal W(\phi_{k},\psi_{k})(x,\xi)},
\end{multline}
so that using that $\mathcal W(\phi_{k},\psi_{k})$ belongs to $\mathscr S(\RZ)$, we get
\begin{multline*}
\limsup_{\substack{\RZ\ni X, \val X\rightarrow+\io}}\Bigl[
\val{\mathcal W(u,v)(X)}\Bigr]
\\\le 2^n\norm{u-\phi_{k}}_{L^2(\R^n)} \norm{v}_{L^2(\R^n)} 
+2^n\norm{v-\psi_{k}}_{L^2(\R^n)} \norm{\phi_{k}}_{L^2(\R^n)} ,
\end{multline*}
and thus, taking the limit when $k\rightarrow+\io$, we obtain
$$
\lim_{\substack{\RZ\ni X, \val X\rightarrow+\io}}\Bigl[
\val{\mathcal W(u,v)(X)}\Bigr]=0,
$$
completing the proof of Theorem \ref{thm776655}.
\end{proof}
\begin{rem}\rm
 Let $u$ be in $L^2(\R^n)$ be an even function. We then have 
\begin{equation}\label{}
 \mathcal W(u,u)(0,0)=2^n\norm{u}_{L^2(\R^n)}^2=\norm{\mathcal W(u,u)}_{L^\io(\RZ)}.
\end{equation}
 On the other hand if $u$ is odd we have 
\begin{equation}\label{}
 \mathcal W(u,u)(0,0)=-2^n\norm{u}_{L^2(\R^n)}^2=-\norm{\mathcal W(u,u)}_{L^\io(\RZ)},
\end{equation}
showing that for odd functions the minimum of the Wigner distribution is negative (we assume $u\not=0$ in $L^2(\R^n)$) and attained at 0.
Let us check for instance the (odd) function $u_{1}$ of Remark \ref{rem13}.
We have 
\begin{multline*}
2\norm{u_{1}}_{L^2(\R)}^2=2\int x^2 e^{-2\pi x^2}dx=4\int_{0}^{+\io}\frac{t}{2\pi}
e^{-t} (2\pi)^{-1/2}\frac12 t^{-1/2} dt\\
=\frac{2\Gamma(3/2)}{ (2\pi)^{3/2}}=\frac{\Gamma(1/2)}{ (2\pi)^{3/2}}=\frac{1}{2^{3/2}\pi}=-\mathcal W (u_{1},u_{1})(0,0).
\end{multline*}
 \end{rem}
\subsubsection{On weak versions of the Wigner distribution}
\index{weak versions of the Wigner distribution}
Let $u,v$ be in the space  $\mathscr S'(\R^{n})$
of tempered distributions. Then we can define as above the tempered distribution
$\Omega(u,v)$ in $\RZ$:
we set
\begin{multline}
\poscal{\Omega(u,v)(x,z)}{\Phi(x,z)}_{\mathscr S'(\R^{2n}),\mathscr S(\R^{2n})}
\\=\poscal{u(x_{1})\otimes\bar v(x_{2})}{\Phi(\frac{x_{1}+x_{2}}{2},x_{1}-x_{2})}_{\mathscr S'(\R^{2n}),\mathscr S(\R^{2n})},
\end{multline}
and then we define the Wigner distribution $\mathcal W(u,v)$ as the Fourier
transform with respect to $z$ of  $\Omega(u,v)$, meaning that 
\begin{equation}
\poscal{\mathcal W(u,v)}{\Psi}_{\mathscr S'(\R^{2n}),\mathscr S(\R^{2n})}
=\poscal{\Omega(u,v)}{\mathcal F_{2}\Psi}_{\mathscr S'(\R^{2n}),\mathscr S(\R^{2n})},
\end{equation}
where 
$$
(\mathcal F_{2}\Psi)(x,\xi)=\int_{\R^{n}} e^{-2i\pi z\cdot \xi} \Psi (x,z) dz.
$$
Of course $\mathcal W(u,v)$ is only a tempered distribution  on $\RZ$ and we have the inversion formula, using the notations of Remark \ref{rem.334455},
\begin{equation}
\Omega(u,v)=\mathcal F_{2}\mathcal C_{2}\mathcal W(u,v).
\end{equation}
The above remarks show that there is no difficulty to extend the definition of the joint Wigner distribution $\mathcal W(u,v)$ to the case where $u,v$ are both tempered distributions on $\R^{n}$. Some properties are surviving from the $L^{2}$ theory, in particular the inversion formula, but one should be reasonably cautious at avoiding to write brackets of duality as integrals.
Theorem 2 in \cite{MR1039210} gives a more complete version of the following result.
\begin{theorem}\label{thm.112lkj}
 Let $u\in L^{2}(\R^{n})$ such that $\mathcal W(u,u)\in L^{1}(\RZ)$. Then $u$ belongs to 
 $L^{p}(\R^{n})$ for all $p\in [1,+\io]$
 and we have 
 $$
 \norm{u}_{L^{1}(\R^{n})}
  \norm{u}_{L^{\io}(\R^{n})}\le 2^{n} \norm{\mathcal W(u,u)}_{L^{1}(\R^{2n})}.
 $$
\end{theorem}
\begin{nb}
 We refer the reader to our Section
 \ref{sec.meager} and in particular to 
our
 Theorem \ref{thm.67kjhg} on page \pageref{thm.67kjhg}
 showing that the set of $u$ in $L^{2}(\R^{n})$ such that 
 $\mathcal W(u,u)$ belongs to $L^{1}(\RZ)$ is meager.
\end{nb}
\begin{proof}
Thanks to Theorem \ref{thm776655}, we have  $\mathcal W(u,u)\in L^{p}(\RZ)$ for all $p\in [1,+\io]$
and we have in a weak sense,
 $$
 u(x+\frac z2)\bar u(x-\frac z2)=\int e^{2i\pi z\cdot \xi}\mathcal W(u,u)(x,\xi) d\xi, 
 $$
 so that 
 \begin{equation}\label{623+++}
 u(x_{1}) \bar u(x_{2})=\int e^{2i\pi (x_{1}-x_{2})\cdot \xi}\mathcal W(u,u)(\frac{x_{1}+x_{2}}2,\xi) d\xi, 
\end{equation}
and thus
$$
\int\val{u(x_{1})}\val{u(x_{2})} dx_{1}\le \iint \bigl\vert\mathcal W(u,u)(\frac{x_{1}+x_{2}}2,\xi)\bigr\vert d\xi dx_{1}
= 2^{n}\norm{\mathcal W(u,u)}_{L^{1}(\RZ)},
$$
i.e.
$$
\norm{u}_{L^{1}(\R^{n})}\norm{u}_{L^{\io}(\R^{n})}\le 2^{n}\norm{\mathcal W(u,u)}_{L^{1}(\RZ)},
$$
proving the lemma.
\end{proof}
 \subsubsection{Composition Formulas}
 \index{composition formula}
 The following lemma is classical (see e.g \cite{MR416349}, \cite{MR424081}); however
 we shall provide a proof for the convenience of the reader.
 \begin{lem}\label{lem.113esd}
 Let $u,v, f,g$ be in $L^2(\R^n)$. Then
\begin{equation}\label{131131}
 \poscal{u}{g}_{L^2(\R^n)} \poscal{f}{v}_{L^2(\R^n)}=\iint
 \mathcal W(u,v)(x,\xi)\mathcal W(f,g)(x,\xi) dx d\xi.
\end{equation}
In other words, the Weyl symbol of the rank-one operator
$
u\mapsto  \poscal{u}{g}_{L^2(\R^n)} f
$
is $\mathcal W(f,g)$.
In particular if $f=g$ is a unit vector in $L^2(\R^n
)$ we find that 
$\mathcal W(f,f)$ is the Weyl symbol of the orthogonal projection onto $\C f$. \end{lem}
\begin{proof}
 Both functions $\mathcal W(u,v),
 \mathcal W(f,g)$ belong to $L^2(\RZ)$, so that the integral on the right-hand-side of \eqref{131131}
 actually makes sense. Also 
 $\mathcal W(u,v)$
  is the partial Fourier transform with respect to the variable $z$
  of $(x,z)\mapsto u(x+z/2)\bar v(x-z/2)$, thus applying Plancherel formula\footnote{\label{foot3}We refer of course to the formula
  $
  \poscal{\hat u}{\hat v}_{L^{2}(\R^{n})}=\poscal{u}{v}_{L^{2}(\R^{n})},
  $
  when using the \emph{complex} Hilbert space $L^{2}(\R^{n})$.
  Note however that Formula \eqref{res444} is using the \emph{real} duality between 
  $\mathscr S(\R^{n})$ and $\mathscr S'(\R^{n})$ so that to check,
  with $\mathscr S^{*}(\R^{N})$ standing for the anti-dual of $\mathscr S(\R^{N})$
  (i.e. continuous anti-linear forms on $\mathscr S(\R^{N})$), we have also 
  \vskip-17pt
 \begin{multline}
  \poscal{\widehat T}{\widehat \phi}_{\mathscr S^{*}(\R^{N}), \mathscr S(\R^{N})}=
  \poscal{\widehat T}{\overline{\widehat \phi}}_{\mathscr S^{'}(\R^{N}), \mathscr S(\R^{N})}
  = \poscal{ T}{
 \widehat{\overline{\widehat \phi}}
  }_{\mathscr S^{'}(\R^{N}), \mathscr S(\R^{N})}=\poscal{ T}{
 {\overline{\phi}}
  }_{\mathscr S^{'}(\R^{N}), \mathscr S(\R^{N})}
  \\=
  \poscal{ T}{
 {{\phi}}
  }_{\mathscr S^{*}(\R^{N}), \mathscr S(\R^{N})}.
\end{multline}
  } we obtain that  
  \begin{multline*}
\iint
 \mathcal W(u,v)(x,\xi)\mathcal W(f,g)(x,\xi) dx d\xi
 \\=
  \iint u(x+z/2)\bar v(x-z/2)
f (x-z/2)\bar g(x+z/2)dx dz
\\
=\poscal{u}{g}_{L^2(\R^n)} \poscal{f}{v}_{L^2(\R^n)}.
\end{multline*}
The last property follows from \eqref{eza654}.
\end{proof}
Using Section 2.1.5 in \cite{MR2599384},
we obtain that 
for $a,b \in \mathscr S(\RZ)$ 
$$
\opw{a}\opw{b}=\iint_{\RZ\times\RZ} a(Y) b(Z) 2^{2n} \sigma_{Y}\sigma_{Z} dY dZ.
$$
We get 
\begin{equation}\label{sharp0}
\opw{a}\opw{b}=\opw{a\sharp b},
\end{equation}
with 
\begin{align}\
(a\sharp b)(X)&=
2^{2n}
\iint_{\RZ\times\RZ}  e^{-4i\pi[X-Y,X-Z]} 
a(Y) b(Z) dY dZ
\label{2.wecofo}
\\&=
\iint_{\RZ\times\RZ}  e^{-2i\pi\poscal{\Xi}{Z}} 
a\bigl(X+\frac{\sigma^{-1} \Xi}2\bigr) b(Z+X) d\Xi dZ,
\label{wecof3}
\\
&=\int_{\RZ}
e^{2i\pi \poscal{X}{\Xi}}a\bigl(X+\frac{\sigma^{-1} \Xi}2\bigr)\widehat b(\Xi) d\Xi,
\label{weco++}
\end{align}
where $[\cdot, \cdot]$ is the symplectic form \eqref{sympfor} and $\sigma$ is \eqref{ffqq99}.
Formula  \eqref{wecof3} is interesting since very close to the group $J^{t}$ defined in Formula (4.1.14) of 
\cite{MR2599384}.
\begin{lem}\label{lem.54jhgf}
 Let $u_{0}, u_{1}, u_{2}, u_{3}$ be in $L^{2}(\R^{n})$.
 Then we have for all $X\in \RZ$, 
\begin{equation}\label{}
 \val{\poscal{u_{1}}{u_{2}}_{L^{2}}}
\val{\mathcal W(u_{0},u_{3})(X)}\le 2^{n}
\bigl(\val{\mathcal W(u_{0}, u_{2})}\ast\val{\mathcal W(\check u_{1}, u_{3})}\bigr)(X).
\end{equation}
\end{lem}
\begin{proof}
According to Lemma \ref{lem.113esd},
we have for $v\in L^{2}(\R^{n})$,
\begin{multline*}
\opw{\mathcal W(u_{0},u_{2}})\opw{\mathcal W(u_{1},u_{3}})v=
\opw{\mathcal W(u_{0},u_{2}})\bigl(\poscal{v}{u_{3}}_{L^{2}(\R^{n})}u_{1}\bigr)
\\
=\poscal{v}{u_{3}}_{L^{2}(\R^{n})}\poscal{u_{1}}{u_{2}}_{L^{2}(\R^{n})} u_{0}
\\
=\poscal{u_{1}}{u_{2}}_{L^{2}(\R^{n})}\opw{\mathcal W(u_{0},u_{3})}v, 
\end{multline*}
so that with the notation \eqref{sharp0}, we get
\begin{equation}\label{1258++}
\mathcal W(u_{0},u_{2})\sharp
\mathcal W(u_{1},u_{3})=\poscal{u_{1}}{u_{2}}_{L^{2}(\R^{n})}
\mathcal W(u_{0},u_{3}),
\end{equation}
and using  \eqref{weco++}, we get
\begin{multline}
\bigl(\mathcal W(u_{0},u_{2})\sharp
\mathcal W(u_{1},u_{3})\bigr)(x,\xi)
\\
=\iint
e^{2i\pi(x\cdot\eta+\xi\cdot y)}
\mathcal W(u_{0},u_{2})\bigl(x-\frac y2, \xi+\frac\eta 2\bigr)
\overbrace{\mathcal F\bigl(\mathcal W(u_{1},u_{3})\bigr)(\eta, y)}^{\mathcal A(u_{1}, u_{3})(-\eta, -y)} dy d\eta,
\end{multline}
 where $\mathcal F$ stands for the Fourier transformation and $\mathcal A$ for the ambiguity function
 (cf. \eqref{1210}).
 With Formula \eqref{1211++}, we obtain
 \begin{multline}
\bigl(\mathcal W(u_{0},u_{2})\sharp
\mathcal W(u_{1},u_{3})\bigr)(x,\xi)
\\
=\iint
e^{4i\pi(-x\cdot\eta+\xi\cdot y)}
\mathcal W(u_{0},u_{2})(x-y, \xi-\eta)
\mathcal W(\check u_{1}, u_{3})(y, \eta )dy d\eta 2^{n},
\end{multline}
yielding from \eqref{1258++} for any $X\in \RZ$, 
\begin{multline}\label{}
\poscal{u_{1}}{u_{2}}_{L^{2}}
\mathcal W(u_{0},u_{3})(X)
\\=\int_{\RZ}
e^{4i\pi[X,Y]}
\mathcal W(u_{0},u_{2})(X-Y)
\mathcal W(\check u_{1}, u_{3})(Y)dY 2^{n},
\end{multline}
 which implies the lemma.
\end{proof}
\subsubsection{$L^{2}$-boundedness}
\begin{theorem}\label{thm.117jhg}
 Let $a$ be a semi-classical symbol on $\RZ$,
i.e. a  smooth function of $(x,\xi)$ depending on $h\in (0,1]$ such that
\begin{equation}\label{semi00++}
\forall l\in \N,\quad
p_{l}(a)=
\sup_{\substack{
(x,\xi)\in \RZ, h\in (0,1]\\\val \alpha+\val \beta\le l
}
}\val{(\p_{x}^\alpha\p_{\xi}^\beta \text{a})(x,\xi,h)}h^{-\frac {\val \alpha+\val \beta}{2}}<+\io.
\end{equation}
Then the operator $\opw{a(x,\xi,h)}$ is bounded on $L^{2}(\R^{n})$ and such that 
\begin{equation}\label{}
\norm{\opw{a(x,\xi,h)}}_{\mathcal B(L^{2}(\R^{n}))}\le c_{n} p_{\ell_{n}}(a),
\end{equation}
where $c_{n}$ and $\ell_{n}$ depend only on $n$.
\end{theorem}
\begin{proof}
 Theorem 1.2 in A.~Boulkhemair's article  \cite{MR1696697} is providing that result (and more)
  with $\ell_{n}=[n/2]+1$.
  Note also that Theorem 1.1.4 in \cite{MR2599384} is providing an elementary proof
  of the above result for the ordinary quantization  of $a$ given by 
  \index{ordinary quantization}
  \index{{~\bf Notations}!$\ops{0}{a}$}
  \begin{multline}
(\ops{0}{a}u)(x)=\int e^{2i\pi x\cdot \xi} a(x,\xi,h) \hat u(\xi) d\xi\\=\iint e^{2i\pi (x-y)\cdot \xi} a(x,\xi,h) u(y) dy d\xi.
\end{multline}
\end{proof}
\begin{nb}\rm
Formula  \eqref{weco++} appears as 
\begin{equation}(a\sharp b)(X)=
\Bigl(\OPS{0}{a\bigl(X-\frac{\sigma \Xi}2}b\Bigr)
(X),
\end{equation}
where $\ops{0}{\cdot}$ stands for the ordinary quantization in $2n$ dimensions.
\end{nb}
The following classical result is a consequence of Theorem \ref{thm.117jhg}.
\begin{theorem}\label{thm.117j++}
 Let  $C^{\infty}_{b}(\RZ)$ be the set of bounded smooth complex-valued functions on $\RZ$ such that all derivatives are bounded and 
 let $a$ be in $C^{\infty}_{b}(\RZ)$. Then the operator $\opw{a}$ is bounded on $L^{2}(\R^{n})$  
 and the $\mathcal B(L^{2}(\R^{n}))$ norm of $\opw{a}$ is bounded above by a fixed semi-norm of $a$ in the Fr\'echet space  $C^{\infty}_{b}(\RZ)$.
\end{theorem}
\subsubsection{On the Heisenberg Uncertainty Relations}
\index{Heisenberg uncertainty relations}
Let $u\in \mathscr S(\R)$. We have, using the notations \eqref{deuxipi},
\begin{equation}\label{1259pp}
2\re\poscal{D_{x} u}{i x u}_{L^{2}(\R)}=\poscal{[D_{x}, ix] u}{u}_{L^{2}(\R)}=\frac{1}{2\pi}\norm{u}^{2}_{L^{2}(\R)}, 
\end{equation}
implying in particular
\begin{equation}
\norm{D_{x}u}_{L^{2}(\R)}\norm{xu}_{L^{2}(\R)}\ge \frac{1}{4\pi}\norm{u}^{2}_{L^{2}(\R)}, 
\end{equation}
which is an equality for $u(x)=e^{-\pi x^{2}}$; moreover we infer also from 
\eqref{1259pp} that 
\begin{align}\label{}
\poscal{\pi(D_{x}^{2}+x^{2}) u}{u}&\ge \frac{1}{2}\norm{u}^{2}_{L^{2}(\R)}, 
\end{align}
and for 
\begin{equation}
q_{\mu}(x,\xi)=\sum_{1\le j\le n}\mu_{j}(x_{j}^{2}+\xi_{j}^{2}), \quad 0\le \mu_{1}\le \dots\le \mu_{n},
\end{equation}
the inequality\begin{equation}\label{1263az}
\poscal{\OPW{\pi q_{\mu}(x,\xi)}u}{u}_{L^{2}(\R^{n})}\ge \norm{u}^{2}_{L^{2}(\R^{n})}\frac12\underbrace{\sum_{1\le j\le n} \mu_{j}
}_{\substack{\text{defined as}\\
\trace_{+}(q_{\mu})}},\end{equation}
which is an equality for $u(x)=e^{-\pi\val x^{2}}$. Note that  the above (optimal)
inequality can be reformulated  as 
\begin{equation}
\iint_{\RZ} \pi q_{\mu}(x,\xi)\mathcal W(u,u)(x,\xi) dx d\xi\ge  \norm{u}^{2}_{L^{2}(\R^{n})}\frac12\trace_{+}(q_{\mu}).
\end{equation}
\index{fundamental matrix}
Note also that   with the symplectic matrix $\sigma$ defined by \eqref{ffqq99},
the so-called fundamental matrix of $q_{\mu}$ is defined by 
\begin{multline}
F_{q_{\mu}}=\sigma^{-1}Q_{\mu}=\mat22{0}{-I}{I}{0}\mat22{M}{0}{0}{M}=\mat22{0}{-M}{M}{0},
\\ \text{with\quad}M=\diag(\mu_{1},\dots, \mu_{n}),
\end{multline}
so that 
\index{{~\bf Notations}!$\trace_{+}$}
\begin{equation}
\spectrum F_{q_{\mu}}=\{\pm i\mu_{j}\}_{1\le j\le n}, \quad \trace_{+}(q_{\mu})=
\sum_{\substack{\lambda \text{ eigenvalue of $F_{q_{\mu}}$}\\
\text{with }\im \lambda >0}} \lambda/i.
\end{equation}
With the notations 
\index{creation}
\index{annihilation}
 \begin{align}
 \begin{cases}
 C_{j}=D_{x_{j}}+i x_{j}, & \text{ creation operators}, 
 \\
 C_{j}^{*}=D_{x_{j}}-i x_{j},&
 \text{annihilation operators}, 
\end{cases}
 \end{align}
we see that 
$$
\pi[C_{j}^{*},C_{j}]=\pi[D_{x_{j}}-i x_{j}, D_{x_{j}}+i x_{j}]=I,
$$
and 
\begin{equation}
\opw{q_{\mu}}=\pi\sum_{1\le j\le n} \mu_{j}C_{j} C_{j}^{*}+
\frac12\trace_{+}(q_{\mu}),
\end{equation}
which provides another proof  of \eqref{1263az}.
\begin{lem}[Quantum Mechanics must deal with unbounded operators\footnote{Thus QM must involve infinite dimensional Hilbert spaces and unbounded operators on them.}]
Let $\mathbb H$ be a Hilbert space and let $J,K\in \mathcal B(\mathbb H)$; then the commutator $[J,K]\not=\Id$.
\end{lem}
\begin{proof}
Let $J,K$ be bounded operators with $[J,K]=\Id$. Then for all $N\in \N^{*}$, we have 
\begin{equation}\label{1269ou}
[J, K^{N}]=NK^{N-1}.
\end{equation}
Indeed, this is true for $N=1$ and if it holds for some $N\ge 1$, we find that
\begin{multline*}
[J,K^{N+1}]=JK^{N}K-K^{N+1} J=[J,K^{N}] K+K^{N}JK-K^{N+1} J
\\=[J,K^{N}] K+K^{N}(JK-KJ) =[J,K^{N}] K+K^{N}=(N+1) K^{N}, \quad\text{qed}.
\end{multline*}
Note that \eqref{1269ou} implies that for all $N\in \N^{*}$, we have $K^{N}\not=0$: of course $K\not=0$ since $[J,K]=\Id$ and 
if we had $K^{N}=0$ for some $N\ge 2$,
\eqref{1269ou} would imply $K^{N-1}=0$
and eventually $K=0$.
As a consequence, we get from \eqref{1269ou} that for all $N\ge 2$, 
$$
N\norm{K^{N-1}}_{\mathcal B(\mathbb H)}\le 2\norm{J}_{\mathcal B(\mathbb H)}\norm{K^{N}}_{\mathcal B(\mathbb H)}\le 2\norm{J}_{\mathcal B(\mathbb H)}\norm{K}_{\mathcal B(\mathbb H)}\norm{K^{N-1}}_{\mathcal B(\mathbb H)},
$$
implying since $\norm{K^{N-1}}_{\mathcal B(\mathbb H)}>0$,
that
$$
\forall N\ge 2, \quad N\le 2\norm{J}\norm{K},
$$
which is impossible and proves the lemma.
\end{proof}
\index{Hardy inequality}
\begin{lem}[Hardy's inequality: the study of non-self-adjoint operators may be useful to determine lowerbounds of self-adjoint operators]
Let $n\in \N, n\ge 3$;
let $u$ in $L^{2}(\R^{n})$ such that $\nabla u \in L^{2}(\R^{n}), \val x^{-1} u \in L^{2}(\R^{n})$.
Then we have 
\begin{equation}
\norm{\nabla u}_{L^{2}(\R^{n})}^{2}\ge \bigl(\frac{n-2}{2}\bigr)^{2}
\norm{\val x^{-1}{u}}_{L^{2}(\R^{n})}^{2}.
\end{equation}
\end{lem}
\begin{proof}
 We write first
 $$
 \sum_{1\le j\le n}\norm{(D_{x_{j}}-i\phi_{j}) u}^{2}_{L^{2}(\R^{n})}
 =\poscal{\val{D}^{2}u}{u}_{L^{2}(\R^{n})}+\poscal{\val \phi^{2}u}{u}_{L^{2}(\R^{n})}-\frac{1}{2\pi}
 \poscal{(\dive \phi)u}{u}_{L^{2}(\R^{n})},
 $$
 so that with 
 $
 \phi(x)=\frac{\nu x}{2\pi \val x^{2}},
 $
 we get the operator inequality 
 $$
 \val D^{2}+\frac{\nu^{2} }{4\pi^{2}\val x^{2}}\ge \frac{\nu(n-2)}{4\pi^{2}\val x^{2}},
 \quad
\text{ so that }\quad
{ -\Delta}\ge \val x^{-2} \underbrace{\nu(n-2-\nu)}_{\text{largest at }
\nu=(n-2)/2},
 $$
  proving the lemma.
\end{proof}
\begin{nb}\rm 
 A modern approach to the Heisenberg Uncertainty Principle should certainly begin with reading C. Fefferman's article \cite{MR707957} as well as E. Lieb's book \cite{MR2766495}.
\end{nb}
\subsubsection{Non-negative quantizations formulas}\label{sec.nonneg}
\index{non-negative quantization formulas}
\begin{lem}\label{lem112}
 Let $\chi$ be an even  function in $\mathscr S(\RZ)$ with $L^2(\RZ)$ norm equal to 1. We define 
 \begin{equation}\label{gamchi}
\Gamma_{\chi}=\bar \chi\sharp \chi.
\end{equation}
Then the function $\Gamma_{\chi}$ belongs to $\mathscr S(\RZ)$, is real-valued even and is such that 
$$\int_{\RZ}\Gamma_{\chi}(X) dX=1.$$
Let $u$ be in $L^2(\R^n)$. Then the convolution
$\mathcal W(u,u)\ast \Gamma_{\chi}$ is non-negative.
As a result, the operator with Weyl symbol $a\ast \Gamma_{\chi}$ is a non-negative operator whenever $a$ is a non-negative function.
\end{lem}
 \begin{proof}
 According to  the book \cite{MR2599384}, the composition formula \eqref{2.wecofo}
 is bilinear continuous from  $\mathscr S(\RZ)^2$ into  $\mathscr S(\RZ)$ and we have also
 $$
 \overline{a\sharp b}=\bar b\sharp \bar a.
 $$
 so that $\Gamma_{\chi}$ is indeed real-valued. Moreover, we have 
 \begin{multline*}
 \int_{\RZ} \Gamma_{\chi}(X) dX
 =
 2^{2n}
\iiint_{(\RZ)^3}  e^{-4i\pi[X-Y,Y-Z]} 
\bar\chi(Y) \chi(Z) dY dZdX
\\
=\int \val{\chi(Y)}^2 dY=1,
\end{multline*}
and
\begin{multline*}
\Gamma_{\chi}(-X)=
2^{2n}
\iint_{\RZ\times\RZ}  e^{-4i\pi[-X-Y,-X-Z]} 
\bar \chi(Y) \chi(Z) dY dZ
\\
=2^{2n}
\iint_{\RZ\times\RZ}  e^{-4i\pi[-X+Y,-X+Z]} 
\bar \chi(Y) \chi(Z) dY dZ=
\Gamma_{\chi}(X).
\end{multline*}
We have also
\begin{align*}
\bigl(\mathcal W&(u,u)\ast \Gamma_{\chi}\bigr)(Y)=
\int_{\RZ} \mathcal W(u,u)(Y-X) \Gamma_{\chi}(X) dX
\\
&=\int_{\RZ} \mathcal W(u,u)(Y+X) \Gamma_{\chi}(X) dX
=\int_{\RZ} \mathcal W(u,u)(T_{Y}(X)) \Gamma_{\chi}(X) dX
\\&=
\int_{\RZ} \mathcal W(\tau_{-Y}u,\tau_{-Y}u)(X) \Gamma_{\chi}(X) dX
=
\int_{\RZ} \mathcal W(\tau_{-Y}u,\tau_{-Y}u)(X) (\bar \chi\sharp\chi)(X) dX
\\
&=\poscal{
\opw{\bar \chi\sharp\chi}
\tau_{-Y}u}{\tau_{-Y}u}_{L^2(\R^n)}
=\norm{\opw{\chi}\tau_{-Y}u}_{L^2(\R^n)}^2\ge 0,
\end{align*}
proving the first statement of non-negativity. Let $a$ be a non-negative function, say in $L^1(\RZ)$; we have 
\begin{align*}
\opw{a\ast \Gamma_{\chi}}&=2^n\iint a(Y)\Gamma_{\chi}(X-Y) \sigma_{X} dY dX
\\&=\int a(Y)\int(\bar \chi\sharp \chi)(X-Y) 2^n\sigma_{X}dX dY
\\
&=
\int a(Y)\int(\bar \chi\sharp \chi)(T_{-Y}(X)) 2^n\sigma_{X}dX dY
=
\int a(Y)\tau_{Y}
\opw{\bar \chi\sharp \chi}
\tau_{-Y} dY
\\
&=
\int a(Y)\tau_{Y}
\opw{\bar \chi}
\opw{\chi}
\tau_{-Y} dY
\\
&=\int a(Y)\underbrace{\bigl[\opw\chi\tau_{-Y}\bigr]^*\bigl[\opw\chi\tau_{-Y}\bigr] }_{\text{non-negative operator}}dY\ge 0,
\end{align*}
if $a(Y)\ge 0$ for all $Y\in \RZ$ and this concludes the proof.\end{proof}
We can write as well
\begin{equation}\label{134}
\opw{a\ast \Gamma_{\chi}}
=\int_{\RZ} a(Y)\bigl[\tau_{Y}\opw\chi\tau_{-Y}\bigr]^*
\bigl[\tau_{Y}\opw\chi\tau_{-Y}\bigr] dY=
\int_{\RZ} a(Y)\Sigma_{\chi}(Y)dY,
\end{equation}
with
\begin{multline}\label{135}
\Sigma_{\chi}(Y)=\bigl[\tau_{Y}\opw\chi\tau_{-Y}\bigr]^*
\bigl[\tau_{Y}\opw\chi\tau_{-Y}\bigr]
\\=
\bigl(\opw{\chi(\cdot-Y)}
\bigr)^*
\opw{\chi(\cdot-Y)}.
\end{multline}
\begin{rem}
 \rm The Gaussian case in the previous lemma gives rise to the standard non-negativity  properties of coherent states. In fact choosing $\chi(X)=2^ne^{-2\pi\val X^2}$, we see that $\chi$ is even, belongs to the Schwartz space and 
 $$
 \norm{\chi}_{L^2(\RZ)}^2= 2^{2n}\int_{\RZ} e^{-4\pi \val X^2} dX=2^{2n}4^{-2n/2}=1.
 $$
 We have also\footnote{\label{foot4}Proposition 4.1.1 in \cite{MR2599384} is useful to compute the Fourier transform of Gaussian functions and is a notable asset of the Fourier normalization given in Footnote \ref{foot1} page \pageref{foot1}.}
\begin{align*}
& \Gamma_{\chi}(X)=
  2^{4n}
\iint_{(\RZ)^2}  e^{-4i\pi[X-Y,X-Z]} 
e^{-2\pi(\val Y^2+\val Z^2)}dY dZ
\\
&=
 2^{3n}
\int_{\RZ}  e^{4i\pi[Y,X]} 
e^{-2\pi(\val {X+Y}^2+\val Y^{2})}dY 
=2^{3n}
\int_{\RZ}  e^{4i\pi[Y,X]} 
e^{-2\pi(\val {Y+\frac X2}^2+{\val {Y-\frac X 2}}^{2})}dY 
\\&=
 2^{3n}
 e^{-\pi\val X^{2}}
\int_{\RZ}  e^{4i\pi[Y,X]} 
e^{-4\pi\val {Y}^2}dY
= 2^{3n}e^{- \pi \val{X}^2}4^{-n}e^{- \pi \val{X}^2}=\chi(X).
\end{align*}
In that case we find that $\opw\chi$ is a rank-one orthogonal projection on the fundamental state $\Psi_{0}$ of the Harmonic Oscillator $\pi(\val{D_{x}}^2+\val x^2)$.
According to \eqref{herm1}
the one-dimensional $k$-th Hermite function is 
\begin{equation}\label{hermite}
\psi_{k}(x)=\frac{(-1)^{k}}{2^{k}\sqrt{k! }   }  2^{1/4} e^{\pi x^{2}}\left(\frac{d}{\sqrt \pi dx}\right)^{k}(e^{-2\pi x^{2}}),
\end{equation}
so that 
$
\Psi_{0}(x)=2^{n/4} e^{-\pi \val x^2}.
$
We calculate
\begin{multline*}
\Gamma(x,\xi)=\mathcal W(\Psi_{0},\Psi_{0})(x,\xi)=2^{n/2}\int_{\R^n} e^{-\pi(\val{x+z/2}^2+\val{x-z/2}^2)} e^{-2i\pi z \xi}dz
\\=
2^{n/2} e^{-2\pi \val x^2}\int_{\R^n} e^{-\pi z^2/2} e^{-2i\pi z \xi}dz
=
2^{n} e^{-2\pi \val x^2}e^{-2\pi\val \xi^2}=\chi(x,\xi).
\end{multline*}
\index{anti-Wick quantization}
\index{{~\bf Notations}!$\ops{\rm aw}{a}$}
The anti-Wick quantization  of a symbol $\text{\tt a}$ is defined as  (see e.g. M. Shubin's book \cite{MR1852334})
\begin{equation}\label{awick}
\ops{\rm aw}{\text{\tt a}}
=\int_{\RZ}\text{\tt a}(Y) \Sigma_{Y} dY,
\end{equation}
where 
$\Sigma_{Y}$ is the rank-one orthogonal projection given
by
\begin{equation}\label{}
\Sigma_{y,\eta} u=\poscal{u}{\tau_{y,\eta}\Psi_{0}}\tau_{y,\eta}\Psi_{0}.
\end{equation}
\end{rem}
\begin{rem}
 \rm
 It is interesting to notice that to produce non-negativity of the operator with Weyl symbol $\text{\tt a}\ast \Gamma_{\chi}$ when $\text{\tt a}$ is a non-negative function, we do not use the non-negativity of $\Gamma_{\chi}$ as a function, which by the way does not always hold (except in the Gaussian cases),
 but we use the fact that the quantization of $\Gamma_{\chi}$ is non-negative, as it is defined as $\opw{\bar \chi\sharp\chi}=(\opw\chi)^{*}\ \opw\chi$.
\end{rem}
\begin{rem}\label{rem115}
 \rm
Another important remark 
is concerned with the Taylor expansion of $\text{\tt a}\ast \Gamma_{\chi}$: we have 
\begin{align*}
(\text{\tt a}\ast \Gamma_{\chi})(X)&=\int \text{\tt a}(X-Y) \Gamma_{\chi}(Y) dY
=\int \text{\tt a}(X+Y) \Gamma_{\chi}(Y) dY
\\
&=\int \Bigl(\text{\tt a}(X) +\text{\tt a}'(X)Y+\int_{0}^1(1-\theta)\text{\tt a}''(X+\theta Y) Y^2\Bigr)\Gamma_{\chi}(Y) dY
\\
&=\text{\tt a}(X)+
\int\!\!\!\int_{0}^1(1-\theta)\text{\tt a}''(X+\theta Y) Y^2
\Gamma_{\chi}(Y) dY.
\end{align*}  
As a result the difference  $(\text{\tt a}\ast \Gamma_{\chi})-\text{\tt a}$ depends only on the second derivative of $\text{\tt a}$.
If for instance $\text{\tt a}$ is a semi-classical symbol, i.e. a  smooth function of $(x,\xi)$ depending on $h\in (0,1]$ such that
\begin{equation}\label{semi00}
\forall (\alpha, \beta)\in \N^n\times \N^n,\quad
\sup_{(x,\xi)\in \RZ, h\in (0,1]}\val{(\p_{x}^\alpha\p_{\xi}^\beta \text{\tt a})(x,\xi,h)}h^{-\frac {\val \alpha+\val \beta}{2}}<+\io.
\end{equation}
then the difference 
$\ops{\rm aw}{\text{\tt a}}-\opw{\text{\tt a}}$ is bounded on $L^2(\R^n)$ with an $O(h)$
operator-norm, so that if $\text{\tt a}$ happens also to be non-negative,
we find
$$
\opw{\text{\tt a}}=\underbrace{\opw{\text{\tt a}}-
\opw{\text{\tt a}\ast \Gamma_{\chi}}
}_{\substack{O(h)\\
\text{as an operator,}
\\
\text{cf. Theorem \ref{thm.117jhg}
}}}+\underbrace{
\opw{\text{\tt a}\ast \Gamma_{\chi}}
}_{\substack{\ge 0\\
\text{as an operator}
}},
$$
and we obtain  a version of the so-called Sharp G\aa rding Inequality,
\begin{equation}\label{}
\opw{\text{\tt a}} +C h\ge 0\quad \text{(as an operator)}.
\end{equation}
\end{rem}
\begin{theorem}
 Let $\chi$ be an even function in the Schwartz space $\mathscr S(\RZ)$ with $L^2(\RZ)$
 norm equal to 1 and let $\Gamma_{\chi}$ be given by \eqref{gamchi}.
 For 
 $a\in L^\io(\RZ)$, we define 
 \begin{equation}\label{}
\op{\chi,a}=\opw{ a\ast \Gamma_{\chi}}.
\end{equation}
Then $\op{\chi,a}$ is a bounded operator in $L^2(\R^n)$ and we have 
\begin{equation}\label{1311}
\norm{\op{\chi,a}}_{\mathcal B(L^2(\R^n))}\le \norm{a}_{L^\io(\RZ)}.
\end{equation}
Moreover, if $a$ is valued in some interval $J$ of the real line, we have
the operator inequalities
\begin{equation}\label{1312}
\inf J\le \op{\chi,a}\le \sup J.
\end{equation}
In particular if $a(x,\xi)\ge 0$ for all $(x,\xi)\in \RZ$, we have the operator-inequality
$\op{\chi,a}\ge 0$.  
\end{theorem}
\begin{nb}\rm
 The non-negativity of the anti-Wick quantization \eqref{awick} and its avatars Husimi (\cite{husimi}), Coherent States,
 Gabor wavelets (see e.g. \cite{HF-TS}),
are  particular cases of the above theorem. More information on this topic is available in Section 2.4 of the book \cite{MR2599384}. Another remark is that this result can easily be extended to matrix-valued symbols as in Remark 2 page 79 of L.~H\"ormander's \cite{MR2304165} and even to symbols valued in $\mathcal B(H)$, where $H$ is a Hilbert space.
\end{nb}
\begin{proof}
 We start with Formulas \eqref{134}, \eqref{135}, entailing
 $$
 \op{\chi,a}=
\int_{\RZ} a(Y)\Sigma_{\chi}(Y)dY,
$$
with
$
\Sigma_{\chi}(Y)=
\bigl[\opw{\chi(\cdot-Y)}\bigr]^*
\opw{\chi(\cdot-Y)}=
\tau_{Y}\opw{\bar \chi\sharp \chi}\tau_{-Y} .
$
We note that 
$$
\op{\chi,1}=\int_{\RZ}\tau_{Y}\opw{\bar \chi\sharp \chi}\tau_{-Y}dY,
$$
so has Weyl symbol
$
X\mapsto\int_{\RZ} \Gamma_{\chi}(X-Y) dY=1
$
from Lemma \ref{lem112} and thus  $\op{\chi,1}=\Id$.
We infer that for $u,v\in \mathscr S(\R^n)$, 
$$
\poscal{ \op{\chi,a} u}{v}_{L^2(\R^n)}=
\int_{\RZ} a(Y)\poscal{\opw{\chi(\cdot-Y)} u}{\opw{\chi(\cdot-Y)} v}dY,
$$
so that with any $\nu>0$,
\begin{multline*}
\val{\poscal{ \op{\chi,a} u}{v}_{L^2(\R^n)}}
\\\le \norm{a}_{L^\io(\RZ)}
\int_{\RZ}\frac12\bigl(\nu
\norm{\opw{\chi(\cdot-Y)} u}_{L^2(\R^n)}^2+\nu^{-1}\norm{\opw{\chi(\cdot-Y)} v}_{L^2(\R^n)}^2
\bigr) dY
\\
=
 \norm{a}_{L^\io(\RZ)}
 \frac12
\bigl(\nu\poscal{\op{\chi,1}u}{u}_{L^2(\R^n)}
+\nu^{-1}\poscal{\op{\chi,1}v}{v}_{L^2(\R^n)}
\bigr) 
\\=
 \norm{a}_{L^\io(\RZ)}
 \frac12
\bigl(\nu\norm{u}_{L^2(\R^n)}^2
+\nu^{-1}\norm{v}_{L^2(\R^n)}^2
\bigr),
\end{multline*}
and taking the infimum of the right-hand-side with respect to $\nu$,
we obtain
$$
\val{\poscal{ \op{\chi,a} u}{v}_{L^2(\R^n)}}
\le 
 \norm{a}_{L^\io(\RZ)}
\norm{u}_{L^2(\R^n)}\norm{v}_{L^2(\R^n)},
$$
proving 
\eqref{1311}.
To prove \eqref{1312},
it is enough to prove the last statement in the theorem
which follows immediately from \eqref{134},  \eqref{135} since each operator $\Sigma_{Y}$ is non-negative.
The proof of the theorem is complete.\end{proof}
It is nice to have examples of non-negative quantizations, but somehow more importantly, it is crucial to relate these quantizations to the mainstream quantization, that is to the Weyl quantization. This is what we do in the next theorem, dealing with semi-classical symbols.
\index{sharp G\aa rding inequality}
\begin{theorem}[Sharp G\aa rding Inequality]\label{garding} Let $a$ be a function defined on $\R^n\times\R^n\times (0,1]$ such that $a(x,\xi,h)$ is smooth for all $h\in (0,1]$ and such that
\begin{equation}\label{1314}\forall (\alpha, \beta)\in \N^n\times\N^n,\quad
\sup_{(x,\xi,h)\in\R^n\times\R^n\times(0,1]}\val{(\p_{x}^\alpha\p_{\xi}^\beta a)(x,\xi, h)} h^{-\val \beta}<+\io.
\end{equation}
Let us assume that the function $a$ is valued in $\R_{+}$. Then, there exists a constant $C$ such that 
\begin{equation}\label{}
\opw{a}+C h\ge 0.
\end{equation}
\end{theorem}
\begin{proof}
 We have given a proof of this result in Remark \ref{rem115} but with a different definition for a semi-classical symbol (see \eqref{semi00}).
 Starting with our definition above in \eqref{1314}, we define
\begin{equation}\label{semi11}
 b(x,\xi,h)= a(h^{1/2} x, h^{-1/2} \xi, h),
\end{equation}
 and we see that $b$ satisfies the estimates  \eqref{semi00} and is a non-negative function  so that, applying Remark  \ref{rem115}, we can find a constant $C$ such that 
 $$
\opw{b}+C h\ge 0.
 $$
 We note now that Segal's formula \eqref{segal} applied to the symplectic mapping
 $$(x,\xi)\mapsto (h^{1/2}x, h^{-1/2}\xi),$$
 shows that $\opw{b}$ is unitarily equivalent to $\opw{a}$, providing the sought result.
\end{proof}
\begin{nb}\rm
 Several versions of the above theorem can be found in the literature,
 in particular Theorem 18.1.14 in \cite{MR2304165}. The first proof of this result was given in 1966 by L.~H\"ormander in \cite{MR233064} for scalar-valued symbols
and a proof for systems was given by 
P. Lax \& L. Nirenberg in \cite{MR0234105} on the same year.
Far-reaching refinements of that inequality were given by C. Fefferman \& D.H. Phong,
 who proved in \cite{MR507931} in 1978 that, under the same assumption as in Theorem \ref{garding} for scalar-valued symbols, they obtain the much stronger
\index{Fefferman-Phong inequality}
 \begin{equation}\label{fefpho}
 \opw{a}+C h^2\ge 0.
\end{equation}
 A thorough discussion of these questions is given in Section 18.6 of \cite{MR2304165} and in Section 2.5 of \cite{MR2599384} (see also \cite{MR3333053}).\end{nb}
\subsection{Examples}
\subsubsection{Hermite functions}
\index{Hermite functions}
\index{Laguerre polynomials}
We can easily calculate the Wigner distribution of Hermite functions and since the Wigner distributions respect tensor products as partial Fourier transforms, it is enough to do so in one dimension. 
With $\psi_{k}$ given in \eqref{hermite}, the Wigner distribution $\mathcal W(\psi_{k},\psi_{k})$ appears as the Weyl symbol of $\mathbb P_{k;1}=\mathbb P_{k}$ as defined in \eqref{hardec}.
We find that the Weyl   symbol of $\mathbb P_{0;n}$, following \eqref{6.knb44}, is 
$$
2^n e^{-2\pi(\val x^2+\val\xi^2)}.
$$
More generally, the paper \cite{MR1643938} provides in one dimension
\begin{equation}\label{wigherm}
\mathcal W(\psi_{k},\psi_{k})(x,\xi)=
(-1)^k 2 e^{-2\pi (x^2+\xi ^2)} L_{k}\bigl(4\pi  (x^2+\xi ^2)\bigr),
\end{equation}
where $L_{k}$ is the standard Laguerre polynomial with degree $k$ (see \eqref{laguerre}). As a result, the Weyl symbol of $\mathbb P_{k;n}$ is equal to $\pi_{k,n}(x,\xi)$ with 
\begin{equation}\label{}
\pi_{k,n}(x,\xi)=
(-1)^k 2^n e^{-2\pi (\val x^2+\val \xi ^2)}\sum_{\alpha\in \N^n, \val \alpha=k}
\prod_{1\le j\le n} L_{\alpha_{j}}\bigl(4\pi(x_{j}^2+\xi_{j}^2)\bigr).
\end{equation}
Note that the leading term in the polynomial $(-1)^kL_{k}(t)$ is $t^k/k!$ and this implies that the set $$
\{(x,\xi)\in \R^2, \mathcal W(\psi_{k},\psi_{k})(x,\xi)<0\}
$$ where  $\mathcal W(\psi_{k},\psi_{k})$ is given by \eqref{wigherm} is a relatively compact
open subset of $\R^2$:
Indeed we have
$$
W(\psi_{k}, \psi_{k})(X)=2e^{-2\pi \val X^{2}}\Bigl\{\frac{(4\pi \val X^{2})^{k}}{k!}\Bigr\}
\underbrace{\Bigl(1+\sum_{0\le l\le k-1} a_{l}(4\pi\val{X}^{2})^{-(k-l)}\Bigr)}_{
\ge 1/2 \text{ for $\val X\ge R_{0}$}
} 
$$
which implies that
$
\{X\in \R^{2}, \val X\ge \max(R_{0}, 1)\}\subset\{X\in \R^{2},W(\psi_{k}, \psi_{k})(X) >0 \}
$
and thus $\{W(\psi_{k}, \psi_{k})(X)\le 0\}\subset\{
\val X< \max(R_{0}, 1)\}.
$
\subsubsection{One-sided exponentials}
Let us define for $ a>0$,
\begin{equation}\label{}
f_{a}(t) =H(t) a^{1/2}e^{-at/2}.
\end{equation}
We have 
\begin{align}
\mathcal W(f_{a}, f_{a})(x,\xi)&=a H(x)\int_{\val z\le 2x}e^{-2i\pi z \xi}
e^{-\frac{a}{2}(x+z/2)}e^{-\frac{a}{2}(x-z/2)} dz
\notag\\
&=a H(x) e^{-x a}\int_{\val z\le 2x}e^{-2i\pi z \xi}
dz
\notag\\
&=2a H(x) e^{-x a}\int_{0}^{2x}
\cos \left(z2\pi \xi\right)
dz
\notag\\
&=a H(x) e^{-x a}
\frac{\sin \left(4\pi x \xi\right)}{\pi \xi}.
\label{exa1111}
\end{align}
We can check
$$
\iint\mathcal W(f_{a}, f_{a})(x,\xi) dx d\xi=\frac{a}{\pi}\int_{x=0}^{+\io}
e^{-ax}\int \frac{\sin \left(4\pi x \xi\right)}{ \xi} d\xi dx=1=\norm{f_{a}}_{L^2(\R)}^2,
$$
and  since 
\begin{equation}
\int_{\R}\frac{\sin^2t}{t^2} dt=\pi,
\end{equation}
we verify (see Lemma \ref{lem.113esd} and \eqref{wigner}),
$$
\iint\mathcal W(f_{a}, f_{a})(x,\xi)^2 dx d\xi=\frac{a^2}{\pi^2}\int_{x=0}^{+\io}
e^{-2ax}\int \frac{\sin^2 \left(4\pi x \xi\right)}{ \xi^2} d\xi dx=1=\norm{f_{a}}_{L^2(\R)}^4.
$$
On the other hand, the ambiguity function  $\mathcal A(f_{a},f_{a})$ is the inverse Fourier transform of $\mathcal W$ and we have 
\begin{multline*}
\mathcal A(f_{a},f_{a})(\eta, y)= \frac a \pi\iint
H(x) e^{-x (a-2i\pi \eta)}
\frac{\sin \xi}{ \xi}
e^{2i \pi \frac{y}{4\pi x}\xi} dx d\xi
\\
=  a \int_{\val y/2}^{+\io} e^{-x (a-2i\pi \eta)}
dx =\frac{a e^{-\frac12\val y(a-2i\pi \eta)}}{a-2i\pi \eta},
\end{multline*}
which corresponds to Formula (9) in \cite{GJM}
noting that with our notations, we have 
$$
\mathcal A(f,f)(\eta, y)=\wt{\mathcal A}(f,f)(y,-\eta),
$$
where $\wt{
\mathcal A}(f,f)
$ is the normalization chosen in \cite{GJM}.
Going back to the Wigner distribution, that simple example is interesting since we have
\begin{equation*}
\bigl\{(x,\xi), \mathcal W(f_{a},f_{a})(x,\xi)<0\bigr\}
=\cup_{k\in \N}\bigl\{(x,\xi)\in (0,+\io)\times \R^*,\
\frac{k}{2}+\frac{1}{4}
<x\val \xi<\frac{k}{2}+\frac{1}{2}
\bigr\},
\end{equation*}
and we see that the Lebesgue measure of
$$E_{k}=\bigl\{(x,\xi)\in (0,+\io)\times \R^*,\
\frac{k}{2}+\frac{1}{4}
<x\val \xi<\frac{k}{2}+\frac{1}{2}
\bigr\},
$$
is infinite since 
$$
\val{E_{k}}=2\int_{0}^{+\io}\frac{dx}{4x}=+\io.
$$
Moreover the function 
$\mathcal W(f_{a}, f_{a})(x,\xi) $ does not belong to $L^1(\R^2)$ since 
$$
\iint H(x) e^{-x a}
\Val{\frac{\sin \left(4\pi x \xi\right)}{\pi \xi}} dx d\xi
\ge \iint_{(0,+\io)^2}
e^{-x a}\Val{\frac{\sin \eta}{\pi\eta}} dxd\eta=+\io.
$$
As a consequence, we have, 
using the notation for $\alpha\in \R$,
\begin{equation}
\alpha_{\pm}=\max(\pm\alpha, 0),
\end{equation}
\begin{equation}\label{}
\iint \bigl(\mathcal W(f_{a}, f_{a})(x,\xi)\bigr)_{+} dx d\xi=
\iint \bigl(\mathcal W(f_{a}, f_{a})(x,\xi)\bigr)_{-} dx d\xi=+\io,
\end{equation}
since the real-valued function $\mathcal W(f_{a}, f_{a})$
does not belong to $L^1(\R^2)$ and is such that 
$$
\iint\mathcal W(f_{a}, f_{a})(x,\xi)dx d\xi= \norm{f_{a}}^2_{L^2(\R)}=1.
$$
We shall see in Section \ref{sec.jhae22}
several important consequences of that phenomenon for the quantization of the indicatrix of some subsets of $\R^2$, such as 
\begin{equation}\label{}
E_{\pm}=\bigl\{(x,\xi), \pm \mathcal W(f_{a}, f_{a})(x,\xi)>0\bigr\}.
\end{equation}
\subsubsection{Box functions}
\index{box functions}
We start with 
\begin{equation}\label{}
\beta_{0}(t)=\mathbf 1_{[-\frac 12, \frac 12]}(t),
\end{equation}
for which a straightforward calculation gives
\begin{equation}\label{}
\mathcal W(\beta_{0}, \beta_{0})(x,\xi)=\mathbf 1_{[-\frac 12, \frac 12]}(x)\frac{
\sin\bigl(2\pi(1-2\val x) \xi\bigr)
}{\pi \xi}.
\end{equation}
More generally for real parameters $a\le b$, defining
\begin{multline*}
\beta=(b-a)^{-1/2}\mathbf 1_{[a,b]}(x) e^{2i\pi \omega x}, \quad\text{we find}\\
\mathcal W(\beta,\beta)(x,\xi)=[(b-a)\pi(\xi-\omega)]^{-1}
\Bigl(
\mathbf 1_{[a,\frac{a+b}2]}(x)\sin[4\pi(\xi-\omega)(x-a)]
\\+\mathbf 1_{[\frac{a+b}2,b]}(x)\sin[4\pi(\xi-\omega)(b-x)]
\Bigr).
\end{multline*}
Checking now
\begin{equation}\label{}
\beta_{1}(t)=\mathbf 1_{[-\frac 12, \frac 12]}(t)\sign t,
\end{equation}
we find after a simple (but this time a bit  tedious) calculation
\begin{multline}\label{}
\mathcal W(\beta_{1}, \beta_{1})(x,\xi)=\mathbf 1\bigl(\val x\le \frac14\bigr)\ \frac{2\sin(4\pi \val x \xi)
-\sin\bigl(2\pi(1-2\val x) \xi\bigr)}{\pi \xi}
\\
+\mathbf 1\bigl(\frac14\le \val x\le \frac12\bigr)\ \frac{
\sin\bigl(2\pi(1-2\val x) \xi\bigr)
}{\pi \xi}.
\end{multline}
\subsection{Integrals of the Wigner distribution on subsets of the phase space}
\begin{lem}
 Let $E$ be a measurable subset with finite Lebesgue measure
 of the phase space $\R^n\times \R^n$ and let $\mathbf 1_{E}$ be the indicator function of the set $E$. Then the operator with Weyl symbol  $\mathbf 1_{E}$  is bounded self-adjoint on $L^2(\R^n)$ and for any $u\in L^2(\R^n)$, we have
 \begin{equation}\label{flandrin0}
\poscal{\opw{\mathbf 1_{E}} u}{u}_{L^2(\R^n)}=\iint _{E}\mathcal W(u,u)(x,\xi) dx d\xi.
\end{equation}
\end{lem}
\begin{proof}
 It follows immediately from \eqref{eza654} and \eqref{norm01}.
\end{proof}
\begin{rem}
 A consequence of the above formula is that a spectral analysis of the operator
 $\opw{\mathbf 1_{E}}$
 would display interesting extremalization properties for the right-hand-side of 
 \eqref{flandrin0};
 for instance if 
 $$
 \lambda_{-}=\inf\bigl(\text{spectrum}(\opw{\mathbf 1_{E}})\bigr), 
 \quad
 \lambda_{+}=\sup\bigl(\text{spectrum}(\opw{\mathbf 1_{E}})\bigr),
 $$
 we obtain that for $u$ normalized in $L^{2}(\R^{n})$, we have 
 \begin{equation}\label{}
 \lambda_{-} \le 
 \iint _{E}\mathcal W(u,u)(x,\xi) dx d\xi\le \lambda_{+}.
\end{equation}
In particular, if $\lambda_{-}$ is an eigenvalue related to a normalized eigenfunction $u_{-}$,
 (resp. if $\lambda_{+}$ is an eigenvalue related to a normalized eigenfunction $u_{+}$),
 we get for all $u$ normalized in $L^{2}(\R^{n})$,
 \begin{multline}\label{}
 \iint _{E}\mathcal W(u_{-},u_{-})(x,\xi) dx d\xi
 \le \iint _{E}\mathcal W(u,u)(x,\xi) dx d\xi
 \\
\text{resp.}
 \le 
 \iint _{E}\mathcal W(u_{+},u_{+})(x,\xi) dx d\xi.
\end{multline}
 \end{rem}
We shall see below several examples where the operator
$\opw{\mathbf 1_{E}}$ is bounded on $ L^2(\R^n)$ with an $E$ having infinite Lebesgue measure. We may note in particular that 
$$
\opw{\mathbf 1_{\RZ}}=\Id,
$$
and for a given non-zero linear form $L(x,\xi)$ on $\RZ$ and 
\begin{equation}\label{trivial}
E=\{(x,\xi)\in \RZ, L(x,\xi) \in J\},\quad\text{where $J$ is a subset of $\R$},
\end{equation}
we may find affine  symplectic coordinates $(y,\eta)$
on $\RZ$ such that $L(x,\xi)=y_{1}$, implying with \eqref{segal} that 
$\opw{\mathbf 1_{E}}$ is unitarily equivalent to the orthogonal projection
$$
u\mapsto u(y)\mathbf 1_{J}(y_{1}).
$$
Although in that case, the quantization of the indicatrix of $E$ given by \eqref{trivial}
is trivial,
we shall see below that in many cases, including some rather explicit ones,
the Weyl  quantization of the rough Hamiltonian $\mathbf 1_{E}(x,\xi)$
could be far from a projection and may have a rather complicated spectrum
with a supremum which could be strictly larger than 1 and an infimum which could be negative.
\par
In some sense, although we have the  trivial identity 
$\mathbf 1_{E}(x,\xi)^2=\mathbf 1_{E}(x,\xi)$, we shall see that the quantization process by the Weyl formula is destroying that property; 
to understand integrals of the Wigner distribution on subsets of the phase space,
Formula \eqref{flandrin0} forces us to consider the Weyl quantization  
of the function $\mathbf 1_{E}(x,\xi)$ and the Heisenberg Uncertainty Principle shows that non-commutation properties
are  governing operators
and these properties are of course distorting the classical identities
satisfied by classical Hamiltonians.
\par
 We must point out as well 
that we do not have here at our disposal a semi-classical version 
of our quantization which could ensure some bridge between classical properties and operator-theoretic results as it is the case for the quantization of nice smooth semi-classical symbols depending on a small parameter $h$ such as a $\moo$
function $a(x,\xi, h)$ satisfying \eqref{1314}.
In particular for a symbol $a$ satisfying  \eqref{1314},
we have the following result: if for all $(x,\xi,h)\in \R^n\times\R^n\times (0,1]$ we have
$a(x,\xi,h)\le 1$, then there exists a semi-norm $C$ of the symbol $a$ such that
\begin{equation}\label{}
\Id-\opw{a}+ Ch^2\ge 0\qquad\text{i.e.}\quad
\opw{a}\le \Id + C h^2,
\end{equation}
an inequality following from the Fefferman-Phong Inequality 
(cf.\hskip1pt\eqref{fefpho})
which implies as well the following lemma.
\begin{lem}
Let $a$ be a semi-classical symbol of order 0, i.e. a smooth function satisfying 
\eqref{1314}
 such that for  all $(x,\xi,h)\in \R^n\times\R^n\times (0,1]$ we have
$$
0\le a(x,\xi, h)\le 1.
$$
Then there exists a semi-norm $C$ of the symbol $a$ such that 
$$
-Ch^2\le \opw{a}\le \Id + C h^2.
$$
 \end{lem}
\section{Quantization of radial functions and Mehler's formula}
\index{Mehler formula}
This section and the following  are essentially based upon the author's paper \cite{MR3951885}.
\subsection{Basic formulas in one dimension}
In this section,
 we work in one dimension and consider a function $F$ in the Schwartz class of $\R$.
 We want to calculate somewhat explicitly the Weyl quantization of $F(x^{2}+\xi^{2})$
 and also extend that computation to the case where $F$ is merely $L^{\io}(\mathbb R)$. 
 We have, say for $F$ in the Wiener algebra $\mathscr W(\R)=${ \tt Fourier}$\bigl(L^{1}(\R)\bigr)$,
 $$
 \opw{F(x^{2}+\xi^{2})}=\int_{\R}\hat F(\tau) 
 \opw{e^{2i\pi \tau(x^{2}+\xi^{2})}}
  d\tau,
 $$
as an absolutely converging integral of  a function defined on $\R$ (equipped with the Lebesgue measure)
valued in $\mathcal B(L^{2}(\R))$
(bounded endomorphisms of $L^{2}(\R)$).
In fact applying Mehler's Formula \eqref{mehler+n},
we find 
$$
\underbrace{ 
\opw{e^{2i\pi \tau(x^{2}+\xi^{2})}}
}_{\substack{\text{operator with Weyl symbol}
\\
e^{2i\pi \tau(x^{2}+\xi^{2})}
}}=\cos(\arctan \tau) \underbrace{e^{2i\pi (\arctan \tau) 
\opw{x^{2}+\xi^{2}}
}}_{\substack{
\text{exponential }e^{iM},\\\text{ with $M$ self-adjoint  operator}
\\=2\pi(\arctan \tau) \opw{x^{2}+\xi^{2}}
}},
$$
so that, using the spectral decomposition \eqref{hardec} 
of the Harmonic Oscillator $$\opw{\pi(x^{2}+\xi^{2})},$$
we get 
 \begin{align*}
\opw{F(x^{2}+\xi^{2})}
&=
 \int_{\R}\hat F(\tau) \sum_{k\ge 0}e^{2i(\arctan \tau)(k+\frac12)}\mathbb P_{k}\frac{d\tau}{\sqrt{1+\tau^{2}}}\\
 &=
  \sum_{k\ge 0}\int_{\R}\hat F(\tau) e^{2i(k+\frac12)\arctan \tau}\frac{d\tau}{\sqrt{1+\tau^{2}}}
  \mathbb P_{k},
 \end{align*}
where the use of Fubini theorem is justified by
$$
\int_{\R}\val{\hat F(\tau)}\frac{d\tau}{\sqrt{1+\tau^{2}}}<+\io,\quad \mathbb P_{k}\ge 0, \sum_{k\ge 0} \mathbb P_{k}=\Id.
$$
We have 
\begin{multline*}
\int_{\R}\hat F(\tau) e^{2i(k+\frac12)\arctan \tau}\frac{d\tau}{\sqrt{1+\tau^{2}}}
\\=\int_{\R}\hat F(\tau) 
\bigl(
\cos(\arctan \tau\bigr)+i\sin(\arctan \tau)
\bigr)^{2k+1}
\frac{d\tau}{\sqrt{1+\tau^{2}}},
\end{multline*}
and, using Section \ref{sec.arctan}, we get 
$$
\int_{\R}\hat F(\tau) e^{2i(k+\frac12)\arctan \tau}\frac{d\tau}{\sqrt{1+\tau^{2}}}
=\int_{\R}\hat F(\tau) 
\bigl(
1+i\tau
\bigr)^{2k+1}
\frac{d\tau}{(1+\tau^{2})^{k+1}}.
$$
We have proven the following lemma.
\begin{lem}\label{lemm21}
 Let $F$ be a tempered distribution on $\R$
 such that 
 $\hat F$ is locally integrable and such that
 \begin{equation}\label{condit}
\int_{\R}\val{\hat F(\tau)}\frac{d\tau}{\sqrt{1+\tau^{2}}}<+\io.
\end{equation}
 Then the operator $\opw{F(x^{2}+\xi^{2})}$ has the spectral decomposition 
 \begin{align}\label{for212}
\opw{F(x^{2}+\xi^{2})}& =\sum_{k\ge 0}
\int_{\R}
\frac{\hat F(\tau) (1+i\tau)^{2k+1} }{(1+\tau^{2})^{k+1}}d\tau
 \ \mathbb P_{k}
 \\
 & =\sum_{k\ge 0}
\int_{\R}
\frac{\hat F(\tau) (1+i\tau)^{k} }{(1-i\tau)^{k+1}}d\tau
 \ \mathbb P_{k},
\end{align}
where the orthogonal projections $\mathbb P_{k}$ are defined in \eqref{hardec}. 
\end{lem}
\subsection{Higher dimensional questions}
We work now  in $n$ dimensions and consider a function $F$ in the Schwartz class of $\R$.
 We want to calculate somewhat explicitly the Weyl quantization of $F\bigl(\sum_{1\le j\le n}\mu_{j}(x_{j}^{2}+\xi_{j}^{2})
 \bigr)$, where the $\mu_{j}$ are positive parameters, denoted by 
 $$
\OPW{F(\sum_{1\le j\le n}\mu_{j}(x_{j}^{2}+\xi_{j}^{2}))}, 
\quad q_{\mu}(x,\xi)=\sum_{1\le j\le n}\mu_{j}(x_{j}^{2}+\xi_{j}^{2}),
 $$
 and also extend that computation to the case where $F$ is merely $L^{\io}(\mathbb R)$. 
 We have, say for $F$ in the Wiener algebra $\mathscr{ W}(\R)=${ \tt Fourier}$\bigl(L^{1}(\R)\bigr)$,
 $$
 \OPW{F\bigl(q_{\mu}(x,\xi)\bigr)}
 =\int_{\R}\hat F(\tau) 
\OPW{ e^{2i\pi \tau\sum_{1\le j\le n}\mu_{j}(x_{j}^{2}+\xi_{j}^{2})}}
 d\tau,
 $$
as an absolutely converging integral of  a function defined on $\R$ (equipped with the Lebesgue measure)
valued in $\mathcal B(L^{2}(\R^n))$
(bounded endomorphisms of $L^{2}(\R^n)$).
In fact applying Mehler's Formula \eqref{mehler+n},
we find  by tensorisation,
\begin{equation}\label{mehler11}
\underbrace{ 
\OPW{e^{2i\pi \tau\sum_{1\le j\le n}\mu_{j}(x_{j}^{2}+\xi_{j}^{2})}}
}_{\substack{\text{operator with Weyl symbol}
\\
e^{2i\pi \tau q_{\mu}(x,\xi)}
}}=\prod_{1\le j\le n}\cos(\arctan( \tau \mu_{j})) \underbrace{e^{2i\pi (\arctan (\tau \mu_{j})) \opw{x_{j}^{2}+\xi_{j}^{2}}
}}_{\substack{
\text{exponential }e^{iM_{j}},\\\text{ with $M_{j}$ self-adjoint  operator}
\\=2\pi(\arctan (\tau \mu_{j})) 
\opw{x_{j}^{2}+\xi_{j}^{2}}
}},
\end{equation}
so that, using the spectral decomposition \eqref{app.paln}
of the Harmonic Oscillator 
we get 
 \begin{align*}
\OPW{F(q_{\mu}(x,\xi))}
&=
 \int_{\R}\hat F(\tau) \sum_{\alpha\in \N^n}
 \prod_{1\le j\le n}e^{2i(\arctan( \tau \mu_{j}))(\alpha_{j}+\frac12)}\mathbb P_{\alpha_{j}}
\frac1{\sqrt{1+(\tau \mu_{j})^{2}}} {d\tau}\\
 &=
  \sum_{\alpha\in \N^n}\int_{\R}\hat F(\tau) 
  \prod_{1\le j\le n}e^{2i(\alpha_{j}+\frac12)\arctan( \tau \mu_{j})}
  \frac1{\sqrt{1+(\tau \mu_{j})^{2}}} 
  {d\tau}
    \mathbb P_{\alpha},
 \end{align*}
where the use of Fubini theorem is justified by
$$
\int_{\R}\val{\hat F(\tau)}\frac{d\tau}{\sqrt{1+\tau^{2}}}<+\io,\quad \mathbb P_{\alpha}\ge 0, \sum_{\alpha} \mathbb P_{\alpha}=\Id.
$$
We have 
\begin{multline*}
\int_{\R}\hat F(\tau) 
\prod_{1\le j\le n}e^{2i(\alpha_{j}+\frac12)\arctan( \tau \mu_{j})}
  \frac1{\sqrt{1+(\tau \mu_{j})^{2}}} 
d\tau
\\=\int_{\R}\hat F(\tau) \prod_{1\le j\le n}
\bigl(
\cos(\arctan(\mu_{j} \tau)\bigr)+i\sin(\arctan(\mu_{j} \tau))
\bigr)^{2\alpha_{j}+1}
  \frac1{\sqrt{1+(\tau \mu_{j})^{2}}} 
d\tau
\end{multline*}
and, using Section \ref{sec.arctan}, we get 
\begin{multline*}
\int_{\R}\hat F(\tau) 
\prod_{1\le j\le n}e^{2i(\alpha_{j}+\frac12)\arctan( \tau \mu_{j})}
  \frac1{\sqrt{1+(\tau \mu_{j})^{2}}} 
d\tau
\\=\int_{\R}\hat F(\tau) 
\prod_{1\le j\le n}
\frac{(1+i\tau \mu_{j})^{2\alpha_{j}+1}}{(1+(\tau\mu_{j})^2)^{\alpha_{j}+\frac12}}
  \frac1{\sqrt{1+(\tau \mu_{j})^{2}}} 
d\tau.\end{multline*}
We have proven the following lemma.
\begin{lem}\label{lemm21n}
 Let $F$ be a tempered distribution on $\R$
 such that 
 $\hat F$ is locally integrable and such that
 \begin{equation}\label{conditn}
\int_{\R}\val{\hat F(\tau)}\frac{d\tau}{\sqrt{1+\tau^{2}}}<+\io.
\end{equation}
 Then the operator $ \OPW{F(\sum_{1\le j\le n}\mu_{j}(x_{j}^{2}+\xi_{j}^{2})
 )}$ has the spectral decomposition 
 \begin{multline}\label{for212n}
\OPW{F\bigl(\sum_{1\le j\le n}\mu_{j}(x_{j}^{2}+\xi_{j}^{2})\bigr)}
 =\sum_{\alpha\in \N^n}
\int_{\R}\hat F(\tau) 
\prod_{1\le j\le n}\frac{(1+i\tau\mu_{j})^{2\alpha_{j}+1} }{(1+\tau^{2}\mu_{j}^2)^{\alpha_{j}+1}}d\tau
 \ \mathbb P_{\alpha}
 \\
 =\sum_{\alpha\in \N^n}
\int_{\R}\hat F(\tau) 
\prod_{1\le j\le n}\frac{(1+i\tau\mu_{j})^{\alpha_{j}} }{(1-i\tau \mu_{j})^{\alpha_{j}+1}}d\tau
 \ \mathbb P_{\alpha},
\end{multline}
where $\mathbb P_{\alpha}$ is the rank-one orthogonal projection onto $\Psi_{\alpha}$ given by
\eqref{psidim}.
\end{lem}
\begin{lem}\label{lemm21n+}
 Let $F$ be as in Lemma \ref{lemm21n+} and let us assume that all the $\mu_{j}$ are equal to $\mu$ (positive).
 Then 
 \begin{equation}\label{}
\OPW{F\bigl(\mu\sum_{1\le j\le n}(x_{j}^{2}+\xi_{j}^{2})\bigr)}
 =\sum_{k\ge 0}\int_{\R}\hat F(\tau) \frac{(1+i\tau\mu)^{k} }{(1-i\tau \mu)^{k+n}}d\tau
 \mathbb P_{k;n},
\end{equation}
with
 \begin{equation}\label{}
  \mathbb P_{k;n}=\sum_{\substack{\alpha\in \N^n\\\val \alpha=k}}\mathbb P_{\alpha},
\end{equation}
where $\mathbb P_{\alpha}$ is the rank-one orthogonal projection onto $\Psi_{\alpha}$ given by
\eqref{psidim}.
\end{lem}
\begin{proof}
 With all the $\mu_{j}$ equal to $\mu>0$, we find
 $$
 \prod_{1\le j\le n}\frac{(1+i\tau\mu_{j})^{\alpha_{j}} }{(1-i\tau \mu_{j})^{\alpha_{j}+1}}=
  \prod_{1\le j\le n}\frac{(1+i\tau\mu)^{\alpha_{j}} }{(1-i\tau \mu)^{\alpha_{j}+1}}
  =\frac{(1+i\tau\mu)^{\val\alpha} }{(1-i\tau \mu)^{\val\alpha+n}},
 $$
 which depends only on $\val \alpha$, so that  applying the previous lemma gives
 $$
 \Bigl(F\bigl(\mu\sum_{1\le j\le n}(x_{j}^{2}+\xi_{j}^{2})
 \bigr)\Bigr)^{w}
 =\sum_{k\ge 0}\int_{\R}\hat F(\tau) \frac{(1+i\tau\mu)^{k} }{(1-i\tau \mu)^{k+n}}d\tau
 \mathbb P_{k;n},
 $$
 giving the sought result.
\end{proof}
\section{Conics with eccentricity smaller than 1}
\index{indicatrix of a disc}
\subsection{Indicatrix of a disc}\label{sec.22one}
\label{secdim}
Let us assume now that, with some $a\ge 0$,
$$
F=\indic{[-\frac a{2\pi}, \frac a{2\pi}]},
\text{
so that 
}\quad 
F(x^{2}+\xi^{2})=\mathbf 1_{\{{2\pi(x^{2}+\xi^{2})\le a}\}}.
$$
According to Section \ref{12345},
we have
$
\hat F(\tau) = \frac{\sin a\tau}{\pi \tau},
$
so that \eqref{condit} holds true. We find in this case,
\begin{equation}\label{556699+}
\opw{F(x^{2}+\xi^{2})}
=\sum_{k\ge 0} F_{k}(a)\mathbb P_{k},\quad 
F_{k}(a)=
\int_{\R}
\frac{\sin a\tau}{\pi \tau}
\frac{(1+i\tau)^{k}}{(1-i\tau)^{k+1}}d\tau,
\end{equation}
so that (note that $F_{k}(a)$ is real-valued since $F$ is real-valued and thus the operator  
$\opw{F(x^{2}+\xi^{2})}$ is self-adjoint), and 
for $a>0$,
using the result \eqref{fprime} of the Appendix page \pageref{sub555+},
we obtain 
\begin{align*}
F'_{k}(a)&=\frac1{\pi}\int_{\R}
{\cos a\tau}
\frac{(1+i\tau)^{k}}{(1-i\tau)^{k+1}}d\tau
\\
&=\frac1{2\pi}\int_{\R}
e^{i a \tau}\left\{
\frac{(1+i\tau)^{k}}{(1-i\tau)^{k+1}}
+\frac{(1-i\tau)^{k}}{(1+i\tau)^{k+1}}\right\} d\tau
\\&=\frac1{2\pi}\int_{\R}
e^{i a \tau}\left\{
\frac{i^{k}(\tau-i)^{k}}{(-i)^{k+1}(\tau+i)^{k+1}}
+\frac{(-i)^{k}(\tau+i)^{k}}{i^{k+1}(\tau-i)^{k+1}}\right\} d\tau
\\&=\frac{(-1)^{k}}{2i\pi}\int_{\R}
e^{i a \tau}\left\{-
\frac{(\tau-i)^{k}}{(\tau+i)^{k+1}}
+\frac{(\tau+i)^{k}}{(\tau-i)^{k+1}}\right\} d\tau.
\end{align*}
We shall now calculate explicitly both integrals above: let $1<R$ be given and let us consider the closed path
\begin{equation}\label{pathcir}
\gamma_{R}=[-R,R]
\cup\underbrace{\{R e^{i\theta}\}_{0\le  \theta\le \pi}}_{\gamma_{2;R}}.
\end{equation}
\begin{figure}[ht]
\centering
\scalebox{0.35}{\includegraphics{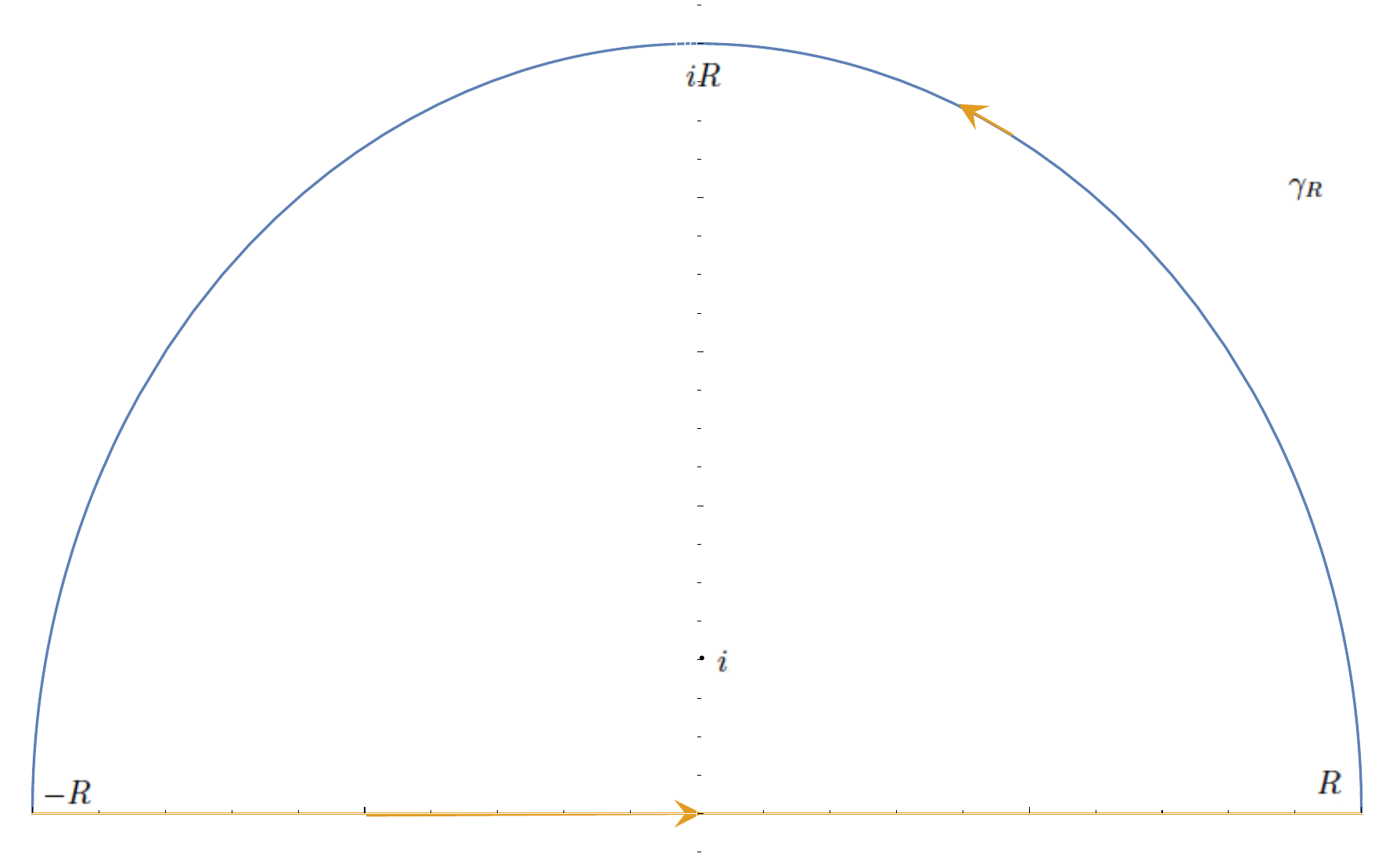}}\par
\caption{$
\gamma_{R}=[-R,R]\cup
\{R e^{i\theta}\}_{0\le  \theta\le \pi}
$}
\label{pic1}
\end{figure}
We have 
\begin{multline*}
\frac1{2 i\pi}\int_{\gamma_{R}}e^{i a \tau}\left\{-
\frac{(\tau-i)^{k}}{(\tau+i)^{k+1}}
+\frac{(\tau+i)^{k}}{(\tau-i)^{k+1}}\right\} d\tau=\text{\tt Res}(e^{ia\tau}\frac{(\tau+i)^{k}}{(\tau-i)^{k+1}} ;i)
\\
= \frac1{k!}(\frac{d}{d\tau})^{k}\bigl\{ e^{i a\tau}(\tau+i)^{k}\bigr\}_{\vert \tau=i},
\end{multline*}
and we note that, for $a>0$, 
$$
\lim_{R\rightarrow+\io}
\int_{\gamma_{2;R}}e^{i a \tau}\left\{-
\frac{(\tau-i)^{k}}{(\tau+i)^{k+1}}
+\frac{(\tau+i)^{k}}{(\tau-i)^{k+1}}\right\} d\tau=0,
$$
since for $R \ge 2$, 
\begin{multline*}
\int_{0}^{\pi} \val{e^{ia Re^{i\theta}}}
\left\vert-
\frac{(R e^{i\theta}-i)^{k}}{(R e^{i\theta}+i)^{k+1}}
+\frac{(R e^{i\theta}+i)^{k}}{(R e^{i\theta}-i)^{k+1}}\right\vert
\val{iR e^{i\theta}} d\theta
\\\le 
\int_{0}^{\pi} e^{-aR\sin \theta}\left\vert-
\frac{(e^{i\theta}-iR^{-1})^{k}}{( e^{i\theta}+iR^{-1})^{k+1}}
+\frac{(e^{i\theta}+iR^{-1})^{k}}{(e^{i\theta}-iR^{-1})^{k+1}}\right\vert d\theta
\\
\le 
\int_{0}^{\pi} e^{-aR\sin \theta}
d\theta \sup_{0\le \rho\le 1/2}\left\{
\frac{(1+\rho)^{k}}{( 1-\rho)^{k+1}}
+\frac{(1+\rho)^{k}}{(1-\rho)^{k+1}}\right\}. 
\end{multline*}
For $a>0$, we obtain $\lim_{R\rightarrow+\io }\int_{0}^{\pi} e^{-aR\sin \theta}
d\theta=0$ by dominated convergence. 
As a result, we get 
$$F'_{k}(a)=
(-1)^{k} \frac1{k!}(\frac{d}{d\tau})^{k}\bigl\{ e^{i a\tau}(\tau+i)^{k}\bigr\}_{\vert \tau=i}
=(-1)^{k} \frac1{k!}(\frac{d}{\frac i ad\epsilon})^{k}\bigl\{ e^{-a-\epsilon}(i+i\frac\epsilon a+i)^{k}\bigr\}_{\vert \epsilon=0},
$$
that is 
$$
F'_{k}(a)=\frac{(-1)^{k}}{k!} e^{-a} (\frac{d}{d\epsilon})^{k}\bigl\{ e^{-\epsilon}
(2a+\epsilon)^{k}
\bigr\}_{\vert \epsilon=0}.
$$
We note that $F'_{k}$ belongs to $L^{1}(\R_{+})$ as the product of $e^{-a}$ by a polynomial.
We have also that 
\begin{equation}\label{321654}
\lim_{a\rightarrow+\io}F_{k}(a)=1 \quad \text{(see the Appendix page \pageref{sub555}),}
\end{equation}
and this yields
$$
F_{k}(a)=1+\int_{+\io}^{a} F'_{k}(b) db
=1-\int_{a}^{+\io}
\frac{(-1)^{k}}{k!} e^{-b} (\frac{d}{d\epsilon})^{k}\bigl\{ e^{-\epsilon}
(2b+\epsilon)^{k}
\bigr\}_{\vert \epsilon=0} db,
$$
so that 
\begin{equation}\label{11111}
F_{k}(a)=1-e^{-a}P_{k}(a),
\end{equation}
with 
\begin{multline}\label{88888}
P_{k}(a)
=\frac{(-1)^{k}}{k!} \int_{0}^{+\io}
e^{-t} (\frac{d}{d\epsilon})^{k}\bigl\{ e^{-2\epsilon }
(a+t+\epsilon)^{k}
\bigr\}_{\vert \epsilon=0} dt\\
=\frac{(-1)^{k}}{k!} \int_{0}^{+\io}
e^{t} (\frac{d}{d\epsilon})^{k}\bigl\{ e^{-2\epsilon -2t}
(a+t+\epsilon)^{k}
\bigr\}_{\vert \epsilon=0} dt
\\
=\frac{(-1)^{k}}{k!} \int_{0}^{+\io}
e^{t} (\frac{d}{dt})^{k}\bigl\{ e^{-2t}
(a+t)^{k}
\bigr\} dt.
\end{multline}
We see that $P_{k}$ is a polynomial with leading monomial $\frac{2^{k}a^{k}}{k!}$
(by a direct computation)
 and $P_{k}(0)=1$
 (since $0=F_{k}(0)=1-P_{k}(0)$)
  and moreover,
using Laguerre polynomials (see e.g. \eqref{laguerre} in our Section \ref{sec.laguerre}), we obtain
\begin{align}
P_{k}(a)&=\frac{(-1)^{k}}{k!} \int_{0}^{+\io}
e^{-t}  e^{2t+2a}(\frac{d}{2dt})^{k}\bigl\{ e^{-2t-2a}
(2a+2t)^{k}
\bigr\} dt\label{onedim7}
\\
&=(-1)^{k}\int_{0}^{+\io} e^{-t}L_{k}(2t+2a) dt,\label{onedim6}
\end{align}
and this gives in particular 
\begin{multline}\label{5555}
P'_{k}(a)=(-1)^{k}\int_{0}^{+\io} e^{-t}2L'_{k}(2t+2a) dt\\=(-1)^{k}\Bigl\{
[e^{-t}L_{k}(2t+2a)]_{t=0}^{t=+\io}+\int_{0}^{+\io} e^{-t} L_{k}(2t+2a) dt 
\Bigr\}
\\=(-1)^{k+1}L_{k}(2a)+P_{k}(a).
\end{multline}
Moreover we have from  \eqref{88888}, for $k\ge 1$, 
\begin{align*}
&P'_{k}(a)=\frac{(-1)^{k}}{k!}\int_{0}^{+\io}
e^{t} (\frac{d}{dt})^{k}\bigl\{ e^{-2t}
k(a+t)^{k-1}
\bigr\} dt\\&
=\frac{(-1)^{k}}{k!}\int_{0}^{+\io}
e^{t} \frac{d}{dt}(\frac{d}{dt})^{k-1}\bigl\{ e^{-2t}
k(a+t)^{k-1}
\bigr\} dt
\\&=
\frac{(-1)^{k}}{k!}\Bigl\{
\bigl[e^{t}(\frac{d}{dt})^{k-1}\bigl\{ e^{-2t}
k(a+t)^{k-1}
\bigr\} \bigr]^{t=+\io}_{t=0}
-\int_{0}^{+\io}
e^{t} (\frac{d}{dt})^{k-1}\bigl\{ e^{-2t}
k(a+t)^{k-1}
\bigr\} dt
\Bigr\}
\\
&
=\frac{(-1)^{k-1}}{(k-1)!}
(\frac{d}{dt})^{k-1}\bigl\{ e^{-2t}
(a+t)^{k-1}
\bigr\}_{\vert t=0}
+\frac{(-1)^{k-1}}{(k-1)!}\int_{0}^{+\io}
e^{t} (\frac{d}{dt})^{k-1}\bigl\{ e^{-2t}
(a+t)^{k-1}
\bigr\} dt
\end{align*}
\begin{align*}
&=
\frac{(-1)^{k-1}}{(k-1)!}
e^{2t+2a}(\frac{d}{2dt})^{k-1}\bigl\{ e^{-2t-2a}
(2a+2t)^{k-1}
\bigr\}_{\vert t=0}
\\ &\hskip150pt
+\frac{(-1)^{k-1}}{(k-1)!}\int_{0}^{+\io}
e^{t} (\frac{d}{dt})^{k-1}\bigl\{ e^{-2t}
(a+t)^{k-1}
\bigr\} dt
\\&=(-1)^{k-1}L_{k-1}(2a)+P_{k-1}(a),
\end{align*}
so that
\begin{equation}\label{lagk-1}\forall k\ge 1,\quad
P'_{k}(a)=(-1)^{k-1}L_{k-1}(2a)+P_{k-1}(a)=(-1)^{k+1}L_{k}(2a)+P_{k}(a).
\end{equation}
This implies for $N\ge 1,$
$$
\sum_{1\le k\le N} P_{k}(a)-\sum_{1\le k\le N}(-1)^{k}L_{k}(2a)
=\sum_{0\le k\le N-1} P_{k}(a)+\sum_{0\le k\le N-1}(-1)^{k}L_{k}(2a),
$$
yielding 
\begin{equation*}
P_{N}(a)-\underbrace{P_{0}(a)}_{=1=L_{0}(a)}=\sum_{1\le k\le N}(-1)^{k}L_{k}(2a)+\sum_{0\le k\le N-1}(-1)^{k}L_{k}(2a),
\end{equation*}
and
\begin{equation}\label{222222}
P_{N}(a)=\sum_{0\le k\le N}(-1)^{k}L_{k}(2a)+\sum_{0\le k\le N-1}(-1)^{k}L_{k}(2a).
\end{equation}
Note that the previous formula holds as well for $N=0$, since $P_{0}=1=L_{0}$.
\vs
Although the function $\R_{+}\ni a \mapsto F_{k}(a)$ has no monotonicity properties, we prove below that 
$\R_{+}\ni a \mapsto P_{k}(a)$ is indeed increasing. For that purpose, let us use \eqref{lagk-1},
which implies
\begin{align*}
P'_{k}(a)&=(-1)^{k-1}L_{k-1}(2a)+P_{k-1}(a),\quad k\ge 1,\\
P_{k-1}(a)&=P_{k-2}(a)+(-1)^{k-2}L_{k-2}(2a)+(-1)^{k-1}L_{k-1}(2a), \quad k\ge 2,
\\
P'_{k}(a)&=2(-1)^{k-1}L_{k-1}(2a)+(-1)^{k-2}L_{k-2}(2a)+P_{k-2}(a),\quad k\ge 2.
\end{align*}
We claim that that for $k\ge 1$,
\begin{equation}\label{666666}
P'_{k}(a)=2\sum_{0\le l\le k-1}(-1)^{l}L_{l}(2a).
\end{equation}
That property holds for $k=1$ since $P_{1}(a)=1+2a$: we check
$
P'_{1}(a)=2.
$
Moreover we have 
\begin{align*}
P'_{k+1}(a)&=(-1)^{k}L_{k}(2a)+P_{k}(a)\qquad\text{\footnotesize (from the first equation in \eqref{lagk-1})}
\\
\text{\footnotesize (using \eqref{222222})\quad }&=(-1)^{k}L_{k}(2a)
+\sum_{0\le l\le k}(-1)^{l}L_{l}(2a)+\sum_{0\le l\le k-1}(-1)^{l}L_{l}(2a)
\\
&=2\sum_{0\le l\le k}(-1)^{l}L_{l}(2a),\quad \text{\tt qed.}
\end{align*}
As a byproduct we find from \eqref{inelag}
\begin{equation}\label{}\forall a\ge 0,\quad
P'_{k}(a)\ge 0,
\end{equation}
which implies that for $a\ge 0$, $P_{k}(a)\ge P_{k}(0)=1$.
We have proven the following
\begin{lem}\label{keylem}
 The polynomial $P_{k}(a)=e^{a}\bigl(1- F_{k}(a)\bigr)$ is increasing  on $\R_{+}$, $P_{k}(0)=1$.
\end{lem}
Let us take a look at the first  $P_{k}$: we have
\par
{\footnotesize\begin{align*}
P_{0}(a)&= 1,\\
P_{1}(a)&= 1+2 a,\\
P_{2}(a)&=  1+2 a^2,\\
P_{3}(a)&=  1+2 a-2 a^2+\frac{4 a^3}{3},\\
P_{4}(a)&=  1+4 a^2-\frac{8 a^3}{3}+\frac{2 a^4}{3},\\
P_{5}(a)&=  1+2a-4 a^2+\frac{16 a^3}{3}-2 a^4+\frac{4 a^5}{15},\\
P_{6}(a)&=  1+6 a^2-8 a^3+\frac{14 a^4}{3}-\frac{16 a^5}{15}+\frac{4 a^6}{45},\\
P_{7}(a)&=1+2
a-6 a^2+12 a^3-\frac{26 a^4}{3}+\frac{44 a^5}{15}-\frac{4 a^6}{9}+\frac{8 a^7}{315},\\
P_{8}(a)&=1+8a^2-16 a^3+\frac{44 a^4}{3}-\frac{32 a^5}{5}+\frac{64 a^6}{45}-\frac{16 a^7}{105}+\frac{2 a^8}{315},\\
P_{9}(a)&=
1+2 a-8 a^2+\frac{64 a^3}{3}-\frac{68 a^4}{3}+\frac{184
a^5}{15}-\frac{32 a^6}{9}+\frac{176 a^7}{315}-\frac{2 a^8}{45}+\frac{4 a^9}{2835}
,\\
P_{10}(a)&{\text{\small $ =
1+10 a^2-\frac{80 a^3}{3}+\frac{100 a^4}{3}-\frac{64 a^5}{3}+\frac{344
a^6}{45}-\frac{496 a^7}{315}+\frac{58 a^8}{315}-\frac{32 a^9}{2835}+\frac{4 a^{10}}{14175}
$}}
,\\
P_{11}(a)&{\text{\small $ =
1 + 2 a - 10 a^2 +\frac{ 100 a^3}3 - \frac{140 a^4}3 +  \frac{104 a^5}3 - 
\frac{664 a^6}{45} +  \frac{1184 a^7}{315} -\frac{26 a^8}{45}$}} \\&\hskip240pt {\text{\small $ 
+ \frac{148 a^9}{2835} - 
  \frac {4 a^{10}}{1575} + \frac{8 a^{11}}{155925}$}}
  ,\\
  P_{12}(a)&{\text{\small $ =
1 + 12 a^2 - 40 a^3 +   \frac{190 a^4}3 -   \frac{160 a^5}3 + 
  \frac{1184 a^6}{45} - 
    \frac {2512 a^7}{315} +   \frac{478 a^8}{315} - 
    \frac{512 a^9}{2835}$}} \\&\hskip240pt {\text{\small $ 
 +   \frac{184 a^{10}}{14175} -   \frac{16 a^{11}}{31185} + 
    \frac{4 a^{12}}{467775}.
$}}
\end{align*}}
We note as well that 
\begin{equation}\label{polpik}
P_{k}(x)=\sum_{0\le m\le k}\frac{x^{m}}{m!}\sum_{m\le l\le k} 2^{l}(-1)^{k-l}\binom{k}{l},
\end{equation}
since from \eqref{88888},
\begin{align*}
P_{k}(a)
&=\frac{(-1)^{k}}{k!} \int_{0}^{+\io}
e^{t} (\frac{d}{dt})^{k}\bigl\{ e^{-2t}
(a+t)^{k}
\bigr\} dt
\\
&= (-1)^{k}\sum_{0\le m\le k}\int_{0}^{+\io}
e^{-t} 
\frac{(-2)^{k-m}}{(k-m)!} \frac{k!}{(k-m)!m!}(a+t)^{k-m}
dt
\\
&= (-1)^{k}\sum_{0\le m\le k}\int_{0}^{+\io}
e^{-t} 
\frac{(-2)^{k-m}}{(k-m)!} \frac{k!}{(k-m)!m!}\sum_{0\le l\le k-m}
a^{l} t^{k-l-m}\binom{k-m}{l}
dt
\\
&= (-1)^{k}\sum_{\substack{0\le m\le k\\0\le l\le k-m}}
\frac{(-2)^{k-m}}{(k-m)!} \frac{k!}{(k-m)!m!}
a^{l} (k-l-m)!\binom{k-m}{l}
\\
&= \sum_{\substack{0\le l+m\le k}}
\frac{(-1)^{m}2^{k-m}}{(k-m)!} \frac{k!}{m!}
a^{l} \frac{1}{l!}
=\sum_{0\le l\le k}\frac{a^{l}}{l!}\sum_{l\le m'\le k}{(-1)^{k-m'}2^{m'}}\binom{k}{m'},\quad\text{\tt qed.}
\end{align*}
\begin{lem}\label{lem13}
 With the polynomial $P_{k}$ defined by \eqref{onedim6}, we have
 \begin{equation}\label{}
\begin{cases}
P_{k}(a)&=2\sum_{0\le l\le k-1} (-1)^{l}L_{l}(2a)+(-1)^{k}L_{k}(2a),
\\
P'_{k}(a)&=2\sum_{0\le l\le k-1} (-1)^{l}L_{l}(2a).
\end{cases}
\end{equation}
\end{lem}
\begin{proof}
We may use the already proven \eqref{222222}, \eqref{666666}, but we may also prove this directly by 
induction on $k$.
\end{proof}
\begin{pro}\label{pro.2424}
 Let $F_{k}$ be given by \eqref{11111} with $P_{k}$ defined by \eqref{88888}. We have
 \begin{align}
F_{k}(a)&=1-e^{-a}P_{k}(a)\le 1-e^{-a}=F_{0}(a)\quad \text{for $a\ge 0$},
\label{1214}\\
F'_{k}(a)&=e^{-a}\bigl(P_{k}(a)-P'_{k}(a)
\bigr)=e^{-a}(-1)^{k}L_{k}(2a),
\label{}\\
F'_{k}(0)&=(-1)^{k},\quad \lim_{a\rightarrow +\io} F'_{k}(a)=0_{+}, \quad F_{k}(0)=0,
\lim_{a\rightarrow +\io} F_{k}(a)=1_{-}.\label{}
\end{align}
\end{pro}
\begin{proof}
 We use  \eqref{11111}, \eqref{666666} and \eqref{222222} for the three first equalities, Lemma \ref{keylem} for the first inequality.
 The fourth equality follows from $L_{k}(0)=1$,
 while the fifth is due to the fact that the leading monomial of $(-1)^{k}L_{k}(2a)$ is $2^{k}a^{k}/k!$. The two last equalities are a consequence of the first line.
\end{proof}
\begin{rem}
{\rm The zeroes of $F'_{k}$ on the positive half-line are the positive zeroes of the Laguerre polynomial $L_{k}$
divided by $2$.
When $k$ is even (resp. odd)
the function $F_{k}$ is positive increasing  (resp. negative decreasing) near $0$, then oscillates 
with changes of monotonicity at each $a$ such that $L_{k}(2a)=0$ and when $2a$ is larger than the largest zero of $L_{k}$, the function $F_{k}$ is increasing, smaller than 1,  with limit 1 at infinity.
\par
Typically we have $F_{2l}(0)=0, F'_{2l}(0)=+1,$
\begin{equation}\label{jjjkkk}
\quad 0<a_{1,2l}<a_{2}<\dots<a_{2l-1,2l}<a_{2l,2l}\quad\text{the zeroes of $L_{2l}(2a)$,} 
\end{equation}
 $F_{2l}$ vanishes simply at $b_{0}=0$ and at  $b_{j}\in (a_{j}, a_{j+1})$ for $1\le j\le 2l-1$, also at $b_{2l}>a_{2l}$: $2l+1$ zeroes
 with a    positive (resp. negative)
 derivative at $b_{0}, b_{2},\dots,b_{2l}$
 (resp. at $b_{1}, b_{3},\dots,b_{2l-1}$).
 \par
 Moreover, we have $F_{2l+1}(0)=0, F'_{2l+1}(0)=-1,$
\begin{equation}\label{666555}
\quad 0<a_{1, 2l+1}<a_{2,2l+1}<\dots <a_{2l, 2l+1}<a_{2l+1, 2l+1},
\text{the zeroes of $L_{2l+1}(2a)$,} 
\end{equation}
 $F_{2l+1}$ vanishes simply at $b_{0}=0$ and at  $b_{j}\in (a_{j}, a_{j+1})$ for $1\le j\le 2l$, also at $b_{2l+1}>a_{2l+1}$: $2l+2$ zeroes
 with a    positive (resp. negative)
 derivative at $b_{1}, b_{3},\dots,b_{2l+1}$
 (resp. at $b_{0}, b_{2},\dots,b_{2l}$).}
\end{rem}
We note as well that a consequence of the previous remark is that 
\begin{align}
&\min _{a\ge 0} F_{2l}(a)=\min_{1\le j\le l}\{F_{2l}(a_{2j, 2l})\},
\\
&\min _{a\ge 0} F_{2l+1}(a)=\min_{0\le j\le l}\{F_{2l+1}(a_{2j+1, 2l+1})\},
\end{align}
where $(a_{p,k})_{1\le p\le k}$ are defined in \eqref{jjjkkk}, \eqref{666555}.
\begin{theorem}\label{thm-fl0001}
 Let $a\ge 0$ be given and let 
 \begin{equation}\label{fl0001}
D_{a}=\{(x,\xi)\in \R^{2}, x^{2}+\xi^{2}\le \frac{a}{2\pi}\}.
\end{equation}
 Then we have 
 \begin{equation}\label{}
\opw{\mathbf 1_{D_{a}}}=\sum_{k\ge 0}F_{k}(a)\mathbb P_{k}\le 1-e^{-a}.
\end{equation}
\end{theorem}
\begin{proof}
 An immediate consequence of \eqref{556699+} and \eqref{1214}.
 Note that the inequality in the above theorem is due to P.~Flandrin in \cite{conjecture}
 (see also the related references \cite{MR1643942}, \cite{MR1681043}).
\end{proof}
{\bf Curves.} 
Let us display some  curves of $\R_{+}\ni a\mapsto F_{k}(a)=1-e^{-a}P_{k}(a)$.
\begin{figure}[ht]
\centering
\scalebox{0.76}{\includegraphics[angle=90,height=620pt,width=0.85\textwidth]{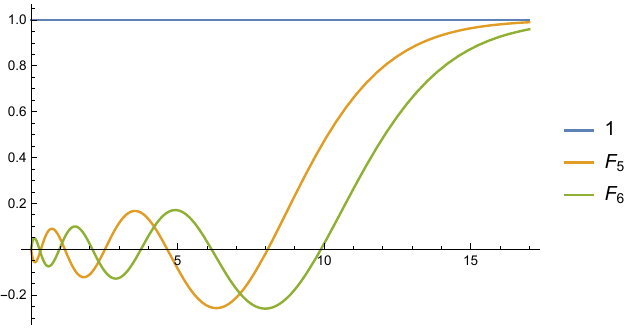}}\vs\vs
\caption{Functions $F_{5}, F_{6}$.}
\label{pic2}
\end{figure}
\vs\vs
\eject
\vs
\begin{figure}[ht]
\centering
\scalebox{1.1}{\includegraphics[angle=90,height=470pt,width=0.75\textwidth]{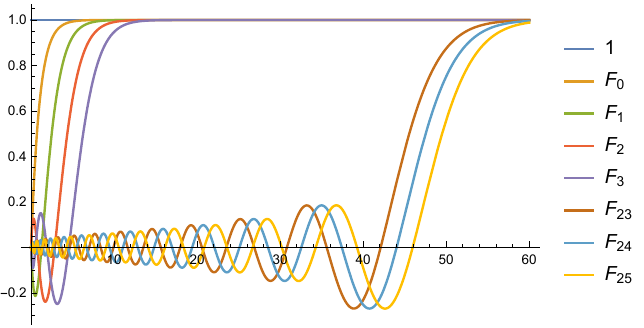}}\vs\vs
\caption{Functions $F_{k}$}
\label{665544}
\vs\vs\vs\vs
\label{pic3}
\end{figure}
\vs
\subsection{Indicatrix of an Euclidean ball}\label{balls}
\index{indicatrix of an Euclidean ball}
The following result displays an explicit spectral decomposition on the Hermite basis  for the Weyl quantization of the characteristic function of Euclidean balls.
\begin{theorem}\label{main}
Let $a\ge 0$ be given and let
 \begin{equation}\label{qanlios}
\mathcal Q_{a,n}=\opw{\mathbf 1\{{2\pi(\val x^{2}+\val\xi^{2})\le a}\}},
\end{equation}
be the Weyl quantization of the characteristic function of the Euclidean ball  of $\RZ$ with center $0$ and  radius $\sqrt{a/(2\pi)}$.
Then we have
\begin{equation}\label{}
\mathcal Q_{a,n}=\sum_{k\ge 0} F_{k;n}(a)\mathbb P_{k;n},
\end{equation}
with
$\mathbb P_{k;n}=\sum_{\alpha\in \N^{n}, \val \alpha=k}\mathbb P_{\alpha},$
where $\mathbb P_{\alpha}$ is the orthogonal projection onto $\Psi_{\alpha}$
(defined in \eqref{psidim}), with $\val \alpha=\sum_{1\le j\le n}\alpha_{j}=k$
and 
\begin{equation}\label{}
F_{k;n}(a)=
\int_{\R}
\frac{\sin a\tau}{\pi \tau}
\frac{(1+i\tau)^{k}}{(1-i\tau)^{k+n}}d\tau.
\end{equation}
\end{theorem}
\vs\vs
The spectral decomposition of the previous theorem
allows a simple recovery of the result of the article \cite{MR2761287} by E.~Lieb
and Y.~Ostrover.
\begin{theorem}\label{cor}
 Let $a\ge 0, \mathcal Q_{a,n}, F_{k;n}$ be defined above.
 Then we have
 \begin{equation}\label{253253}
 F_{k;n}(a)\le 
1-\frac1{\Gamma(n)}\int_{a}^{+\io} e^{-t}t^{n-1}dt=1-\frac{\Gamma(n, a)}{\Gamma(n)},
\end{equation}
and thus we have
\begin{equation}\label{325tyu}
\mathcal Q_{a,n}\le 1-\frac{\Gamma(n, a)}{\Gamma(n)},
\end{equation}
where the incomplete Gamma function $\Gamma(\cdot, \cdot)$ is defined in \eqref{incgam}.
\end{theorem}
\begin{proof}[Proof of Theorems \ref{main} and \ref{cor}]
We use the results of  (the previous) Section \ref{secdim}:
Let us assume now that, with some $a\ge 0$,
$$
F=\indic{[-\frac a{2\pi}, \frac a{2\pi}]},
\text{
so that 
}\quad 
F(\val x^{2}+\val \xi^{2})=\mathbf 1\{{2\pi(\val x^{2}+\val\xi^{2})\le a}\}.
$$
According to Section \ref{12345},
we have
$
\hat F(\tau) = \frac{\sin a\tau}{\pi \tau},
$
so that \eqref{condit} holds true. We find in this case,
following the results of Lemma \ref{lemm21n+},
\begin{align}
\bigl(F(\val x^{2}+\val \xi^{2})\bigr)^{w}&=\sum_{k\ge 0} F_{k;n}(a)\mathbb P_{k;n},\quad 
\mathbb P_{k;n}=\sum_{\alpha\in \N^{n}, \val \alpha=k}\mathbb P_{\alpha},\label{111111}
\\
F_{k;n}(a)&=
\int_{\R}
\frac{\sin a\tau}{\pi \tau}
\frac{(1+i\tau)^{k}}{(1-i\tau)^{k+n}}d\tau,\label{333333}
\end{align}
where $\mathbb P_{\alpha}$ is the orthogonal projection onto $\Psi_{\alpha}$
(defined in \eqref{psidim}), with $\val \alpha=\sum_{1\le j\le n}\alpha_{j}=k$.
This completes the proof of Theorem \ref{main}.
\vs
We postpone the proof of Theorem \ref{cor} until after settling a couple of lemmas.
\begin{lem}\label{lem.kkjjhhg}
Let $(k,n)\in \N\times\N^{*}$. With $F_{k;n}(a)$ given by \eqref{333333}, we have 
\begin{align}
F_{k;n}(a)&=1-e^{-a}P_{k, n}(a),\quad\text{where $P_{k;n}$ is the  polynomial}
\\
P_{k;n}(a)&=
\frac{(-1)^{k+n-1}}{(k+n-1)!}
\int_{0}^{+\io}e^{-t}(t+a)^{n-1}
\Bigl\{e^{s}(\frac{d}{ds})^{n+k-1}\bigl[s^{k}e^{-s}\bigr]\Bigr\}_{\vert s=2t+2a} dt,
\label{329}
\\
P_{k;n}(a)&=\frac{(-1)^{k+n-1}}{(k+n-1)! 2^{n-1}}\int_{0}^{+\io}(t+a)^{n-1}
e^{t}
(\frac{d}{dt})^{n+k-1}\bigl\{(t+a)^{k} e^{-2t}\bigr\}
dt.\label{3210}
\end{align}
\end{lem}
\begin{proof}[Proof of the lemma]
 The lemma holds true for $n=1$ from Proposition \ref{pro.2424}.
We have  for $a>0$, $n\ge 2$,
\begin{align*}
F'_{k;n}(a)&=\frac1 \pi
\int_{\R}
{\cos a\tau}
\frac{(1+i\tau)^{k}}{(1-i\tau)^{k+n}}d\tau
\\&=\frac1{ 2\pi}
\int_{\R}
e^{i a\tau}
\frac{(1+i\tau)^{k}}{(1-i\tau)^{k+n}}d\tau
+
\frac1{ 2\pi}
\int_{\R}
e^{i a\tau}
\frac{(1-i\tau)^{k}}{(1+i\tau)^{k+n}}d\tau
\\&=\frac i{ 2i\pi}
\int_{\R}
e^{i a\tau}
\frac{i^{k}(\tau-i)^{k}}{(-i)^{k+n}(\tau+i)^{k+n}}d\tau
+
\frac i{ 2i\pi}
\int_{\R}
e^{i a\tau}
\frac{(-i)^{k}(\tau+i)^{k}}{i^{k+n}(\tau-i)^{k+n}}d\tau,
\end{align*}
so that 
\begin{align*}
F'_{k;n}(a)&=
i^{1-n} (-1)^{k}
\text{\tt Res}\bigl(e^{ia\tau}\frac{(\tau+i)^{k}}{(\tau-i)^{k+n}} ;i\bigr)
\\
&= \frac{i^{1-n} (-1)^{k}}{(k+n-1)!}(\frac{d}{d\tau})^{k+n-1}\bigl\{ e^{i a\tau}(\tau+i)^{k}\bigr\}_{\vert \tau=i}
\end{align*}
and thus
\begin{align*}
F'_{k;n}(a)&= \frac{i^{1-n} (-1)^{k}}{(k+n-1)!}(\frac{d}{\frac i ad\epsilon})^{k+n-1}\bigl\{ e^{-a-\epsilon}(i+i\frac\epsilon a+i)^{k}\bigr\}_{\vert \epsilon=0}
\\
&= \frac{i^{1-n} (-1)^{k}a^{n-1}}{i^{n-1}(k+n-1)!}(\frac{d}{d\epsilon})^{k+n-1}\bigl\{ e^{-a-\epsilon}(2a+\epsilon )^{k}\bigr\}_{\vert \epsilon=0}
\\
&= e^{a}\frac{(-1)^{k+n-1}a^{n-1}}{(k+n-1)!}(\frac{d}{2d\epsilon})^{k+n-1}\bigl\{ e^{-2a-2\epsilon}(2a+2\epsilon )^{k}\bigr\}_{\vert \epsilon=0},
\end{align*}
that is 
\begin{align*}
F'_{k;n}(t)&=\frac{(-1)^{k+n-1}}{(k+n-1)!} e^{t} t^{n-1}(\frac{d}{d s})^{k+n-1}\bigl\{ e^{-s}
s^{k}
\bigr\}_{\vert s=2t}
\\
&=\frac{(-1)^{k+n-1}}{(k+n-1)!2^{n-1}} e^{t} t^{n-1}(\frac{d}{d t})^{k+n-1}\bigl\{ e^{-2t}
t^{k}
\bigr\}.
\end{align*}
We have also that $\lim_{a\rightarrow+\io}F_{k;n}(a)=1$ (following the arguments of Section \ref{sec.22one}) and this yields
\begin{align*}
F_{k;n}(a)&=1-\frac{(-1)^{k+n-1}}{(k+n-1)!2^{n-1}} \int_{a}^{+\io}e^{t} t^{n-1}(\frac{d}{d t})^{k+n-1}\bigl\{ e^{-2t}t^{k}\bigr\}
dt
\\&=
1-e^{-a}\frac{(-1)^{k+n-1}}{(k+n-1)!2^{n-1}} \int_{0}^{+\io} (t+a)^{n-1}e^{t}(\frac{d}{d t})^{k+n-1}\bigl\{ e^{-2t}
(t+a)^{k}\bigr\}
dt,
\end{align*}
concluding the proof of the Lemma.
\end{proof}
Let us go back to Formula \eqref{329}, written as 
\begin{align}
&\frac{(-1)^{k+n-1}}{2^{n-1}}
\int_{0}^{+\io}e^{-t}
\Bigl\{
\frac{(2t+2a)^{n-1}}{(k+n-1)!}
(\frac{d}{d\epsilon}-1)^{n+k-1}\bigl[(\epsilon+2t+2a)^{k}\bigr]\Bigr\}_{\vert \epsilon=0} dt=
\notag\\
&P_{k;n}(a)=
\frac{(-1)^{k+n-1}}{2^{n-1}}
\int_{0}^{+\io}e^{-t}
L_{k+n-1}^{1-n}(2t+2a) dt,
\label{3211}
\end{align}
where the generalized Laguerre polynomial $L_{k+n-1}^{1-n}$
is defined  by \eqref{laggen} (note that $1-n+k+n-1=k$ which not negative).
\begin{lem}\label{ggff}
 Let $n\in \N^*$, $k\in \N$ and let $P_{k;n}$ be the polynomial defined in Lemma \ref{lem.kkjjhhg} (and thus in \eqref{3211}). Then we have
 \begin{align}
P_{k;n}(X)-P'_{k;n}(X)&=\frac{(-1)^{k+n-1}}{2^{n-1}}L_{k+n-1}^{1-n}(2X), \quad P_{k;n}(0)=1,
\label{3212}\\
\text{for $n\ge 2$},&\quad P'_{k;n}=P_{k;n-1}.
\label{3213}
\end{align}
\end{lem}
\begin{proof}
 From \eqref{3211}, we find
 \begin{equation}\label{3214}
P'_{k;n}(a)=
\frac{(-1)^{k+n-1}}{2^{n-1}}
\int_{0}^{+\io}e^{-t}
2(L_{k+n-1}^{1-n})'(2t+2a) dt
\end{equation}
\begin{align}
&=
\frac{(-1)^{k+n-1}}{2^{n-1}}\Bigl\{\bigl[
e^{-t}
(L_{k+n-1}^{1-n})(2t+2a) 
\bigr]_{t=0}^{t=+\io}
+
\int_{0}^{+\io}e^{-t}
L_{k+n-1}^{1-n}(2t+2a) dt\Bigr\}
\notag
\\&=\frac{(-1)^{k+n}}{2^{n-1}}
L_{k+n-1}^{1-n}(2a) +P_{k;n}(a),
\notag
\end{align}
and since $0=F_{k;n}(0)=1-P_{k;n}(0)$, this proves \eqref{3212}.
Using now \eqref{3211} and \eqref{958}, we find that 
 \begin{align*}
P_{k;n}(a)&=
\frac{(-1)^{k+n}}{2^{n-1}}
\int_{0}^{+\io}\frac{d}{dt}\bigl\{e^{-t}\bigr\}
L_{k+n-1}^{1-n}(2t+2a) dt
\\
&=\frac{(-1)^{k+n}}{2^{n-1}}\Bigl\{
\bigl[e^{-t}L_{k+n-1}^{1-n}(2t+2a) \bigr]_{t=0}^{t=+\io}-
\int_{0}^{+\io}e^{-t}
2(L_{k+n-1}^{1-n})'(2t+2a) dt\Bigr\}
\\
&=\frac{(-1)^{k+n}}{2^{n-1}}\Bigl\{
-L_{k+n-1}^{1-n}(2a)+
\int_{0}^{+\io}e^{-t}
2(L_{k+n-2}^{2-n})(2t+2a) dt\Bigr\}
\\
&=\underbrace{\frac{(-1)^{k+n-1}}{2^{n-1}}
L_{k+n-1}^{1-n}(2a)}_{\substack{P_{k;n}(a)-P'_{k;n}(a)\\\text{from \eqref{3212}}}}
+\underbrace{\frac{(-1)^{k+n-2}}{2^{n-2}}\int_{0}^{+\io}e^{-t}
L_{k+n-2}^{2-n}(2t+2a) dt}_{\substack{P_{k;n-1}(a)\\\text{from \eqref{3211}}}},
\end{align*}
so that  for $n\ge 2$, $k\in \N$, we obtain \eqref{3213}, completing the proof of the lemma.
\end{proof}
\begin{lem}
 Let $k,n,P_{k;n}$ be as in Lemma \ref{ggff}. Then we have 
 \begin{equation}\label{fdre}
\forall j\in \Iintv{0,n-1},
\quad  \left(\frac{d}{dX}\right)^j P_{k;n}=P_{k;n-j}.
\end{equation}
Moreover, for all $a\ge 0$ and all $k\in \N$,
\begin{equation}\label{1234*}
P_{k;n}(a)\ge P_{0;n}(a)=\frac{1}{(n-1)!}\int_{0}^{+\io}
e^{-t}(t+a)^{n-1} dt =e^a\frac{\Gamma (n,a)}{\Gamma (n)}.
\end{equation}
\end{lem}
\begin{proof}
 Formula \eqref{fdre} follows immediately by induction from \eqref{3213} since the latter is proving  \eqref{fdre} for $j=1, n\ge 2, k\in \N.$
 Assuming that \eqref{fdre} holds true for some $1\le j<n$, all $k\in \N$,
 we have
 $
 P_{k;n}^{(j)}=P_{k, n-j}
 $
 and if $j+1<n$, we obtain from \eqref{3213} that
 $$
 P_{k,n-j-1}=P'_{k, n-j}=P_{k;n}^{(j+1)},
 $$
 proving \eqref{fdre}. The property \eqref{1234*}
 holds true for $n=1$.
 From  \eqref{3213} and $P_{k;n+1}(0)=1$, we find that 
 $
 P_{k;n+1}(a)=1+\int_{0}^a P_{k;n}(s) ds
 $
 and assuming that \eqref{1234*} holds true for $n$, we obtain for $a\ge 0$,
\begin{align*}
 P_{k;n+1}(a)&\ge 1+\int_{0}^a
 \frac{1}{(n-1)!}\int_{0}^{+\io}
e^{-t}(t+s)^{n-1} dt
 ds
 \\
 &=1+
 \int_{0}^{+\io}
e^{-t}\Bigl[\frac{(t+s)^{n}}{n!}\Bigr]^{s=a}_{s=0}  dt
 \\&=
 1+\frac{1}{n!}
 \int_{0}^{+\io}
e^{-t}
\bigl((t+a)^{n}-t^n \bigr)dt=\frac{1}{n!}
 \int_{0}^{+\io}
e^{-t}
(t+a)^{n} dt,
\end{align*}
completing the proof of the lemma.
\end{proof}
We can now prove Theorem \ref{cor}: since 
$
F_{k;n}(a)=1- e^{-a} P_{k;n}(a)
$
the estimate \eqref{fdre} implies indeed
$
F_{k;n}(a)\le \frac{\Gamma(n,a)}{\Gamma(n)},
$
concluding the proof.
\end{proof}
\begin{rem}\rm
 Our methods of proof in one and more dimensions are quite similar: 
\begin{itemize}
\item  Using Mehler's Formula,
 we diagonalize in the Hermite basis the quantization of the indicatrix of the Euclidean ball 
\begin{equation*}\label{}
 D_{a;n}=\{(x,\xi)\in \RZ, 2\pi\bigl(\val x^2+\val\xi^2\bigr)\le a\}.
\end{equation*}
 \item Once we get the diagonalization
 $$
\opw{ \mathbf 1_{D_{a;n}}}=\sum_{k\in \N} F_{k;n}(a) \mathbb P_{k;n},
 $$
 we study explicitly the functions $F_{k;n}$ and prove that 
 $$
 F_{k;n}(a)= 1-e^{-a}
P_{k;n}(a), $$
where $P_{k;n}$ is a polynomial given in terms of the generalized Laguerre polynomials
$$
P_{k;n}(a)=
\frac{(-1)^{k+n-1}}{2^{n-1}}
\int_{0}^{+\io}e^{-t}
L_{k+n-1}^{1-n}(2t+2a) dt.
$$
\item
Following the Flandrin paper  \cite{conjecture},
we use Feldheim inequality in  \cite{MR0001401} to tackle the case $n=1$, and next we use an induction on $n$, made possible  by the relationship between the standard and the generalized Laguerre polynomials.
It is interesting to note that the functions $F_{k;n}$
have no monotonicity properties:
with value 0 at 0, they have an oscillatory behavior for $a\le a_{k,n}$ and for $a$ large enough, increase monotonically  to 1 (see for instance Figures \ref{pic2} and 
\ref{pic3}
in the 1D case); the inequality  $F_{k;n}(a)\le 1-e^{-a}$ holds true for all $a\ge 0$ in all dimensions. On the other hand the polynomials $P_{k;n}$ are increasing and larger than 1 on the positive half-line.\end{itemize}
\vs
 The key ingredients are thus Mehler's formula and Feldheim inequality, but it should be pointed out that  the arguments proving Feldheim inequality
 (Formula (6.8) and Theorem 12) in the R.~Askey \& G.~Gasper's article \cite{MR0430358} are also based upon a version of Mehler's Formula
which appears thus as the basic result for our investigation.
The paper \cite{MR2761287} by E.~Lieb
and Y.~Ostrover has a slightly different line of arguments and takes advantage of symmetry properties of the sphere. We shall go back to this
in a situation where the symmetry is absent, such as for some general ellipsoids.
\end{rem}
\subsection{Ellipsoids in the phase space}
\index{ellipsoids in the phase space}
\subsubsection{Preliminaries}
We provide below a couple of remarks on ellipsoids in higher dimensions.
Let us first recall a particular case  of Theorem 21.5.3 in \cite{MR2304165}.
\index{symplectic reduction of quadratic forms}
\begin{theorem}[Symplectic reduction of quadratic forms]\label{thm312}
 Let $q$ be a positive-definite quadratic form  on $\R^{n}\times \R^{n}$ equipped with the canonical symplectic form \eqref{sympfor}. Then there exists $S$ in the symplectic group $Sp(n,\R)$
 of $\RZ$
 and $\mu_{1}, \dots, \mu_{n}$
 positive  such that  for all $X=(x,\xi)\in \R^{n}\times \R^{n}$,
\begin{equation}\label{ellipse}
 q(SX)=\sum_{1\le j\le n}\mu_{j}(x_{j}^{2}+\xi_{j}^{2}).
\end{equation}
\end{theorem}
Note that an interesting consequence of this theorem  is that, considering a general ellipsoid in $\RZ$ (with center of gravity at $0$), 
$$\mathbb E=\{X\in \RZ, q(X)\le 1\}$$
where $q$ is a positive definite quadratic form, we are able to find symplectic coordinates
such that $q$ is given by \eqref{ellipse}.
Note however that no further simplification is possible and that the $\mu_{j}$ are symplectic invariants of $\mathbb E$.
Note that the volume of $\mathbb E$ is given by
$$
\val{\mathbb E}_{2n}=\frac{\pi^{n}}{n!{\mu_{1}\dots \mu_{n}}}.
$$
\subsubsection{Spectral decomposition for the quantization of the characteristic function of the ellipsoid}
Let $a_{1},\dots, a_{n}$ be positive numbers.
We consider the ellipsoid $E(a_{1}, \dots, a_{n})$ given by
\begin{equation}\label{ellip++}
E(a)=E(a_{1}, \dots, a_{n})=\{(x,\xi)\in \R^{n}\times\R^{n}, 2\pi\sum_{1\le j\le n}
\frac{x_{j}^{2}+\xi_{j}^{2}}{a_{j}}\le 1
\}.
\end{equation}
We define on $\R^n$
 the function
$$
F(X_{1}, \dots, X_{n})=\indic{[-1,1]}(\frac{2\pi}{a_{1}}X_{1}+\dots+
\frac{2\pi}{a_{n}}X_{n}).
$$
\begin{theorem}\label{thm.specell}
Let $a=(a_{j})_{1\le j\le n}$ be positive numbers and let $E(a)$ be defined by \eqref{ellip++}. Then we have 
\begin{equation}\label{}
\opw{\indic{E(a)}}=\sum_{\alpha\in \N^n} F_{\alpha}(a)\mathbb P_{\alpha},
\end{equation}
where $ \mathbb P_{\alpha}$ is defined in \eqref{app.paln} and 
\begin{equation}\label{}
F_{\alpha}(a)=1-K_{\alpha}(a),
\end{equation}
with
\begin{equation}\label{kalpha}
K_{\alpha}(a)=\int_{
\substack{\sum t_{j}/a_{j}\ge 1\\t_{j}\ge 0}}
e^{-(t_{1}+\dots+t_{n})}
\prod_{1\le j\le n} 
(-1)^{\alpha_{j}}L_{\alpha_{j}}(2t_{j})
dt,
\end{equation}
\end{theorem}
\begin{rem}\label{rem-312}\rm
 For all $\alpha\in \N^n$,
 the functions $F_{\alpha}, K_{\alpha}$ are holomorphic on
 \begin{equation}\label{336}
\mathcal U=\{a\in \C^n, \forall j\in  \Iintv{1,n}, \re a_{j}>0\}.
\end{equation}
Indeed let $K$ be a compact subset  of $\mathcal U$; there exists $\rho>0$ such that
$$
\forall (a_{1},\dots, a_{n})\in K,\quad\min_{1\le j\le n}\re a_{j}\ge \rho,
$$
and as a result for $a\in K$, we have for $s\in \R_{+}^n$
$$
\val{e^{-(a_{1}s_{1}+\dots+a_{n}s_{n})}
\prod_{1\le j\le n} 
(-1)^{\alpha_{j}}L_{\alpha_{j}}(2a_{j}s_{j})
}\le e^{-\rho(s_{1}+\dots + s_{n})}C_{K,\alpha}(1+\val s)^{\val \alpha},
$$
so that 
\begin{multline*}
\int_{
\substack{\sum s_{j}\ge 1\\s_{j}\ge 0}}\sup_{a\in K}\val{
e^{-(a_{1}s_{1}+\dots+a_{n}s_{n})}
\prod_{1\le j\le n} 
(-1)^{\alpha_{j}}L_{\alpha_{j}}(2a_{j}s_{j})}
ds
\\
\le\int_{
\substack{\sum s_{j}\ge 1\\s_{j}\ge 0}}
 e^{-\rho(s_{1}+\dots + s_{n})}C_{K,\alpha}(1+\val s)^{\val \alpha}ds
 \le C_{K,\alpha}
 \int_{
\R^n}
 e^{-\rho\sigma_{n}\val s}(1+\val s)^{\val \alpha}ds<+\io.
\end{multline*}
Since we have
 \begin{equation*}
K_{\alpha}(a)=\int_{
\substack{\sum s_{j}\ge 1\\s_{j}\ge 0}}
e^{-(a_{1}s_{1}+\dots+a_{n}s_{n})}
\prod_{1\le j\le n} 
(-1)^{\alpha_{j}}L_{\alpha_{j}}(2a_{j}s_{j})
ds a_{1}\dots a_{n},
\end{equation*}
this proves the sought  holomorphy.
\end{rem}
\begin{proof}[Proof of the theorem]
 We have 
\begin{align*}
\opw{\indic{E(a)}}&=\bigl(F(x_{1}^{2}+\xi_{1}^{2}, \dots, x_{n}^{2}+\xi_{n}^{2})\bigr)^{w}
=\int_{\R^{n}}\hat F(\tau) 
\OPW{e^{2i\pi \sum_{j}\tau_{j}(x_{j}^{2}+\xi_{j}^{2})}}
d\tau
\\&=\sum_{\alpha\in \N^{n}}
\int_{\R^{n}}
\hat F(\tau)  
\prod_{1\le j\le n} \frac{(1+i\tau_{j})^{2\alpha_{j}+1}}{(1+\tau_{j}^{2})^{\alpha_{j}+1}}d\tau \mathbb P_{\alpha}
\\&=\sum_{\alpha\in \N^{n}}
\int_{\R^{n}}
\hat F(\tau)  
\prod_{1\le j\le n} \frac{(1+i\tau_{j})^{\alpha_{j}}}{(1-i\tau_{j})^{\alpha_{j}+1}}
d\tau \mathbb P_{\alpha},
\end{align*}
where $ \mathbb P_{\alpha}$ is defined in \eqref{app.paln}.
On the other hand we have 
\begin{multline*}
\hat F(\tau)=\int e^{-2i\pi \tau\cdot x}
\indic{[-1,1]}(\frac{2\pi}{a_{1}}x_{1}+\dots+
\frac{2\pi}{a_{n}}x_{n}) dx_{1}\dots dx_{n}
\\=a_{1}\dots a_{n}(2\pi)^{-n}
\int e^{-i \sum_{j}\tau_{j}a_{j }y_{j}}
\indic{[-1,1]}(\sum y_{j}) dy,
\end{multline*}
so that,  with $M_{k}$ defined in \eqref{foulag-}, using \eqref{foulag+}, we get
\begin{align*}
&\opw{\indic{E(a)}}
\\
&=a_{1}\dots a_{n}
\sum_{\alpha\in \N^{n}}
\iint_{\R^{n}\times \R^{n}}
e^{-i 2\pi\sum_{j}\tau_{j}a_{j }y_{j}}
\indic{[-1,1]}(\sum y_{j}) dy 
\prod_{1\le j\le n} \frac{(1+i2\pi\tau_{j})^{\alpha_{j}}}{(1-i2\pi\tau_{j})^{\alpha_{j}+1}}
d\tau \mathbb P_{\alpha}
\\
&=a_{1}\dots a_{n}
\sum_{\alpha\in \N^{n}}
\int_{\R^{n}}
\int_{\R^{n}} e^{-i 2\pi\sum_{j}\tau_{j}a_{j }y_{j}}
\indic{[-1,1]}(\sum y_{j}) dy 
\prod_{1\le j\le n} \overline{\hat G_{\alpha_{j}}(\tau_{j})}
d\tau \mathbb P_{\alpha}
\\
&=a_{1}\dots a_{n}\sum_{\alpha\in \N^{n}}
\int_{\R^{n}}\indic{[-1,1]}(\sum y_{j}) 
\prod_{1\le j\le n} G_{\alpha_{j}}(a_{j}y_{j})
dy  \mathbb P_{\alpha}
\\
&=\sum_{\alpha\in \N^{n}}
\int_{\R^{n}}\indic{[-1,1]}(\sum t_{j}/a_{j}) 
\prod_{1\le j\le n} 
(-1)^{\alpha_{j}}H(t_{j})e^{-t_{j}}L_{\alpha_{j}}(2t_{j})
dt  \mathbb P_{\alpha},
\end{align*}
with
\begin{align}\label{337}
F_{\alpha}(a)&=\int_{\R^{n}}\bigl(1-\indic{[1,+\io]}(\sum t_{j}/a_{j}) \bigr)
\prod_{1\le j\le n} 
(-1)^{\alpha_{j}}H(t_{j})e^{-t_{j}}L_{\alpha_{j}}(2t_{j})
dt 
\\
&=1-\int_{\R^{n}}\indic{[1,+\io]}(\sum t_{j}/a_{j})
\prod_{1\le j\le n} 
(-1)^{\alpha_{j}}H(t_{j})e^{-t_{j}}L_{\alpha_{j}}(2t_{j})
dt,\notag
\end{align}
where we have used that $P_{k;1}(0)=1$ (see page \pageref{onedim6})
,
so that setting 
\begin{equation*}
K_{\alpha}(a)=\int_{
\substack{\sum t_{j}/a_{j}\ge 1\\t_{j}\ge 0}}
e^{-(t_{1}+\dots+t_{n})}
\prod_{1\le j\le n} 
(-1)^{\alpha_{j}}L_{\alpha_{j}}(2t_{j})
dt,
\end{equation*}
we have 
$
F_{\alpha}(a)=1-K_{\alpha}(a),$
concluding the proof of the theorem.
\end{proof}
\begin{rem}\label{rem-313}
 We have from \eqref{337}
 \begin{equation}\label{}
 F_{\alpha}(a_{1},\dots, a_{n})=
\int_{\R^{n}}\indic{[0,1]}(\sum_{1\le j\le n} s_{j}) 
\prod_{1\le j\le n} 
(-1)^{\alpha_{j}}H(s_{j})e^{-a_{j}s_{j}}L_{\alpha_{j}}(2a_{j}s_{j}) a_{j}
ds,
\end{equation}
and since  the set
\text{\small$
\{s\in \R_{+}^n, \sum_{1\le j\le n}s_{j}\le 1\}
$}
is compact,
we obtain that $F_{\alpha}$ is an entire function, as well as $K_{\alpha}$ which is indeed given by \eqref{kalpha} on the open subset $\mathcal U$ defined in \eqref{336}.
\end{rem}
\begin{lem}\label{lem320}
  With the notations of Theorem \ref{thm.specell}, we have with $\mu_{j}=1/a_{j}$,
  \begin{multline}\label{3316}
F_{\alpha}(a)=\Bigl(\prod_{1\le j\le n} a_{j}\Bigr)
\int_{\R}\frac{\sin  \tau}{\pi\tau}
\Bigl(\prod_{1\le j\le n}\frac{(a_{j}+i\tau)^{\alpha_{j}}}{
(a_{j}-i\tau)^{\alpha_{j}+1}
} \Bigr) d\tau
\\
=
\int_{\R}\frac{\sin\tau}{\pi\tau}
\Bigl(\prod_{1\le j\le n}\frac{(1+i\tau\mu_{j})^{\alpha_{j}}}{
(1-i\tau\mu_{j})^{\alpha_{j}+1}
} \Bigr) d\tau.
\end{multline}
\end{lem}
\begin{proof}
 Mehler's formula implies in one dimension that
 \begin{equation}\label{}
\opw{e^{2\pi i \tau(x^2+\xi^2)}}
=(1+\tau^2)^{-1/2}
\exp{\bigl[2\pi i(\arctan \tau)(x^2+D_{x}^2)\bigr]},
\end{equation}
and a simple tensorisation gives
$$
\opw{e^{2\pi i \tau\sum_{j}\mu_{j}(x_{j}^2+\xi_{j}^2)}}=
\prod_{j}(1+(\tau \mu_{j})^2)^{-1/2}
\exp{\bigl[2\pi i\sum_{j}(\arctan (\tau \mu_{j}))(x_{j}^2+D_{x_{j}}^2)\bigr]},
$$
so that we have 
\begin{align*}
&
\OPW{F\bigl(\sum_{j}\mu_{j}(x_{j}^2+\xi_{j}^2)\bigr)}
\\&=\int_{\R}
\hat F(\tau) 
\OPW{e^{2\pi i \tau\sum_{j}\mu_{j}(x_{j}^2+\xi_{j}^2)}} 
d\tau
\\
&=\int_{\R}
\hat F(\tau) 
\prod_{j}(1+(\tau \mu_{j})^2)^{-1/2}
\exp{\bigl[2\pi i\sum_{j}(\arctan (\tau \mu_{j}))(x_{j}^2+D_{x_{j}}^2)\bigr]}
 d\tau
 \\
& =
 \sum_{\alpha\in \N^n}
 \int_{\R}
\hat F(\tau) 
\left(\prod_{j}(1+(\tau \mu_{j})^2)^{-1/2}
\exp{\bigl[2 i(\arctan (\tau \mu_{j}))(\alpha_{j}+\frac12)\bigr]}\right)
 d\tau
 \mathbb P_{\alpha}
 \\
&=
 \sum_{\alpha\in \N^n}
 \int_{\R}
\hat F(\tau) 
\left(\prod_{j}(1+(\tau \mu_{j})^2)^{-1/2}
\frac{(1+i\tau \mu_{j})^{2\alpha_{j}+1}}{(1+(\tau\mu_{j})^2)^{\alpha_{j}+\frac12}}\right)
 d\tau
 \mathbb P_{\alpha}
 \\
 &
 = \sum_{\alpha\in \N^n}
 \int_{\R}\hat F(\tau)
\Bigl(\prod_{1\le j\le n}\frac{(1+i\tau\mu_{j})^{\alpha_{j}}}{
(1-i\tau\mu_{j})^{\alpha_{j}+1}
} \Bigr) d\tau \mathbb P_{\alpha},
\end{align*}
and for $F(t)=\mathbf 1_{[-1,1]}(2\pi t)$, we find $\hat F(\tau)=\frac{\sin \tau}{\pi \tau}$
and the sought result.
\end{proof}
\begin{rem}\rm
 It is also possible to provide a direct checking for the above lemma, since with the notations \eqref{foulag-}, \eqref{foulag+}, we have 
 $$
 \frac{(1+i\tau\mu_{j})^{\alpha_{j}}}{
(1-i\tau\mu_{j})^{\alpha_{j}+1}
}=\check{\widehat{G_{\alpha_{j}}}}\bigr(\tau \mu_{j}/(2\pi)\bigl),
 $$
 and thus
\begin{multline*}
 F_{\alpha}(a)= \int_{\R}\hat F(\tau)\prod_{j} \check{\widehat{G_{\alpha_{j}}}}\bigr(\tau \mu_{j}/(2\pi)\bigl) d\tau
 \\=
 \int_{\R}\hat F(\tau)\int_{\R^n}\prod_{j}(-1)^{\alpha_{j}}
L_{\alpha_{j}}(2t_{j}) H(t_{j})e^{-t_{j}}e^{2\pi i \tau \mu_{j}t_{j}/(2\pi)} dt d\tau
\\
=
\int_{\R^n}\prod_{j}(-1)^{\alpha_{j}}
L_{\alpha_{j}}(2t_{j}) H(t_{j})e^{-t_{j}}
F\bigl(\sum_{j}\mu_{j}t_{j}/2\pi\bigr)
 dt.
\end{multline*}
Now since we have 
$F\bigl(\sum_{j}\mu_{j}t_{j}/2\pi\bigr)=
\mathbf 1_{[-1,1]}(\sum_{j}\mu_{j}t_{j})$,
this fits with the expression of $F_{\alpha}$ in Theorem \ref{thm.specell}.
\end{rem}
\begin{rem}\rm
 Another interesting remark is that the expression \eqref{3316} depends obviously only on $\val \alpha$ and $a=a_{1}=\dots=a_{n}$ in the case where all the $a_{j}$ are equal:
 indeed in that case, we have with $\mu=1/a$,
 $$
 \prod_{1\le j\le n}\frac{(1+i\tau\mu_{j})^{\alpha_{j}}}{
(1-i\tau\mu_{j})^{\alpha_{j}+1}
} 
=\frac{(1+i\tau\mu)^{\val\alpha}}{
(1-i\tau\mu)^{\val \alpha+n}
},
 $$
 and this gives another ({\sl a posteriori}) justification  of our calculations in the isotropic case
 of Section \ref{balls}.
 On the other hand, we get also
 the identity
 \begin{equation}\label{}
F_{0_{\N^n}}(a_{1},\dots, a_{n})=
\int_{\R}\frac{\sin\tau}{\pi\tau}
\re\Bigl(\prod_{1\le j\le n}
(1-i\tau\mu_{j})^{-1}
\Bigr)d\tau,
\end{equation}
where the explicit expression \eqref{244} is given for the left-hand-side.
\end{rem}
\begin{lem}
  With the notations of Theorem $\ref{thm.specell}$, the function 
  \text{\small$K_{\alpha_{1}, \dots, \alpha_{n}} (a_{1}, \dots, a_{n})$}
   is symmetric in the variables
 \text{\small$(\alpha_{1}, a_{1};\dots; \alpha_{n}, a_{n})$}, i.e. for a permutation $\pi$ of 
 \text{\small$\{1,\dots, n\}$}, we have 
  \begin{equation}\label{}
K_{\alpha_{\pi(1)}, \dots, \alpha_{\pi(n)}} (a_{\pi(1)}, \dots, a_{\pi(n)})=K_{\alpha_{1}, \dots, \alpha_{n}} (a_{1}, \dots, a_{n}).
\end{equation}
\end{lem}
\begin{proof}
 Formula \eqref{kalpha} yields
  \begin{equation*}
K_{\alpha}(a)=\int_{
\substack{\sum s_{j}\ge 1\\s_{j}\ge 0}}
\prod_{1\le j\le n} \bigl(e^{-a_{j}s_{j}}a_{j}
(-1)^{\alpha_{j}}L_{\alpha_{j}}(2a_{j}s_{j})\bigr)
\ ds,
\end{equation*}
and the domain of integration is invariant by permutation of the variables,
entailing the sought result.
\end{proof}
\begin{lem}\label{lem.213}
 With the notations of Theorem \ref{thm.specell}, we have 
 \begin{align*}
&K_{\alpha_{1}, \dots, \alpha_{n}}(a_{1}, \dots, a_{n})=e^{-a_{n}}P_{\alpha_{n}}(a_{n})
\\&\hskip22pt+\int_{0}^{a_{n}}(-1)^{\alpha_{n}}L_{\alpha_{n}}(2t_{n})
e^{-t_{n}}K_{\alpha_{1},\dots, \alpha_{n-1}}
\bigl(a_{1}(1-t_{n}/a_{n}),\dots, a_{n-1}(1-t_{n}/a_{n})\bigr)
dt_{n}
\\
&=e^{-a_{n}}P_{\alpha_{n}}(a_{n})
\\&\hskip22pt+\int_{0}^{1}(-1)^{\alpha_{n}}L_{\alpha_{n}}(2a_{n}\theta)
e^{-\theta a_{n}}K_{\alpha_{1},\dots, \alpha_{n-1}}
\bigl(a_{1}(1-\theta),\dots, a_{n-1}(1-\theta)\bigr)
d\theta a_{n}.
\end{align*}
\end{lem}
\begin{proof}
The domain
of integration is the disjoint union
{\small$$
\left\{\frac{t_{1}}{a_{1}}+\dots+\frac{t_{n-1}}{a_{n-1}}\ge 1-\frac{t_{n}}{a_{n}}, t_{j}\ge 0,
0\le \frac{t_{n}}{a_{n}}\le 1\right\}
\sqcup\left\{
 \frac{t_{n}}{a_{n}}>1, t_{j}\ge 0, 1\le j\le n-1
\right\},
$$}
so that 
\begin{align*}
&K_{\alpha_{1}, \dots, \alpha_{n}}(a_{1}, \dots, a_{n})=e^{-a_{n}}P_{\alpha_{n}}(a_{n})
\\&\hskip22pt+\int_{0}^{a_{n}}(-1)^{\alpha_{n}}L_{\alpha_{n}}(2t_{n})
e^{-t_{n}}K_{\alpha_{1},\dots, \alpha_{n-1}}
\bigl(a_{1}(1-t_{n}/a_{n}),\dots, a_{n-1}(1-t_{n}/a_{n})\bigr)
dt_{n}
\\
&=e^{-a_{n}}P_{\alpha_{n}}(a_{n})
\\&\hskip22pt+\int_{0}^{1}(-1)^{\alpha_{n}}L_{\alpha_{n}}(2a_{n}\theta)
e^{-\theta a_{n}}K_{\alpha_{1},\dots, \alpha_{n-1}}
\bigl(a_{1}(1-\theta),\dots, a_{n-1}(1-\theta)\bigr)
d\theta a_{n},
\end{align*}
which is the sought result.
\end{proof}
\begin{lem}\label{lem.214}
 With the notations of Theorem \ref{thm.specell}, we have, assuming that the $(a_{j})_{1\le j\le n}$ are positive distinct numbers,
\begin{equation}\label{244}
K_{0,\dots, 0}(a_{1},\dots, a_{n})=\sum_{1\le j\le n} e^{-a_{j}} \frac{
\prod_{k\not=j} a_{k}}{\prod_{k\not=j}(a_{k}-a_{j})}.
\end{equation}
\end{lem}
\begin{proof}
The latter formula is true for $n=1$ since we have $K_{0}(a_{1})= e^{-a_{1}}$.
We have also
\begin{align*}
&K_{0\in \N^n}(a_{1},\dots, a_{n})
=
 e^{-a_{ n}}
+
a_{n}\int_{0}^{1}e^{-\theta a_{n}}
K_{0\in \N^{n-1}}\bigl(
a_{1}(1-\theta)
,\dots,
a_{n-1}(1-\theta)\bigr) d\theta
\\&
=
 e^{-a_{ n}}
+
a_{n}\int_{0}^{1}e^{-\theta a_{n}}
\sum_{1\le j\le n-1} e^{-a_{j}(1-\theta)} \frac{
\prod_{k\not=j} a_{k}}{\prod_{k\not=j}(a_{k}-a_{j})}
 d\theta
\\ &=
 e^{-a_{ n}}
+
a_{n}
\sum_{1\le j\le n-1}\frac{
\prod_{k\not=j} a_{k}}{\prod_{k\not=j}(a_{k}-a_{j})}
\int_{0}^{1}e^{-\theta a_{n}}
 e^{-a_{j}(1-\theta)} 
 d\theta
 \end{align*} 
 \begin{align*}
 & =
 e^{-a_{ n}}
+
\sum_{1\le j\le n-1}\frac{a_{n}
\prod_{k\not=j} a_{k}}{\prod_{k\not=j}(a_{k}-a_{j})} e^{-a_{j}}
\int_{0}^{1}e^{\theta(a_{j}- a_{n})} 
 d\theta
  \\
  &=
 e^{-a_{ n}}
+
\sum_{1\le j\le n-1}\frac{a_{n}
\prod_{k\not=j} a_{k}}{\prod_{k\not=j}(a_{k}-a_{j})} e^{-a_{j}}
\frac{e^{a_{j}-a_{n}}-1}{a_{j}-a_{n}}
\\
 & =
 e^{-a_{ n}}
+
\sum_{1\le j\le n-1}\frac{a_{n}
\prod_{k\not=j} a_{k}}{\prod_{k\not=j}(a_{k}-a_{j})} 
\frac{e^{-a_{n}}-e^{-a_{j}}}{(a_{j}-a_{n})}
\\
 &=
 e^{-a_{ n}}\left(1+\sum_{1\le j\le n-1}\frac{a_{n}
\prod_{k\not=j} a_{k}}{\prod_{k\not=j}(a_{k}-a_{j})} 
\frac{1}{(a_{j}-a_{n})}
\right)
\\
&\hskip94pt+
\underbrace{\sum_{1\le j\le n-1}\frac{a_{n}
\prod_{k\not=j} a_{k}}{\prod_{k\not=j}(a_{k}-a_{j})} 
\frac{e^{-a_{j}}}{(a_{n}-a_{j})}}_{\text{OK}}.
\end{align*}
We need to prove that
$$
\left(1+\sum_{1\le j\le n-1}\frac{a_{n}
\prod_{k\not=j, 1\le k\le n-1} a_{k}}{\prod_{k\not=j, 1\le k\le n-1}(a_{k}-a_{j})} 
\frac{1}{(a_{j}-a_{n})}
\right)=
\frac{
\prod_{1\le l\le n-1} a_{l}}{\prod_{1\le l\le n-1}(a_{l}-a_{n})}.
$$
that is 
\begin{multline*}
\prod_{1\le l\le n-1} a_{l}
=\prod_{1\le l\le n-1}(a_{l}-a_{n})\left(1+\sum_{1\le j\le n-1}\frac{a_{n}
\prod_{k\not=j, 1\le k\le n-1} a_{k}}{\prod_{k\not=j, 1\le k\le n-1}(a_{k}-a_{j})} 
\frac{1}{(a_{j}-a_{n})}
\right),
\end{multline*}
which is 
\begin{multline*}
\prod_{1\le l\le n-1} a_{l}=\prod_{1\le l\le n-1}(a_{l}-a_{n})
+\sum_{1\le j\le n-1}\frac{a_{n}
\prod_{k\not=j, 1\le k\le n-1} a_{k}}{\prod_{k\not=j, 1\le k\le n-1}(a_{k}-a_{j})} 
\frac{\prod_{1\le l\le n-1}(a_{l}-a_{n})}{(a_{j}-a_{n})},
\end{multline*}
i.e.
\begin{align}\label{245}
\prod_{1\le l\le n-1} a_{l}=\prod_{1\le l\le n-1}(a_{l}-a_{n})
+\sum_{1\le j\le n-1}\frac{a_{n}
\prod_{k\not=j, 1\le k\le n-1} a_{k}(a_{k}-a_{n})}{\prod_{k\not=j, 1\le k\le n-1}(a_{k}-a_{j})} 
.
\end{align}
Let us reformulate \eqref{245} as an equality between polynomials (to be proven) with
\begin{align}\label{246}
\prod_{1\le l\le n-1}(a_{l}-X)
+\sum_{1\le j\le n-1}\frac{X
\prod_{k\not=j, 1\le k\le n-1} a_{k}(a_{k}-X)}{\prod_{k\not=j, 1\le k\le n-1}(a_{k}-a_{j})} 
-\prod_{1\le l\le n-1} a_{l}=0,
\end{align}
and let us assume that the $(a_{j})_{1\le j\le n-1}$ are distinct and different from $0$. The polynomial $\mathcal Q$
on the left-hand-side has degree less than $n-1$ and we have
$$
\mathcal Q(0)=0, \qquad\text{and  \quad $ \forall j\in \Iintv{1,n-1},
 $}
$$
$$
 \mathcal Q(a_{j})
=\frac{a_{j}
\prod_{k\not=j, 1\le k\le n-1} a_{k}(a_{k}-a_{j})}{\prod_{k\not=j, 1\le k\le n-1}(a_{k}-a_{j})} 
-\prod_{1\le l\le n-1} a_{l}=0,
$$
so that $\mathcal Q$ has degree less than $n-1$ with $n$ distinct roots and this proves the identity \eqref{246} when the $(a_{j})_{1\le j\le n-1}$ are distinct and all different from $0$, proving \eqref{244} in that case; of course we may assume that all $a_{j}$ are positive and noting from \eqref{kalpha} that $K_{\alpha}$ is continuous on $({\R_{+}^*})^n$,
we get Formula \eqref{244} in all cases where all the $a_{j}$ are positive,
concluding the proof of the lemma.
\end{proof}
\begin{lem}\label{lem-3320}
  With the notations of Theorem \ref{thm.specell}, we have, assuming 
$$  0<a_{1}\le\dots\le a_{n},$$
the inequality
  \begin{equation}\label{}
K_{0\in \N^n}(a_{1},\dots, a_{n})\ge \sum_{1\le j\le n}
e^{-a_{j}}\frac{\prod_{1\le l<j}a_{l}}{(j-1)!}\ge e^{-\min_{1\le j\le n} a_{j}}=\max_{1\le j\le n}e^{-a_{j}}.
\end{equation}
\end{lem}
\begin{rem}\rm
 The above estimate is sharp in the sense that when all the $a_{j}$ are equal to the same $a>0$, we have proven in \eqref{253253} that
 \begin{multline*}
K_{0}(a)=\frac{e^{-a}}{(n-1)!}\int_{0}^{+\io} e^{-s} (s+a)^{n-1}
ds=e^{-a}\sum_{0\le l\le n-1} \frac{a^l}{(n-1-l)!l!}\Gamma(n-l)
\\=e^{-a}\sum_{0\le l\le n-1} \frac{a^l}{l!}
=e^{-a}\sum_{1\le j\le n} \frac{a^{j-1}}{(j-1)!}
= \sum_{1\le j\le n}
e^{-a_{j}}\frac{\prod_{1\le l<j}a_{l}}{(j-1)!}_{\vert a_{1}=\dots=a_{n}=a}.
\end{multline*}
\end{rem}
\begin{proof}The property is true for $n=1$ since $K_{0}(a_{1})=e^{-a_{1}}$.
We check the case $n=2$ with $a_{1}<a_{2}$, and we find 
\begin{multline*}
K_{(0,0)}(a_{1},a_{2})=e^{-a_{1}}+\int_{0}^{a_{1}} e^{-t_{1}}
e^{-a_{2}(1-t_{1}/a_{1})} dt_{1}
\\
=e^{-a_{1}}+ e^{-a_{2}}\frac{e^{a_{2}-a_{1}}-1}{\frac{a_{2}}{a_{1}}-1}=
e^{-a_{1}}+ e^{-a_{2}}a_{1}\frac{e^{a_{2}-a_{1}}-1}{{a_{2}}-{a_{1}}}
\ge e^{-a_{1}}+ e^{-a_{2}}a_{1}.
\end{multline*}
 Let us consider for some $n\ge 3$, $0<a_{1}<\dots<a_{n}$ and inductively,
  \begin{align*}
K_{0\in\N^{n}}&(a_{1}, \dots, a_{n})
\\&=e^{-a_{1}}P_{0}(a_{1})+\int_{0}^{a_{1}}
e^{-t_{1}}K_{0\in \N^{n-1}}
\bigl(a_{2}(1-t_{1}/a_{1}),\dots, a_{n}(1-t_{1}/a_{1})\bigr)
dt_{1}
\\&=e^{-a_{1}}P_{0}(a_{1})+a_{1}\int_{0}^{1}
e^{-a_{1}\theta}K_{0\in \N^{n-1}}
\bigl(a_{2}(1-\theta),\dots, a_{n}(1-\theta)\bigr)d\theta
\\
&\ge e^{-a_{1}}+
a_{1}\int_{0}^{1}
e^{-a_{1}\theta}
\sum_{2\le j\le n}
e^{-a_{j}(1-\theta)}\frac{\prod_{2\le l<j}a_{l}}{(j-2)!}
(1-\theta)^{j-2}d\theta
\\
&=
e^{-a_{1}}+\sum_{2\le j\le n} e^{-a_{j}}
\Bigl(\underbrace{a_{1}
\prod_{2\le l<j}a_{l}}_{\prod_{1\le k<j}a_{k}}\Bigr)
\int_{0}^{1}
e^{(a_{j}-a_{1})\theta}
\frac{1}{(j-2)!}
(1-\theta)^{j-2}d\theta
\\&\ge 
e^{-a_{1}}+\sum_{2\le j\le n} e^{-a_{j}}
\Bigl(\prod_{1\le k<j}a_{k}\Bigr)
\int_{0}^{1}
\frac{1}{(j-2)!}
(1-\theta)^{j-2}d\theta
\\
&=e^{-a_{1}}+\sum_{2\le j\le n} e^{-a_{j}}
\Bigl(\prod_{1\le k<j}a_{k}\Bigr)
\frac{1}{(j-1)!},
\end{align*} 
concluding the proof of the lemma.
\end{proof}
\begin{rem}\rm 
 The reader may have noticed that it is not obvious on Formula \eqref{244}
 \begin{equation*}
K_{0,\dots, 0}(a_{1},\dots, a_{n})=\sum_{1\le j\le n} e^{-a_{j}} \frac{
\prod_{k\not=j} a_{k}}{\prod_{k\not=j}(a_{k}-a_{j})},
\end{equation*}
that $K_{0}$ is an entire function.
Let us start with taking a look at 
\begin{align}
K_{0,0}(a_{1}, a_{2})&= \frac{e^{-a_{1}}a_{2}}{a_{2}-a_{1}}+\frac{e^{-a_{2}}a_{1}}{a_{1}-a_{2}}=\frac{a_{2}e^{-a_{1}}-a_{1}e^{-a_{2}}}{a_{2}-a_{1}}
\notag\\&=e^{-\frac{(a_{1}+a_{2})}{2}}
\frac{a_{2}e^{-\frac{a_{1}}2+\frac{a_{2}}{2}}-a_{1}e^{-\frac{a_{2}}2+\frac{a_{1}}2}}{a_{2}-a_{1}}
\notag\\
&=e^{-\frac{(a_{1}+a_{2})}{2}}
\frac{a_{2}
(\cosh \frac{a_{2}-a_{1}}{2}+\sinh  \frac{a_{2}-a_{1}}{2})
-a_{1}(\cosh \frac{a_{1}-a_{2}}{2}+\sinh  \frac{a_{1}-a_{2}}{2})
}{a_{2}-a_{1}}
\notag\\
&=e^{-\frac{(a_{1}+a_{2})}{2}}
\Bigl[\cosh (\frac{a_{2}-a_{1}}{2})
+
\frac{(a_{2}+a_{1})\sinh ( \frac{a_{2}-a_{1}}{2})}{a_{2}-a_{1}}\Bigr]
\notag\\
&
=e^{-\frac{(a_{1}+a_{2})}{2}}
\Bigl[\cosh (\frac{a_{2}-a_{1}}{2})
+
\frac{\frac12(a_{2}+a_{1})\sinh ( \frac{a_{2}-a_{1}}{2})}{\frac{a_{2}-a_{1}}2}\Bigr]
\notag\\
&
=e^{-\frac{(a_{1}+a_{2})}{2}}
\Bigl[\cosh (\frac{a_{2}-a_{1}}{2})
+
\frac12(a_{2}+a_{1})
\shc (\frac{a_{2}-a_{1}}{2})\Bigr],
\label{entire}
\end{align}
where $\shc$ stands for   the even entire function defined by 
\begin{equation}\label{shc}
\shc t=\frac{\sinh t}{t}.
\end{equation}
\index{{~\bf Notations}!$\shc t=\frac{\sinh t}{t}$}
\end{rem}
We have also from Lemma \ref{lem320}
\begin{equation}
F_{\alpha}(a)
=
\int_{\R}\frac{\sin\tau}{\pi\tau}
\Bigl(\prod_{1\le j\le n}\frac{(1+i\tau\mu_{j})^{\alpha_{j}}}{
(1-i\tau\mu_{j})^{\alpha_{j}+1}
} \Bigr) d\tau,
\end{equation}
and defining the function $F_{\alpha}(a,\lambda)$ as the absolutely converging integral,
\begin{equation}
F_{\alpha}(a,\lambda)
=
\int_{\R}\frac{\sin(\lambda\tau)}{\pi\tau}
\Bigl(\prod_{1\le j\le n}\frac{(1+i\tau\mu_{j})^{\alpha_{j}}}{
(1-i\tau\mu_{j})^{\alpha_{j}+1}
} \Bigr) d\tau,\quad F_{\alpha}(a)= F_{\alpha}(a,1),
\end{equation}
we get
\begin{align*}
&\frac{\p F_{\alpha}}{\p\lambda}(a,\lambda)
=\frac1\pi
\int_{\R}{\cos(\lambda\tau)}
\Bigl(\prod_{1\le j\le n}\frac{(1+i\tau\mu_{j})^{\alpha_{j}}}{
(1-i\tau\mu_{j})^{\alpha_{j}+1}
} \Bigr) d\tau
\\&
=\frac1{2\pi}
\int_{\R}e^{i\lambda \tau}
\Bigl(\prod_{1\le j\le n}\frac{(1+i\tau\mu_{j})^{\alpha_{j}}}{
(1-i\tau\mu_{j})^{\alpha_{j}+1}
} \Bigr) d\tau
\\&\hskip44pt+
\frac1{2\pi}
\int_{\R}e^{i\lambda \tau}
\Bigl(\prod_{1\le j\le n}\frac{(1-i\tau\mu_{j})^{\alpha_{j}}}{
(1+i\tau\mu_{j})^{\alpha_{j}+1}
} \Bigr) d\tau
\\
&=
\frac1{2\pi}
\int_{\R}e^{i\lambda \tau}
\Bigl(
\prod_{1\le j\le n}\frac{(1+i\tau\mu_{j})^{\alpha_{j}}}{
(1-i\tau\mu_{j})^{\alpha_{j}+1}
} 
+
\prod_{1\le j\le n}\frac{(1-i\tau\mu_{j})^{\alpha_{j}}}{
(1+i\tau\mu_{j})^{\alpha_{j}+1}
} 
\Bigr) d\tau
\\&=i\sum_{1\le j\le n}\text{Res}\left(
e^{i\lambda \tau}
\prod_{1\le j\le n}\frac{(1-i\tau\mu_{j})^{\alpha_{j}}}{
(1+i\tau\mu_{j})^{\alpha_{j}+1}
} ;\tau=i/\mu_{j}=ia_{j}
\right)
\\&
=i\sum_{1\le j\le n}\text{Res}\left(
e^{i\lambda \tau}
\prod_{1\le j\le n}\frac{(-i\mu_{j})^{\alpha_{j}}
(ia_{j}+\tau)^{\alpha_{j}
}
}{
(i\mu_{j})^{\alpha_{j}+1}(-ia_{j}+\tau)^{\alpha_{j}+1}
} ;\tau=ia_{j}
\right)
\\&
=\frac1{i^{n-1}}\sum_{1\le j\le n}\text{Res}\left(
e^{i\lambda \tau}
\prod_{1\le j\le n}
(-1)^{\alpha_{j}}\frac{
a_{j}
(ia_{j}+\tau)^{\alpha_{j}
}
}{
(\tau-ia_{j})^{\alpha_{j}+1}
} ;\tau=ia_{j}
\right),
\end{align*}
so that assuming  that the $a_{j}$ are positive and distinct, we get 
\begin{align*}
&\frac{\p F_{\alpha}}{\p\lambda}(a,\lambda)=
\frac1{i^{n-1}}(\prod a_{k})
\\&\times\sum_{1\le j\le n}\frac1{\alpha_{j}!}
\left(\frac{d}{d\tau}\right)^{\alpha_{j}}\left(
e^{i\lambda \tau}
(-1)^{\alpha_{j}}
(ia_{j}+\tau)^{\alpha_{j}}
\prod_{1\le k\le n, k\not=j}
(-1)^{\alpha_{k}}\frac{
(ia_{k}+\tau)^{\alpha_{k}
}
}{
(\tau-ia_{k})^{\alpha_{k}+1}
} \right)_{\vert\tau=ia_{j}}
\\
&=\frac1{i^{n-1}}(\prod_{1\le k\le n} a_{k})
\sum_{1\le j\le n}\frac1{\alpha_{j}!}
\\&\hskip44pt \times
\left(\frac{d}{id\sigma}\right)^{\alpha_{j}}\left(
e^{-\lambda\sigma}
(-1)^{\alpha_{j}}
(ia_{j}+i\sigma)^{\alpha_{j}}
\hskip-12pt
\prod_{1\le k\le n, k\not=j}
(-1)^{\alpha_{k}}\frac{
(ia_{k}+i\sigma)^{\alpha_{k}
}
}{
(i\sigma-ia_{k})^{\alpha_{k}+1}
} \right)_{\vert\sigma=a_{j}}
\end{align*}
\begin{align*}
&=(-1)^{n-1+\val \alpha}(\prod_{1\le k\le n} a_{k})
\sum_{1\le j\le n}\frac1{\alpha_{j}!}
\\&
\hskip44pt\times\left(\frac{d}{d\sigma}\right)^{\alpha_{j}}
\left(
e^{-\lambda\sigma}
(a_{j}+\sigma)^{\alpha_{j}}
\hskip-12pt
\prod_{1\le k\le n, k\not=j}
\frac{
(a_{k}+\sigma)^{\alpha_{k}
}
}{
(\sigma-a_{k})^{\alpha_{k}+1}
} \right)_{\vert\sigma=a_{j}}
\\
&=(\prod_{1\le k\le n} a_{k})
\sum_{1\le j\le n}\frac{(-1)^{\alpha_{j}}}{\alpha_{j}!}\left(\frac{d}{d\sigma}\right)^{\alpha_{j}}\left(
e^{-\lambda\sigma}
(a_{j}+\sigma)^{\alpha_{j}}
\hskip-12pt
\prod_{1\le k\le n, k\not=j}
\frac{
(a_{k}+\sigma)^{\alpha_{k}
}
}{
(a_{k}-\sigma)^{\alpha_{k}+1}
} \right)_{\vert\sigma=a_{j}}.
\end{align*}
Since $F_{\alpha}(a,+\io)=1$, thanks to Lemma \ref{lem93}, we find eventually that 
\begin{align*}
&F_{\alpha}(a)=F_{\alpha}(a,1)=\int_{+\io}^1
\frac{\p F_{\alpha}}{\p\lambda}(a,\lambda) d\lambda+1=1-K_{\alpha}(a),
\\
&K_{\alpha}(a)=
(\prod_{1\le k\le n} a_{k})
\sum_{1\le j\le n}
\frac{(-1)^{\alpha_{j}}}{\alpha_{j}!}
\\&\hskip75pt
\int^{+\io}_{1}
\left(\frac{d}{d\sigma}\right)^{\alpha_{j}}\left(
e^{-\lambda\sigma}
(a_{j}+\sigma)^{\alpha_{j}}
\hskip-12pt
\prod_{1\le k\le n, k\not=j}
\frac{
(a_{k}+\sigma)^{\alpha_{k}
}
}{
(a_{k}-\sigma)^{\alpha_{k}+1}
} \right)_{\vert\sigma=a_{j}} d\lambda
\\
&=
\sum_{1\le j\le n}
\frac{(-1)^{\alpha_{j}}}{\alpha_{j}!}
\\&\int^{+\io}_{1} e^{-\lambda a_{j}}\!\!\!
\left(\frac{d}{d\sigma}-\lambda\right)^{\alpha_{j}}\left(
(a_{j}+\sigma)^{\alpha_{j}} a_{j}
\hskip-12pt
\prod_{1\le k\le n, k\not=j}
\frac{
(a_{k}+\sigma)^{\alpha_{k}
}
a_{k}
}{
(a_{k}-\sigma)^{\alpha_{k}+1}
} \right)_{\vert\sigma=a_{j}} d\lambda.
\\
&=
\sum_{1\le j\le n}
\frac{(-1)^{\alpha_{j}}}{\alpha_{j}!}
\int^{+\io}_{1} e^{-\lambda a_{j}}
\left(\frac{d}{d\sigma}-\lambda\right)^{\alpha_{j}}\left(
(a_{j}+\sigma)^{\alpha_{j}}
\hskip-12pt
\prod_{1\le k\le n, k\not=j}
\frac{
(a_{k}+\sigma)^{\alpha_{k}
}
}{
(a_{k}-\sigma)^{\alpha_{k}+1}
} \right)_{\vert\sigma=a_{j}} d\lambda
\\
&
=
\sum_{1\le j\le n}
\frac{(-1)^{\alpha_{j}}}{\alpha_{j}!}
\int^{+\io}_{a_{j}} \!\!\! e^{-t_{j}}
\\&
\hskip32pt
\left(\frac{d}{da_{j}s}-\frac{t_{j}}{a_{j}}\right)^{\alpha_{j}}\left(
(a_{j}+a_{j}s)^{\alpha_{j}}
\hskip-12pt
\prod_{1\le k\le n, k\not=j}
\frac{ a_{k}
(a_{k}+a_{j}s)^{\alpha_{k}
}
}{
(a_{k}-a_{j}s)^{\alpha_{k}+1}
} \right)_{\vert s=1} {dt_{j}}
\end{align*}
\begin{align*}
\\&=
\sum_{1\le j\le n}
\frac{(-1)^{\alpha_{j}}}{\alpha_{j}!}
\int^{+\io}_{a_{j}} e^{-t}
\left(\frac{d}{ds}-{t}\right)^{\alpha_{j}}\left(
(1+s)^{\alpha_{j}}
\hskip-12pt
\prod_{1\le k\le n, k\not=j}
\frac{ a_{k}
(a_{k}+a_{j}s)^{\alpha_{k}
}
}{
(a_{k}-a_{j}s)^{\alpha_{k}+1}
} \right)_{\vert s=1} {dt}
\\
&=
\sum_{1\le j\le n}
\frac{(-1)^{\alpha_{j}}}{\alpha_{j}!}
\int^{+\io}_{a_{j}} e^{-t}
\left(\frac{d}{ds}-{1}\right)^{\alpha_{j}}\left(
(t+s)^{\alpha_{j}}
\hskip-12pt
\prod_{1\le k\le n, k\not=j}
\frac{ a_{k}
(a_{k}+a_{j}s/t)^{\alpha_{k}
}
}{
(a_{k}-a_{j}s/t)^{\alpha_{k}+1}
} \right)_{\vert s=t} {dt}
\\
&=
\sum_{1\le j\le n}
\frac{(-1)^{\alpha_{j}}}{\alpha_{j}!}
\int^{+\io}_{a_{j}} e^{-t}
\\
&\hskip44pt\left(\frac{d}{d(s+t)}-{1}\right)^{\alpha_{j}}\left(
(t+s)^{\alpha_{j}}
\hskip-12pt
\prod_{1\le k\le n, k\not=j}
\frac{ ta_{k}
(t(a_{k}-a_{j})+a_{j}(s+t))^{\alpha_{k}
}
}{
(t(a_{k}+a_{j})-a_{j}(s+t))^{\alpha_{k}+1}
} \right)_{\vert s+t=2t} {dt}
\end{align*}
\begin{align*}
\\
&=
\sum_{1\le j\le n}
{(-1)^{\alpha_{j}}}
\\&\times
\int^{+\io}_{a_{j}} e^{-t}
\left(\frac{d}{ds}-{1}\right)^{\alpha_{j}}\left(
\frac{s^{\alpha_{j}}}{\alpha_{j}!}
\prod_{1\le k\le n, k\not=j}
\frac{ ta_{k}
(t(a_{k}-a_{j})+a_{j}s)^{\alpha_{k}
}
}{
(t(a_{k}+a_{j})-a_{j}s)^{\alpha_{k}+1}
} \right)_{\vert s=2t} {dt}
\end{align*}
\begin{align*}
&=
\sum_{1\le j\le n}
{(-1)^{\alpha_{j}}} e^{-a_{j}}
\int^{+\io}_{0} e^{-t}
\\&\times 
\left(\frac{d}{ds}-{1}\right)^{\alpha_{j}}\left(
\frac{s^{\alpha_{j}}}{\alpha_{j}!}
\prod_{1\le k\le n, k\not=j}
\frac{ (t+a_{j})a_{k}
\bigl((t+a_{j})(a_{k}-a_{j})+a_{j}s\bigr)^{\alpha_{k}
}
}{
\bigl((t+a_{j})(a_{k}+a_{j})-a_{j}s\bigr)^{\alpha_{k}+1}
} \right)_{\vert s=2t+2a_{j}} {dt}.
\end{align*}
We have also to deal with 
$$
\prod_{1\le k\le n, k\not=j}
\frac{ (t+a_{j})a_{k}
\bigl((t+a_{j})(a_{k}-a_{j})+a_{j}s\bigr)^{\alpha_{k}
}
}{
\bigl((t+a_{j})(a_{k}+a_{j})-a_{j}s\bigr)^{\alpha_{k}+1}
} 
$$
and 
$$
\bigl((t+a_{j})(a_{k}+a_{j})-a_{j}(2t+2a_{j})\bigr)=
a_{j}(a_{k}+a_{j})-2a_{j}^2+t(a_{k}-a_{j})=(t+a_{j})(a_{k}-a_{j})
$$
$$
(t+a_{j})(a_{k}+a_{j})-a_{j}s=
(t+a_{j})(a_{k}-a_{j})+a_{j}(2t+2a_{j}-s)
$$
so that 
\begin{multline}\label{simplek}
K_{\alpha}(a)=
\sum_{1\le j\le n}
{(-1)^{\alpha_{j}}} e^{-a_{j}}
\int^{+\io}_{0} e^{-t}
\\\times\left(\frac{d}{ds}-{1}\right)^{\alpha_{j}}\!\!\!
\biggl(
\frac{s^{\alpha_{j}}}{\alpha_{j}!}
\!\!\prod_{1\le k\le n, k\not=j}\!\!\!\!\!\!\!\!\!
\frac{ (t+a_{j})a_{k}
\bigl((t+a_{j})(a_{k}+a_{j})+a_{j}(s-2t -2a_{j})\bigr)^{\alpha_{k}
}
}{
\bigl(
(t+a_{j})(a_{k}-a_{j})-a_{j}(s-2t-2a_{j})
\bigr)^{\alpha_{k}+1}
} \biggr)_{\vert s=2t+2a_{j}} \hskip-20pt{dt}.
\end{multline}
\vs
\subsection{A conjecture on integrals of products of Laguerre polynomials}\label{sec.oiu666}
We formulate in this section  a conjecture on the behaviour of the functions $K_{\alpha}(a)$;
 as displayed in the previous sections, we know several  useful elements for the analysis of these functions, including some quite explicit expression. However, in the non-isotropic case, we were not able to prove the estimate $F_{\alpha}(a)\le 1$, equivalent to $K_{\alpha}(a)\ge 0$, except for the case $\alpha=0$. We are thus reduced to conjectural statements.
 \begin{conjecture}\label{conj-quad}Let $n\ge 1$ be an integer and let $\alpha=(\alpha_{1},\dots, \alpha_{n})\in \N^n$. For $a=(a_{1}, \dots, a_{n})\in (0,+\io)^n$ we define
 \begin{equation}\label{}
K_{\alpha}(a)=\int_{
\substack{t=(t_{1},\dots, t_{n})\in \R_{+}^n\\\sum_{1\le j\le n} t_{j}/a_{j}\ge 1}}
e^{-(t_{1}+\dots+t_{n})}
\prod_{1\le j\le n} 
(-1)^{\alpha_{j}}L_{\alpha_{j}}(2t_{j})
dt,
\end{equation}
where $L_{k}$ stands for the classical Laguerre polynomial
\begin{equation}\label{}
L_{k}(X)=\Bigl(\frac{d}{dX}-1\Bigr)^k\frac{X^k}{k!}.
\end{equation}
Then we conjecture that, assuming $0< a_{1}\le \dots\le a_{n}$, we have 
\begin{equation}\label{343}
K_{\alpha}(a)\ge
\sum_{1\le j\le n}
e^{-a_{j}}\frac{\prod_{1\le l<j}a_{l}}{(j-1)!}.
\end{equation}
\end{conjecture}
\begin{rem}\rm
 A slightly stronger and more symmetrical version of the above conjecture
 is that for $n, \alpha,a, K_{\alpha}$ as above, we have  
 \begin{equation}\label{343+}
K_{\alpha}(a)\ge
K_{0}(a).
\end{equation}
It is indeed stronger since we have proven in Lemma $\ref{lem-3320}$ that $K_{0}(a)$
 is greater than the right-hand-side of \eqref{343}.
\end{rem}
\begin{theorem}
 The previous conjecture is a proven theorem in the following cases.
 \begin{itemize}
\item[(1)]  When $n=1$.
\item[(2)]  For all $n\ge 1$, when all the $a_{j}$ are equal.
\item[(3)]  For all $n\ge 1$, when $\alpha=0_{\N^n}$.
\item[(4)]  When $n=2$ and $\min(\alpha_{1},\alpha_{2})=0$.
\end{itemize}
\end{theorem}
\begin{proof}
 $\mathbf {(1)}$
 When $n=1$, we have proven above (in Proposition \ref{pro.2424})
that for $\alpha\in \N$, $a>0$,
\begin{equation}\label{}
 K_{\alpha}(a)=e^{-a}P_{\alpha}(a)\ge e^{-a},
\end{equation}
which is indeed \eqref{343+} in that case.
With the notations of Theorem \ref{thm-fl0001}
 (and in particular  where $D_{a}$ is defined in 
 \eqref{fl0001})
this implies
 \begin{equation}\label{}
\opw{\mathbf 1_{D_{a}}}
\le 1-e^{-a},
\end{equation}
an inequality due to P.~Flandrin in the 1988 paper \cite{conjecture}.
\par\no
$\mathbf {(2)}$ Assuming that all the $a_{j}$ are  equal to $a>0$,
 we have proven in Theorem \ref{cor} that for $\alpha\in \N^n$, $\val \alpha=\sum_{1\le j\le n}\alpha_{j}$,
 \begin{equation}\label{}
K_{\alpha}(a, \dots, a)\ge \frac{\Gamma(n, a)}{\Gamma(n)}=
e^{-a}\sum_{1\le j\le n}\frac{a^{j-1}}{(j-1)!}=K_{0}(a,\dots,a)
\end{equation}
since  from \eqref{kalpha}, we have 
\begin{align*}
K_{0}(a,\dots, a)&=\int_{
\substack{\sum t_{j}\ge a\\t_{j}\ge 0}}
e^{-(t_{1}+\dots+t_{n})}
dt\\&=\int_{\substack{t_{n}\ge a\\ t_{j}\ge 0}}e^{-(t_{1}+\dots+t_{n})}
dt+\int_{0}^{a} e^{-t_{n}}\int_{\sum t_{j}\ge a-t_{n}}
e^{-(t_{1}+\dots+t_{n-1})} dt
\\\text{\footnotesize (inductively)\quad}&=e^{-a}+
\int_{0}^{a} e^{-t_{n}} e^{-(a-t_{n})}
\sum_{1\le j\le n-1}\frac{(a-t_{n})^{j-1}}{(j-1)!} dt_{n}
\\
&=e^{-a}\Bigl(1+\sum_{1\le j\le n-1}\frac{a^j}{j!}\Bigr)=
e^{-a}\sum_{1\le j\le n}\frac{a^{j-1}}{(j-1)!},
\end{align*}
proving  \eqref{343+} in that case.
With \begin{equation}\label{fl0001+}
D_{(a)}=\{(x,\xi)\in \R^{2n}, 2\pi \frac{\val x^{2}+\val\xi^{2}}a\le 1\},
\end{equation}
this implies that
\begin{equation}\label{}
\opw{\mathbf 1_{D_{(a)}}}
\le 1-e^{-a}\sum_{1\le j\le n}\frac{a^{j-1}}{(j-1)!},
\end{equation}
an inequality proven in the  2010 article \cite{MR2761287} by E.~Lieb
and Y.~Ostrover.
\par\no
$\mathbf {(3)}$ When $\alpha=0_{\N^n}$, we have proven \eqref{343} in Lemma \ref{lem-3320}.
 \par\no
$\mathbf {(4)}$ When $n=2$, from the case $n=1$ we have 
 $
 K_{\alpha_{2}}(a_{2})= e^{-a_{2}} P_{\alpha_{2}}(a_{2}),
 $
 so that  from Lemma \ref{lem.213}, we obtain 
 $$
 K_{\alpha_{1},\alpha_{2}}(a_{1}, a_{2})=e^{-a_{1}} P_{\alpha_{1}}(a_{1})
 +a_{1}\int_{0}^1 e^{-\theta a_{1}-(1-\theta) a_{2}}
 (-1)^{\alpha_{1}} L_{\alpha_{1}}(2\theta a_{1})P_{\alpha_{2}}(a_{2}(1-\theta)) d\theta,
 $$
 and if $\alpha_{1}$=0, it means that 
  \begin{multline*}
 K_{0,\alpha_{2}}(a_{1}, a_{2})=e^{-a_{1}} 
 +a_{1}\int_{0}^1 e^{-\theta a_{1}-(1-\theta) a_{2}}
P_{\alpha_{2}}(a_{2}(1-\theta)) d\theta
\\\ge 
e^{-a_{1}} 
 +a_{1}\int_{0}^1 e^{-\theta a_{1}-(1-\theta) a_{2}}
 d\theta=K_{0,0}(a_{1},a_{2}),
\end{multline*}
and the reasoning is identical for $\alpha_{2}=0$, 
concluding  the proof of the theorem.
\end{proof}
We are interested in the Weyl quantization of the indicatrix of 
\begin{equation}\label{}
D_{a_{1},\dots, a_{n}}=\{(x,\xi)\in \R^{2n}, 2\pi \sum_{1\le j\le n}\frac{x_{j}^{2}+\xi_{j}^{2}}{a_{j}}\le 1\}, \quad a_{j}>0,
\end{equation}
and
we have a weaker conjecture.
 \begin{conjecture}[A weak form of Conjecture \ref{conj-quad}] \label{conj-quad--}
 With
 $n,\alpha, a, K_{\alpha}$ as in Conjecture \ref{conj-quad},
 we conjecture that 
 \begin{equation}\label{343--}
K_{\alpha}(a)\ge
0.\end{equation}
Note that Inequality \eqref{343--} is equivalent to
\begin{equation}\label{}
\opw{\mathbf 1_{D_{a_{1},\dots, a_{n}}}}
\le 1.
\end{equation}
\end{conjecture}
\begin{rem}\rm
 In the first place, although the second conjecture is much weaker than the first, there is no reason to believe that the weak conjecture should be easier to prove than the first: in particular, in the known cases, it is indeed the proof of the precise statement \eqref{343} which leads to \eqref{343--} and we are not aware of a direct proof of \eqref{343--}, even in one dimension. 
\end{rem}
\vs\no
{\bf A summary of our knowledge on the functions $K_{\alpha}$}.
As proven in Remarks \ref{rem-312} and  \ref{rem-313}, the functions $K_{\alpha}$ are entire functions given on the open subset \eqref{336} by Formula \eqref{kalpha} (see also Formula \eqref{entire}).
Moreover the function $F_{\alpha}(a)=1-K_{\alpha}(a)$ can be expressed as a simple integral for $a_{j}>0$,
\begin{equation}
F_{\alpha}(a_{1}, \dots, a_{n})
=
\int_{\R}\frac{\sin\tau}{\pi\tau}
\Bigl(\prod_{1\le j\le n}\frac{(1+i\tau\mu_{j})^{\alpha_{j}}}{
(1-i\tau\mu_{j})^{\alpha_{j}+1}
} \Bigr) d\tau,
\quad
\mu_{j}=\frac{1}{a_{j}},
\end{equation}
and we have an explicit expression of the function $K_{\alpha}$ as a sum of simple integrals in \eqref{simplek}. However, having an explicit expression
does not mean much and for instance, we do have several explicit expressions for the Laguerre
 polynomials but Inequality \eqref{inelag} remains very hard work,
 requiring a deep understanding of these polynomials.
We have also an induction formula
in Lemma \ref{lem.213}.
 As a further remark, we have the following
 \begin{lem}
 Let $n, \alpha, a, K_{\alpha}$ as in Conjecture $\ref{conj-quad}$. Then we have 
 \begin{align}
&\lim_{a_{n}\rightarrow +\io} K_{\alpha_{1},\dots, \alpha_{n-1}, \alpha_{n}}(a_{1},\dots, a_{n-1}, a_{n})=
K_{\alpha_{1},\dots, \alpha_{n-1}}(a_{1},\dots, a_{n-1}),
\label{1111}
\\
&\lim_{a_{1}\rightarrow {0}_{+}} K_{\alpha_{1}, \alpha_{2},\dots, \alpha_{n}}(a_{1},a_{2},\dots, a_{n})=
1.
\label{2222}
\end{align}
 \end{lem}
 \begin{proof}
Formula \eqref{kalpha}  and Lebesgue Dominated Convergence Theorem imply the first equality \eqref{1111}.
Lemma \ref{lem.213}, in which we may swap the variables $a_{1}$ and $a_{n}$ gives for $a_{1}>0$
\begin{multline*}
K_{\alpha_{1}, \alpha_{2},\dots, \alpha_{n}}(a_{1},a_{2},\dots, a_{n})=e^{-a_{1}}P_{\alpha_{1}}(a_{1})
\\+a_{1}\int_{0}^{1} e^{-\theta a_{1}} (-1)^{\alpha_{1}}L_{\alpha_{1}}(2a_{1}\theta) K_{\alpha_{2},\dots, \alpha_{n}}
\bigl(a_{2}(1-\theta),\dots, a_{n}(1-\theta)\bigr) d\theta,
\end{multline*}
and since $P_{\alpha_{1}}$ is a polynomial such that $P_{\alpha_{1}}(0)=1$, we get \eqref{2222}.
\end{proof}
\vs\no
{\bf Reasons to believe in the conjecture}.
This is true in one dimension, also in $n$ dimensions for spheres and it is a quadratic problem in the sense that ellipsoids are convex subsets of $\RZ$ characterized by an inequality 
$$
\{X\in \RZ, p(X)\le 0\},
$$
where 
 $p$ is a  polynomial of degree 2 with a positive-definite quadratic part.
 We shall see below in this paper that convexity of a set $A$ does not guarantee that the quantization 
 $\opw{\mathbf 1_{A}}$ is smaller than 1 as an operator and that  Flandrin's conjecture is not true,
 but it is hard to believe that  such a phenomenon could occur for ellipsoids.
 We must point out a specific feature of anisotropy related to Mehler's formula \eqref{mehler11}: if all the $\mu_{j}$ are equal to the same $\mu>0$ (this is the isotropic case),
 then,  with $q_{\mu}(x,\xi)=\mu(\val x^2+\val \xi^2)$, we have 
 $$
 \opw{e^{2i\pi \tau q_{\mu}(x,\xi)}}
 = \phi(\tau \mu) e^{2i\arctan(\tau \mu)\sum_{1\le j\le n}\pi(x_{j}^2+D_{j}^2)},
 $$
 where $\phi(\tau \mu)$ is a scalar quantity. As a consequence, if we quantize $F(q_{\mu}(x,\xi))$, we get
 $$
 \OPW{F\bigl(q_{\mu}(x,\xi)\bigr)}
 =\int_{\R}\hat F(\tau)\phi(\tau \mu)
 e^{2i\frac{\arctan(\tau \mu)}{\mu} \pi \opw{q_{\mu}}} d\tau,
 $$
 and thus
$$
\OPW{F\bigl(q_{\mu}(x,\xi)\bigr)}
 =\wt{F}( \opw{q_{\mu}}),
 \qquad 
 \wt{F}(\lambda)=\int_{\R}\hat F(\tau)\phi(\tau \mu) e^{2i\pi \frac{\arctan(\tau\mu)}{\mu}\lambda} d\tau,
$$
and $\OPW{F\bigl(q_{\mu}(x,\xi)\bigr)}$ appears as a function of the self-adjoint operator $\opw{q_{\mu}}$.
Following the same route in the anisotropic case, we get,
with
\begin{align}
q_{\mu}(x,\xi)&=\sum_{1\le j\le n}\mu_{j}(x_{j}^2+\xi_{j}^2),
\\
 \OPW{F\bigl(q_{\mu}(x,\xi)\bigr)}&=
 \int_{\R}\hat F(\tau)\phi(\tau \mu)
 e^{2i\pi \sum_{1\le j\le n}(\frac{\arctan(\tau \mu_{j}}{\mu_{j}})\mu_{j}(x_{j}^2+D_{j}^2)} d\tau,
\label{rrae34}\end{align}
and since  $\frac1{\mu_{j}}\arctan(\tau \mu_{j})$ does depend on $\mu_{j}$ (and not only on
$\tau$), the operator 
$\OPW{F\bigl(q_{\mu}(x,\xi)\bigr)}$ is not a function  of 
the self-adjoint operator $\opw{q_{\mu}}$.
\vs
As a final comment on the strongest form of the Conjecture \eqref{343+},
we would say that it could be seen as a property of the Laguerre polynomials, known in the case  $n=1$, where it stands as follows: 
we define for $k\in \N$, the polynomial $P_{k}$ by 
\begin{equation}\label{}
P_{k}(x)=\int_{0}^{+\io} e^{-t}(-1)^kL_{k}(2x+2t) dt,
\end{equation}
and we have $ P_{k}(0)=1$
 from \eqref{foulag+}. Moreover, we have the inequality (equivalent to \eqref{343+} for $n=1$)
\begin{equation}\label{iuytre}
\forall x\ge 0, \quad P_{k}(x)\ge P_{k}(0).
\end{equation}
We note that 
$
e^{-x} P_{k}(x)=\int_{x}^{+\io} e^{-s}(-1)^kL_{k}(2s) ds,
$
so that the unique solution $P_{k}$
of the Initial Value Problem for the ODE
\begin{equation}\label{dfretz}
P_{k}(x)-P'_{k}(x)=(-1)^kL_{k}(2x), \quad P_{k}(0)=1, 
\end{equation}
does satisfy \eqref{iuytre}.
We note that  from Lemma \ref{lem13}, we have 
$$
P'_{k}(X)=2\sum_{0\le l<k}(-1)^lL_{l}(2X),
$$
so that \eqref{iuytre} is a consequence of Feldheim Inequality
\eqref{inelag}. Let us reformulate  \eqref{343+}, using the polynomials $P_{k}$: for $a_{j}\ge 0$, 
 \begin{multline}\label{}
K_{\alpha}(a)=\int_{
\substack{t=(t_{1},\dots, t_{n})\in \R_{+}^n\\\sum_{1\le j\le n} t_{j}/a_{j}\ge 1}}
\prod_{1\le j\le n} \frac{\p}{\p t_{j}}\Bigl\{-e^{-t_{j}}
P_{\alpha_{j}}(t_{j})\Bigr\}
dt 
\\\ge K_{0}(a)= \int_{
\substack{t=(t_{1},\dots, t_{n})\in \R_{+}^n\\\sum_{1\le j\le n} t_{j}/a_{j}\ge 1}}
\prod_{1\le j\le n} \frac{\p}{\p t_{j}}\Bigl\{-e^{-t_{j}}
\Bigr\}
dt,
\end{multline}
which is equivalent to 
 \begin{multline}\label{}
\int
H\bigl(1-\sum_{1\le j\le n} s_{j}\bigr)
\prod_{1\le j\le n}H(s_{j}) \frac{\p}{\p s_{j}}\Bigl\{-e^{-a_{j}s_{j}}
P_{\alpha_{j}}(a_{j}s_{j})\Bigr\}
ds
\\\le\int H\bigl(1-\sum_{1\le j\le n} s_{j}\bigr)
\prod_{1\le j\le n} a_{j} H(s_{j})e^{-
a_{j}s_{j}}
ds,
\end{multline}
\index{{~\bf Notations}!$H(x)=\mathbf 1_{[0,+\io)}(x)$}
\index{Heaviside function}
where $H=\mathbf 1_{\R_{+}}$ (Heaviside function).
This  is equivalent to 
 \begin{multline}\label{}
\int
H\bigl(1-\sum_{1\le j\le n} s_{j}\bigr)
\prod_{1\le j\le n}H(s_{j}) e^{-a_{j}s_{j}}\left(a_{j}-\frac{\p}{\p s_{j}}\right)\Bigl\{
P_{\alpha_{j}}(a_{j}s_{j})\Bigr\}
ds
\\\le\int H\bigl(1-\sum_{1\le j\le n} s_{j}\bigr)
\prod_{1\le j\le n} a_{j}H(s_{j})e^{-
a_{j}s_{j}}
ds,
\end{multline}
i.e. to
 \begin{multline}\label{}
\int H\bigl(1-\sum_{1\le j\le n} s_{j}\bigr)
\prod_{1\le j\le n} H(s_{j})e^{-
a_{j}s_{j}}
\\
\times\left(\prod_{1\le j\le n}a_{j}-
\prod_{1\le j\le n}\left(a_{j}-\frac{\p}{\p s_{j}}\right)\Bigl\{
P_{\alpha_{j}}(a_{j}s_{j})\Bigr\}
\right)
ds\ge 0.
\end{multline}
Note that for $n=1$, it means  for $a\ge 0$, 
\begin{multline*}
\int_{0}^1e^{-as}\bigl(a-aP_{k}(as)+aP'_{k}(as)\bigr)ds
={1-e^{-a}}+\int_{0}^1\frac{d}{ds}\bigl\{
e^{-as}P_{k}(as)\bigr\}
\\=
{1-e^{-a}}+e^{-a}P_{k}(a)-P_{k}(0)
=e^{-a}\bigl(P_{k}(a)-1\bigr)\ge 0,
\end{multline*}
which holds true from \eqref{iuytre}.
\begin{rem}\rm
 There are several classical results on products of Laguerre polynomials, in particular the article \cite{MR1574144},
 {\it On some expansions in Laguerre polynomials} by A.~Erd\'elyi
 and also the paper \cite{MR3187821},
 {\it  Linearization of the products of the generalized Lauricella polynomials and the multivariate Laguerre polynomials via their integral representations}
 by
 Shuoh-Jung Liu, 
 Shy-Der Lin, 
 Han-Chun Lu
  and H. M. Srivastava.
  However it seems that the non-negativity of the polynomials
  $P_{\alpha;1}, P'_{\alpha;1}$ do not suffice to tackle the conjecture in two dimensions and more.
 \end{rem}
\section{Parabolas}
\index{epigraph of a parabola}
\subsection{Preliminary remarks}
We start with a picture, demonstrating that the epigraph
of a parabola is an increasing union of ellipses.
\begin{figure}[ht]
\centering
\vskip-22pt
\scalebox{0.45}{\includegraphics{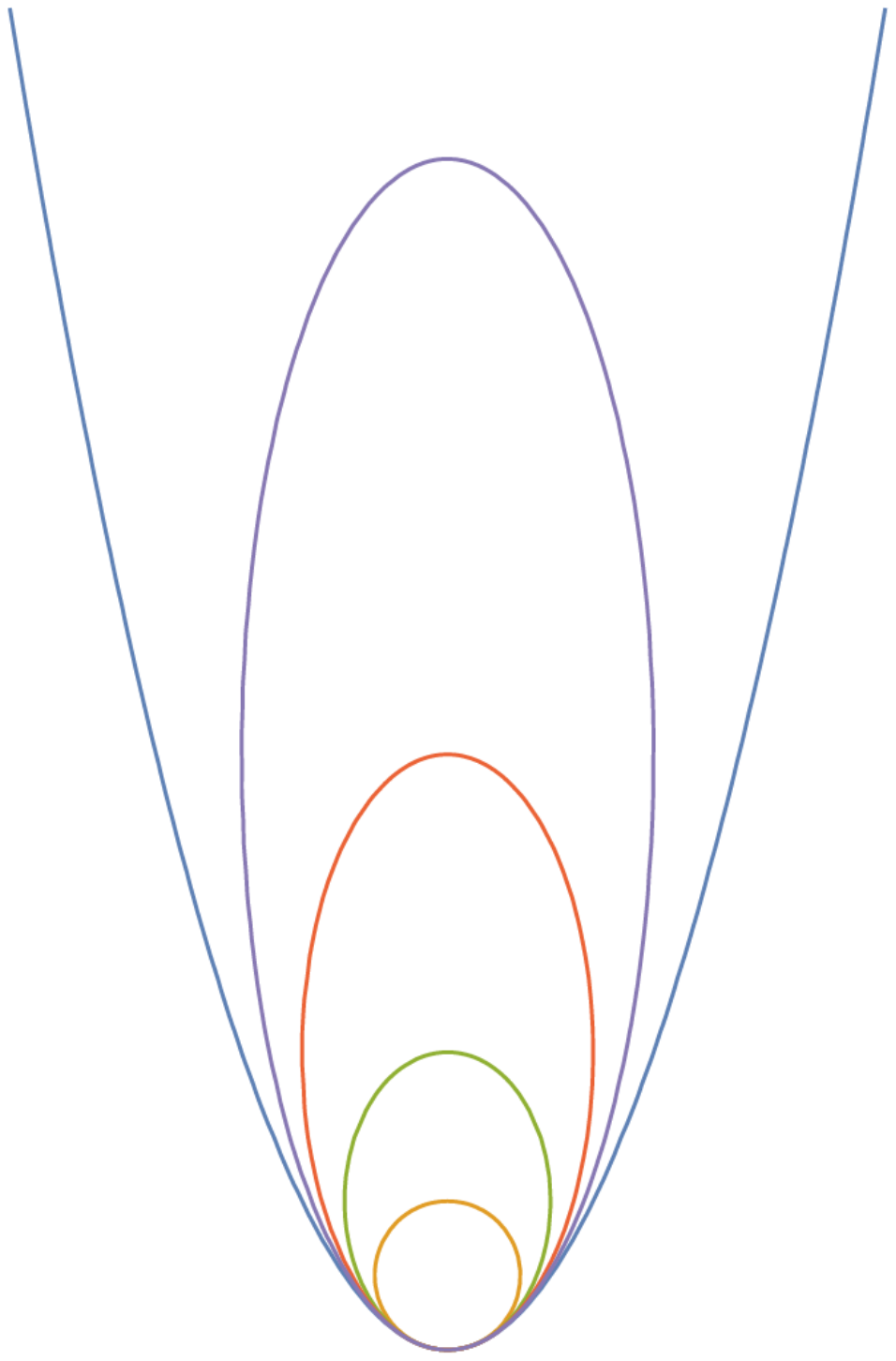}}\par
\vskip-22pt
\caption{The epigraph of a  parabola is an increasing union of ellipses.}
\label{pic4}
\end{figure}
It is easy to see that the epigraph of a parabola, i.e. the set
$
\{(x,\xi)\in \mathbb R^2, \xi> x^2\}
$
 is a countable increasing union of ellipses
in the sense that
\begin{equation}\label{ellunie}
\mathcal P=\{(x,\xi)\in \mathbb R^2, \xi> x^2\}=\cup_{k\ge 1}\underbrace{\{(x,\xi)\in \mathbb R^2, \xi> x^2+k^{-2} \xi^2\}}_{\mathcal E_{k}}.
\end{equation}
Note that for $k\ge 1$ we have $\mathcal E_{k}\subset \mathcal E_{k+1}\subset \mathcal P$ since 
$ x^2+k^{-2} \xi^2\ge x^2+(k+1)^{-2} \xi^2> x^2,
$
from the fact that  $\xi>0$ on $\mathcal E_{k}$. Moreover, if $\xi>x^2$ and $k>{\xi}/{\sqrt{\xi-x^2}}$, we get $(x,\xi)\in \mathcal E_{k}$.
\begin{rem}\label{rem.ttrree} \rm
 The ellipse $\mathcal E_{k}$ is symplectically equivalent to a circle with area $\frac{\pi k^3}{4}$
 since
 \begin{multline*}
 x^2+k^{-2}\xi^2-\xi=x^{2}+k^{-2}\bigl(\xi-\frac{k^2}{2}\bigr)^2-\frac{k^2}4
 =(\lambda^{-1} y)^{2}+k^{-2}\bigl(\lambda \eta-\frac{k^2}{2}\bigr)^2-\frac{k^2}4
 \\=\lambda^{-2} y^{2}+\lambda^2k^{-2}\bigl(\eta-\frac{k^2}{2\lambda}\bigr)^2-\frac{k^2}4,
\end{multline*}
so that choosing $\lambda$ such that $\lambda^{-2}=\lambda^2 k^{-2}$, e.g. $\lambda=\sqrt k$, we get
$$
 x^2+k^{-2}\xi^2-\xi=k^{-1}\bigl( y^2+(\eta-\frac{k^2}{2\lambda})^2\bigr)-\frac{k^2}4,
$$
and
$
\mathcal E_{k}=\{(y,\zeta)\in \R^2, y^2+\zeta^2<\frac{k^3}{4}\},
$
where $(y,\zeta)$ are the affine symplectic coordinates
$$
y= xk^{1/2}, \quad \zeta=\xi k^{-1/2}-\frac{k^{3/2}}{2}.
$$
\end{rem}
\begin{lem}\label{lem.hh88}
 Let $u\in \mathscr S(\R)$. Then $\mathcal W(u,u)$ belongs to $\mathscr S (\R^2)$
 and with $\mathcal {E},\mathcal E_{k}$ defined by \eqref{ellunie}, we have
 $$
 \iint_{\xi> x^2} \mathcal W(u,u)(x,\xi) dx d\xi=\lim_{k\rightarrow+\io}
  \iint_{\mathcal E_{k}} \mathcal W(u,u)(x,\xi) dx d\xi\le \norm{u}_{L^2(\R)}^2.
 $$
\end{lem}
\begin{proof}
 Since $\mathcal W(u,u)$ belongs to $\mathscr S(\RZ)\subset L^1(\RZ)$, we may apply the Lebesgue Dominated Convergence Theorem and \eqref{ellunie} to obtain the equality in the lemma. On the other hand Theorem \ref{thm-fl0001} and Remark \ref{rem.ttrree} imply
 $$
  \iint_{\mathcal E_{k}} \mathcal W(u,u)(x,\xi) dx d\xi=
  \poscal{
 \opw{ \mathbf 1_{\mathcal E_{k}}}
   u}{u}\le (1-e^{-\frac{\pi k^3}{2}})\norm{u}_{L^2(\R)}^2\le \norm{u}_{L^2(\R)}^2,
 $$
 and the sought result. 
 \end{proof}
 \begin{rem}\rm
Moreover, Theorem \ref{thm-fl0001} and the expression of $F_{0}(a)=1-e^{-a}$ imply that with $\psi_{0}$ defined in \eqref{herm1}, we have 
 $$
  \iint_{\mathcal E_{k}} \mathcal W(\psi_{0},\psi_{0})(x,\xi) dx d\xi=
  \poscal{\opw{\mathbf 1_{\mathcal E_{k}}} \psi_{0}}{\psi_{0}}=\norm{\psi_{0}}_{L^2(\R)}^2
  (1-e^{-\pi k^3/3}),
 $$
 so that from Lemma \ref{lem.hh88}, we have 
 $
  \iint_{\mathcal P} \mathcal W(\psi_{0},\psi_{0})(x,\xi) dx d\xi=\norm{\psi_{0}}_{L^2(\R)}^2,
 $
 entailing
 \begin{equation}\label{flanpa}
 \sup_{\phi\in \mathscr S(\R), \norm{\phi}_{L^2(\R)}=1}\iint_{\mathcal P} \mathcal W(\phi,\phi)(x,\xi) dx d\xi=1.
\end{equation}
 \end{rem}
 \begin{rem}\rm
 We want to study the operator with Weyl symbol $H(\xi-x^2)$ ($H=\mathbf 1_{\R_{+}}$ is the Heaviside function) and since $\xi-x^2$ is a polynomial with degree less than 2,
 see from \eqref{gfcd44+}
that $\OPW{H(\xi-x^2)}$ commutes with $D_{x}-x^2= e^{2\pi ix^3/3}D_{x}e^{-2\pi ix^3/3}$, and the latter has (continuous) spectrum $\R$:
we expect thus that $\OPW{H(\xi-x^2)}$ should have continuous spectrum and be conjugated to a Fourier multiplier.
\end{rem}
\subsection{Calculation of the kernel}
The  Weyl symbol of the operator $\opw{\mathbf 1_{\mathcal P}}$  is 
$$
H(\xi-x^2),
$$
($\mathcal P$ is defined in \eqref{ellunie},
$H$ is the Heaviside function $H=\mathbf 1_{\R_{+}}$),
corresponding to the distribution kernel
$
k_{\mathcal P}(x,y)
$
obtained from Proposition \ref{pro1716} by (we use freely integrals meaning only Fourier transform in the distributional sense),
\begin{multline*}k_{\mathcal P}(x,y)=
\int e^{2i\pi (x-y)\xi}H(\xi-(\frac{x+y}{2})^2) d\xi
= \int e^{2i\pi (x-y)(\xi+(\frac{x+y}{2})^2)}H(\xi) d\xi
\\
=e^{2i\pi (x-y)(\frac{x+y}{2})^2}\frac12\bigl(\delta_{0}(y-x)+\frac{1}{i\pi(y-x)}\bigr)
\\=
\frac{\delta_{0}(y-x)}2+\frac{e^{2i\pi (x-y)(\frac{x+y}{2})^2}}{2i\pi(y-x)}.
\end{multline*}
We have 
\begin{multline*}
4 (x-y)(\frac{x+y}{2})^2=(x^2-y^2)(x+y)=x^3-y^3+x^2 y-y^2 x
\\
=\frac43(x^3-y^3)+\frac13(y-x)^3,
\end{multline*}
so that
\begin{equation}\label{unipar}
k_{\mathcal P}(x,y)=e^{i\frac{2\pi}3 x^3}\left(
\frac{\delta_{0}(y-x)}2+\frac{e^{i\frac\pi 2
\frac13(y-x)^3
}}{2i\pi(y-x)}\right)e^{-i\frac{2\pi}3y^3},
\end{equation}
and the operator $\opw{\mathbf 1_{\mathcal P}}$ is unitarily equivalent to the operator  with kernel
\begin{equation}\label{}
\tilde k(x,y)=\frac{\delta_{0}(y-x)}2+\frac{e^{i\frac\pi 6
(y-x)^3}}{2i\pi(y-x)}.
\end{equation}
We have proven the following result.
\begin{lem}\label{lem45}
 The operator with Weyl symbol $\R^2\ni(x,\xi)\mapsto\mathbf 1_{\R_{+}}(\xi-x^2)$ has the distribution kernel
 $$k_{\mathcal P}(x,y)=e^{i\frac{2\pi}3 x^3}
 \left(\frac{\delta_{0}(y-x)}2+\frac{e^{i\frac\pi 6
(y-x)^3}}{2i\pi(y-x)}\right)e^{-i\frac{2\pi}3 y^3},
 $$
 and is thus unitarily equivalent to 
 \begin{equation}\label{422}
\frac{\Id}{2}+\text{convolution with\quad }\frac{ie^{-i\pi t^3/6}}{2\pi}\textrm{\rm pv}\frac1{ t}.
\end{equation}
\end{lem}
\begin{lem}\label{lem46}
 The distribution $\frac{ie^{-i\pi t^3/6}}{2\pi}\textrm{\rm pv}\frac1{ t}$
 has the Fourier transform
 \begin{equation}\label{}
\frac{1}{2\pi}\int\frac{\sin(2\pi as\tau+\frac{s^3}{3})}{s}ds,\quad a=(2/\pi)^{1/3}.
\end{equation}
The operator \eqref{422} is the Fourier multiplier $\omega(D_{t})$ with 
\begin{equation}\label{}
\omega(\tau)=\frac12\left(1+\frac 1\pi\int_{-\io}^{+\io}
\frac{\sin(s\eta+\frac{s^3}{3})}{s}ds\right), \quad \eta
=2^{4/3}\pi^{2/3}\tau.
\end{equation}
\end{lem}
\begin{proof}
We calculate  in the distribution sense ($t=as, a=(2/\pi)^{1/3}$),
\begin{multline*}
\int e^{-2i\pi t\tau}i\frac{e^{-i\pi t^3/6}}{2\pi t} dt =\frac{i}{2\pi}\int
e^{-2i\pi as\tau}\frac{e^{-i\pi a^3s^3/6}}{ s} ds
=\frac{i}{2\pi}\int
\frac{(-i)\sin(\frac{s^3}{3}+2\pi as\tau)}{ s} ds
\\
=\frac{1}{2\pi}\int\frac{\sin(2\pi as\tau+\frac{s^3}{3})}{s}ds,
\end{multline*}
so that 
with $\eta=2\pi a\tau$, we get
$$
\omega(\tau)=\frac12\left(1+\frac 1\pi\int_{-\io}^{+\io}
\frac{\sin(s\eta+\frac{s^3}{3})}{s}ds\right)=\frac12\bigl(1-F(\eta)\bigr)=G(\eta),
$$
proving the lemma.
\end{proof}
\begin{lem}\label{lem.nnbbxx} We have, with $\eta= 2^{4/3}\pi^{2/3} \tau$,
\begin{align}
&\omega(\tau)=
\frac12\left(1+\frac 1\pi\int_{-\io}^{+\io}
\frac{\sin(s\eta+\frac{s^3}{3})}{s}ds\right)=G(\eta), \quad \omega(0)=\frac23=G(0),\label{425}
\\
&G'(\eta)=\frac{1}{2\pi}\int_{\R}
\cos(s\eta+\frac{s^3}{3}) ds=\re \frac{1}{2\pi}\int_{\R}
\exp{i(s\eta+\frac{s^3}{3})} ds=\ai(\eta),\label{426}
\\
&G(\eta)=\frac{2}{3}+\int_{0}^{\eta}\ai(\xi) d\xi, \label{427}\end{align}
 where $\ai$ is the Airy function defined as the  inverse Fourier transform of $t\mapsto e^{i(2\pi t)^3/3}$.
\end{lem}
\begin{proof} We have 
\begin{multline}\label{sin3}
 \frac 1\pi\int_{-\io}^{+\io}
\frac{\sin(\frac{s^3}{3})}{s}ds= \frac 1\pi\int_{-\io}^{+\io}
\frac{\sin(\sigma)}{3^{1/3} \sigma^{1/3}}
3^{1/3} \frac13 \sigma^{-2/3} d\sigma
\\=\frac{1}{3\pi}\int_{-\io}^{+\io}\frac{\sin \sigma}{\sigma} d\sigma=\frac13,
\end{multline}
proving \eqref{425}.
We have also
$$
G(\eta)=\frac12+\im\Bigl\{\text{Inverse Fourier Transform}\big\{
y\mapsto e^{i(2\pi y)^3/3}\text{pv}(\frac1{2\pi y})
\bigr\}\Bigr\},
$$
and thus
\begin{multline*}
G'(\eta)=\im\Bigl\{\text{Inverse Fourier Transform}\big\{
y\mapsto e^{i(2\pi y)^3/3}i
\bigr\}\Bigr\}
\\=\im\left(
\int e^{2i\pi y\eta} e^{i(2\pi y)^3/3} i dy\right)=
\im\left(\frac1{2\pi}\int e^{i t\eta} e^{it^3/3} idt\right)=\ai(\eta),
\end{multline*}
which is \eqref{426}, implying \eqref{427}.
\end{proof}
\begin{lem}\label{lem48}
 With $G$ defined in Lemma \ref{lem.nnbbxx}, we get that $G$ is an entire function, real-valued on the real line such that 
 \begin{equation}\label{428}
\lim_{\eta\rightarrow+\io} G(\eta)=1, \quad \lim_{\eta\rightarrow-\io} G(\eta)=0, 
\end{equation}
and moreover with $\eta_{0}$ the largest zero of the Airy function ($\eta_{0}\approx-2.33811$), the function $G$ has an absolute minimum at $\eta_{0}$ with $G(\eta_{0})\approx -0.274352$,
\begin{equation}\label{429}
\forall \eta\in \R, \quad G(\eta_{0})\le G(\eta)<1.
\end{equation}
\end{lem}
\begin{proof}
 The first statements follow from  Lemma \ref{lem.nnbbxx} and \eqref{428} is implied by
 \eqref{427} 
 and   \eqref{996}, \eqref{99o}. 
 The strict inequality in \eqref{429} follows for $\eta\ge 0$ from \eqref{426} since $\ai$ is positive on $[0,+\io)$ so that $G$ is strictly increasing there from $G(0)=2/3$ to $G(+\io)=1$.
 The other statements are proven in Section \ref{secairy} of the Appendix.
\end{proof}
\subsection{The main result}\label{secmainpara}
Collecting the results of Lemmas \ref{lem45}, \ref{lem46}, \ref{lem.nnbbxx}, \ref{lem48} and of Section \ref{secairy} in the Appendix,
we have proven the following theorem.
\begin{theorem} Let $H(\xi-x^2)=\mathbf 1\{(x,\xi)\in \R^2, \xi\ge x^2\}$ be the indicatrix
 of the epigraph of the parabola with equation $\xi= x^2$. Then the operator with Weyl symbol $H(\xi-x^2)$ is unitary equivalent to the Fourier multiplier 
 $G(2^{4/3}\pi^{2/3} \tau)$ where 
 \begin{equation}\label{}
 G(\eta)=\frac23+\int_{0}^\eta \ai(\xi) d\xi=\int_{-\io}^\eta \ai(\xi) d\xi, \qquad\text{($\ai$ is the Airy function)}.
\end{equation}
The function $G$ is entire on $\C$, real valued on the real line and such that $$G(\R)=[G(\eta_{0}), 1),$$ where $\eta_{0}$ is the largest zero of the Airy function
 \begin{equation}\label{}
 \text{we have $\eta_{0}\approx -2.338107410,\qquad G(\eta_{0})\approx -0.2743520591$.}
\end{equation}
The operator with Weyl symbol $H(\xi-x^2)$ is self-adjoint bounded on $L^2(\R)$ with norm $1$,
with  spectrum equal to  $[G(\eta_{0}), 1]$ (continuous spectrum) and
\begin{equation}\label{}
 \forall u\in L^2(\R),\qquad
G(\eta_{0})\norm{u}_{L^2(\R)}^2\le  \iint_{\xi\ge x^2}\mathcal W(u,u)(x,\xi) dxd\xi\le \norm{u}_{L^2(\R)}^2.
\end{equation}
\end{theorem}
\begin{figure}[ht]
\centering
{\includegraphics[angle=0,height=300pt,width=0.9\textwidth]{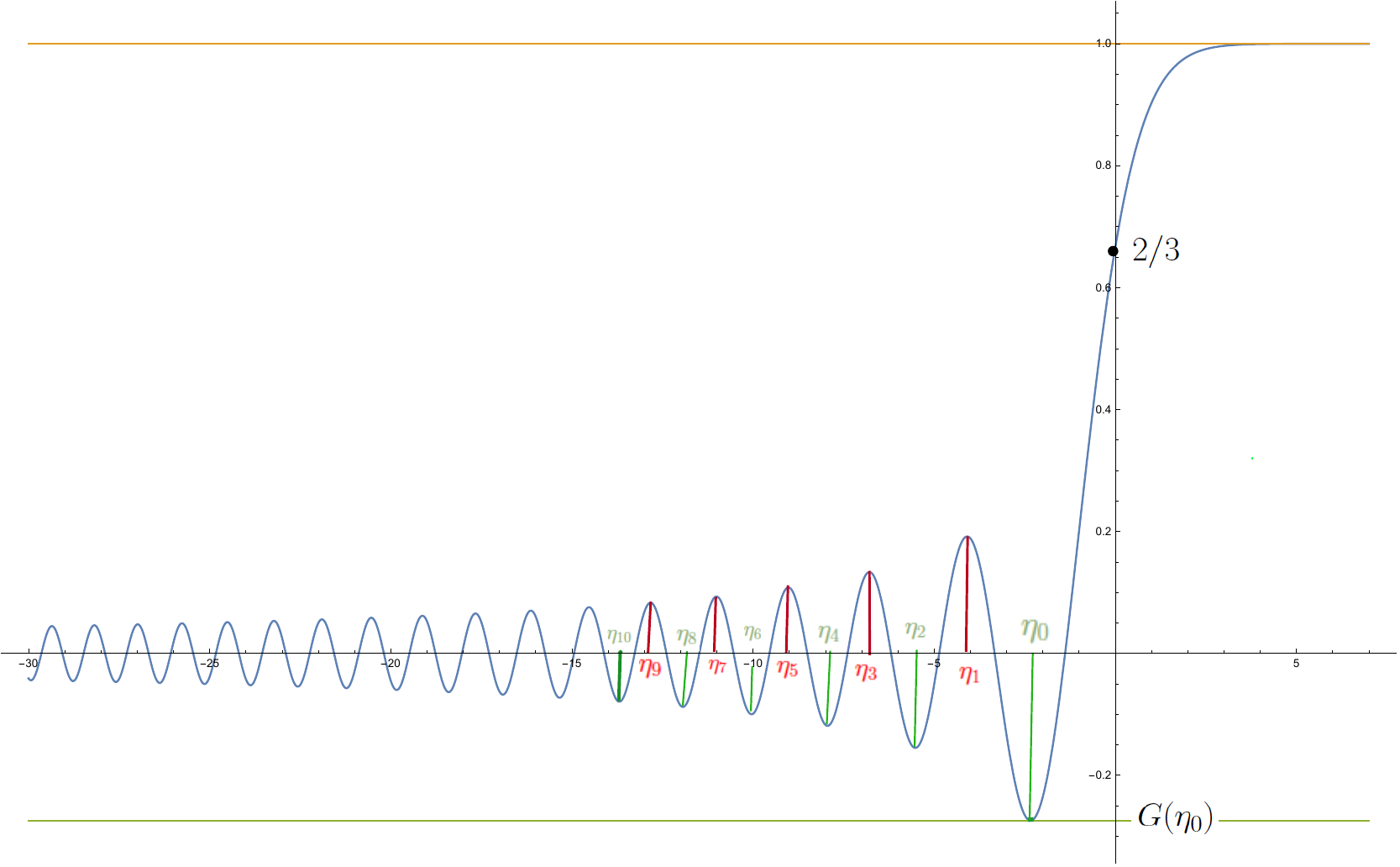}}\par
\caption{The function $G$. More details on $G$  are given  in our Appendix \ref{secairy}.}
\label{pic5}
\end{figure}
\subsection{Paraboloids, a conjecture}\label{sec.parcj1}
We are interested now in multi-dimensional versions of the previous results, namely, we would like to find a bound for integrals of the Wigner distribution  on paraboloids of $\RZ$ for $n\ge 2$.
Let us start with recalling Theorem 21.5.3 in \cite{MR2304165}, a version of which was given in our Theorem \ref{thm312} in the positive-definite case.
\subsubsection{On non-negative quadratic forms}
\begin{theorem}[Symplectic reduction of quadratic forms, Theorem 21.5.3 in \cite{MR2304165}]
 Let $q$ be a non-negative quadratic form  on $\R^{n}\times \R^{n}$ equipped with the canonical symplectic form \eqref{sympfor}. Then there exists $S$ in the symplectic group $Sp(n,\R)$
 of $\RZ$,
 $$\text{
 $r\in\{0,\dots,n\}$,
 $\mu_{1}, \dots, \mu_{r}$
 positive,
 and $s\in \N$  such that $r+s\le n$,}
 $$ so that  for all $X=(x,\xi)\in \R^{n}\times \R^{n}$,
\begin{equation}\label{ellipse+}
 q(SX)=\sum_{1\le j\le r}\mu_{j}(x_{j}^{2}+
 \xi_{j}^{2})+\sum_{r+1\le j\le r+s}x_{j}^{2}.
\end{equation}
\end{theorem}
\begin{defi}
 Let $n\in \N^*$ and let $\R^{2n}$ be equipped with the canonical
 symplectic form \eqref{sympfor}.  
 Let $q$ be a non-negative quadratic form on $\R^{2n}$ with rank $2n-1$ and $T$ be a non-zero vector in $\RZ$ such that $q(\sigma T)=0$.
 A paraboloid $\mathcal P$ of $\RZ$ with vertex $0$ and shape $(q,T)$
 is defined  by
 \begin{equation}\label{}
\mathcal P=\{X\in \RZ, q(X)\le [X,T]\}.
\end{equation}
A paraboloid $\mathcal Q$ with vertex $m\in \RZ$ and shape $(q,T)$
is defined as
\begin{equation}\label{}
\mathcal Q=\mathcal P+m,
\end{equation}
where  $\mathcal P$ is a paraboloid  with vertex $0$ and shape $(q,T)$.
\end{defi}
\begin{rem}\rm
 We can find some symplectic coordinates such that 
 $$
 q(X)-[X,T]=\sum_{1\le j\le r}\mu_{j}(x_{j}^{2}+
 \xi_{j}^{2})+\sum_{r+1\le j\le r+s}x_{j}^{2}+\sum_{1\le j\le n} (x_{j} \tau_{j}-\xi_{j} t_{j}),
 $$
 with $2r+s=2n-1$.
 We can get rid of the linear terms $x_{j} \tau_{j}-\xi_{j} t_{j}$
 when $1\le j\le r$ by writing
 $$
 \mu_{j}(x_{j}^{2}+
 \xi_{j}^{2})+x_{j} \tau_{j}-\xi_{j} t_{j}=\mu_{j}\bigl(x_{j}+\frac{\tau_{j}}{2\mu_{j}}
 \bigr)^2+
 \mu_{j}\bigl(\xi_{j}-\frac{t_{j}}{2\mu_{j}}
 \bigr)^2-\frac1{4\mu_{j}}(t_{j}^2+\tau_{j}^2),
 $$
 and also of $x_{j} \tau_{j}$ for $r+1\le j\le r+s$, since
 $$
x_{j}^2+ x_{j}\tau_{j}=(x_{j}+\frac{\tau_{j}}2)^2-\frac{\tau_{j}^2}{4}.
 $$
 We are left with using affine symplectic coordinates $(y,\eta)$ so that 
 \begin{multline*}
 q(X)-[X,T]=\sum_{1\le j\le r}\mu_{j}(y_{j}^{2}+
 \eta_{j}^{2})+\sum_{r+1\le j\le r+s}y_{j}^{2}\\-\sum_{r+1\le j\le r+s} \eta_{j} t_{j}
 +\sum_{r+s+1\le j\le n}(y_{j}\tau_{j}-\eta_{j}t_{j})-a.
\end{multline*}
 Since we have $2r+s=2n-1$, we get $r+s+1=2n-r$: we cannot have 
$ r+s+1\le n$
since it would imply that $2n-r\le n$ and thus $r\ge n$, which is incompatible with 
$2r+s=2n-1, r, s\ge 0.$
We get then that
$
s=2l+1, r=n-1-l
$
and since $r+s\le n, 1\le s$, we have $l=0$, 
$s=1$, $r=n-1$,
 and
 $$
 q(X)-[X,T]=
 \sum_{1\le j\le n-1}\mu_{j}(y_{j}^{2}+
 \eta_{j}^{2})+y_{n}^2-\eta_{n}t_{n}-a,
 $$
 and $t_{n}\in \R^{*}$.
 With $y_{n}=t^{1/3} \tilde y_{n}, \eta_{n}=t^{-1/3}\tilde \eta_{n}$,
 we get
 $$
 q(X)-[X,T]= \sum_{1\le j\le n-1}\mu_{j}(y_{j}^{2}+\eta_{j}^{2})+t^{2/3}(\tilde y_{n}^{2}-\tilde \eta_{n}-a t^{-2/3}),
 $$
 and the inequality $q(X)-[X,T]\le 0$ is equivalent to 
 $$
 \sum_{1\le j\le n-1}t^{-2/3}\mu_{j}(y_{j}^{2}+\eta_{j}^{2})+\tilde y_{n}^{2}\le \tilde \eta_{n}+a t^{-2/3}.
 $$
 We can  thus assume {\sl ab initio} that our paraboloid is given by the inequality 
 \begin{equation}\label{}
  \sum_{1\le j\le n-1}\nu_{j}(x_{j}^{2}+
 \xi_{j}^{2})+x_{n}^2\le \xi_{n}.
\end{equation}
 \end{rem}
 \subsubsection{On the kernel for the paraboloid}
 We shall consider the paraboloid
 \begin{equation}\label{}
\mathcal P_{n}=\{(x,\xi)\in\RZ,\ x_{n}^2+\sum_{1\le j\le n-1}(x_{j}^2+\xi_{j}^2)\le \xi_{n} \}.
\end{equation}
We have, with $X'=(x';\xi')=(x_{1},\dots, x_{n-1};\xi_{1},\dots, \xi_{n-1})$,
\begin{align*}
&P=\OPW{H(\xi_{n}-x_{n}^2-\val{X'}^2)}
=
\int_{\R}\hat H(\tau) 
\opw{e^{2i\pi \tau(\xi_{n}-x_{n}^2)}}
\opw{e^{-2i\pi \tau\val{X'}^2}}
d\tau
\\
&=\sum_{k\ge 0}\int_{\R}\hat H(\tau) 
 \mathbb P_{k; n-1}\otimes
\opw{e^{2i\pi \tau(\xi_{n}-x_{n}^2)}}
e^{-i(\arctan \tau)(2k+n-1)}(1+\tau^2)^{-\frac{(n-1)}{2}} d\tau
\\
&=\frac12 \Id
\\
&\hs\hs+\frac1{2i\pi}\sum_{k\ge 0}
 \mathbb P_{k; n-1}\otimes 
\int_{\R} 
\opw{e^{2i\pi \tau(\xi_{n}-x_{n}^2)}}
\frac1\tau
\Bigl(\frac{1-i\tau}
{(1+\tau^2)^{1/2}}\Bigr)^{2k+{n-1}}
(1+\tau^2)^{-\frac{(n-1)}{2}} d\tau
\\
&=\frac12 \Id+\frac1{2}\sum_{k\ge 0}
\mathbb P_{k; n-1}\otimes 
\int_{\R} 
\opw{e^{2i\pi \tau(\xi_{n}-x_{n}^2)}}
\frac{(1-i\tau)^k}{i\pi \tau(1+i\tau)^{k+n-1}}
d\tau.
\end{align*}
Let $k(x_{n}, y_{n})$ be the kernel of the operator  in the integral: we have 
$$
k(x_{n}, y_{n})=e^{\frac{2i\pi}{3}(x_{n}^3-y_{n}^3)}e^{-\frac{i\pi}{6}(x_{n}-y_{n})^3}
\frac{i}{\pi(x_{n}-y_{n})}
\frac{(1+i(x_{n}-y_{n}))^k}{(1-i(x_{n}-y_{n}))^{k+n-1}}.
$$
As a result, we find that $P$ is unitarily equivalent to $\tilde P$, with
\begin{equation}\label{}
2\tilde P=\sum_{k\ge 0}
 \mathbb P_{k; n-1}\otimes
\bigl(
I_{n}+\text{convolution with }
\frac{ie^{-\frac{i\pi}{6}x_{n}^3}}{\pi x_{n}}
\frac{(1+ix_{n})^k}{(1-ix_{n})^{k+n-1}}
\bigr)
.
\end{equation}
We define
\begin{multline}\label{}
\omega_{k,n-1}(\tau)=\frac12+\int 
\frac{ie^{-\frac{i\pi}{6}t^3}}{2\pi t}
\frac{(1+it)^k}{(1-it)^{k+n-1}} e^{-2i\pi t\tau} dt
\\
=\frac12+\int 
\frac{e^{\frac{i\pi}{6}t^3}}{2i\pi t}
\frac{(1-it)^k}{(1+it)^{k+n-1}} e^{2i\pi t\tau} dt,
\end{multline}
and we get that 
\begin{equation}\label{}
\tilde P=\sum_{k\ge 0}\mathbb P_{k; n-1}\otimes 
\omega_{k,n-1}(D_{x_{n}}).
\end{equation}
We note that for $n=1$, the sum is reduced to $k=0$ with $\mathbb P_{0;0}=I$,
so that we recover Formula \eqref{425} with $\omega_{0,0}=\omega$.
We find also that 
\begin{equation}\label{449}
\omega_{k,n-1}'(\tau)=\int 
e^{\frac{i\pi}{6}t^3}
\frac{(1-it)^k}{(1+it)^{k+n-1}} e^{2i\pi t\tau} dt,
\end{equation}
in the sense that the inverse Fourier transform 
of $t\mapsto e^{\frac{i\pi}{6}t^3}
\frac{(1-it)^k}{(1+it)^{k+n-1}}$ is the distribution derivative of $\omega_{k,n-1}$.
Going back to the normalization of 
Lemma \ref{lem.nnbbxx}, we have, with $\eta= 2^{4/3}\pi^{2/3} \tau$,
\begin{align}\label{}
G_{k,n-1}(\eta)&=\omega_{k,n-1}(\tau),
\\
G_{k,n-1}'(\eta)&=2^{-4/3}\pi^{-2/3}
\int 
e^{\frac{i\pi}{6}t^3}
\frac{(1-it)^k}{(1+it)^{k+n-1}} e^{2^{-\frac13}i\pi^{\frac13} t\eta} dt,
\notag
\\
&\underbrace{=}_{t=\pi^{-\frac13} 2^{\frac13} s}\frac{1}{2\pi}
\int 
e^{\frac{is^3}{3}}
\frac{(1-i\pi^{-1/3} 2^{1/3} s)^k}{(1+i
\pi^{-1/3} 2^{1/3} s)^{k+n-1}} e^{is\eta} ds:=A_{k,n-1}(\eta).
\end{align}
We have $A_{0,0}=\ai$ and $A_{k,n-1}$ is an entire function, real-valued on the real line;
we have 
$$
G_{k,n-1}(\eta)=\int_{-\infty}^\eta A_{k,n-1}(\xi)d\xi, \quad G_{k,n-1}(+\io)=1.
$$
\begin{rem}\rm We claim that
 the asymptotic properties of the functions $A_{k,n-1}$ are analogous to the properties of the standard Airy function and we have indeed from \eqref{449},
 \begin{equation}\label{}
 \omega_{k,n-1}'(\tau)=(1-iD)^k(1+iD)^{-k-n+1}\mathcal F^{-1}(
 e^{\frac{i\pi}{6}t^3}).
\end{equation}
We claim as well that 
$$
-\frac12<\inf_{k\ge 0,\eta \in \R} G(\eta)<0, \quad \sup_{k\ge 0,\eta \in \R} G(\eta)=1,
$$
so that $\tilde P$ is bounded on $L^2(\R^n)$ and 
\begin{equation}\label{}
\int_{\xi_{n}\ge x_{n}^2+\sum_{1\le j\le n-1}(x_{j}^2+\xi_{j}^2)}
\mathcal W(u,u)(x,\xi) dx d\xi\le \norm{u}_{L^2(\R^n)}^2.
\end{equation}
\end{rem}
\section{Conics with eccentricity greater than 1}
\index{hyperbolic convex sets}
We want to consider now integrals of the Wigner distribution on ``hyperbolic'' 
convex subsets of the plane such as 
\begin{equation}\label{5014}
\mathcal C_{\sigma}=\{(x,\xi)\in \R^2, x\xi\ge \sigma, x\ge 0\},
\end{equation}
where $\sigma$ is a non-negative parameter.
It is convenient to start with the limit-case where $\sigma=0$ and 
$\mathcal C_{0}=\{(x,\xi)\in \R^2, x\ge 0, \xi\ge 0\}$
(we will label $\mathcal C_{0}$ as the \emph{quarter-plane}).
The indicator function of $\mathcal C_{0}$ is $H(x) H(\xi)$ where $H=\mathbf 1_{\R_{+}}$ is the Heaviside function.
\par\no
{\bf Acknowledgements.}
The author is grateful to Thomas Duyckaerts for sharp comments on a first version of this section.
\begin{nb}\rm
The  reader will see a great similarity between our calculations below in this section  and the J.G. Wood \& A.J. Bracken paper \cite{MR2131219} (see also \cite{MR1728097}).
This article   is 
 very important for the problem at stake -- Integrating the Wigner distribution on subsets of the phase space -- and was a wealthy source of information for us, although as a mathematician,
 the author  has a quite rigid relationship with calculations, and feels the need to justify formal manipulations;
for instance, we may point out that the test functions used in  \cite{MR2131219} are homogeneous distributions of type
 $$
 x_{\pm}^{-\frac12+i\omega}, \quad \omega\in \R,
 $$
  which are not in $L^{2}(\R)$ (not even in $L^{2}_{\textrm{loc}}$),
 a situation which raises some difficulties, first when you try to normalize in $L^{2}$ these test functions and also when trying to give a non-formal meaning to their images under the operator with Weyl symbol $H(x) H(\xi)$, images which are not clearly defined.
 In  our  joint  paper \cite{DDL}
with B.~Delourme and T.~Duyckaerts, proving that Flandrin's conjecture is not true, we followed numerical arguments which were quite  apart from 
the arguments of \cite{MR2131219}.
However,  in this article, we do follow many of the arguments of \cite{MR2131219},
along with avoiding formal calculations.
\end{nb}
\subsection{The quarter-plane,
a counterexample to Flandrin's conjecture}\label{sec.51}
\index{quarter-plane}
\index{counterexample to Flandrin conjecture}
\subsubsection{Preliminaries}\label{sec.prelim}
We study in this section the operator 
\begin{equation}\label{quarter-plane}
A_{0}=\opw{H(x)H(\xi)}
\end{equation}
where $H=\mathbf 1_{\R_{+}}$, 
that is the Weyl quantization of the characteristic function of the  first quarter of the plane.\begin{lem}\label{lem.51fdsq}
 The operator $A_{0}$ given by \eqref{quarter-plane} is bounded self-adjoint on $L^{2}(\R)$.
 \end{lem}
\begin{proof}
Since the Weyl symbol of $A_{0}$ is real-valued, $A_{0}$ is formally self-adjoint and it is enough to prove that $A_{0}$ is bounded on $L^{2}(\R)$.
Let us start with recalling the classical formulas
\begin{equation}\label{}
\hat H(t)=\frac{\delta_{0}(t)}{2}+\frac{1}{2i \pi}\text{\rm pv}\left(\frac1t\right),
\quad
\widehat{\sign}=\frac{1}{i \pi}\text{\rm pv}\left(\frac1t\right),
\end{equation}
useful below.
 The kernel\footnote{There is no difficulty at defining the product $S\bigl((x+y)/2\bigr) T(x-y)$
 for $S,T$ tempered distributions on the real line  since we may use the tensor product
 with
 $$
 \poscal{S(\frac{x+y}2) T(x-y)}{\Phi(x,y)}_{\mathscr S'(\R^{2}), \mathscr S(\R^{2})}
 = \poscal{S(x_{1}) \otimes T(x_{2})}{\Phi(x_{1}+\frac{x_{2}}{2},x_{1}-\frac{x_{2}}2)}_{\mathscr S'(\R^{2}), \mathscr S(\R^{2})}.
 $$
 However, we shall not use directly Formula \eqref{513iop}, since want to avoid formal manipulation involving for instance meaningless products such as $H(x) H(y) k_{0}(x,y)$.
 We refer the reader  to footnote \ref{footwave} on page \pageref{footwave} and to Remark \ref{rem.5487} for more details on this matter.
 } of $A_{0}$
is
\begin{equation}\label{513iop}
k_{0}(x,y)= H(x+y)\hat H(y-x)=H(x+y)\frac12\Bigl(\delta_{0}(y-x)+\frac{1}{i\pi}
\text{pv}\frac{1}{y-x}\Bigr).
\end{equation}
For $\lambda>0$, we define
$
A_{0,\lambda}=\left(H(x)\indic{[0,\lambda]}(\xi)\right)^{w},
$
whose distribution-kernel is the $L^{\io}(\RZ)$ function
$$
k_{0,\lambda}(x,y)=H(x+y)e^{i\pi(x-y)\lambda}\frac{\sin(\pi(x-y)\lambda)}{\pi(x-y)}.
$$
We can thus notice that 
\begin{multline}\label{}
k_{0,\lambda}(x,y)= \overbracket[0.5pt][3pt]{H(x)H(y)
e^{i\pi(x-y)\lambda}\frac{\sin(\pi(x-y)\lambda)}{\pi(x-y)}}^{k_{0,\lambda}^{\flat}(x,y)}
\\+\underbracket[0.5pt][3pt]{H(x+y)\bigl(H(-x) H(y)+H(x)H(-y)\bigr)
\frac{\sin(\pi(x-y)\lambda)}{\pi(x-y)}
e^{i\pi(x-y)\lambda}
}_{
k_{0,\lambda}^{\sharp}(x,y)
},
\end{multline}
and the operator with distribution-kernel $k_{0,\lambda}^{\flat}$ is 
$$
H\opw
{\indic{[0,\lambda]}(\xi)}H,\text{\ that is\quad  $H\indic{[0,\lambda]}(D)H$,}
$$
 where $H$ stands for the operator of  multiplication
by the Heaviside function $H$.
On the other hand, the operator with distribution kernel $k_{0,\lambda}^{\sharp}$
is such that
\begin{multline*}
\val{k_{0,\lambda}^{\sharp}(x,y)}\le H(x+y)\frac{H(-x) H(y)+H(x)H(-y)}{\pi\val{x-y}}
\\=H(x+y)\frac{H(-x) H(y)}{\pi(y-x)}+H(x+y)\frac{H(x)H(-y)}{\pi(x-y)}.
\end{multline*}
According to 
Proposition \ref{pro.hardy} in our Appendix,
the Hardy operator and the modified Hardy operators are bounded on $L^{2}(\R)$
and 
we obtain 
that, for $\phi,\psi\in \mathscr S(\R^{n})$, with $H=H(x), \check H=H(-x)$, 
\begin{multline}
\Val{\iint H(x) \indic{[0,\lambda]}(\xi)W(\phi,\psi)(x,\xi) dx d\xi}
\\
\le \norm{H \phi}_{L^{2}(\R)} \norm{H \psi}_{L^{2}(\R)} 
+ \frac12\norm{H \phi}_{L^{2}(\R)}  
\norm{\check H \psi}_{L^{2}(\R)}+\frac12 \norm{\check H \phi}_{L^{2}(\R)}  
\norm{H \psi}_{L^{2}(\R)},
\end{multline}
so that 
\begin{multline}\label{515ert}
\val{\poscal{A_{0}\phi}{\psi}_{\mathscr S^{*}(\R), \mathscr S(\R)}}
\\=
\Bigl\vert\iint H(x)H(\xi)\overbracket[0.5pt][3pt]{W(\phi,\psi)(x,\xi)}^{\in \mathscr S(\R^{2})} dx d\xi
\Bigr\vert=
\lim_{\lambda\rightarrow+\io}\Val{\iint H(x) \indic{[0,\lambda]}(\xi)W(\phi,\psi)(x,\xi) dx d\xi}
\\
\le \norm{H \phi}_{L^{2}(\R)} \norm{H \psi}_{L^{2}(\R)} 
+\frac12 \norm{H \phi}_{L^{2}(\R)}  
\norm{\check H \psi}_{L^{2}(\R)}
+\frac12 \norm{\check H \phi}_{L^{2}(\R)}  
\norm{H \psi}_{L^{2}(\R)},
\end{multline}
yielding the $L^{2}$-boundedness of the operator $A_{0}$, and this concludes the proof of the lemma.
\end{proof}
\begin{rem}\label{rem.5487}{\rm
That cumbersome detour with the operator $A_{0,\lambda}$ is useful to ensure that the operator $A$ is indeed bounded on $L^{2}(\R)$.
The kernel $k_{0}$
of $A_{0}$ is a distribution of order 1 and the product $H(x) H(y) k_{0}(x,y)$ is not a priori
meaningful, even when $k$ is a Radon measure\footnote{\label{footwave}Even a wave-front-set approach, which would allow the product
$H(x)\text{pv}(1/(y-x))$, does not offer a meaning for the product
$H(x)H(y)\text{pv}(1/(y-x))$ since the wave-front-set
of $\text{pv}(1/(y-x))$ is located on the conormal of the first diagonal (i.e. $\{(x,x;\xi,-\xi)\}_{x\in \R, \xi\in \R^{*}}$),
whereas the wave-front set at $(0,0)$ of $H(x) H(y)$ contains all directions and in particular is antipodal to the conormal of the diagonal at $(0,0)$.
}. However with the proven $L^{2}$-boundedness of $A_{0}$, the products of operators
$HA_{0}H$,
$\check H A_{0} H$, 
$HA_{0}\check H$, $\check H A_{0}\check H$
make sense and for instance we may approximate in the strong-operator-topology the operator 
$HA_{0}H$ by the operator 
$
\chi(\cdot/\varepsilon) A\chi(\cdot/\varepsilon),
$
where $\chi$ is a smooth function supported in $[1,+\io)$ and equal to $1$ on $[2,+\io)$.
We have indeed
$$
HAH=\bigl(H-\chi(\cdot/\varepsilon)\bigr) A H+\chi(\cdot/\varepsilon) A \bigl(H-\chi(\cdot/\varepsilon)\bigr) 
+\chi(\cdot/\varepsilon)A \chi(\cdot/\varepsilon),
$$
so that for $u\in L^{2}(\R)$,
$
HAH u=\lim_{\varepsilon\rightarrow 0_{+}}\chi(\cdot/\varepsilon)A \chi(\cdot/\varepsilon)u.
$
The operator with kernel
$$
H(x+y)\chi(x/\varepsilon)\chi(y/\varepsilon) \textrm{\rm pv}\frac{1}{i\pi(y-x)}=\chi(x/\varepsilon)\chi(y/\varepsilon)  \text{\rm pv}\frac{1}{i\pi(y-x)},
$$
converges strongly towards  the operator 
$H(\sign D) H$.
}
\end{rem}
\begin{pro}\label{pro.hgf654}
 Let $A_{0}=\opw{H(x)H(\xi)}$ be 
  the operator
with Weyl symbol
$H(x) H(\xi)$, a priori sending $\mathscr S(\R)$ into 
$\mathscr S'(\R)$. Then $A_{0}$ can be uniquely extended to a self-adjoint bounded operator
on $L^{2}(\R)$
with
\begin{equation}\label{511rez}
\norm{A_{0}}_{\mathcal B(L^{2}(\R))}\le \frac{1+\sqrt{2}}{2}\approx 
1.207
\end{equation}
\end{pro}
\begin{nb}
 The bound above can be  significantly improved (see Proposition 
 \ref{pro.521kjs}
 for optimal bounds)
  and moreover we will show below that the spectrum of $A_{0}$ actually  intersects $(1,+\io)$.
  In fact it is easier to start with the information that $A_{0}$ is indeed bounded on $L^{2}(\R)$.
\end{nb}
\begin{proof}
The $L^{2}(\R)$-boundedness of $A_{0}$ is given by Lemma \ref{lem.51fdsq}.
We are left with proving the bound \eqref{511rez}: we note that \eqref{515ert}
implies
$$
\val{\poscal{A_{0} u}{u}_{L^{2}(\R)}}\le \norm{Hu}_{L^{2}(\R)}^{2}+
\norm{Hu}_{L^{2}(\R)}\norm{\check Hu}_{L^{2}(\R)},
$$
proving the proposition, since the eigenvalues of the quadratic form $\R^{2}\ni(x_{1}, x_{2})
\mapsto x_{1}^{2}+ x_{1}x_{2}$
are $(1\pm \sqrt 2)/2$.
\end{proof}
We can do much better and actually diagonalize the operator $A_{0}$, using as in Proposition \ref{pro.hardy} logarithmic coordinates on each half-line.
We state a lemma on ``diagonal'' terms whose proof is already given above.
\begin{lem}[Diagonal terms]\label{lem.52swxq}
Let $A_{0}$ be the operator with Weyl symbol $H(x) H(\xi)$.
With $H$ standing as well for the operator of multiplication by $H(x)$, we have 
\begin{gather}
HA_{0}H=H H(D) H=H\frac{(\Id+\sign D)}2H.
\label{516yov}
\end{gather}
\end{lem}
\begin{lem}[Off-diagonal terms]\label{lem.52ezrr}
 Let $B_{0}=2\re \check H A_{0}H=\check H A_{0}H+H A_{0}\check H$. 
 Then we have for all $u\in L^{2}(\R)$,
\begin{equation}\label{519jhg}
 \val{\poscal{B_{0} u}{u}_{L^{2}(\R)}}\le\frac12\norm{Hu}_{L^{2}(\R)}\norm{\check Hu}_{L^{2}(\R)}.
\end{equation}
\end{lem}
\begin{proof}[Proof of the Lemma]
For $u\in \mathscr S(\R)$ such that $0\notin \supp u$, we define for $t\in \R$,
\begin{equation}\label{519jhf}
\phi_{1}(t)= u(e^{t}) e^{t/2}, \quad \phi_{2}(t)=u(-e^{t}) e^{t/2},
\end{equation}
so that 
\begin{equation}
\norm{Hu}_{L^{2}(\R)}^{2}=\norm{\phi_{1}}_{L^{2}(\R)}^{2}, \quad \norm{\check Hu}_{L^{2}(\R)}^{2}=\norm{\phi_{2}}_{L^{2}(\R)}^{2}.
\end{equation}
We have 
 \begin{align}
 &\poscal{B_{0} u}{u}_{L^{2}(\R)}=
 \iint 
 \frac{H(x+y)\bigl(\check H(x) H(y)+H(x)\check H(y)\bigr)}{2i\pi (y-x)} u(y) \bar u(x) dy dx
 \notag\\&\hskip22pt =
 \iint \frac{H(-e^{s}+e^{t}) e^{\frac{s+t}2}}{2i\pi(e^{t}+e^{s})}\phi_{1} (t)
 \bar \phi_{2}(s) ds dt 
 -
  \iint \frac{H(e^{s}-e^{t}) e^{\frac{s+t}2}}{2i\pi(e^{t}+e^{s})}\phi_{2} (t)
 \bar \phi_{1}(s) ds dt 
 \notag\\
 &\hskip22pt =
  \iint \frac{H(t-s)}{4i\pi
  \cosh(\frac{s-t}2)
  }\phi_{1} (t)
 \bar \phi_{2}(s) ds dt 
 -
  \iint \frac{H(s-t)}{4i\pi\cosh(\frac{s-t}2)}\phi_{2} (t)
 \bar \phi_{1}(s) ds dt 
 \notag
 \end{align}
so that 
\begin{gather}
\poscal{B_{0} u}{u}_{L^{2}(\R)}=
 \poscal{\tilde S_{0}\ast \phi_{1}}{\phi_{2}}_{L^{2}(\R)}
 +\poscal{S_{0}\ast \phi_{2}}{\phi_{1}}_{L^{2}(\R)},
 \\
 \tilde S_{0}(t)=\frac{\check H(t)}{4i\pi\cosh (t/2)}, \quad S_{0}(t)=\frac{ iH(t)}{4\pi\cosh (t/2)}.
\label{5110kj}\end{gather}
We calculate
$$
\int_{0}^{+\io}\frac{dt}{4\pi \cosh (t/2)}=\frac{1}{2\pi}[\arctan(\sinh (t/2))]_{0}^{+\io}=\frac14=
\int_{-\io}^{0}\frac{dt}{4\pi \cosh (t/2)},
$$
so that 
\begin{equation}
\val{\poscal{B_{0} u}{u}_{L^{2}(\R)}}\le \frac12\norm{\phi_{1}}_{L^{2}(\R)}
\norm{\phi_{2}}_{L^{2}(\R)}=\frac12\norm{Hu}_{L^{2}(\R)}\norm{\check Hu}_{L^{2}(\R)},
\end{equation}
proving the estimate of the lemma for $u\in \mathscr S(\R)$ such that $0\notin \supp u$.
We use now that we already know that $B_{0}$ is a bounded self-adjoint operator on $L^{2}(\R)$: let $u$ be a function in $L^{2}(\R)$ and let $(\phi_{k})_{k\ge 1}$ be a sequence\footnote{Such a sequence is easy to find: a first step is to find  a sequence $(\tilde \phi_{k})_{k\ge 1}$ in the Schwartz space converging in $L^{2}(\R)$ towards $u$, then consider with a given $\omega\in \moo(\R;[0,1])$ such that $\omega(t)=0$ for $\val t\le 1$ and $\omega(t)=1$ for $\val t\ge 2$, 
$
\phi_{k}(x) =\omega(kx)\tilde \phi_{k}(x).
$} in $\mathscr S(\R)$ such that each $\phi_{k}$ vanishes in a neighborhood of 0 so that $\lim_{k}\phi_{k}=u$ in $L^{2}(\R)$.
We find 
that 
\begin{multline*}
\val{\poscal{B_{0}u}{u}_{L^{2}(\R)}}\le \val{\poscal{B_{0}(u-\phi_{k})}{u}_{L^{2}(\R)}}
+
 \val{\poscal{B_{0}\phi_{k}}{u-\phi_{k}}_{L^{2}(\R)}}
+
 \val{\poscal{B_{0}\phi_{k}}{\phi_{k}}_{L^{2}(\R)}}
\\ \le  \norm{B_{0}}_{\mathcal B(L^{2}(\R))}
\bigl(
\norm{u-\phi_{k}}_{L^{2}(\R)}\norm{u}_{L^{2}(\R)}
+
\norm{u-\phi_{k}}_{L^{2}(\R)}\norm{\phi_{k}}_{L^{2}(\R)}\bigr)
\\+\frac12 \norm{H\phi_{k}}_{L^{2}(\R)} \norm{\check H\phi_{k}}_{L^{2}(\R)},
\end{multline*}
providing readily the result of the lemma since the multiplication by $H$ and $\check H$ are bounded operators on $L^{2}(\R)$.
\end{proof}
\begin{rem}\rm
 The estimate \eqref{519jhg} and Lemma \ref{lem.52swxq} are already improving
 \eqref{511rez},
 since  the eigenvalues of the quadratic form $\R^{2}\ni(x_{1}, x_{2})
\mapsto x_{1}^{2}+ \frac12x_{1}x_{2}$
are $(2\pm \sqrt 5)/4$, so that the right-hand-side of  \eqref{511rez} can be replaced by
$(2+ \sqrt 5)/4\approx 1.059$. Anyhow, we shall provide below a diagonalization of $A_{0}$ and optimal bounds.
\end{rem}
\begin{nb}\rm 
 We shall be a little faster in the sequel on the ``cumbersome'' detours to avoid formal multiplication of kernels by Heaviside functions but the reader should keep in mind that it is an important point to secure $L^{2}(\R)$-boundedness \emph{before} any further manipulation of the kernels.
\end{nb}
\subsubsection{An isometric isomorphism}
\begin{rem}\label{rem.kjhg43}\rm
 The mapping  $\Psi$ defined by 
 \begin{equation}\label{5113oi}
\begin{matrix}
\Psi:\ L^{2}(\R)&\longrightarrow &L^{2}(\R;\C^{2})\\
u&\mapsto& \Bigl((Hu)(e^{t}) e^{t/2}, (\check Hu)(-e^{t}) e^{t/2}\Bigr)
\end{matrix}
\end{equation}
is an isometric isomorphism of Hilbert spaces:
indeed we have
$$
\norm{u}^{2}_{L^{2}(\R)}=
\int_{\R}\val{u(e^{t})}^{2} e^{t}dt+\int_{\R}\val{u(-e^{t})}^{2} e^{t}dt.
$$
Moreover if $(\phi_{1}, \phi_{2})\in L^{2}(\R;\C^{2})$, we may define  for $x\in \R^{*}$
$$
u(x)=H(x) \phi_{1}(\ln x)x^{-1/2}+\check H(x) \phi_{2}(\ln \val x)\val x^{-1/2},
$$
and we have 
$
\Psi (u)(t)= \bigl(\phi_{1}(t),\phi_{2}(t)\bigr).
$
\end{rem}
\begin{rem}\label{rem.jhgf55}\rm
 Using Lemma \ref{lem.52swxq} and Notations \eqref{519jhf}
 we see that 
 \begin{align}
 \poscal{HA_{0}Hu}{u}_{L^{2}(\R)}&=\frac12 \norm{\phi_{1}}^{2}_{L^{2}(\R)}
 +\iint\frac{1}{2i\pi}\text{pv}\frac{e^{(s+t)/2}}{e^{t}-e^{s}}\phi_{1}(t) \bar \phi_{1}(s) dsdt
\notag\\
&=\frac12 \norm{\phi_{1}}^{2}_{L^{2}(\R)}
 +\iint\frac{1}{4i\pi}\text{pv}\frac{1}{\sinh(\frac{t-s}2)}\phi_{1}(t) \bar \phi_{1}(s) dsdt
 \notag\\&=
 \int_{\R}\val{\hat\phi_{1}(\tau)}^{2}\bigl(\frac12+\hat T_{0}(\tau)\bigr) d\tau,
\label{65tgfd}\end{align}
with 
\begin{equation}\label{5114pi}
T_{0}(t)=\frac{t}{4\sinh(t/2)}\text{pv}\frac{i}{\pi t}.
\end{equation}
We have 
\begin{gather}\label{5116gg}
\hat T_{0}=\sign\ast \rho_{0}, \quad \text{with\quad }\rho_{0}(\tau)=\int\frac{t}{4\sinh (t/2)}
e^{-2i\pi t\tau} dt,
\end{gather}
and we note that the function $\rho_{0}$ belongs to $\mathscr S(\R)$,
as the Fourier transform of a function  in $\mathscr S(\R)$.
Also we have 
$$
\int \rho_{0}(\tau) d\tau=\hat \rho_{0}(0)=\frac12,
$$
and this yields with 
$
\frac{d}{d\tau}\left\{\frac 12+\hat T_{0}\right\}=2\rho_{0}
$
(which follows from \eqref{5116gg})
and
\begin{equation}\label{5117yg}
\frac 12+\hat T_{0}(\tau)=1-\int_{\tau}^{+\io} 2\rho_{0}(\tau') d\tau',
\end{equation}
since 
$$
\frac{d}{d\tau}\left\{\frac 12+\hat T_{0}+\int_{\tau}^{+\io} 2\rho_{0}(\tau') d\tau'\right\}= 0
\quad\text{and\quad}
\lim_{\tau\rightarrow+\io} (\sign\ast \rho_{0})(\tau)=\frac12.
$$
\end{rem}
\begin{theorem}\label{thm.calculs--}
Let $A_{0}$ be the operator with Weyl symbol $H(x) H(\xi)$.
 The operator $A_{0}$ is bounded self-adjoint  on $L^{2}(\R)$ so that we may define,
 with $\Psi$ defined in \eqref{5113oi}, 
 \begin{equation}\label{}
\wt A_{0}=\Psi A_{0}\Psi^{-1}.
\end{equation}
The operator $\wt A_{0}$ is the Fourier multiplier on $L^{2}(\R;\C^2)$ given by the matrix
\begin{equation}\label{5219+--}
\renewcommand\arraystretch{1.7}
\mathcal M_{0}(\tau)=
\begin{pmatrix}
\frac12+\hat T_{0}(\tau)&\hat S_{0}(\tau)\\
\overline{\hat S_{0}(\tau)}&0
\end{pmatrix},
\end{equation}
where $T_{0}, S_{0}$ are defined respectively in 
\eqref{5114pi},
\eqref{5110kj}.
 In particular we have
with $\Phi=(\phi_{1}, \phi_{2})\in L^{2}(\R;\C^{2})$,
\begin{equation}\label{5219--}
\poscal{\wt A_{0}\Phi}{\Phi}_{L^{2}(\R;\C^{2})}
=\int_{\R}
e^{2i\pi t\tau}
\poscal{\mathcal M_{0}(\tau)\hat \Phi(\tau)}{\hat \Phi(\tau)}_{\C^{2}} d\tau.
\end{equation}
\end{theorem}
\begin{rem}\label{rem.57uchg}
As a consequence of Theorem \ref{thm.calculs--}, we find that the spectrum of the self-adjoint bounded operator $A_{0}$ is the closure of the set of eigenvalues of the matrices $\mathcal M_{0}(\tau)$ when $\tau$ runs on the real line.
 \end{rem}
\begin{proof}
 The proof follows readily from Remarks \ref{rem.kjhg43}, \ref{rem.jhgf55} and Lemmas
\ref{lem.52swxq},
\ref{lem.52ezrr}.
\end{proof}
\begin{lem}\label{lem.510--}
 Let $\mathcal N$ be   a $2\times 2$ Hermitian matrix
 $$
 \renewcommand\arraystretch{1.2}
\mathcal N=
\begin{pmatrix}
a_{11}&
a_{12}
\\
\overline{a_{12}}&
0
\end{pmatrix}.
 $$
 Then the eigenvalues $\lambda_{-}\le \lambda_{+}$ of $\mathcal N$ are such that
 \begin{equation}\label{goal00}
\lambda_{-}<0<1< \lambda_{+},
\end{equation}
if and only if
\begin{equation}\label{cond00}
a_{12}\not=0\text{ \quad and \quad}\val{a_{12}}^2>1-a_{11}.
\end{equation}
\end{lem}
\begin{proof}
 The characteristic polynomial of $\mathcal N$
 is $p(\lambda)=
 \lambda^2-a_{11}\lambda-\val{a_{12}}^2$ and since $a_{11}$ is real-valued,
 has two real roots $\lambda_{-}\le \lambda_{+}$. If \eqref{cond00} holds true, the roots are distinct and
 $$
 p(0)=-\val{a_{12}}^2<0, \quad p(1)=1-a_{11}-\val{a_{12}}^2<0,
 $$
 implying \eqref{goal00}. Conversely, if \eqref{goal00} is satisfied, then $p(0), p(1)$ are both negative,  implying \eqref{cond00}, completing the proof of the lemma.
\end{proof}
\begin{lem}\label{lem.51--}
 Let us define for $\omega\in \R$,
\begin{equation}\label{}
I(\omega)=\frac{1}{4\pi}
\int_{0}^{+\io}\frac{\sin(t\omega)}
{\cosh(t/2)} 
dt.
\end{equation}
Then we have
\begin{equation}\label{}
I(\omega)=\frac{1}{4\pi \omega}+O(\omega^{-3}), \quad \val \omega\rightarrow+\io.
\end{equation}
\end{lem}
\begin{proof}
Indeed we have for $\omega\in \R^{*}$,
\begin{multline*}
I(\omega)=-\frac{1}{4\pi \omega}
\int_{0}^{+\io}\frac{\frac{d}{dt}\cos(t\omega)}
{\cosh(t/2)} 
dt
=\frac{1}{4\pi \omega}\Bigl(1
-\int_{0}^{+\io}\frac{\cos(t\omega)}
{(\cosh(t/2))^{2}} \frac12\sinh(t/2)
dt\Bigr)
\\
=\frac{1}{4\pi \omega}\bigl(1+g(\omega)\bigr), \end{multline*}
with
\begin{align*}
g(\omega)&=
-\int_{0}^{+\io}{
\frac{d}{\omega dt}\{\sin(t\omega)\}}
{\sech(t/2)} \frac12\tanh(t/2)
dt
\\&=\frac1{2\omega}\int_{0}^{+\io}\sin(t\omega)\frac{d}{dt}\bigl\{{\sech(t/2)} \tanh(t/2)\bigr\} dt
\\
&=-\frac1{2\omega^{2}}\int_{0}^{+\io}\frac{d}{dt}
\bigl\{\cos(t\omega)\bigr\}\frac{d}{dt}\bigl\{{\sech(t/2)} \tanh(t/2)\bigr\} dt
\\
&=\frac1{2\omega^{2}}\left\{\int_{0}^{+\io}
\cos(t\omega)\frac{d^{2}}{dt^{2}}\bigl\{{\sech(t/2)} \tanh(t/2)\bigr\} dt+\frac12\right\}=O(\omega^{-2}),
\end{align*}
proving the lemma.
\end{proof}
\begin{pro}\label{pro.510uhc}
 The matrix $\mathcal M_{0}(\tau)$ defined in \eqref{5219+--} is equal to 
\begin{equation}\label{5123hg}
 \mathcal M_{0}(\tau)=\renewcommand\arraystretch{1.7}\mat22{a_{11}(\tau)}{a_{12}(\tau)}{\overline{a_{12}(\tau)}}{0}
\end{equation}
with
 \begin{equation}\label{5124hg}
 1-a_{11}(\tau)=\int_{\tau}^{+\io} 2\rho_{0}(\tau') d\tau', \quad 
 a_{12}(\tau)=\frac{i}{4\pi}\int_{0}^{+\io}\frac{ 1}{\cosh (t/2)} e^{-2i\pi \tau t} dt.
\end{equation}
We have 
\begin{align}
1-a_{11}(\tau)&=O(\tau^{-N}) \quad \text{for any $N$ when $\tau\rightarrow+\io$},
\label{5127tt}\\
\re (a_{12}(\tau))&=\frac{1}{8\pi^{2}\tau}+O(\tau^{-3})
\quad\text{when $\tau\rightarrow+\io$.}
\label{5128hh}\end{align}
\end{pro}
\begin{proof}
 Formulas \eqref{5123hg}, \eqref{5124hg} follow from  Theorem \ref{thm.calculs--},
 \eqref{5117yg} and \eqref{5110kj}.
 The estimates \eqref{5127tt} follow from the fact that $\rho_{0}$ belongs to the Schwartz class
 and  \eqref{5128hh} is a reformulation of Lemma \ref{lem.51--}.
 \end{proof}
\begin{theorem}\label{thm.54}
 Let $A_{0}$ be the operator  with Weyl symbol $H(x) H(\xi)$, where $H$ is the Heaviside function. Then $A_{0}$ is a bounded self-adjoint operator on $L^{2}(\R)$ such that
 \begin{equation}\label{5314}
\inf\bigl(\text{\rm spectrum}(A_{0})\bigr)<0<1<\sup\bigl(\text{\rm spectrum}(A_{0})\bigr).
\end{equation}
\end{theorem}
\begin{proof}
Using Remark \ref{rem.57uchg} and Proposition \ref{pro.510uhc}
we find that for $\tau$ large enough, Conditions
\eqref{cond00} are satisfied, proving readily 
\eqref{5314}.
\end{proof}
\begin{cor}[A counterexample to Flandrin's conjecture]\label{cor.55}
There exists a function $\phi_{0}\in \mathscr S(\R)$, with $L^{2}(\R)$ norm equal to 1 such that
  \begin{equation}\label{}
\iint_{x\ge 0, \xi\ge 0} \mathcal W(\phi_{0}, \phi_{0})(x,\xi) dx d\xi >1.
\end{equation}
There exists $a>0$ such that 
$\iint_{0\le x\le a, 0\le \xi\le a} \mathcal W(\phi_{0}, \phi_{0})(x,\xi) dx d\xi >1.$ 
\end{cor}
\begin{rem}\rm
On page 2178 of \cite{conjecture},
we find the sentence
 {\it ``it is conjectured that 
\begin{equation}\label{5133yu}
\forall u\in L^{2}(\R), \quad
\iint_{\mathcal C} \mathcal W(u,u)(x,\xi) dx d\xi\le \norm{u}_{L^{2}(\R)}^{2},
\end{equation}
is true for any convex domain $\mathcal C$''},
a quite mild commitment for the validity of \eqref{5133yu},
although that statement was referred to later on as {\it Flandrin's conjecture}
in the literature. 
The second part of the above corollary is providing a disproof of that conjecture based upon an ``abstract'' argument used in the proof of Theorem \ref{thm.54}; the result of that corollary was already known 
via a numerical analysis argument
after our joint work \cite{DDL} with B.~Delourme and T.~ Duyckaerts.\end{rem}
\begin{proof}
 From Theorem \ref{thm.54}, we find $u_{0}\in L^{2}(\R)$ such that 
 $$
 \norm{u_{0}}_{L^{2}(\R)}^{2}<\poscal{A_{0}u_{0}}{u_{0}}. 
 $$
 Let $\psi\in \mathscr S(\R)$: we have 
 \begin{multline*}
 \val{\poscal{A_{0}u_{0}}{u_{0}}-\poscal{A_{0}\psi}{\psi}}=\val{\poscal{A_{0}(u_{0}-\psi)}{u_{0}}+\poscal{A_{0}\psi}{u_{0}-\psi}}
 \\
 \le \norm{A_{0}}_{\mathcal B(L^{2}(\R))}\norm{u_{0}-\psi}_{L^{2}(\R)}
 \bigl(\norm{u_{0}}_{L^{2}(\R)}+
\norm{\psi}_{L^{2}(\R)}
 \bigr),
\end{multline*}
and thus if $(\psi_{k})_{k\ge 1}$ is a sequence of $\mathscr S(\R)$ converging towards $u_{0}$ in $L^{2}(\R)$, we get
\begin{multline*}
\norm{u_{0}}_{L^{2}(\R)}^{2}<\poscal{A_{0}u_{0}}{u_{0}}
\\\le \poscal{A_{0}\psi_{k}}{\psi_{k}}+
\underbrace{\norm{A_{0}}_{\mathcal B(L^{2}(\R))}\norm{u_{0}-\psi_{k}}_{L^{2}(\R)}
 \bigl(\norm{u_{0}}_{L^{2}(\R)}+
\norm{\psi_{k}}_{L^{2}(\R)}
 \bigr)}_{= \sigma_{k}, \text{ goes to 0 when $k\rightarrow+\io$.}}.
\end{multline*}
There exists $k_{0}\ge 1$ such that for $k\ge k_{0}$, we have 
$$
0\le \sigma_{k}\le \frac12\bigl(
\poscal{A_{0}u_{0}}{u_{0}}-\norm{u_{0}}_{L^{2}(\R)}^{2}
\bigr)=\frac{\varepsilon_{0}}2, \quad \varepsilon_{0}>0.
$$ 
We obtain that for $k\ge k_{0}$, 
$$
\norm{u_{0}}_{L^{2}(\R)}^{2}<\poscal{A_{0}u_{0}}{u_{0}}\le 
\poscal{A_{0}\psi_{k}}{\psi_{k}}+\frac{\varepsilon_{0}}2,
$$
and thus
\begin{multline*}
\norm{\psi_{k}}_{L^{2}(\R)}^{2}=
\underbrace{ \norm{\psi_{k}}_{L^{2}(\R)}^{2}-
\norm{u_{0}}_{L^{2}(\R)}^{2}}_{=\theta_{k}, \text{ goes to 0 when $k\rightarrow+\io$}}+\norm{u_{0}}_{L^{2}(\R)}^{2}
\\= \theta_{k}
+\poscal{A_{0}u_{0}}{u_{0}}-\varepsilon_{0}\le \theta_{k}+ 
\poscal{A_{0}\psi_{k}}{\psi_{k}}+\frac{\varepsilon_{0}}2-\varepsilon_{0}
\\=\poscal{A_{0}\psi_{k}}{\psi_{k}}+\theta_{k}-\frac{\varepsilon_{0}}2.
\end{multline*}
Choosing now $k\ge k_{0}$ and $k$ large enough to have $\theta_{k}<\varepsilon_{0}/4$,
we get
$$
\norm{\psi_{k}}_{L^{2}(\R)}^{2}\le \poscal{A_{0}\psi_{k}}{\psi_{k}}-\frac{\varepsilon_{0}}4< \poscal{A_{0}\psi_{k}}{\psi_{k}},
$$
and since for $\tilde \phi=\psi_{k}$, the Wigner distribution  $\mathcal W(\tilde \phi, \tilde \phi)$ belongs to 
$\mathscr S(\R^{2})$,
we have 
$$\norm{\tilde \phi}_{L^{2}(\R)}^{2}<
\poscal{A_{0}\tilde \phi}{\tilde \phi}=\iint H(x) H(\xi) \mathcal W(\tilde \phi, \tilde \phi)(x,\xi) dx d\xi,
$$
and noting that  this strict inequality above  implies that $\tilde \phi\not=0$, we may set $\phi_{0}=\tilde \phi/\norm{\tilde \phi}$ and get
the first statement in the corollary.
\begin{nb}\rm
 The proof above is complicated by the fact that the identity
 $$
 \poscal{a^{w} u}{u}_{L^{2}(\R^{n})}=\iint_{\RZ} a(x,\xi) \mathcal W(u,u)(x,\xi) dx d\xi,
 $$
 is valid a priori for $u\in \mathscr S(\R^{n})$ (and in that case $\mathcal W(u,u)$ belongs to
 $\mathscr S(\R^{2n})$), but could be meaningless as a Lebesgue integral even for $\opw{a}$ bounded on $L^{2}(\R^{n})$ and $u\in L^{2}(\R^{n})$, since we shall have $\mathcal W(u,u)\in L^{2}(\RZ)$ but not in $L^{1}(\RZ)$
 (we shall see in Section \ref{sec.baire} that generically the Wigner distribution of a pulse $u$  in $L^{2}(\R^{n})$ does {\bf not} belong to $L^{1}(\RZ)$).
\end{nb}
Since $\mathcal W(\phi, \phi)$ belongs to the Schwartz space of $\R^{2}$, the Lebesgue Dominated Convergence Theorem  provides the last statement in the Corollary.
\end{proof}
\begin{nb}\rm
 The reader will  notice that the results
 of the incoming Section \ref{sec.52} in the special case $\sigma=0$
 imply the results of Section \ref{sec.51}, which could be then erased, say at the second reading.
 However, as far as the first -- and maybe only -- reading is concerned, we checked that most of the computational arguments in the next section are much more involved 
 and it seemed worth while to the author to avoid unnecessary complications
 for the disproof of Flandrin's conjecture 
 via the quarter-plane example and set apart the more involved examples of the hyperbolic regions tackled in Section \ref{sec.52}.\end{nb}
\subsection{Hyperbolic regions}\label{sec.52}
\index{hyperbolic regions}
We consider in this section  the \eqref{5014}
set $\mathcal C_{\sigma}$ with  a non-negative $\sigma$. 
\subsubsection{A preliminary observation}
We want  to consider the operator $A_{\sigma}$
with Weyl symbol $H(x) H(x\xi-\sigma)$
and as in Section \ref{sec.prelim}, we would like to secure the fact that $A_{\sigma}$ is bounded on $L^{2}(\R)$. 
\begin{claim}\label{cla.516aze}
 For all $\sigma\ge 0$ the operator $A_{\sigma}$ is bounded self-adjoint on $L^{2}(\R)$.
\end{claim}
\begin{proof}[Proof of the Claim]
 Let us choose 
 \begin{equation}\label{521cut}
\text{$\chi_{0}\in\moo(\R;[0,1])$ with }
\begin{cases} 
\text{$\chi_{0}(t)=0$, for $t\le 1$,}
\\
\text{$\chi_{0}(t)=1$,  for $t\ge 2$.}
\end{cases}
\end{equation}
For $\phi, \psi\in \mathscr S(\R)$, we have
\begin{multline}\label{521aze}
\poscal{(A_{0}-A_{\sigma})\phi}{\psi}_{\mathscr S^{*}(\R), \mathscr S(\R)}=\iint H(x) H(\xi)H(\sigma-x\xi)\underbrace{\mathcal W(\phi, \psi)(x,\xi)}_{\in \mathscr S(\R^{2})} dx d\xi
\\
=\lim_{\epsilon\rightarrow 0_{+}}\iint\chi_{0}(x/\epsilon)
H(\xi)H(\sigma-x\xi)\mathcal W(\phi, \psi)(x,\xi) dx d\xi.
\end{multline}
The kernel $k_{\sigma, \epsilon}$ of the operator with Weyl symbol $\chi_{0}(x/\epsilon)
H(\xi)H(\sigma-x\xi)$
is
\begin{equation}\label{}
\ell_{\sigma, \epsilon}(x,y)=
\chi_{0}\bigl(\frac{x+y}{2\epsilon}\bigr)e^{2i\pi\sigma\frac{x-y}{x+y}}
\frac{\sin(\frac{2\pi \sigma(x-y)}{x+y})}{\pi(x-y)},
\end{equation}
and we have 
\begin{align}
\iint \ell_{\sigma,\epsilon}(x,y)&\phi(y) \bar \psi(x) dy dx
=\iint\chi_{0}\bigl(\frac{x+y}{2\epsilon}\bigr)e^{2i\pi\sigma\frac{x-y}{x+y}}
\frac{\sin(\frac{2\pi \sigma(x-y)}{x+y})}{\pi(x-y)}\phi(y) \bar \psi(x) dxdy
\notag\\
&=\iint\chi_{0}\bigl(\frac{x+y}{2}\bigr)e^{2i\pi\sigma\frac{x-y}{x+y}}
\frac{\sin(\frac{2\pi \sigma(x-y)}{x+y})}{\pi\epsilon(x-y)}\phi(\epsilon y) \bar \psi(\epsilon x) \epsilon^{2}dxdy
\notag\\&=
\iint
\underbrace{\chi_{0}\bigl(\frac{x+y}{2}\bigr)e^{2i\pi\sigma\frac{x-y}{x+y}}
\frac{\sin(\frac{2\pi \sigma(x-y)}{x+y})}{\pi(x-y)}}_{m_{\sigma}(x,y)}
\underbrace{\phi(\epsilon y)\epsilon^{1/2} }_{\phi_{\epsilon}(y)}
\underbrace{\bar \psi(\epsilon x) \epsilon^{1/2}}_{\bar \psi_{\epsilon}(x)}dy dx.
\label{523aze}
\end{align}
We note that, assuming as we may that $\sigma>0$, 
\begin{multline}\label{524aze}
\val{m_{\sigma}(x,y) H(x) H(y)}=\chi_{0}\bigl(\frac{x+y}{2}\bigr)
\Val{
\frac{\sin(\frac{2\pi \sigma(x-y)}{x+y})}{\frac{2\pi\sigma(x-y)}{x+y}}}
\frac{2\sigma H(x) H(y)}{x+y}
\\\le \frac{2\sigma H(x) H(y)}{x+y},
\end{multline}
and 
\begin{multline}\label{525aze}
\val{m_{\sigma}(x,y) \check H(x) H(y)}=\chi_{0}\bigl(\frac{x+y}{2}\bigr)
\Val{
\frac{\sin(\frac{2\pi \sigma(x-y)}{x+y})}{{\pi (x-y)}}}
{ \check H(x) H(y)}
\\\le \frac{ \check H(x) H(y)}{\pi(y-x)},
\end{multline}
as well as
\begin{multline}\label{526aze}
\val{m_{\sigma}(x,y) \check H(y) H(x)}=\chi_{0}\bigl(\frac{x+y}{2}\bigr)
\Val{
\frac{\sin(\frac{2\pi \sigma(x-y)}{x+y})}{{\pi (x-y)}}}
{ \check H(y) H(x)}
\\\le \frac{ \check H(y) H(x)}{\pi(x-y)}.
\end{multline}
As a consequence, since we have also
$
m_{\sigma}(x,y) \check H(x) \check H(y)\equiv 0,
$
the inequalities \eqref{524aze},  \eqref{525aze}, \eqref{526aze}, 
the identities 
\eqref{523aze}, \eqref{521aze} and 
Proposition \ref{pro.hardy}
imply that 
\begin{multline*}
\val{\poscal{(A_{0}-A_{\sigma})\phi}{\psi}_{\mathscr S^{*}(\R), \mathscr S(\R)}}
\le 2\pi \sigma \underbrace{\norm{H\phi_{\epsilon}}_{L^{2}(\R)}}_{\norm{H\phi_{}}_{L^{2}(\R)}}
\norm{H\psi_{\epsilon}}_{L^{2}(\R)}
+
 \underbrace{\norm{\check H\phi_{\epsilon}}_{L^{2}(\R)}}_{\norm{\check H\phi}_{L^{2}(\R)}}
\norm{H\psi_{\epsilon}}_{L^{2}(\R)}
\\+
 \norm{H\phi_{\epsilon}}_{L^{2}(\R)}
\norm{\check H\psi_{\epsilon}}_{L^{2}(\R)},
\end{multline*}
proving that 
$A_{0}-A_{\sigma}$ is bounded on $L^{2}(\R)$; with Proposition \ref{pro.hgf654}, this implies that 
$A_{\sigma}$ is also  bounded on $L^{2}(\R)$, proving the claim.
\end{proof}
\begin{nb}\rm
 With that important piece of information in Claim \ref{cla.516aze}, we shall be less strict in our manipulations of the kernels and accept below  some abuse of language in these matters. 
\end{nb}
The Weyl quantization of $\mathbf 1_{\mathcal C_{\sigma}}$ has the kernel
\begin{equation}\label{521+}
k_{\sigma}(x,y)=H(x+y) e^{4i\pi\sigma(\frac{x-y}{x+y})}\frac12\Bigl(\delta_{0}(y-x)+\frac{1}{i\pi}\text{pv}\frac1{y-x}\Bigr),
\end{equation}
a formula to be compared to \eqref{513iop}.
Using the Schwartz function $\phi_{0}$ of Corollary \ref{cor.55}, we get from the Lebesgue Dominated Convergence Theorem that  for $\sigma$ small enough,
\begin{equation}\label{}
\poscal{
\opw{\mathbf 1_{\mathcal C_{\sigma}}}
\phi_{0}}{\phi_{0}}_{L^{2}(\R)}=
\iint_{x\xi\ge \sigma, x>0} \mathcal W(\phi_{0}, \phi_{0})(x,\xi) dx d\xi >1.
\end{equation}
However, this argument does not work for large positive $\sigma$ and we must go back to a direct calculation.
 \subsubsection{Diagonal terms}
Denoting by $A_{\sigma}$ the operator with kernel \eqref{521+}
(and Weyl symbol $H(x\xi-\sigma) H(x)$),
we find that  for 
$u\in \mathscr S(\R)$, $u_{+}=Hu$,
we have 
\begin{align*}
\poscal{A_{\sigma}Hu}{Hu}_{L^{2}(\R)}
&=\iint
e^{4i\pi\sigma(\frac{x-y}{x+y})}\frac12\Bigl(\delta_{0}(y-x)+\frac{1}{i\pi}\text{pv}\frac1{y-x}\Bigr)
u_{+}(y)\bar u_{+}(x) dy dx
\\
&=\frac12\norm{u_{+}}^{2}_{L^{2}(\R_{+})}
+\iint_{\R^{2}}
e^{4i\pi\sigma(\frac{e^{s}-e^{t}}{e^{s}+e^{t}})}\frac{1}{2i\pi}\text{pv}\frac1{e^{t}-e^{s}}
u_{+}(e^{t})\bar u_{+}(e^{s}) e^{s+t}ds dt
\\&=
\frac12\norm{u_{+}}^{2}_{L^{2}(\R_{+})}
+\iint_{\R^{2}}
e^{4i\pi\sigma
\tanh(\frac{s-t}2)
}\frac{1}{2i\pi}\text{pv}\frac{e^{(s+t)/2}}{e^{t}-e^{s}}
\phi_{1}(t)\bar \phi_{1}(s)ds dt,
\end{align*}
with \begin{equation}\label{}
\phi_{1}(t)=u_{+}(e^{t}) e^{t/2},
\qquad\text{so that 
$
\norm{\phi_{1}}_{L^{2}(\R)}=\norm{u_{+}}_{L^{2}(\R_{+})}.
$}
\end{equation}
We get 
$$
\poscal{A_{\sigma}Hu}{Hu}_{L^{2}(\R)}
=
\frac12\norm{u}^{2}_{L^{2}(\R_{+})}
+\frac{1}{4i\pi}\iint_{\R^{2}}
\frac{e^{4i\pi\sigma
\tanh(\frac{s-t}2)
}}{\sinh(\frac{t-s}2)}
\phi(t)\bar \phi(s)ds dt,
$$
and noting that 
$
\sinh x=x C(x),
$
with $C$ even such that $1/C\in \mathscr S(\R)$,
we find 
\begin{align}
\poscal{A_{\sigma}Hu}{Hu}_{L^{2}(\R)}
&=
\frac12\norm{\phi_{1}}^{2}_{L^{2}(\R)}
-\frac{1}{2i\pi}\iint_{\R^{2}}
\frac{e^{4i\pi\sigma
\tanh(\frac{s-t}2)
}}{(s-t)C(\frac{s-t}2)}
\phi(t)\bar \phi(s)ds dt
\notag\\
&=\frac12\norm{\phi_{1}}^{2}_{L^{2}(\R)}
+\poscal{T_{\sigma}\ast \phi_{1}}{\phi_{1}}_{L^{2}(\R)}
\notag\\
&=\int_{\R} \val{\hat \phi_{1}(\tau)}^{2}\bigl(\frac12+\hat T_{\sigma}(\tau)\bigr) d\tau,
\label{524}
\end{align}
with 
\begin{equation}\label{525+}
T_{\sigma}(t)=
\frac{1}{4}\frac{t e^{4i\pi\sigma
\tanh(\frac{t}2)
}}{\sinh(t/2)}\text{pv}{\frac{i}{\pi t}}.
\end{equation}
We note that 
\begin{align}
&\hat T_{\sigma}(\tau)=\sign\ast \rho_{\sigma},\quad\text{with}
\label{526-+}
\\&\rho_{\sigma}(\tau)=\frac{1}{4}
\int\frac{t e^{4i\pi\sigma
\tanh(\frac{t}2)
}}{\sinh(t/2)} e^{-2i\pi t \tau} dt,\quad\rho_{\sigma}\in \mathscr S(\R),
\label{526-}
\end{align}
since the function $\R\ni t\mapsto \frac{te^{4i\pi\sigma
\tanh(\frac{t}2)
}}{\sinh(t/2)} $ belongs to the Schwartz space\footnote{Indeed, the iterated derivatives of $\tanh$ are  polynomials of $\tanh$ (check this by induction on the order of derivatives) and thus bounded on the real line; since the function $t\mapsto t/\sinh(t/2)$ belongs to the Schwartz space, this proves that the above product is in $\mathscr S(\R)$.}. 
Note also that the function $\rho_{\sigma}$ is real-valued on the real line.
This entails that 
\begin{equation}\label{526,5}
\frac{d}{d\tau}\Bigl\{\frac12+\hat T_{\sigma}\Bigr\}=2\rho_{\sigma},
\end{equation}
and since 
$$
\rho_{\sigma}(\tau)=\frac14
\mathcal F\Bigl\{t\mapsto \frac{t
e^{4i\pi\sigma \tanh(t/2)}
}{\sinh({t}/2)} \Bigr\},
\quad\text{implying }\quad \int_{\R} \rho_{\sigma}(\tau) d\tau=\frac12,
$$
we get that 
\begin{equation}\label{526}
\lim_{\tau\rightarrow\pm\io}\hat T_{\sigma}(\tau)=\pm \frac12.
\end{equation}
This yields that 
\begin{equation}\label{529}
\frac12+\hat T_{\sigma}(\tau)-1=\int_{+\io}^{\tau} 2\rho_{\sigma}(\tau') d\tau'=-1+
\int_{-\io}^{\tau} 2\rho_{\sigma}(\tau') d\tau',
\end{equation}
where the last equality follows from \eqref{526}: indeed we have for $\tau>0$, from \eqref{526,5},
\begin{equation}\label{6543}
\frac12+\hat T_{\sigma}(\tau)-1=\int_{+\io}^{\tau} 2\rho_{\sigma}(\tau') d\tau'=-1+
\int_{-\io}^{\tau} 2\rho_{\sigma}(\tau') d\tau',
\end{equation}
and for $\tau<0$,
$$
\frac12+\hat T_{\sigma}(\tau)=\int_{-\io}^{\tau}2 \rho_{\sigma}(\tau') d\tau'
=
1+
\int_{+\io}^{\tau}2 \rho_{\sigma}(\tau') d\tau'.
$$
We note that
\begin{equation}\label{5210}
\forall N\in \N, \quad
\sup_{\tau\in \R}\val \tau^{N}\Val{\frac12+\hat T_{\sigma}(\tau)-H(\tau)}<+\io.
\end{equation}
Indeed  for $\tau>0$, we have, using $\rho_{\sigma}\in \mathscr S(\R)$, 
$$
\Val{\tau^{N}\int_{+\io}^{\tau} \rho_{\sigma}(\tau') d\tau'}\le \int_{\tau}^{+\io}\val{\rho_{\sigma}(\tau')} {\tau'}^{N} d\tau'\le 
\int_{0}^{+\io}\val{\rho_{\sigma}(\tau')} {\tau'}^{N} d\tau<+\io.
$$
Also,  for $\tau<0$,
we have 
$$
\Val{\tau^{N}\int_{-\io}^{\tau} \rho_{\sigma}(\tau') d\tau'}\le \int_{-\io}^{\tau}\val{\rho_{\sigma}(\tau')} \val{\tau'}^{N} d\tau'\le 
\int_{-\io}^{0}\val{\rho_{\sigma}(\tau')} \val{\tau'}^{N} d\tau<+\io.
$$
This means that the Fourier multiplier $\frac12+\hat T_{\sigma}(\tau)$ is somehow ``exponentially close'' to $H(\tau)$ for large  values of $\val\tau$ and in particular close to 1 for large positive values of $\tau$.
We have also
\begin{multline}\label{5211}
\hat T_{\sigma}(\tau)=\frac i{4\pi}\int_{\R}e^{-2i\pi \tau t}
\frac{e^{4i\pi\sigma
\tanh(\frac{t}2)
}}{\sinh(t/2)}
dt
\\=
\frac 1{2\pi}\int_{0}^{+\io}
\frac{
\sin(2\pi t\tau-4\pi \sigma\tanh(t/2))
}{\sinh(t/2)}
dt.
\end{multline}
The next lemma  provides more precise estimates than \eqref{5210}.
\begin{lem}\label{lem.57} Let $\tau>0, \sigma\ge 0$.
 Defining $a_{11}(\tau, \sigma)=\frac12+\hat T_{\sigma}(\tau)$ as given by \eqref{525+}, we have 
 \begin{equation}\label{5214++}
\val{1-a_{11}(\tau, \sigma)}\le 
2e^{-{\pi^2\tau}} e^{4\pi \sigma}.
\end{equation}
\end{lem}
\begin{proof}
Using \eqref{6543} and Lemma \ref{lem.912}, we find that for $\tau>0$,
\begin{multline*}
\val{1-a_{11}(\tau, \sigma)}\le 2\int_{\tau}^{+\io}
\val{\rho_{\sigma}(\tau')} d\tau'
\le 2\int_{\tau}^{+\io}
\val{\rho_{\sigma}(\tau')} d\tau'
\\\le
12e^{4\pi \sigma}\int_{\tau}^{+\io}e^{-{\pi^2\tau'}} d\tau'
=e^{4\pi \sigma}\frac{12}{\pi^{2}}e^{-{\pi^2\tau}}, 
\end{multline*}
 entailing the sought result.
 \end{proof}
 \subsubsection{Off-diagonal terms}
We want now to check the off-diagonal terms:
we have with $u\in \mathscr S(\R)$, 
\begin{align}
&u_{+}=Hu,\ u_{-}=\check H u,
\\
&\phi_{1}(t)= u_{+}(e^{t}) e^{t/2}, \quad \phi_{2}(t)= u_{-}(-e^{t}) e^{t/2}
\end{align}
\begin{align}
&\poscal{A_{\sigma}\check Hu}{Hu}_{L^{2}(\R)}
\notag\\
&=\iint
e^{4i\pi\sigma(\frac{x-y}{x+y})}
\frac{H(x+y)\check H(y) H(x)}{2i\pi}\text{pv}\frac1{y-x}
u_{-}(y)\bar u_{+}(x) dy dx
\notag\\
&=\iint
e^{4i\pi\sigma(\frac{e^{s}+e^{t}}{e^{s}-e^{t}})}
\frac{H(e^{s}-e^{t})}{2i\pi}\text{pv}\frac1{-e^{t}-e^{s}}
\phi_{2}(t)\bar \phi_{1}(s) e^{\frac{t+s}2}dt ds
\notag\\
&=\iint
e^{4i\pi\sigma\text{coth}(\frac{s-t}2)}
\ \frac{iH(s-t)}{4\pi}\frac1{\cosh(\frac{t-s}2)}
\phi_{2}(t)\bar \phi_{1}(s) dt  ds
\notag\\
&=\frac{i}{4\pi}\iint
e^{4i\pi\sigma\text{coth}(\frac{s-t}2)}
\ 
{H(s-t)}\frac1{\cosh(\frac{s-t}2)}
\phi_{2}(t)\bar \phi_{1}(s) dt  ds
\notag\\
&=\poscal{S_{\sigma}\ast \phi_{2}}{\phi_{1}}_{L^{2}(\R)},
\label{524++}
\end{align}
with 
\begin{equation}\label{5213}
S_{\sigma}(t)=\frac{i}{4\pi}H(t)
\frac{e^{4i\pi\sigma\text{coth}(\frac{t}2)}}{\cosh(\frac{t}2)},
\end{equation}
and 
\begin{multline}\label{5215++}
\hat S_{\sigma}(\tau)=
\frac{i}{4\pi}
\int H(t)
\frac{e^{4i\pi\sigma\text{coth}(\frac{t}2)}}{\cosh(\frac{t}2)} e^{-2i\pi t\tau} dt 
\\
=
\frac{i}{4\pi}
\int_{0}^{+\io} 
\frac{\cos(4\pi \sigma \coth(t/2)-2\pi t\tau)}{\cosh(\frac{t}2)}  dt 
-
\frac{1}{4\pi}
\int_{0}^{+\io} 
\frac{\sin(4\pi \sigma \coth(t/2)-2\pi t\tau)}{\cosh(\frac{t}2)}  dt 
\\
=
\frac{i}{4\pi}
\int_{0}^{+\io} 
\frac{\cos(2\pi t\tau-4\pi \sigma \coth(t/2))}{\cosh(\frac{t}2)}  dt 
+
\frac{1}{4\pi}
\int_{0}^{+\io} 
\frac{\sin(2\pi t\tau-4\pi \sigma \coth(t/2))}{\cosh({t}/2)}  dt. 
\end{multline}
Note that from \eqref{525+}, \eqref{526-}, we have 
\begin{align*}
\hat T_{\sigma}(\tau)=
\frac{i}{4\pi}\int\frac{e^{4i\pi\sigma
\tanh(\frac{t}2)
}}{\sinh(t/2)}e^{-2i\pi t\tau} dt
=
\frac{1}{2\pi}\int_{0}^{+\io}\frac{
\sin(2\pi t\tau-4\pi \sigma\tanh(t/2))}{\sinh(t/2)}dt.
\end{align*}
\subsubsection{An isometric isomorphism}
\begin{theorem}\label{thm.calculs}
 Let $\sigma\ge 0$ be given, let $\mathcal C_{\sigma}$ be the set defined by \eqref{5014} and let $A_{\sigma}$ be the operator with Weyl symbol $\mathbf 1_{\mathcal C_{\sigma}}$,
 (whose kernel is given by \eqref{521+}). 
 The operator $A_{\sigma}$ is bounded self-adjoint  on $L^{2}(\R)$ so that we may define,
 with $\Psi$ defined in \eqref{5113oi},
 \begin{equation}\label{}
\wt A_{\sigma}=\Psi A_{\sigma}\Psi^{-1}.
\end{equation}
The operator $\wt A_{\sigma}$ is the Fourier multiplier on $L^{2}(\R;\C^2)$ given by the matrix
\begin{equation}\label{5219+}
\renewcommand\arraystretch{1.7}
\mathcal M_{\sigma}(\tau)=
\begin{pmatrix}
\frac12+\hat T_{\sigma}(\tau)&\hat S_{\sigma}(\tau)\\
\overline{\hat S_{\sigma}(\tau)}&0
\end{pmatrix},
\end{equation}
where $T_{\sigma}, S_{\sigma}$ are defined respectively in \eqref{525+},
\eqref{5211},
 \eqref{5213}.
 In particular we have
with $\Phi=(\phi_{1}, \phi_{2})\in L^{2}(\R;\C^{2})$,
\begin{equation}\label{5219}
\poscal{\wt A_{\sigma}\Phi}{\Phi}_{L^{2}(\R;\C^{2})}
=\int_{\R}
e^{2i\pi t\tau}
\poscal{\mathcal M_{\sigma}(\tau)\hat \Phi(\tau)}{\hat \Phi(\tau)}_{\C^{2}} d\tau.
\end{equation}
\end{theorem}
\begin{proof}
 We have 
 \begin{align*}
 \text{kernel}(HA_{\sigma}H)&=e^{4i\pi\sigma\frac{x-y}{x+y}} H(x)H(y)\hat H(y-x),\\
 \text{kernel}(\check H A_{\sigma}H+ HA_{\sigma}\check H)&=e^{4i\pi\sigma\frac{x-y}{x+y}} 
 H(x+y)\bigl(\check H(x)H(y)+H(x)\check H(y)\bigr)\frac{1}{2i\pi(y-x)},\\
 \check H A_{\sigma}\check H&=0.
 \end{align*}
 Proposition \ref{pro.hardy} in our Appendix is readily giving the $L^{2}$-boundedness (and self-adjointness) of $\check H A_{\sigma}H+ HA_{\sigma}\check H$. We find also that $HA_{\sigma}H-\frac{H}{2}$ has kernel
 $$
e^{4i\pi\sigma\frac{x-y}{x+y}} H(x)H(y)\frac{1}{2i\pi(y-x)},
 $$
 and thus it is enough to study the operator with kernel 
 $$
e^{4i\pi\sigma\frac{e^{s}-e^{t}}{e^{s}+e^{t}}} \frac{e^{\frac{s+t}2}}{2i\pi (e^{t}-e^{s})}=e^{4i\pi \sigma\tanh(\frac{s-t}2)}\frac{1}{4i\pi \sinh(\frac{t-s}2)},
 $$
 which is a convolution operator by
 $$
T_{\sigma}(t)=e^{4i\pi \sigma\tanh(\frac{t}2)}\frac{t}{4 \sinh(\frac{t}2)} \text{pv}\frac{i}{\pi t},
 $$
 given by \eqref{525+}. Formula \eqref{526-} implies  in particular that $\hat T_{\sigma}$ is bounded (and real-valued) on the real line, entailing eventually the boundedness and self-adjoint\-ness of $A_{\sigma}$. 
 Formulas \eqref{524}, \eqref{524++} and \eqref{5213} are providing \eqref{5219},
 completing the proof of the theorem.
\end{proof}
\subsubsection{The main result on hyperbolic regions}
\begin{theorem}\label{thm.59}
 Let $\sigma\ge 0$ be given and let $A_{\sigma}$ be the operator defined in Theorem \ref{thm.calculs}.
 Then $A_{\sigma}$ is a bounded self-adjoint operator on $L^{2}(\R)$ such that
 \begin{equation}\label{5314++}
\inf\bigl(\text{\rm spectrum}(A_{\sigma})\bigr)<0<1<\sup\bigl(\text{\rm spectrum}(A_{\sigma})\bigr).
\end{equation}
The spectrum of $A_{\sigma}$ is the closure of the set  of eigenvalues of $\mathcal M_{\sigma}(\tau)$ for $\tau$ running on the real line.
\end{theorem}
\begin{rem}\label{keyrem}
It is enough to prove that, with a given $\sigma\ge 0$,
there exists $\tau\in \R$ such that 
$\mathcal M_{\sigma}(\tau)$ satisfies  \eqref{cond00}.
\end{rem}
\begin{proof}
We have from \eqref{5219+}, \eqref{5211}, \eqref{5215++},
\begin{gather}\label{2525}
\renewcommand\arraystretch{2.5}
\mathcal M_{\sigma}(\tau)=
\begin{pmatrix}
\frac12+\frac 1{2\pi}\int_{0}^{+\io}
\frac{
\sin(2\pi t\tau-4\pi \sigma\tanh(t/2))
}{\sinh(t/2)}
dt&\cdot&\frac{i}{4\pi}
\int_{0}^{+\io}
\frac{e^{{-2i\pi}(t\tau-\frac{2\sigma}{\tanh(t/2)})}}{\cosh({t}/2)} dt\\
\frac{1}{4i\pi}
\int_{0}^{+\io}
\frac{e^{{2i\pi}(t\tau-\frac{2\sigma}{\tanh(t/2)})}}{\cosh({t}/2)} dt 
&\cdot&0
\end{pmatrix}
\\=
\begin{pmatrix}
 a_{11}(\tau,\sigma)& a_{12}(\tau,\sigma)\\
  a_{21}(\tau,\sigma)& a_{22}(\tau,\sigma)
\end{pmatrix}.
\notag
\end{gather}
On the other hand we have 
\begin{equation}\label{526526}
\overline{a_{12}}=a_{21}=\frac{1}{4i\pi}
\int_{0}^{+\io}
\frac{e^{{2i\pi}(t\tau-\frac{2\sigma}{\tanh(t/2)})}}{\cosh({t}/2)} dt,
\end{equation}
so that 
\begin{equation}\label{5229}
\re a_{12}(\tau, \sigma)=\frac1{4\pi}
\int_{0}^{+\io}
\frac{\sin[{{2\pi}(t\tau-\frac{2\sigma}{\tanh({\frac t2})})]}}{\cosh(\frac t2)} dt.
\end{equation}
We note that the function 
$$t\mapsto
\frac{e^{{2i\pi}(t\tau-\frac{2\sigma}{\tanh(t/2)})}}{\cosh({t}/2)},$$
is holomorphic on $\C\backslash i\pi\Z$, with
simple poles at
$(2\Z+1)i\pi$ (zeroes of $\cosh(t/2)$) and essential singularities at $2\Z i\pi$
 (zeroes of $\sinh(t/2)$).
 We shall need a more explicit quantitative  expression for $a_{21}$ to obtain a precise asymptotic result which could be compared to the estimate \eqref{5214++}.
 The next lemma is proven in \cite{MR2131219}; we provide a proof here for the convenience of the reader.
\begin{lem}\label{keylem++}
 Let $\tau>0, \sigma\ge 0$ be given and let $a_{21}(\tau, \sigma)$ be given by 
 \eqref{526526}. We have 
 \begin{align}\label{keyasy}
\re a_{21}(\tau, \sigma) 
&=\frac{e^{-2\pi^2 \tau}}{4\pi}
\biggl\{\int_{0}^\pi
\Bigl(\frac{e^{2\pi(t\tau-2\sigma\tan(t/2))}-1}{\sin(t/2)}+\frac{\sinh(t/2)-\sin(t/2)}{
\sinh(t/2)\sin(t/2)
}\Bigr)dt
\\
&\hskip75pt+\int_{0}^\pi\frac{1-\cos2\pi(t\tau-2\sigma\tanh(t/2))}{\sinh(t/2)}
dt
\notag
\\
&\hskip110pt-
\int_{\pi}^{+\io}\frac{\cos2\pi(t\tau-2\sigma\tanh(t/2))}{\sinh(t/2)}
dt
\biggr\}.
\notag
\end{align}
\end{lem}
\begin{proof}[Proof of Lemma \ref{keylem++}]
Let $0<\epsilon<\pi/2<\pi<R$ be given. We consider the closed path $\gamma_{\epsilon, R}$
of $\C\backslash i\pi \Z$ with $\text{index}_{\gamma_{\epsilon, R}}(i\pi\Z)\equiv 0$,
 \begin{multline}\label{2929}
\gamma_{\epsilon, R}=[\epsilon, R]\cup[R, R+i\pi]\cup[R+i\pi, \epsilon+i\pi]
\\
\cup\{i\pi+\epsilon e^{i\theta}\}_{0\ge \theta\ge -\pi/2}
\cup i[\pi-\epsilon, \epsilon]
\cup\{\epsilon e^{i\theta}\}_{\pi/2\ge \theta\ge0},
\end{multline}
and we have 
\begin{equation}\label{3030}
\oint_{\gamma_{\epsilon, R}}
\frac{e^{{2i\pi}(z\tau-\frac{2\sigma}{\tanh(z/2)})}}{\cosh({z}/2)} dz=0.
\end{equation}
We note as well that 
\begin{multline}\label{I22222}
I_{2}=\oint_{[R, R+i\pi]}
\frac{e^{{2i\pi}(z\tau-\frac{2\sigma}{\tanh(z/2)})}}{\cosh({z}/2)} dz=
i\int_{0}^\pi
\frac{e^{{2i\pi}((R+it)\tau-\frac{2\sigma}{\tanh(\frac{R+it}{2})})}}{\cosh(\frac{R+it}{2})} dt
\\
=i e^{2i\pi R\tau}\int_{0}^\pi
e^{-2\pi t\tau} e^{-4i\pi \sigma\frac{1+e^{-R-it}}{1-e^{-R-it}}}
\frac{2dt}{e^{\frac{R+it}{2}}(1+e^{-R-it})},
\end{multline}
so that
$$
\val{I_{2}}\le 2  e^{-R/2}\int_{0}^\pi
e^{4\pi \sigma\im\left(\frac{1+e^{-R-it}}{1-e^{-R-it}}\right)}
\frac{dt}{\val{1-e^{-R}}},
$$
and since 
$$
\im\left(\frac{1+e^{-R-it}}{1-e^{-R-it}}\right)
=
\im\frac{(1+e^{-R-it})(1-e^{-R+it})}{\val{1-e^{-R-it}}^2}
=
\frac{-2e^{-R}\sin t}{\val{1-e^{-R-it}}^2}\le 0,
$$
we get 
\begin{equation}\label{3232}
\val{I_{2}}\le e^{-R/2}  \frac{2\pi }{1-e^{-R}},
\quad\text{where $I_{2}$ is defined in \eqref{I22222}}.
\end{equation}
Let us now check\footnote{Let us note for future reference the standard formulas
\begin{equation}\label{formfo}
\cosh \bigl(\frac{i\pi}{2}+z\bigr)=i\sinh z,\quad
\sinh \bigl(\frac{i\pi}{2}+z\bigr)=i\cosh z, \quad \tanh
\bigl(\frac{i\pi}{2}+z\bigr)=\coth z.
\end{equation}
}
\begin{multline}\label{3434}
I_{4}=-\int_{-\pi/2}^0
\frac{e^{{2i\pi}((i\pi+\epsilon e^{i\theta})\tau-{2\sigma}{\coth(\frac{i\pi+\epsilon e^{i\theta}}{2})})}}{\cosh
\frac{i\pi+\epsilon e^{i\theta}}{2}
} i\epsilon e^{i\theta} d\theta
\\
=
-e^{-2\pi^2 \tau}\int_{-\pi/2}^0
\frac{e^{{2i\pi}\left(\epsilon e^{i\theta}\tau-{2\sigma}{\tanh(\frac{\epsilon e^{i\theta}}{2})}\right)}}{i\sinh
\frac{\epsilon e^{i\theta}}{2}
} i\epsilon e^{i\theta} d\theta,
\end{multline}
and since 
\begin{equation*}
\Val{\frac{e^{{2i\pi}\left(\epsilon e^{i\theta}\tau-{2\sigma}{\tanh(\frac{\epsilon e^{i\theta}}{2})}\right)}}{i\sinh
\frac{\epsilon e^{i\theta}}{2}
} i\epsilon e^{i\theta}}\le 2\max_{\val z\le \pi/2}\val{\frac{z}{\sinh z}}
e^{\pi^2 \tau}
e^{4\pi \sigma\sup_{\val z\le \pi/4}\left\vert\frac{\sinh z}{\cosh z}\right\vert
},
\end{equation*}
the Lebesgue Dominated Convergence Theorem gives
\begin{equation}\label{3535}
\lim_{\epsilon\rightarrow 0_{+}} I_{4}=-\pi e^{-2\pi^2 \tau}.
\end{equation}
Defining now 
\begin{equation}\label{3636}
I_{6}=
-\int_0^{\pi/2}
\frac{e^{{2i\pi}(\epsilon e^{i\theta}\tau-{2\sigma}{\coth(\frac{\epsilon e^{i\theta}}{2})})}}{\cosh
\frac{\epsilon e^{i\theta}}{2}
} i\epsilon e^{i\theta} d\theta,
\end{equation}
and noting that 
\begin{align*}
4\pi \sigma\im& \coth(
\frac{\epsilon e^{i\theta}}{2}
)=4\pi \sigma \im\frac{1+e^{-\epsilon e^{i\theta}}}{1-e^{-\epsilon e^{i\theta}}}
=4\pi \sigma \im\frac{(1+e^{-\epsilon e^{i\theta}})(
1-e^{-\epsilon e^{-i\theta}}
)}{\val{1-e^{-\epsilon e^{i\theta}}}^2}
\\
&=4\pi \sigma \im\frac{
e^{-\epsilon e^{i\theta}}
-e^{-\epsilon e^{-i\theta}}
}{\val{1-e^{-\epsilon e^{i\theta}}}^2}
=4\pi \sigma \im\frac{
e^{-\epsilon \cos \theta}
(e^{-i\epsilon\sin \theta}-e^{i\epsilon\sin \theta})
}{\val{1-e^{-\epsilon e^{i\theta}}}^2}
\\&=4\pi \sigma e^{-\epsilon \cos \theta}
\im \frac{
(-2i)\sin(\epsilon \sin \theta)
}{\val{1-e^{-\epsilon e^{i\theta}}}^2}
=-4\pi \sigma e^{-\epsilon \cos \theta}
\frac{
2\sin(\epsilon \sin \theta)
}{\val{1-e^{-\epsilon e^{i\theta}}}^2}\le 0,
\end{align*}
we get that 
$$
\val{I_{6}}\le \int_{0}^{\pi/2}\frac{e^{-2\pi \epsilon\tau\sin \theta}}{\min_{\val{z}\le \pi/4}\val{\cosh z}} d\theta \epsilon\le  \epsilon\frac{\pi/2}{\min_{\val{z}\le \pi/4}\val{\cosh z}},
$$
entailing
\begin{equation}\label{I66666}
\lim_{\epsilon\rightarrow 0_{+}} I_{6}=0.
\end{equation}
With 
\begin{equation}\label{3838}
I_{1}=\oint_{[\epsilon, R]}
\frac{e^{{2i\pi}(z\tau-\frac{2\sigma}{\tanh(z/2)})}}{\cosh({z}/2)} dz,
\end{equation}
we have from \eqref{526526} 
\begin{equation}\label{3939}
\lim_{\substack{\epsilon\rightarrow 0_{+}\\
R\rightarrow+\io}} I_{1}=4i\pi a_{21}.
\end{equation}
We define now
\begin{align*}
&I_{5}=-\oint_{[i\epsilon, i(\pi-\epsilon)]}
\frac{e^{{2i\pi}(z\tau-\frac{2\sigma}{\tanh(z/2)})}}{\cosh({z}/2)} dz
=-\int_{\epsilon}^{\pi-\epsilon}
\frac{e^{{2i\pi}(it\tau-\frac{2\sigma}{\tanh(it/2)})}}{\cosh({it}/2)} idt
\notag\\
&=-
\int_{\epsilon}^{\pi-\epsilon} e^{-2\pi t\tau}
\frac{e^{\frac{-4i\pi \sigma}{i\tan(t/2)}}}{\cos({t}/2)} idt
=-i \int_{\epsilon}^{\pi-\epsilon} e^{-2\pi t\tau}
\frac{e^{\frac{-4\pi \sigma}{\tan(t/2)}}}{\cos({t}/2)} dt
\notag\\&=-i \int_{\epsilon}^{\pi-\epsilon} e^{-2\pi (\pi-s)\tau}
\frac{e^{-\frac{4\pi \sigma}{\tan((\pi-s)/2)}}}{\cos({(\pi-s)}/2)} ds
=-
i e^{-2\pi^2\tau}\int_{\epsilon}^{\pi-\epsilon} e^{2\pi s\tau}
\frac{e^{-\frac{4\pi \sigma \sin(s/2)}{\cos(s/2)}}}{\sin({s}/2)} ds,
\end{align*}
so that 
\begin{align}\label{4040}
\notag\\I_{5}=&
-
i e^{-2\pi^2\tau}\int_{\epsilon}^{\pi-\epsilon} e^{2\pi s\tau}
\frac{e^{-{4\pi \sigma \tan(s/2)}{}}}{\sin({s}/2)} ds.
\end{align}
We have also 
\begin{equation}\label{I33333}
I_{3}=\oint_{[R+i\pi, \epsilon+i\pi]}
\frac{e^{{2i\pi}(z\tau-\frac{2\sigma}{\tanh(z/2)})}}{\cosh({z}/2)} dz
=-\int_{\epsilon}^R
\frac{e^{{2i\pi}((t+i\pi)\tau-\frac{2\sigma}{\tanh((t+i\pi)/2)})}}{\cosh({(t+i\pi)}/2)} dt,
\end{equation}
so that using Formulas \eqref{formfo}, we get
$$
I_{3}=-e^{-2\pi^2\tau}
\int_{\epsilon}^R
\frac{e^{{2i\pi}(t\tau-{2\sigma}{\tanh(t/2)})}}{i\sinh(t/2)} dt,
$$
and
\begin{multline}\label{4242}
I_{3}+I_{5}=i e^{-2\pi^2\tau}\left(
\int_{\epsilon}^R
\frac{e^{{2i\pi}(t\tau-{2\sigma}{\tanh(t/2)})}}{\sinh(t/2)} dt-
\int_{\epsilon}^{\pi-\epsilon} e^{2\pi t\tau}
\frac{e^{-{4\pi \sigma \tan(t/2)}{}}}{\sin({t}/2)} dt
\right)
\\=
i e^{-2\pi^2\tau}\bigg\{
\int_{\epsilon}^{\pi-\epsilon}\Bigl(
\frac{e^{{2i\pi}(t\tau-{2\sigma}{\tanh(t/2)})}}{\sinh(t/2)} -
\frac{e^{2\pi (t\tau
 -2 \sigma \tan(t/2)
 )}}{\sin({t}/2)}\Bigr) dt
\\
+\int_{\pi-\epsilon}^R
\frac{e^{{2i\pi}(t\tau-{2\sigma}{\tanh(t/2)})}}{\sinh(t/2)} dt
\biggr\}.
\end{multline}
From \eqref{3030}, \eqref{2929}, \eqref{I22222}, \eqref{3434},
\eqref{3636}, \eqref{3838}, \eqref{4040},
\eqref{I33333}, we find that
$$
I_{1}=-I_{2}-(I_{3}+I_{5})-I_{4}-I_{6},
$$
so that taking the  limit of both sides\footnote{$I_{1},I_{2}, I_{4}, I_{6}, I_{3}+I_{5}$ do have limits when $\epsilon\rightarrow0_{+}, R\rightarrow+\io$.}  when $\epsilon\rightarrow0_{+}, R\rightarrow+\io$
we get, thanks to 
 \eqref{3939},  \eqref{3232}, \eqref{4242},
  \eqref{3535}, \eqref{I66666},
  \begin{multline}\label{}
4i\pi a_{21}=
\\
-i e^{-2\pi^2\tau}\bigg\{
\int_{0}^{\pi}\Bigl(
\frac{e^{{2i\pi}(t\tau-{2\sigma}{\tanh(t/2)})}}{\sinh(t/2)} -
\frac{e^{2\pi (t\tau
 -2 \sigma \tan(t/2)
 )}}{\sin({t}/2)}\Bigr) dt
+\int_{\pi}^{+\io}
\frac{e^{{2i\pi}(t\tau-{2\sigma}{\tanh(t/2)})}}{\sinh(t/2)} dt
\biggr\}\\
+\pi e^{-2\pi^2 \tau},
\end{multline}
implying that 
  \begin{multline*}\label{}
a_{21}=
\\
\frac{e^{-2\pi^2\tau}}{4\pi} \bigg\{
\int_{0}^{\pi}\Bigl(
-\frac{e^{{2i\pi}(t\tau-{2\sigma}{\tanh(t/2)})}}{\sinh(t/2)} +
\frac{e^{2\pi (t\tau
 -2 \sigma \tan(t/2)
 )}}{\sin({t}/2)}\Bigr) dt
-\int_{\pi}^{+\io}
\frac{e^{{2i\pi}(t\tau-{2\sigma}{\tanh(t/2)})}}{\sinh(t/2)} dt
\biggr\}\\
-\frac{i}{4} e^{-2\pi^2 \tau}
\end{multline*}
that is 
\begin{multline}\label{4444}
a_{21}=
\frac{e^{-2\pi^2\tau}}{4\pi} 
\int_{0}^{\pi}\Bigl(
\frac{e^{2\pi (t\tau
 -2 \sigma \tan(t/2)
 )}}{\sin({t}/2)}
-\frac{\cos{{2\pi}(t\tau-{2\sigma}{\tanh(t/2)})}}{\sinh(t/2)} 
\Bigr) dt
\\
-\frac{e^{-2\pi^2\tau}}{4\pi} \int_{\pi}^{+\io}
\frac{\cos{{2\pi}(t\tau-{2\sigma}{\tanh(t/2)})}}{\sinh(t/2)} dt
\\
-i\frac{e^{-2\pi^2\tau}}{4\pi} 
\int_{0}^{\pi}
\frac{\sin{{2\pi}(t\tau-{2\sigma}{\tanh(t/2)})}}{\sinh(t/2)} dt 
-\frac{i}{4} e^{-2\pi^2 \tau}
\\
-i\frac{e^{-2\pi^2\tau}}{4\pi} \int_{\pi}^{+\io}
\frac{\sin{{2\pi}(t\tau-{2\sigma}{\tanh(t/2)})}}{\sinh(t/2)} dt,
\end{multline}
yielding
\begin{multline}
\re a_{21}=
\frac{e^{-2\pi^2\tau}}{4\pi} 
\int_{0}^{\pi}\Bigl(
\frac{e^{2\pi (t\tau
 -2 \sigma \tan(t/2)
 )}}{\sin({t}/2)}
-\frac{\cos{{2\pi}(t\tau-{2\sigma}{\tanh(t/2)})}}{\sinh(t/2)} 
\Bigr) dt
\\
-\frac{e^{-2\pi^2\tau}}{4\pi} \int_{\pi}^{+\io}
\frac{\cos{{2\pi}(t\tau-{2\sigma}{\tanh(t/2)})}}{\sinh(t/2)} dt,
\end{multline}
completing the proof of Lemma  \ref{keylem++}.
\end{proof}
\begin{rem}\rm
 Formula \eqref{4444} also yields
 \begin{multline*}
\im a_{12}=-\im a_{21}=\frac{e^{-2\pi^2\tau}}{4\pi} \Bigl\{
\int_{0}^{\pi}
\frac{\sin{{2\pi}(t\tau-{2\sigma}{\tanh(t/2)})}}{\sinh(t/2)} dt 
+\pi 
\\
+ \int_{\pi}^{+\io}
\frac{\sin{{2\pi}(t\tau-{2\sigma}{\tanh(t/2)})}}{\sinh(t/2)} dt\Bigr\},
\end{multline*}
and since from \eqref{2525}, we have 
\begin{equation}\label{}
a_{11}=
\frac12+\frac 1{2\pi}\int_{0}^{+\io}
\frac{
\sin(2\pi t\tau-4\pi \sigma\tanh(t/2))
}{\sinh(t/2)}
dt,
\end{equation}
this gives
\begin{equation}\label{zeze}
\im a_{12}=
\frac{e^{-2\pi^2\tau}}{4\pi}\bigl(2\pi(a_{11}-\frac12) +\pi\bigr)=
\frac{e^{-2\pi^2\tau}}{2} a_{11}.
\end{equation}
 \end{rem}
To complete the proof of Theorem \ref{thm.59},  it will be enough, according to Lemma
\ref{lem.510--}, to prove that, for $\tau\rightarrow+\io$,
$
\val{a_{12}}^2\gg 1-a_{11}.
$
To achieve that, we note from \eqref{zeze} that the imaginary part of $a_{12}$ is useless and we shall prove simply that 
$$
{(\re a_{12})}^2\gg 1-a_{11}.
$$ 
To get this we are going to use \eqref{5214++}  and a precise asymptotic behavior for 
${(\re a_{12})}^2$ displayed in the next lemma and issued from the explicit formula
\eqref{keyasy}.
\begin{lem}\label{lem.514}
 Let $\tau\ge 1, \sigma\ge 0$ be given and let $a_{21}(\tau, \sigma)$ be given by 
 \eqref{526526}. We have then 
 \begin{align}\label{final}
\re a_{21}(\tau,\sigma)\ge
\frac{e^{-8\pi\sqrt\tau
\sqrt\sigma}}{8\pi^{3}\tau}
- \frac{1}{2\pi}e^{-2\pi^{2}\tau}.
\end{align}
 \end{lem}
\begin{proof}[Proof of the lemma]
 Since for $t\ge 0$ we have
 $\sinh (t/2)-\sin(t/2)\ge 0$,
 we get from  \eqref{keyasy},
 \begin{multline}\label{}
\re a_{21}(\tau,\sigma)\ge
\frac{e^{-2\pi^2 \tau}}{4\pi}
\biggl\{\int_{0}^\pi
\frac{e^{2\pi(t\tau-2\sigma\tan(t/2))}-1}{\sin(t/2)}
dt
-
\int_{\pi}^{+\io}\frac{1}{\sinh(t/2)}
dt
\biggr\}
\\
=
\frac{e^{-2\pi^2 \tau}}{4\pi}
\int_{0}^\pi
\frac{e^{2\pi(t\tau-2\sigma\tan(t/2))}-1}{\sin(t/2)}
dt
-
\frac{e^{-2\pi^2 \tau}}{2\pi}\ln\bigl(\coth\frac \pi 4\bigr).
\end{multline}
Let us define 
\begin{equation}\label{5251}
\omega=2\pi \tau,\quad \kappa=2\pi \sigma,\quad \nu=\kappa^{1/2}\omega^{-1/2},
\quad
\phi_{\nu}(s)= s-\nu^2\tan s.
\end{equation}
We have 
$$
2\pi\bigl( t\tau-2\sigma\tan(t/2))=2\pi\tau\bigl(t-2\nu^2\tan(t/2)\bigr)=
4\pi\tau\bigl(\frac t2-\nu^2\tan\frac t2\bigr)=2\omega\phi_{\nu}(t/2).
$$
We have thus 
\begin{equation}\label{5252}
\re a_{21}(\tau,\sigma)\ge
\frac{e^{-\pi \omega}}{2\pi}
\int_{0}^{\pi/2}
\frac{e^{2\omega \phi_{\nu}(s)}-1}{\sin s}
ds
-
\frac{e^{-\pi \omega}}{2\pi}
\underbrace{\ln\bigl(\coth\frac \pi 4\bigr)}_{\approx 0.421908}.
\end{equation}
Defining
\begin{equation}\label{5254}
\psi_{\nu}(\omega)
=
\frac{e^{-\pi \omega}}{2\pi}
\int_{0}^{\pi/2}
\frac{e^{2\omega \phi_{\nu}(s)}-1}{\sin s}
ds,
\end{equation}
we can use \eqref{5251}, \eqref{5252} and \eqref{9521} to get
whenever  $\tau>0$,
 \begin{align*}
2\pi\re a_{21}(\tau,\sigma)\ge
\frac{e^{-8\pi\sqrt\tau
\sqrt\sigma}}{\pi^{2}\tau}
 \Bigl(\frac{1}{2}
-\frac{1}{4\tau}\Bigr)
- e^{-2\pi^{2}\tau}, 
\end{align*}
so that  for $\tau\ge 1$ we find 
\begin{align}
2\pi\re a_{21}(\tau,\sigma)\ge
\frac{e^{-8\pi\sqrt\tau
\sqrt\sigma}}{4\pi^{2}\tau}
- e^{-2\pi^{2}\tau},
\end{align}
yielding the lemma.
\end{proof}
We  eventually go back  to the proof of Theorem \ref{thm.59}: let $\sigma>0$ be given.
From Lemma \ref{lem.514}  and \eqref{5214++}, 
we have  for $\tau\ge 1$,
\begin{align*}
&\val{1-a_{11}(\tau, \sigma)}\le 
2e^{-{\pi^2\tau}} e^{4\pi \sigma},
\\
&\re a_{21}(\tau,\sigma)\ge
\frac{e^{-8\pi\sqrt\tau
\sqrt\sigma}}{8\pi^{3}\tau}
- \frac{1}{2\pi}e^{-2\pi^{2}\tau}
=
\frac{e^{-8\pi\sqrt\tau
\sqrt\sigma}}{8\pi^{3}\tau}\Bigl(1-\frac{4\pi^{2}\tau e^{8\pi\sqrt\tau
\sqrt\sigma}}{e^{2\pi^{2}\tau}}\Bigr).
\end{align*}
This entails that  for $\tau\ge \tau_{0}(\sigma)$, we have 
\begin{equation}\label{boua21}
\re a_{21}(\tau,\sigma)\ge\frac{e^{-8\pi\sqrt\tau
\sqrt\sigma}}{16\pi^{3}\tau}, 
\end{equation}
and thus  $a_{21}\not=0$ and 
\begin{equation}\label{tryye}
\val{a_{21}(\sigma,\tau)}^{2}\ge \frac{e^{-16\pi\sqrt\tau
\sqrt\sigma}}{2^{8}\pi^{6}\tau^{2}}>\val{1-a_{11}(\tau, \sigma)},
\end{equation}
where the last inequality above  holds true (thanks to \eqref{5214++}) whenever
$$
2e^{-{\pi^2\tau}} e^{4\pi \sigma}<\frac{e^{-16\pi\sqrt\tau
\sqrt\sigma}}{2^{8}\pi^{6}\tau^{2}},
$$
which is indeed true for $\tau\ge \tau_{1}(\sigma)$.
As a result for $\tau\ge \max(4\sigma, 4, \tau_{0}(\sigma), \tau_{1}(\sigma))$,
we obtain that \eqref{tryye} is satisfied so that Remark \ref{keyrem}
implies the result of Theorem \ref{thm.59}, completing our proof.
 \end{proof}
\begin{rem}\rm
 The functions $\tau_{0}(\sigma), \tau_{1}(\sigma)$ can be determined rather easily, the first one by the condition
 $$
 \tau\ge \tau_{0}(\sigma)\Longrightarrow\frac{4\pi^{2}\tau e^{8\pi\sqrt\tau
\sqrt\sigma}}{e^{2\pi^{2}\tau}}\le \frac12,
 $$
 whereas the second one should satisfy
 $$
 \tau\ge \tau_{1}(\sigma)\Longrightarrow
 e^{4\pi \sigma}
 2^{9}\pi^{6}\tau^{2}
 e^{16\pi\sqrt\tau
\sqrt\sigma}
 <
e^{{\pi^2\tau}}.
 $$
\end{rem}
\subsection{Comments and further results}
\subsubsection{Qualitative explanations on the various computations}\label{sec531}
We would like to go back to our proofs that 
\begin{equation}\label{master11}
\val{a_{12}(\tau,\sigma)}^2\gg \val{1-a_{11}(\tau,\sigma)}, \quad \tau\rightarrow+\io,
\end{equation}
which is our key argument via Lemma  \ref{lem.510--} and give a couple of qualitative explanations which may enlighten the calculations. It is of course much simpler to begin with the case $\sigma=0$: in that case, according to Proposition \ref{pro.510uhc} and \eqref{5116gg},
 we have 
\begin{equation}\label{check1}
1-a_{11}(\tau,0)=\int_{\tau}^{+\io}2\rho_{0}(\tau') d\tau,
\quad
2\rho_{0}(\tau)=
\int
\underbrace{\left(\frac{t/2 }{\sinh (t/2)}\right)}_{\substack{=f_{0}(t),\
\text{$f_{0}\in \mathscr S(\R)$}
\\\text{holomorphic }
\\\text{on $\val{ \im t}<2\pi.$}
}}e^{-2i\pi t \tau} ds,\end{equation}
so that 
$
2\rho_{0}(\tau)=\hat f_{0} (\tau).
$
We get thus readily that $\rho_{0}$ belongs to the Schwartz space, as the Fourier transform
of a function in the Schwartz space and this implies  in particular that $1-a_{11}(\tau,0)$ has fast decay towards 0 when $\tau\rightarrow+\io$, as proven in Proposition \ref{pro.510uhc}.
We note also that \eqref{zeze}
gives
$
\im a_{12}(\tau,0)^{2}=
{e^{-4\pi^2\tau}} a_{11}(\tau,0)^{2}/4,
$
and since the limit of $a_{11}$ is 1, we do not expect any help from the imaginary part of $a_{12}$ to proving \eqref{master11}.
Turning our attention to $\re a_{12}$ in \eqref{5128hh},
we have, 
\begin{equation}\label{sigma0}
4\pi\re a_{21}(\tau,0)=
\int_{0}^{+\io}
\frac{\sin{(2\pi t\tau)}}
{\cosh({t}/2)} dt,
\end{equation}
which is the sine-Fourier transform of the function $t\mapsto H(t)\sech(t/2)=g_{0}(t)$, which has a singularity at $t=0$: as a consequence, thanks to Lemma \ref{lem.9110}, the Fourier transform 
$\widehat{g_{0}}$ cannot be rapidly decreasing, cannot even belong to  $L^{1}(\R)$ (that would imply that $g_{0}$ is continuous). Moreover the sine-Fourier transform above is the Fourier transform of the odd part of $g_{0}$, $g_{\text{odd}}(t)=\sech(t/2) \sign t $, which is also singular at 0,
thus $\widehat{g_{\text{odd}}}$ cannot be rapidly decreasing and is an odd function,
which is enough to prove, without more calculations, that \eqref{master11} holds true.
In Section \ref{sec.51}, we used a more explicit argument, with providing an equivalent   of \eqref{sigma0} equal to $1/(2\pi \tau)$ near $+\io$.
Summing-up, \eqref{master11} in the case $\sigma=0$
follows from the existence of a singularity of the function $g_{0}$ above, which is discontinuous at 0.
\par
Let us now take a look at the case $\sigma>0$, which turns out to be more
computationally involved.
We have from \eqref{526526}
\begin{align}\label{}
4\pi i{a_{21}}(\tau, \sigma)
&=
\int_{\R}H(t) \sech(t/2) e^{-i4\pi \sigma\coth(t/2)}
e^{2i\pi t\tau}
dt=\check{\widehat{g_{\sigma}}}(\tau),
\\
g_{\sigma}(t)&=H(t) \sech(t/2) e^{-i4\pi \sigma\coth(t/2)}.
\end{align}
The single discontinuity at $t=0$ of $g_{\sigma}$ when $\sigma>0$ is much wilder than for $\sigma=0$: in the latter case, we had only a jump discontinuity with different limits on both sides,
whereas when $\sigma>0$, we have an essential discontinuity with an oscillatory behaviour in $(-1,+1)$ when $t\rightarrow 0_{+}$ for the real and imaginary parts of $a_{12}$.
However,  $g_{\sigma}$ belongs to all $L^{p}(\R), p\in[1,+\io]$, so that its Fourier transform
belongs to $L^{p}(\R), p\in[2,+\io]$:
we expect then that both sides of \eqref{master11} have limit $0$ for $\tau\rightarrow +\io$
and we must prove that 
$1-a_{11}$ decays much faster  than $a_{12}$.
Looking at a slightly simplified model and using the notations \eqref{5251},
we define for $\omega, \nu$ positive, a function $\alpha$ presumably close to $4\pi i a_{21}$,
given by 
\begin{equation}
\alpha(\omega, \nu)=\int_{0}^{+\io}
e^{i2\omega\mu_{\nu}(s)} \sech(s)  ds, \quad \mu_{\nu}(s)=s-\frac{ \nu^{2}} s, \quad \mu'_{\nu}(s)=1+\frac{\nu^{2}}{s^{2}}.
\end{equation}
Trying our hand with  the stationary phase method, we look at 
\begin{multline*}
\alpha(\omega, \nu)=\frac{1}{2i\omega}\int_{0}^{+\io}
\frac{d}{ds}\left\{e^{i2\omega\mu_{\nu}(s)}\right\} \frac{\sech(s) }{\mu'_{\nu}(s)} ds
\\
=\frac{1}{2i \omega}\int_{0}^{+\io}
\frac{d}{ds}\left\{e^{i2\omega\mu_{\nu}(s)}\right\} \frac{s^{2}\sech(s) }{s^{2}+\nu^{2}} ds
\\=\frac{i}{2\omega}\int_{0}^{+\io}
e^{i2\omega\mu_{\nu}(s)}
\frac{d}{ds}\left\{\frac{s^{2}\sech(s) }{s^{2}+\nu^{2}} 
\right\},
\end{multline*}
since the boundary term vanishes.
Iterating that computation shows that $\alpha(\omega,\nu)= O_{\sigma}(\omega^{-N})$ for all $N$
when $\omega\rightarrow+\io$,
meaning that the information of fast decay for $1-a_{11}$ will not suffice to get \eqref{master11}.
Also, it is worth noticing  that no fast decay of the function $\alpha$
occurs when $\omega\rightarrow-\io$, otherwise Lemma \ref{lem.9110} would give smoothness for
the function $s\mapsto e^{-2i\kappa/s}H(s)\sech s$:
in fact we see also that for $\sigma>0$, $\tau=-\lambda$, $\lambda>0$, we have 
$$
2\pi i{a_{21}}(-\lambda, \sigma)
=
\int_{0}^{+\io} \sech(s) e^{-i4\pi \sigma\coth(s)}
e^{-4i\pi s\lambda}
ds,
$$
and the phase function is
$
\tilde \mu(s)=-4i\pi(s\lambda+\sigma\coth(s))
$
and we have 
$$
\frac{d}{ds}\bigl\{s\lambda+\sigma\coth(s)\bigr\}=\lambda-\frac{\sigma(1-\tanh^{2}s)}{\tanh^{2}s}
=\frac{(\lambda+\sigma) \tanh^{2}s-\sigma}{\tanh^{2}s},
$$
which does vanish at $\tanh s=\sigma/(\lambda+\sigma)$.
As a result we could say that, for $\sigma>0$,
 the $\moo$
 wave-front-set (see e.g. Section 8.1 in \cite{MR1996773})
 of the function $g_{\sigma}$ is reduced to $\{0\}\times (-\io,0)$.
 It turns out that we can show that the Gevrey-2 wave-front-set of $g_{\sigma}$
 is
 $\{0\}\times\R^{*}$,
 and it is expressed via the lowerbound estimate \eqref{final};
 the route that we took for proving this was an explicit calculation of $\re a_{12}$,
following the paper \cite{MR2131219}.
Finally the upper bound \eqref{5214++} can be improved as
\begin{equation}\label{5214++++}
\val{1-a_{11}(\tau, \sigma)}\le 
C_{\sigma,\epsilon}e^{-{
(\pi-\epsilon)2\pi\tau}}, \quad \epsilon>0, 
\end{equation}
and is expressing the fact the the function
$t\mapsto \frac{te^{4i\pi\sigma
\tanh(\frac{t}2)
}}{\sinh(t/2)} $ is analytic on the real line,
with a radius of convergence on the real line  bounded below by $\pi$
(cf. Proposition \ref{pro.92}).
 \subsubsection{More results and examples: $\ell^{p}$ balls, corners}\label{sec.532zae}
 \index{corners}
For $a,\phi_{0}$ like in  Corollary \ref{cor.55},
defining
$$
\Omega_{p}=\{(x,\xi)\in \R^{2}, \val{x-\frac a2}^{p}+\val{\xi-\frac a2}^{p}<\bigl(\frac a 2\bigr)^{p}\},
$$
since $\mathcal W(\phi_{0},\phi_{0})\in \mathscr S(\R^{2})$, we get  
$$
\lim_{p\rightarrow+\io}\iint_{\Omega_{p}}  \mathcal W(\phi_{0},\phi_{0})(x,\xi) dx d\xi
=\iint_{[0,a]^{2}} \mathcal W(\phi_{0},\phi_{0})(x,\xi) dx d\xi>\norm{\phi_{0}}_{L^{2}(\R)}^{2},
$$
proving that the spectrum of $\opw{\indic{\Omega_{p}}}$ intersects $(1,+\io)$ for $p$ large enough,
showing that a counterexample to Flandrin's conjecture can be a convex analytic open bounded set.
Moreover, defining
$$
Q_{a}=
\{(x,\xi) \in \R^{2}, \val x+\val \xi\le a/\sqrt 2\},
$$
we note that $Q_{a}$ is obtained  by rotation and translation of $[0,a]^{2}$
so that 
we can find $\phi_{1}$ in the Schwartz space such that 
$$
\iint_{Q_{a}} \mathcal W(\phi_{1},\phi_{1})(x,\xi) dx d\xi>\norm{\phi_{1}}_{L^{2}(\R)}^{2}.
$$
Since we have  
$$
\lim_{p\rightarrow 1}\hskip-13pt\iint_{\substack{\val x^{p}+\val \xi^{p}\le (a/\sqrt 2)^{p}}}
\hskip-12pt \mathcal W(\phi_{1},\phi_{1})(x,\xi) dx d\xi=
 \iint_{Q_{a}} \mathcal W(\phi_{1},\phi_{1})(x,\xi) dx d\xi>\norm{\phi_{1}}_{L^{2}(\R)}^{2},
$$
we get that for $p-1$ small enough we have
\begin{equation}
\iint_{\substack{\val x^{p}+\val \xi^{p}\le (a/\sqrt 2)^{p}}}
\hskip-12pt \mathcal W(\phi_{1},\phi_{1})(x,\xi) dx d\xi>\norm{\phi_{1}}_{L^{2}(\R)}^{2},
\end{equation}
proving that $\ell^{p}$ balls are counterexamples to Flandrin's conjecture for $p-1$ or $1/p$ small enough.
\vs
\index{convex cones}
{\it Convex affine cones with aperture strictly less than $\pi$ of $\R^{2}$} are translations and rotations of
\begin{equation}\label{539iua}
\Sigma_{\theta_{0}}=\{(x,\xi)\in \R^{2}\backslash (\R_{-}\times \{0\}),\ \arg (x+i\xi)\in (0,\theta_{0}) \}, \quad \text{for some $\theta_{0}\in (0,\pi)$.}
\end{equation}
The vertex of $\Sigma_{\theta_{0}}$ and its rotations is defined as  0
 and the vertex of the translation of vector $T_{0}$ of
 $\Sigma_{\theta_{0}}$ is defined as $T_{0}$. 
 We note that all
convex affine cones with aperture strictly less than $\pi$  are symplectically equivalent in $\R^{2}$,
since $\Sigma_{\theta_{0}}$ is symplectically equivalent to (the interior of)
the quarter plane $\Sigma_{\pi/2}$:
indeed let $\theta_{0}$ be in $(0,\pi)$; the symplectic matrix $M_{\theta_{0}}$ defined by
$$
M_{\theta_{0}}=\mat22{1}{-\cotan \theta_{0}}{0}{1},
$$
is such that 
$
M_{\theta_{0}}\matdu{1}{0}=\matdu{1}{0}, \quad 
M_{\theta_{0}}\matdu{\cos \theta_{0}}{\sin \theta_{0}}=\matdu{0}{\sin \theta_{0}},
$
proving that  $$M_{\theta_{0}}\Sigma_{\theta_{0}}=\Sigma_{\pi/2}.$$
The next  result follows from Theorem 1.3 in \cite{DDL}
and
shows that many counterexamples to 
 Flandrin's conjecture can be be obtained.
\begin{theorem}\label{thm.517ytu}
 Let $K$ be a subset of the closure of a convex affine cone with aperture strictly less than $\pi$ 
 and vertex $X_{0}$ such that $K$ contains a neighborhood of the vertex in the cone\footnote{We shall say that the set $K$ has a corner.}. Then there exists
 $\lambda>0$ such that, with 
 $$
 K_{\lambda}=X_{0}+\lambda(K-X_{0}),
 $$
 there exists $\phi\in \mathscr S(\R)$ such that
 \begin{align}\label{539}
\iint_{K_{\lambda}} \mathcal W(\phi, \phi)(x,\xi) dx d\xi>\norm{\phi}_{L^{2}(\R)}^{2}.
\end{align}
\end{theorem}
\begin{nb}
 Note that \eqref{539} implies that $\phi$ is not the zero function. Also, taking $K$ convex produces another counterexample to Flandrin's conjecture since $K_{\lambda}$ will be then convex, but
 we do not need that assumption to proving the result.
\end{nb}
\begin{proof}
 There is no loss of generality at assuming $X_{0}=0$ and  $$
[0,\rho_{0}]^{2}\subset K\subset\overline{\Sigma}_{\pi/2}, \quad \rho_{0}>0.
 $$
 Using Corollary \ref{cor.55}, we find $\phi_{0}\in \mathscr S(\R)$ (so that 
 $\mathcal W(\phi_{0},\phi_{0})\in \mathscr S(\R^{2})$) such that
 $$\lim_{\lambda\rightarrow+\io}
 \iint_{K_{\lambda}}\mathcal W(\phi_{0},\phi_{0})(x,\xi) dx d\xi=
 \iint_{{\Sigma}_{\pi/2}}\mathcal W(\phi_{0},\phi_{0})(x,\xi) dx d\xi>\norm{\phi_{0}}_{L^{2}(\R)}^{2},
 $$
 implying for $\lambda$ large enough that 
 $
 \iint_{K_{\lambda}}\mathcal W(\phi_{0},\phi_{0})(x,\xi) dx d\xi>\norm{\phi_{0}}_{L^{2}(\R)}^{2},
 $
 which is the sought result.
\end{proof}
\subsection{Numerics}
\begin{defi}\label{def.518}
Let $\sigma\ge 0$ be given.
 With the $2\times 2$ Hermitian matrix $\mathcal M_{\sigma}$ given by \eqref{2525},
 we define for $\tau\in \R$, 
 \begin{align}
\lambda_{+}(\tau, \sigma)
&=\frac12\left(a_{11}(\tau, \sigma)\!+\!\sqrt{a_{11}^{2}(\tau,\sigma)
+4\val{a_{12}(\tau,\sigma)}^{2}
}\right),
\label{est001-}\\
\lambda_{-}(\tau, \sigma)
&=\frac12\left(a_{11}(\tau, \sigma)\!-\!\sqrt{a_{11}^{2}(\tau,\sigma)
+4\val{a_{12}(\tau,\sigma)}^{2}
}\right).\label{est002-}
\end{align}
 \end{defi}
 \begin{rem}According to \eqref{zeze}, we have
\begin{align}
\lambda_{+}(\tau, \sigma)
&=\frac12\left(a_{11}(\tau, \sigma)\!+\!\sqrt{a_{11}^{2}(\tau,\sigma)
\bigl(1+e^{-4\pi^{2}\tau}\bigr)
+4{ (\re a_{12}(\tau,\sigma)})^{2}
}\right),
\label{est001++}\\
\lambda_{-}(\tau, \sigma)
&=\frac12\left(a_{11}(\tau, \sigma)-\sqrt{a_{11}^{2}(\tau,\sigma)
\bigl(1+e^{-4\pi^{2}\tau}\bigr)
+4{ (\re a_{12}(\tau,\sigma)})^{2}
}\right),\label{est002++}
\end{align} 
so that the knowledge of $a_{11}$ and $\re{a_{12}}$ suffices for  expressing  $\lambda_{\pm}$.
\end{rem}
An immediate consequence of Theorem \ref{thm.59} is 
\begin{theorem}
Let $\sigma\ge 0$ be given and let $A_{\sigma}$ be the self-adjoint operator bounded in $L^{2}(\R)$ defined in Theorem \ref{thm.59}. With the notations of Definition \ref{def.518},
we have 
\begin{align}
M_{\sigma}:=\sup\{\text{\rm spectrum} (A_{\sigma})\}
&=\sup_{\tau\in \R}
\lambda_{+}(\tau, \sigma),
\label{est001}\\
m_{\sigma}:=\inf\{\text{\rm spectrum} (A_{\sigma})\}
&=\inf_{\tau\in \R}
\lambda_{-}(\tau, \sigma).
\label{est002}
\end{align}
Moreover for all $\sigma\ge 0$ we have 
\begin{equation}
m_{\sigma}<0<1<M_{\sigma}.
\end{equation}

\end{theorem}
\subsubsection{The quarter-plane: \texorpdfstring{$\sigma=0$}{sigma=0}}
Of course, as shown by the respective calculations of Sections \ref{sec.51} and \ref{sec.52}, the  case $\sigma=0$, dealing with the quarter-plane 
is much simpler than the cases where $\sigma>0$. Nonetheless we know explicitly a spectral decomposition of the operator with Weyl symbol
$H(x) H(\xi)$ from Theorem \ref{thm.calculs},
but we can calculate without difficulty
numerical expressions of $M_{0},m_{0}$ as defined in \eqref{est001}, \eqref{est002}.
\begin{pro}\label{pro.521kjs}We have  from \eqref{9536}, \eqref{5229},
$$
 a_{11}(\tau,0)
 =\frac{1}{1+e^{-4\pi^{2}\tau}}, \quad 
 \re a_{12}(\tau,0)=\frac1{4\pi}
\int_{0}^{+\io}
\sin(2\pi t\tau)
\sech(t/2) dt,
$$
and we can use these formulas and \eqref{est001++}, \eqref{est002++},
\eqref{est001}, \eqref{est002} to calculate numerically 
\begin{align}
{\color{red} M_{0}}&{\color{red}\approx \mathbf {  1.00767997007003}}, \quad (\text{$\lambda_{+}(\tau,0)$ at  $\tau\approx 0.138815397930141$}),
\label{m05dfte}\\
m_{0}&\approx -0.155939843191243,\quad (\text{$\lambda_{-}(\tau,0)$ at  $\tau\approx -0.0566304954736227).$}
\label{m05df++}\end{align}
\end{pro}
\vs\vs
\begin{figure}[ht]
\centering
\scalebox{0.5}{\includegraphics{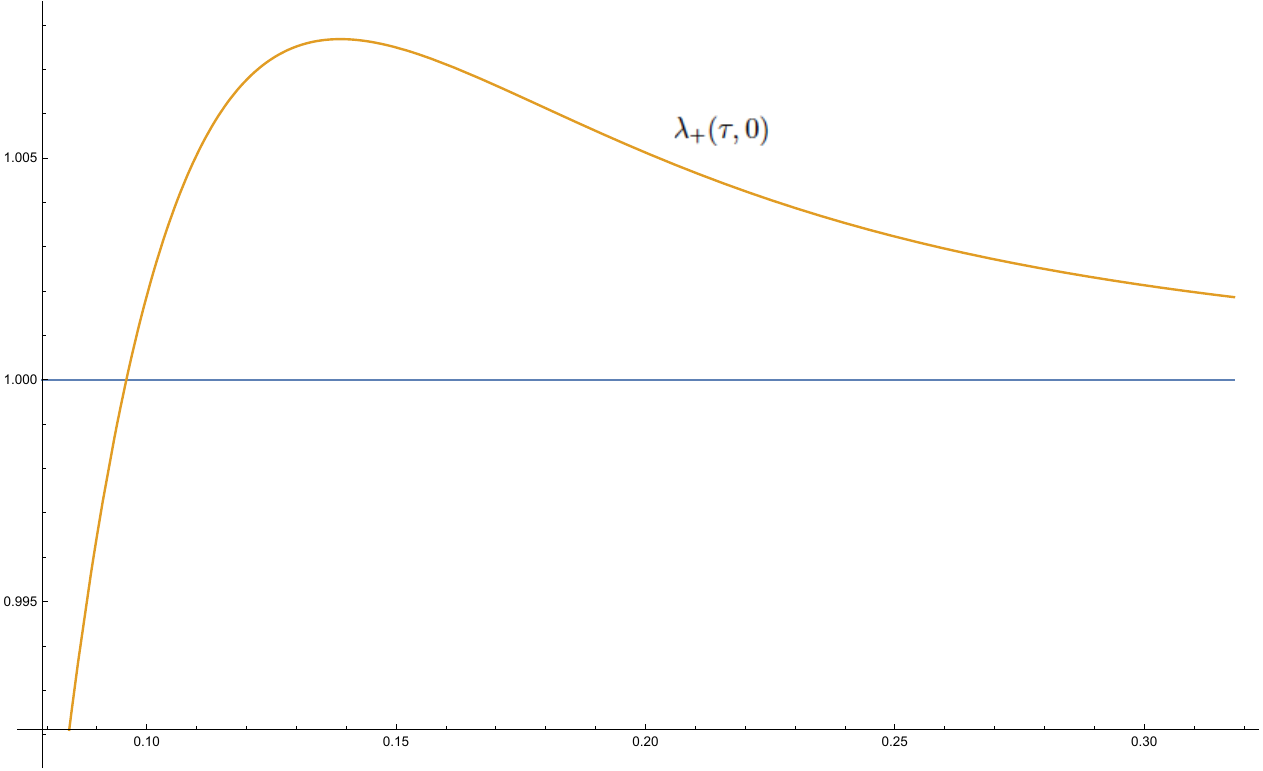}}\par
\caption{
The function $\tau\mapsto \lambda_{+}(\tau, 0)$ near its maximum, well above 1.
}
\label{pic6}
\end{figure}
\begin{figure}[ht]
\centering
\scalebox{0.5}{\includegraphics{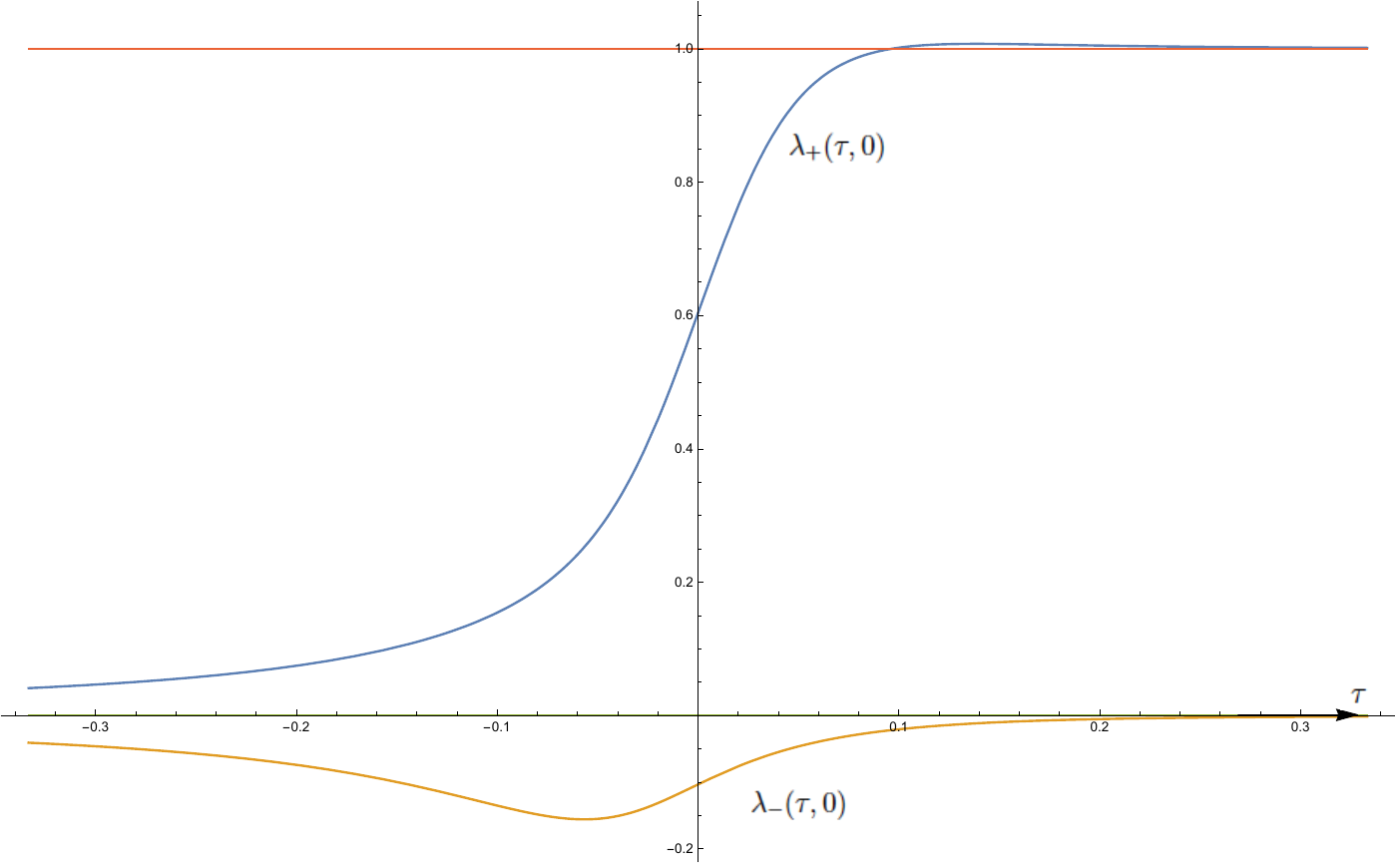}}\par
\caption{
The functions $\tau\mapsto \lambda_{+}(\tau, 0), \lambda_{-}(\tau,0)$.
}
\label{pic7}
\end{figure}
\subsubsection{On hyperbolic regions}
We want now to tackle the case $\sigma>0$.
In order to use the expressions 
\eqref{9536}, \eqref{keyasy} 
respectively for $a_{11}$ and $a_{12}$,
we need first to evaluate  the residue term in \eqref{9536}.
The mapping $z\mapsto \tanh z$ is a biholomorphism of neighborhoods of $0$ in the complex plane, so that  
we have for $z$ near the origin,
\begin{align}
&\zeta=\tanh z,\  d\zeta=(1-\zeta^{2}) dz,\quad
 z=\arcth \zeta=\frac12\ln\left(\frac{1+\zeta}{1-\zeta}\right),
 \\
& \frac{e^{2i\omega z -2i\kappa\coth z}}{\cosh z} dz
=\left(\frac{1+\zeta}{1-\zeta}\right)^{i\omega} e^{-2i\frac\kappa\zeta}
\ \frac{2}{\left(\frac{1+\zeta}{1-\zeta}\right)^{1/2}+\left(\frac{1-\zeta}{1+\zeta}\right)^{1/2}}
\frac{d\zeta}{(1-\zeta^{2})}
\notag\\
&\hskip88pt
=(1+\zeta)^{-\frac12+i\omega}(1-\zeta)^{-\frac12-i\omega}e^{-2i\frac\kappa\zeta}
d\zeta,
\end{align}
so that 
\begin{equation}\label{resnum}
\res{ \frac{e^{2i\omega z -2i\kappa\coth z}}{\cosh z} ,0}
=\res{(1+\zeta)^{-\frac12+i\omega}(1-\zeta)^{-\frac12-i\omega}e^{-2i\frac\kappa\zeta},0}.
\end{equation}
\vs
\begin{pro}
Let $\sigma\ge 0$ be given.
Then for any $\tau\in \R$, 
using  the no\-ta\-tions, 
$\omega=2\pi \tau$, $\kappa=2\pi \sigma$,
we have, 
 for any $\rho\in (0,1)$, 
 \begin{multline}\label{num111}
 a_{11}(\tau, \sigma)
 =\frac{1}{1+e^{-2\pi \omega}}
\\
+\frac{e^{-\pi \omega}}{1+e^{-2\pi \omega}}
\frac{\rho}{2\pi}
\im\left\{
\int_{-\pi}^{\pi}
\exp\left(i\omega\Lg\Bigl(\frac{1+\rho e^{i\theta}}{1-\rho e^{i\theta}}\Bigr)\right)
\frac{e^{-\frac{2i\kappa e^{-i\theta}}{\rho}}  e^{i\theta}}{\sqrt{1-\rho^{2} e^{2i\theta}}}
\  d\theta
\right\}.
\end{multline}
\begin{multline}\label{num222}
\re a_{21}(\tau, \sigma) 
=\frac{e^{-\pi \omega}}{2\pi}
\biggl\{2\int_{0}^{\pi/2}
\frac{e^{(s\omega-\kappa\tan s )}\sinh(s\omega-\kappa\tan s )}{\sin s}
ds
\\
+
\ln\bigl(\coth\frac \pi 4\bigr)
+2\int_{0}^{\pi/2}\frac{\sin^{2}(s\omega-\kappa\tanh s)}{\sinh s}
ds
-
\int_{\pi/2}^{+\io}\frac{\cos2(s\omega-\kappa\tanh s)}{\sinh s}
ds
\biggr\},
\end{multline}
\begin{align}
\im a_{12}(\tau,\sigma)&=
\frac{e^{-\pi \omega}}{2} a_{11}(\tau,\sigma).\hskip33pt&\hskip66pt~
\label{rappel}
\end{align}
\end{pro}
\begin{proof}
 Formula \eqref{num111} follows from \eqref{resnum} and \eqref{9536}
 whereas \eqref{num222}
 is \eqref{keyasy} after  a change of variable $t=2s$,
  where the second integral term inside the brackets is evaluated (cf. Lemma \ref{lem.913}); Formula
 \eqref{rappel} is a reminder of \eqref{zeze}.
\end{proof}
\begin{nb}\rm
 Our choice for $\rho$ in the numerical calculations of \eqref{num111}  is $\rho=3/4$,
 which is a good compromise between using a value of $\rho$ clearly away from $1$ (to avoid singularities coming from small denominators in the $\Lg$ term) and minimize the oscillations  and size
 coming from the term $\exp(-2i \kappa \rho^{-1}e^{-i\theta})$; note that the modulus of the latter is
 $$
 \exp(-2 \kappa \rho^{-1}\sin \theta),
 $$
 which is a smooth function of $\rho$ (flat at 0) when $\theta\in [0,\pi]$, but is unbounded for $\rho\rightarrow 0_{+}$  when $\theta\in (-\pi,0)$. There is no surprise here since although the residue does not depend on the choice of $\rho\in (0,1)$, we cannot get the value of that residue by letting $\rho$ go to 0 because of the part of the path in the lower half-plane.
 The argument of $\exp(-2i \kappa \rho^{-1}e^{-i\theta})$ is 
 $
 -2\kappa\rho^{-1}\cos \theta
 $
 and taking $\rho$ too small would be de\-vastating for the calculations
 because of the strong oscillations triggered by the term
 $
 \exp( -2i\kappa\rho^{-1}\cos \theta)
 $
 all over the circle.
 Of course for the evaluation of 
 $\Lg\Bigl(\frac{1+\rho e^{i\theta}}{1-\rho e^{i\theta}}\Bigr)$
 is easier for $\rho$ small, but we have to take into account the constraints in that direction mentioned above.
\end{nb}
\begin{rem}\rm
 It seems easier numerically for the evaluation of $a_{11}$
 to use \eqref{num111}
 rather than any other expression
 (see e.g. Lemma \ref{lem.57}, \eqref{2525}, \eqref{9522}).
 However the following formula could be interesting, theoretically and numerically:
  recalling that $\sinc x=\frac{\sin x}{x}$, 
  we have from \eqref{2525}
\begin{multline}\label{num111+++}
a_{11}(\tau, \sigma)=\frac12+\frac{2\omega}\pi
\int_{0}^{+\io}
\sinc(2\omega s)\frac{s}{\sinh s} \cos (2\kappa \tanh s) ds
\\
-\frac{2\kappa}\pi
\int_{0}^{+\io}
\sinc(2\kappa s)\frac{1}{\cosh s} \cos (2\omega s) ds,
\end{multline}
 but it turns out that numerical calculations involving \eqref{num111+++}
 seem to be less reliable than the methods using \eqref{num111}.
 \end{rem}
 We can also take a look at the following curves.
 \begin{figure}[ht]
\centering
\scalebox{0.65}{\includegraphics[angle=0,height=400pt,width=1.5\textwidth]{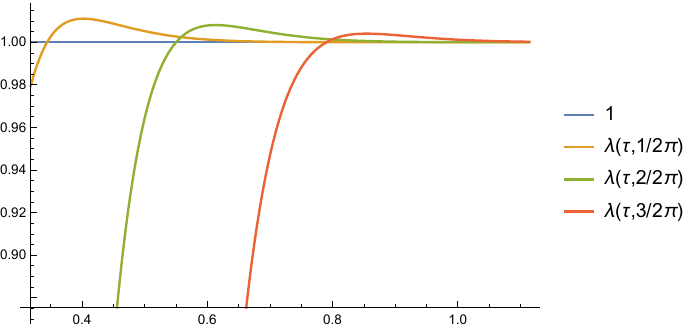}}\par
\caption{
Functions $\lambda_{+}(\tau,\kappa/2\pi)$ with $\kappa=1,2,3$: their maxima are strictly greater than 1.
}
\label{pic8}
\end{figure}
\begin{rem}\rm
 In the above figure, in order to put the three curves on the same picture, we have used three different logarithmic scales on the vertical axis, namely we have drawn
 $$
 \tau\mapsto 1+ \alpha_{j}\Lg\bigl(\lambda_{+}(\tau, \sigma_{j})\bigr), \quad 1\le j\le 3,
 \sigma_{j}=j/2\pi,
 \alpha_{1}=20, \alpha_{2}=100, \alpha_{3}=500.
 $$
 Of course we have 
 $$
 1+ \alpha_{j}\Lg\bigl(\lambda_{+}(\tau, \sigma_{j})\bigr)>1\Longleftrightarrow
 \Lg\bigl(\lambda_{+}(\tau, \sigma_{j})\bigr)>0\Longleftrightarrow
 \lambda_{+}(\tau, \sigma_{j})>1,
 $$
 so that the piece of curves in Figure \ref{pic8} which are above 1 are indeed
 corresponding to curves of $\tau\mapsto \lambda_{+}(\tau,\sigma_{j})$ which go strictly above the threshold 1. We have also
 \begin{align*}\label{}
\max_{\tau}\lambda_{+}(\tau,\sigma_{1})&\approx1+55{\scriptstyle\times}10^{-5}
\hs\text{ at $\tau\approx 0.402030$, }
\\
\max_{\tau}\lambda_{+}(\tau,\sigma_{2})&\approx1+8{\scriptstyle\times}10^{-5}
\hs\text{ at $\tau\approx 0.613262$,}
\\
\max_{\tau}\lambda_{+}(\tau,\sigma_{3})&\approx1+10^{-5}
\hs\text{ at $\tau\approx 0.854746$.}
\end{align*}
We are glad to have a theoretical proof of Theorem \ref{thm.59}
since the numerical analysis of cases where $\sigma$ is large, say larger than 10, seem to be very difficult to achieve,
at least through a standard use of {\tt Mathematica}.
The reason for that is quite clear since using our Lemma \ref{lem.510--}, we did study the function
$\beta$ defined by 
\begin{equation}\label{funbet}
\beta(\tau,\sigma)=\val{a_{12}(\tau,\sigma)}^{2}+a_{11}(\tau,\sigma)-1,
\end{equation}
and proved that for each $\sigma\ge 0$ there exists $T_{0}(\sigma)$ such that for all $\tau\ge T_{0}
(\sigma)$ we have $\beta(\tau,\sigma)>0$ and $a_{12}(\tau,\sigma)\not=0$.
Thanks to Lemma \ref{lem.57} and \eqref{boua21} we knew that for $\tau\ge T_{0}
(\sigma)$, we had 
$$
\val{1-a_{11}}\le 2e^{-{\pi^2\tau}} e^{4\pi \sigma}\ll
\frac{e^{-16\pi\sqrt\tau
\sqrt\sigma}}{2^{8}\pi^{6}\tau^{2}}\le (\re a_{21})^{2}\le \val{a_{12}}^{2},
$$
where the second inequality $\ll$ is in fact comparing for $\sigma$ fixed two exponential decays. The numerical analysis of that inequality is certainly quite difficult when $\sigma$ and $\tau$ are large since both sides are converging to zero quite fast for $\sigma$ fixed and $\tau\rightarrow +\io$;
of course taking the logarithm of both sides 
looks quite reasonable, but in practice does not seem really easy numerically.
When $\sigma=0$, the situation is much better,
since we had to compare (cf. Subsection \ref{sec531}) an exponential decay 
$\val{1-a_{11}}\le 2e^{-{\pi^2\tau}}$
to a polynomial decay
$$
\val{\re a_{12}}^{2}\sim \frac{1}{2^{6} \pi^{4} \tau^{2}},\quad \tau\rightarrow +\io,
$$
and this could be an {\sl a posteriori} explanation for which our numerical argument in \cite{DDL}
worked smoothly  to disprove Flandrin's conjecture.
So to pick-up the quarter-plane (\eqref{5014} with $\sigma=0$)
to produce a counterexample to that conjecture was indeed a very wise choice:
if you choose instead $\mathcal C_{\sigma}$ for $\sigma$ large, our Theorem \ref{thm.59} shows that it is also a counterexample to Flandrin's conjecture\footnote{As a convex subset of the plane on which the integral of the Wigner distribution of some normalized pulse is $>1$. }, but we have a theoretical proof  for that Theorem and if we were depending on a numerical analysis, it is quite likely that checking numerically the positivity of the function $\beta$ defined in \eqref{funbet}
could be rather difficult, even say for $\sigma=10$.
\end{rem}
\section{Unboundedness is Baire generic}\label{sec.baire}
\index{Baire genericity}
In this section we show that for plenty of subsets $E$ of the phase space $\RZ$, the operator
$\opw{\mathbf 1_{E}}$ is not bounded on $L^2(\R^n)$.
\par\no
{\bf Acknowledgements.}
The author is grateful to H.G.~Feichtinger and K.~Gr\"ochenig for sharp comments on a first version of this section.
\subsection{Preliminaries}
\subsubsection{Prolegomena}
\begin{lem}\label{lem51}
 Let $u,v\in L^2(\R^n)$ and let $\mathcal W(u,u),$
 $\mathcal W(v,v),$ be their    Wigner distributions.
 Then 
 we have 
 $$
 \norm{\mathcal W(u,u)-\mathcal W(v,v)}_{L^2(\RZ)}
 \le \norm{u-v}_{L^2(\R^n)}\bigl(\norm{u}_{L^2(\R^n)}
    +
    \norm{v}_{L^2(\R^n)}\bigr).
 $$
 As a consequence if a sequence $(u_{k})$ is converging in $L^2(\R^n)$, then the sequence  $(\mathcal W(u_{k}, u_{k}))$ converges in $L^2(\RZ)$ towards $\mathcal W(u,u)$.
\end{lem}
\begin{proof}
 We have by sesquilinearity
 $$
 \mathcal W(u,u)-\mathcal W(v,v)= \mathcal W(u-v,u)+ \mathcal W(v,u-v),
 $$
 so that
 \begin{multline*}
 \norm{\mathcal W(u,u)-\mathcal W(v,v)}_{L^2(\RZ)}\le 
   \norm{\mathcal W(u-v,u)}_{L^2(\RZ)}
   +
    \norm{\mathcal W(v,u-v)}_{L^2(\RZ)}
    \\
    \underbracket[0.75pt][0.5pt]{\hskip3pt=\hskip3pt}_{\eqref{norm}}\norm{u-v}_{L^2(\R^n)}\bigl(\norm{u}_{L^2(\R^n)}
    +
    \norm{v}_{L^2(\R^n)}\bigr),
\end{multline*}
proving the lemma.
\end{proof}
\begin{lem}\label{lem.62llkk}
Let $(u_{k})$ be a converging sequence in $L^2(\R^n
)$  with limit $u$.
Let us assume that there exists $C_{0}\ge 0$ such that 
\begin{equation}\label{}
\forall k\in \N, \quad\iint \vert\mathcal W(u_{k}, u_{k})(x,\xi)\vert dx d\xi\le C_{0}.
\end{equation}
Then we have 
$
\iint \vert\mathcal W(u, u)(x,\xi)\vert dx d\xi\le C_{0}
.
$
\end{lem}
\begin{proof}
 Let $R>0$ be given. We check
\begin{multline*}
 \iint_{\val x^2+\val \xi^2\le R^2} \vert\mathcal W(u, u)(x,\xi)
 -\mathcal W(u_{k}, u_{k})(x,\xi)
 \vert dx d\xi
\\ \le
  \iint_{\val x^2+\val \xi^2\le R^2}  \vert\mathcal W(u-u_{k}, u)(x,\xi)\vert dx d\xi
  + \iint_{\val x^2+\val \xi^2\le R^2}  \vert\mathcal W(u_{k}, u-u_{k})(x,\xi)\vert dx d\xi
  \\
  \le \sqrt{\val{\mathbb B^{2n} }R^{2n}}
  \bigl(\norm{\mathcal W(u-u_{k}, u)}_{L^2(\RZ)}
  +
  \norm{\mathcal W(u_{k}, u-u_{k})}_{L^2(\RZ)}
  \bigr)
  \\
 =
   \sqrt{\val{\mathbb B^{2n} }R^{2n}}
   \norm{u-u_{k}}_{L^2(\R^n)}
  \bigl(\norm{u}_{L^2(\R^n)}
  +
  \norm{u_{k}}_{L^2(\R^n)}
  \bigr),
\end{multline*}
and thus
\begin{align*}
& \iint_{\val x^2+\val \xi^2\le R^2} \vert\mathcal W(u, u)(x,\xi)\vert dxd\xi
 \le  \iint_{\val x^2+\val \xi^2\le R^2} \vert\mathcal W(u_{k}, u_{k})(x,\xi)\vert dxd\xi
\\ &\hskip155pt +
  \sqrt{\val{\mathbb B^{2n} }R^{2n}}
   \norm{u-u_{k}}_{L^2(\R^n)}
  \bigl(\norm{u}_{L^2(\R^n)}
  +
  \norm{u_{k}}_{L^2(\R^n)}
  \bigr)
  \\
  &\hskip75pt \le C_{0}
  +
  \sqrt{\val{\mathbb B^{2n} }R^{2n}}
   \norm{u-u_{k}}_{L^2(\R^n)}
  \bigl(\norm{u}_{L^2(\R^n)}
  +
  \norm{u_{k}}_{L^2(\R^n)}
  \bigr),
\end{align*}
implying  for all $R>0$,
$$
 \iint_{\val x^2+\val \xi^2\le R^2} \vert\mathcal W(u, u)(x,\xi)\vert dxd\xi
 \le C_{0},
$$
and thus the sought result.
\end{proof}
\subsubsection{An explicit construction}\label{sec.612tre}
We just calculate
$\mathcal W(v_{0},v_{0})$ for 
\begin{equation}\label{explicit}                 
v_{0}=\mathbf 1_{[-1/2,1/2]}.
\end{equation}
\begin{rem}\label{rem.6363}
 When $u$ is  supported in a closed convex set $J$, we have in the integral \eqref{wigner}
 defining $\mathcal W$, 
 $
 x\pm\frac z2\in J\Longrightarrow x\in J,
 $
 so that 
 $\supp \mathcal W(u,u)\subset J\times\R^{n}$.
\end{rem}
We have
\begin{align*}
\mathcal W(v_{0},v_{0})(x,\xi)=\int_{\substack{-1/2\le x+z/2\le 1/2\\
-1/2\le x-z/2\le 1/2
}} e^{2i\pi z\xi}dz,
\end{align*}
and the integration domain is 
$$
-\min(1-2x,1+2x)=\max(-1-2x,2x-1)\le z\le\min( 1-2x, 1+2x),
$$
which is empty unless $1-2x, 1+2x\ge 0$ i.e. $x\in [-1/2, +1/2]$, and moreover we have the equivalence
$$
1-2x\le 1+2x\Longleftrightarrow x\ge 0,
$$
so that 
\begin{multline}\label{ex1234}
\mathcal W(v_{0},v_{0})(x,\xi)=H(x)\int_{-(1-2x)}^{1-2x}e^{2i\pi z\xi}dz
+H(-x)\int_{-(1+2x)}^{1+2x}e^{2i\pi z\xi}dz
\\
=H(x)\frac{e^{2i\pi \xi(1-2x)}-e^{-2i\pi\xi(1-2x)}}{2i\pi \xi}+H(-x)
\frac{e^{2i\pi \xi(1+2x)}-e^{-2i\pi\xi(1+2x)}}{2i\pi \xi}
\\=
\mathbf 1_{[0,1/2]}(x)\frac{\sin(2\pi \xi(1-2x))}{\pi \xi}+
\mathbf 1_{[-1/2,0]}
\frac{\sin(2\pi \xi(1+2x))}{\pi \xi}.
\end{multline}
More generally for $a,b,\omega$ real numbers with $a<b$
and 
\begin{equation}\label{54gfdd}
u_{a,b,\omega}(x)=(b-a)^{-1/2}\mathbf 1_{[a,b]}(x) e^{2i\pi \omega x},
\end{equation}
we have 
\begin{multline}
\mathcal W(u_{a,b,\omega},u_{a,b,\omega})(x,\xi)
\\=
\frac{\bigl(
\mathbf 1_{[a,\frac{a+b}2]}(x)\sin[4\pi(\xi-\omega)(x-a)]
+\mathbf 1_{[\frac{a+b}2,b]}(x)\sin[4\pi(\xi-\omega)(b-x)]
\bigr)}{(b-a)\pi(\xi-\omega)}.
\end{multline}
We check now, using \eqref{ex1234}, for $N>0$,
\begin{align*}
\iint\val{ \mathcal W(v_{0},v_{0})(x,\xi)} dx d\xi&\ge\int_{0\le x\le 1/4}
\int_{0}^N\Val{\frac{\sin(2\pi \xi(1-2x))}{\pi \xi}} d\xi dx
\\
&=
\int_{0\le x\le 1/4}
\int_{0}^{N2\pi (1-2x)}\Val{\frac{\sin\eta}{\pi \eta}} d\eta dx
\\&\ge 
\int_{0\le x\le 1/4}
\int_{0}^{N\pi }\Val{\frac{\sin\eta}{\pi \eta}} d\eta dx
=\frac14\int_{0}^{N\pi }\Val{\frac{\sin\eta}{\pi \eta}} d\eta,
\end{align*}
so that 
\begin{equation}\label{gdfr44}
\iint\val{ \mathcal W(v_{0},v_{0})(x,\xi)} dx d\xi=+\infty.
\end{equation}
\begin{pro}
 Let $a,b,\omega$ be real numbers with $a<b$ and let us define $u_{a,b,\omega}$ by 
 \eqref{54gfdd}.
 Then we have
 \begin{equation}\label{626yyy}
\iint \left\vert\mathcal W(u_{a,b,\omega}, u_{a,b,\omega})(x,\xi)\right\vert dx d\xi=+\io.
\end{equation}
\end{pro}
\begin{nb} Since $u_{a,b,\omega}$ is  a normalized $L^{2}(\R)$ function,
we also have  from \eqref{norm}, \eqref{hg55ss} that the real-valued $
W(u_{a,b,\omega}, u_{a,b,\omega})$
 does satisfy
 \begin{align*}
&\int\left\vert\int {W(u_{a,b,\omega}, u_{a,b,\omega})(x,\xi) }dx \right\vert d\xi=
\int\left\vert\int {W(u_{a,b,\omega}, u_{a,b,\omega})(x,\xi) }d\xi \right\vert dx
\\&\hskip285pt =\norm{u_{a,b,\omega}}_{L^{2}(\R)}^{2}=1,
\\
 &\hskip55pt \iint {W(u_{a,b,\omega}, u_{a,b,\omega})(x,\xi) }^2dx d\xi=\norm{u_{a,b,\omega}}_{L^{2}(\R)}^{4}=1.
\end{align*}
We shall see in the next sections that most of the time in the Baire Category sense, we have  for $u\in L^2(\R^n)$,
$
\iint \val{W(u,u)(x,\xi) }dx d\xi=+\io.
$
\end{nb}
\begin{proof}
 The proof is already given above for $v_{0}=u_{-1/2, 1/2, 0}$. Moreover we have with 
 $$
 \alpha=\frac{1}{b-a}, \quad \beta=\frac{b+a}{2(a-b)},
 $$
 the formula
 $
 v_{0}(y)=e^{-2i\pi \omega(y-\beta)\alpha^{-1}} u_{a,b,\omega}\bigl((y-\beta) \alpha^{-1}\bigr) \alpha^{-1/2},
 $
 so that 
 $u_{a,b,\omega}=\mathcal M v_{0}$,
  where $\mathcal M$ belongs to the group $Mp_{a}(n)$.
 (cf. Section \ref{sec.weyl}) and the covariance property \eqref{segal+} shows that the already proven \eqref{gdfr44} implies  \eqref{626yyy}.
\end{proof}
\subsection{Modulation spaces}
\index{modulation spaces}
\index{Feichtinger algebra}
\index{{~\bf Notations}!$M^{1}(\R^{n})$}
In this section, we use the Feichtinger algebra $M^{1}$, introduced in \cite{MR643206}
(the terminology {\it Feichtinger algebra} goes back to the book \cite{MR1802924}). The survey article \cite{MR3881843} by M.S. Jakobsen
 is a good source for recent developments of the theory
as well as Chapter 12 in the K.~Gr\"ochenig's book \cite{MR1843717}.
We refer the reader to the paper \cite{MR2015328} by K.~Gr\"ochenig \& M. Leinert as well as to 
J. Sj\"ostrand's article \cite{MR1362552} for the use of modulation spaces to proving a non-commutative Wiener lemma.
\subsubsection{Preliminary lemmas}
The following lemmas in this subsection  are well-known (see e.g. Theorem 11.2.5 in \cite{MR1843717}). However we provide a proof for the self-contained\-ness of our survey.
\begin{lem}\label{lem.611mlk} Let $\phi_{0}$ be a non-zero function  in $\mathscr S(\R^{n})$. For 
 $u\in \mathscr S'(\R^{n})$ the following properties are equivalent:
 \begin{itemize}
 \item[(i)] $u\in \mathscr S(\R^{n})$.
 \item[(ii)] $\mathcal W(u,\phi_{0})\in \mathscr S(\R^{2n})$.
 \item[(iii)] $\forall N\in\N, \quad\sup_{X\in \RZ}\val{\mathcal W(u,\phi_{0})(X)}(1+\val X)^{N}<+\io.$
\end{itemize}
\end{lem}
\begin{proof}
Let us assume  (i) holds true; with $\Omega(u,\phi_{0})$ defined in \eqref{funcome},
we find that $\Omega(u,\phi_{0})$ belongs to $\mathscr S(\RZ)$, thus as well as its partial Fourier transform $\mathcal W(u,\phi_{0})$, proving (ii).
We have obviously that (ii) implies (iii). Let us now assume that (iii) holds true.
Using \eqref{623new}, we find
$$
u(x_{1}) \bar \phi_{0}(x_{2})
=\int \mathcal W(u,\phi_{0})(\frac{x_{1}+x_{2}}2,\xi) e^{2i\pi (x_{1}-x_{2}) \xi} d\xi,
$$
and thus
\begin{multline*}
u(x_{1}) \norm{\phi_{0}}_{L^{2}(\R^{n})}^{2}
=\iint \mathcal W(u,\phi_{0})(\frac{x_{1}+x_{2}}2,\xi) e^{2i\pi (x_{1}-x_{2}) \xi}\phi_{0}(x_{2}) d\xi dx_{2}
\\=
\iint \mathcal W(u,\phi_{0})(y,\xi) e^{4i\pi (x_{1}-y) \xi} \phi_{0}(2y-x_{1})d\xi dy 2^{n},
\end{multline*}
so that the latter equality, the fact that $\phi_{0}$ belongs to $\mathscr S(\R^{n})$
imply (i)
by differentiation under the integral sign, concluding the proof of the lemma.
\end{proof}
\begin{lem}\label{lem.21kjhg}
Let $\phi_{0}, \phi_{1}$ be non-zero functions in $L^{2}(\R^{n})$.
 Let $u\in L^{2}(\R^{n})$ such that $\mathcal W(u,\phi_{0})$ belongs to $L^{1}(\RZ)$.
 Then $\mathcal W(u,\phi_{1})$  belongs as well to $L^{1}(\RZ)$.
\end{lem}
\begin{proof}
According to Lemma \ref{lem.54jhgf} applied to 
$u_{0}=u, u_{1}=u_{2}=\phi_{0}, u_{3}=\phi_{1}$, we have 
$$
\norm{\phi_{0}}_{L^{2}}^{2}\mathcal W(u,\phi_{1})\in L^{1}(\RZ),
$$ 
since $\mathcal W(u,\phi_{0})$ belongs to $L^{1}(\RZ)$ as well as  $\mathcal W(\check \phi_{0}, \phi_{1})$.
\end{proof}
\begin{lem}\label{lem.879kol}
 Let $u\in L^{2}(\R^{n})$.
 The following properties are equivalent.
\begin{itemize}
 \item[(i)] For all $\phi\in \mathscr S(\R^{n})$, we have $\mathcal W(u,\phi)\in L^{1}(\RZ)$.
 \item[(ii)] For a non-zero $\phi\in \mathscr S(\R^{n})$, we have $\mathcal W(u,\phi)\in L^{1}(\RZ)$.
 \item[(iii)] $\mathcal W(u,u)$ belongs to $L^{1}(\RZ)$.
\end{itemize}
 \end{lem}
 \begin{proof}
 We have obviously (i) $\Longrightarrow$ (ii) and, conversely, Lemma \ref{lem.21kjhg} yields 
 (ii) $\Longrightarrow$ (i).
 Assuming (i) and using Lemma \ref{lem.54jhgf}
 with $u_{0}=u_{3}=u$, $u_{1}=u_{2}=\phi\in \mathscr S(\R^{n})$,
 we get
 $$
  \norm{\phi}_{L^{2}}^{2}
\val{\mathcal W(u,u)(X)}\le 2^{n}
\bigl(\val{\mathcal W(u, \phi)}\ast\val{\mathcal W(\check \phi, u)}\bigr)(X),
$$
so that choosing a non-zero $\phi$ in the Schwartz space,
we obtain (iii).
Conversely, assuming (iii) and using again Lemma \ref{lem.54jhgf}
with $u_{0}=u_{2}=u$, $u_{3}=\phi\in \mathscr S(\R^{n})$, $u_{1}=\psi\in \mathscr S(\R^{n})$,
 we find 
\begin{equation}\label{641mmm}
 \val{\poscal{\psi}{u}_{L^{2}}}
\val{\mathcal W(u,\phi)(X)}\le 2^{n}
\bigl(\underbrace{\val{\mathcal W(u, u)}}_{\in L^{1}(\RZ)}\ast\val{\underbrace{\mathcal W(\check \psi, \phi)}_{\in \mathscr S(\RZ)}}\bigr)(X).
\end{equation}
Assuming as we may $u\not=0$,  we can choose $\psi\in \mathscr S(\R^{n})$ such that $\poscal{\psi}{u}_{L^{2}}\not=0$,
so that \eqref{641mmm} implies (i).
\end{proof}
\index{inversion formula}
\begin{lem}\label{lem.614inv}
 Let $u_{1}, u_{2}, u_{3}\in L^{2}(\R^{n})$. Then we have the \emph{inversion formula},
 \begin{equation}\label{}
\opw{\mathcal W(u_{1}, u_{2})} u_{3}=\poscal{u_{3}}{u_{2}}_{L^{2}(\R^{n})} u_{1}.
\end{equation}
\end{lem}
\begin{proof}
 It is an immediate consequence of Lemma \ref{lem.113esd}.
\end{proof}
\subsubsection{The space $M^{1}(\R^{n})$}
\begin{defi}\label{defi64}
The space  $M^{1}(\R^{n})$ is defined as the set of $u\in L^{2}(\R^{n})$ such that, for all $\phi\in \mathscr S(\R^{n})$, 
$\mathcal W(u , \phi)$ belongs to $L^{1}(\RZ)$.
According to Lemma \ref{lem.879kol}, $M^{1}(\R^{n})$ is also the set of $u\in L^{2}(\R^{n})$ such that $\mathcal W(u,u)\in L^{1}(\RZ)$ as well as the
the set of $u\in L^{2}(\R^{n})$ such that, for a non-zero $\phi\in \mathscr S(\R^{n})$, 
$\mathcal W(u , \phi)$ belongs to $L^{1}(\RZ)$.
 \end{defi}
 \index{{~\bf Notations}!$\psi_{0}$}
 \begin{pro}\label{pro.54lkqs} Let $\psi_{0}$ be the standard fundamental state of the Harmonic Oscillator 
 $\pi(D_{x}^{2}+x^{2})$
 given by 
 \begin{equation}\label{psi000}
\psi_{0}(x)=2^{n/4} e^{-\pi\val x^{2}}.
\end{equation}
Then $M^{1}(\R^{n})\ni u\mapsto \norm{\mathcal W(u,\psi_{0})}_{L^{1}(\RZ)}$ is a norm on $M^{1}(\R^{n})$.
Let $\psi$ be a non-zero function in $\mathscr S(\R^{n})$: then 
$M^{1}(\R^{n})\ni u\mapsto \norm{\mathcal W(u,\psi)}_{L^{1}(\RZ)}$
is a norm on $M^{1}(\R^{n})$,
equivalent to the previous norm.
\end{pro}
\begin{proof}
 The homogeneity and triangle inequality are immediate, let us check the separation: let $u\in L^{2}(\R^{n})$ such that $\mathcal W(u,\psi)=0$. Then we have
 $$
0=\poscal{ \opw{\mathcal W(u,\psi)}\psi}{u}_{L^{2}(\R^{n})}=\norm{u}^{2}_{L^{2}(\R^{n})}\norm{\psi}^{2}_{L^{2}(\R^{n})},
 $$
 proving the sought result. Let $\psi$ be a non-zero function in $\mathscr S(\R^{n})$;
 according to Lemma \ref{lem.54jhgf} applied to 
$u_{0}=u, u_{1}=u_{2}=\psi_{0}, u_{3}=\psi$, 
 we find 
\begin{equation}\label{}
\val{\mathcal W(u,\psi)(X)}\le 2^{n}
\bigl(\val{\mathcal W(u, \psi_{0})}\ast\val{\mathcal W(\psi_{0}, \psi)}\bigr)(X),
\end{equation}
so that we have 
\begin{align}
&\norm{\mathcal W(u,\psi)}_{L^{1}(\RZ)}\le 2^{n}\norm{\mathcal W(\psi_{0}, \psi)}_{L^{1}(\RZ)}
\norm{\mathcal W(u, \psi_{0})}_{L^{1}(\RZ)},
\label{645lkj}
\\
&\norm{\mathcal W(u,\psi_{0})}_{L^{1}(\RZ)}\le 2^{n}\norm{\mathcal W(\psi, \psi_{0})}_{L^{1}(\RZ)}
\norm{\mathcal W(u, \psi)}_{L^{1}(\RZ)},
\label{646lkj}
\end{align}
proving the equivalence of norms.
\end{proof}
\begin{pro}The space 
 $M^{1}(\R^{n})$, equipped with the equivalent norms of Proposition \ref{pro.54lkqs},
 is a Banach space.
The space $\mathscr S(\R^{n})$ is dense in $M^{1}(\R^{n})$.
\end{pro}
\begin{proof}
 Let $(u_{k})_{k\ge 1}$ be a Cauchy sequence  in $M^{1}(\R^{n})$:
 it means that
 $(\mathcal W(u_{k},\psi_{0}))_{k\ge 1}$ is a Cauchy sequence in $L^{1}(\RZ)$,
 thus such that
 \begin{equation}\label{644poi}
\lim_{k}\mathcal W(u_{k},\psi_{0})=U\quad\text{in $L^{1}(\RZ)$.}
\end{equation}
On the other hand, from Lemma \ref{lem.113esd},
 we have
 \begin{equation}\label{645mlk}
 u_{k}-u_{l}=\OPW{\mathcal W(u_{k}-u_{l},\psi_{0})}\psi_{0},
\end{equation}
so that 
\begin{multline*}
  \norm{u_{k}-u_{l}}_{L^{2}(\R^{n})}\le \norm{\OPW{\mathcal W(u_{k}-u_{l},\psi_{0})}}_{\mathcal B(L^{2}(\R^{n}))}
 \\\underbrace{ \le}_{\text{cf. \eqref{norm01}}} 2^{n}
  \norm{\mathcal W(u_{k}-u_{l},\psi_{0})}_{L^{1}(\RZ)},
 \end{multline*}
 implying that 
 $(u_{k})_{k\ge 1}$ is a Cauchy sequence in $L^{2}(\R^{n})$, thus converging towards a function $u$ in $L^{2}(\R^{n}).$ 
 Since from \eqref{norm}, we have  
 $$
 \norm{\mathcal W(u_{k}-u,\psi_{0})}_{L^{2}(\RZ)}=\norm{u_{k}-u}_{L^{2}(\R^{n})},
 $$
 we obtain as well that 
 \begin{equation}\label{limit1}
 \lim_{k}\mathcal W(u_{k},\psi_{0})=\mathcal W(u,\psi_{0})\quad\text{in $L^{2}(\RZ)$},
\end{equation}
and this implies along with \eqref{644poi} that $U=\mathcal W(u,\psi_{0})$ in $\mathscr S'(\RZ)$.
As a result, we have $\mathcal W(u,\psi_{0})\in L^{1}(\RZ)$, so that $u\in M^{1}(\R^{n})$
 and 
$$
\lim_{k}\mathcal W(u_{k},\psi_{0})=\mathcal W(u,\psi_{0})\quad\text{in $L^{1}(\RZ)$},
$$
entailing convergence towards $u$ for the sequence $(u_{k})_{k\ge 1}$ in $M^{1}(\R^{n})$ and the sought completeness.
We are left with the density question and we start with a calculation.
\begin{claim}With the phase symmetry $\sigma_{y,\eta}$ given by \eqref{phase} and $\psi_{0}$ by 
\eqref{psi000} we have for $X, Y\in \RZ$, 
\begin{equation}\label{648poi}
\mathcal W(\sigma_{Y}\psi_{0}, \psi_{0})(X)=2^{n}e^{-2\pi\val{X-Y}^{2}} e^{-4i\pi[X,Y]},
\end{equation}
where the symplectic form is given in \eqref{sympfor}.
\end{claim}
\begin{proof}[Proof of the Claim]
 We have indeed
 \begin{align*}
 \mathcal W(\sigma_{y,\eta}\psi_{0},& \psi_{0})(x,\xi)=\int 
 \bigl(\sigma_{y,\eta}\psi_{0}\bigr)(x+\frac z2)
 \psi_{0}(x-\frac z2)e^{-2i\pi z\cdot \xi} dz
 \\&=
 \int 
\psi_{0}(2y-x-\frac z2) e^{4i\pi \eta\cdot(x+\frac z2-y)}
 \psi_{0}(x-\frac z2)e^{-2i\pi z\cdot \xi} dz
 \\
 &=2^{n/2}\int 
 e^{-\pi\left(\val{2y-x-\frac z2}^{2}
 +\val{x-\frac z2}^{2}
 \right)}
 e^{2i\pi z\cdot(\eta-\xi)} dz
e^{4i\pi \eta\cdot (x-y)}
\\&=
2^{n/2}e^{4i\pi \eta\cdot (x-y)}\int e^{-\frac\pi 2\left(\val{2y-z}^{2}+\val{2(y-x)}^{2}\right)}e^{2i\pi z\cdot(\eta-\xi)} dz
\\
&=2^{n/2}e^{4i\pi \eta\cdot (x-y)}e^{-2\pi\val{y-x}^{2}}
e^{4i\pi y\cdot(\eta-\xi)} 
2^{n/2}
e^{-2\pi\val{\eta-\xi}^{2}},
\end{align*} 
which is the sought formula.
\end{proof}
Let $u$ be a function in $M^{1}(\R^{n})$. For $\varepsilon>0$ we define
$$
u_{\varepsilon}(x)=\int_{\RZ}\mathcal W(u,\psi_{0})(Y) e^{-\varepsilon\val Y^{2}}
2^{n}(\sigma_{Y}\psi_{0})(x) dY,
$$
and we have 
$$
\mathcal W(u_{\varepsilon},\psi_{0})(X)=\int_{\RZ}\mathcal W(u,\psi_{0})(Y) e^{-\varepsilon\val Y^{2}}
2^{n}\mathcal W(\sigma_{Y}\psi_{0}, \psi_{0})(X) dY,
$$
so that Lemma
\ref{lem.611mlk}  and \eqref{648poi} imply readily that $u_{\varepsilon}$ belongs to the Schwartz space.
Moreover we have $u=\opw{\mathcal W(u,\psi_{0})}\psi_{0},$
from Lemma \ref{lem.614inv} and thus 
$$
\mathcal W(u,\psi_{0})(X)=
\int_{\RZ}\mathcal W(u,\psi_{0})(Y)
2^{n}\mathcal W(\sigma_{Y}\psi_{0}, \psi_{0})(X) dY,
$$
so that 
\begin{multline*}
\int_{\RZ}\val{\mathcal W(u_{\varepsilon},\psi_{0})(X)-\mathcal W(u,\psi_{0})(X)}
dX
\\\le2^{n}
\iint_{\RZ\times \RZ}\underbrace{\val{\mathcal W(u,\psi_{0})(Y)}
\val{\mathcal W(\sigma_{Y}\psi_{0}, \psi_{0})(X) }}_{
\substack
{\in L^{1}(\R^{4n}) \text{ from \eqref{648poi}}
\text{ and $u\in M^{1}(\R^{n})$}}
}
\underbrace{\bigl(1-e^{-\varepsilon\val Y^{2}}\bigr)}_{\in [0,1]}dY
dX. 
\end{multline*}
The Lebesgue Dominated Convergence Theorem shows that the integral above tends to $0$ when $\varepsilon\rightarrow 0_{+}$,
proving the convergence in $M^{1}(\R^{n})$ of the sequence $(u_{\varepsilon})$, which completes the proof of the density.
\end{proof}
\begin{theorem} Let $\mathcal M$ be an element of the metaplectic group $Mp(n)$ (Definition \ref{def.41fdhh}). Then $\mathcal M$ is an isomorphism of $M^{1}(\R^{n})$ and we have  for $u\in M^{1}(\R^{n})$, $\phi\in \mathscr S(\R^{n})$,
\begin{equation}\label{647ppp} 
\mathcal W(\mathcal M u, \mathcal M \phi)=\mathcal W(u,\phi)\circ S^{-1},
\end{equation}
where $\mathcal M$ is in the fiber of the  symplectic transformation $S$.
In particular, the space $M^{1}(\R^{n})$
is invariant by the Fourier transformation and partial Fourier transformations, by the rescaling \eqref{mety02},
by the transformations  \eqref{mety01},
 \eqref{mety03} and also
by the phase translations \eqref{phtrans}
and phase symmetries \eqref{phase}. 
\end{theorem}
\begin{proof} Formula \eqref{647ppp} follows readily from 
 \eqref{segal+} and if $u$ belongs to $M^{1}(\R^{n})$,
 we find that 
 $$
 \mathcal W(\mathcal M u, \underbrace{\mathcal M \psi_{0}}_{\in \mathscr S(\R^{n})})=\underbrace{\mathcal W(u,\psi_{0})}_{\in L^{1}(\RZ)}\circ S^{-1},
 $$
 and since $\det S=1$, we have 
 $
\norm{ \mathcal W(\mathcal M u, \mathcal M \psi_{0})}_{L^{1}(\RZ)}=\norm{\mathcal W(u,\psi_{0})}_{L^{1}(\RZ)},
 $
 implying that 
 $
  \mathcal W(\mathcal M u, \mathcal M \psi_{0})
 $
 belongs to $L^{1}(\RZ)$ so that, thanks to Definition \ref{defi64},
 we get that $\mathcal M u$ belongs to $M^{1}(\R^{n})$.
 The same properties are true for $\mathcal M^{-1}$.
\end{proof}
\begin{rem}\rm
From  Definition \ref{defi64}, we see that, for $u\in M^{1}(\R^{n})$, we have  $$\mathcal W(u,u)\in L^{1}(\RZ),$$ and this implies, thanks to Theorem \ref{thm.112lkj}, that $M^{1}(\R^{n})\subset L^{1}(\R^{n})$.
Moreover
we have
$$
\mathcal F\bigl(M^{1}(\R^{n})\bigr)\subset M^{1}(\R^{n}),
$$
since for $u\in M^{1}(\R^{n})$, we have 
$
\mathcal W(\hat u,\psi_{0})=\mathcal W(\hat u,\hat\psi_{0})
$
and thanks to \eqref{647ppp},
$$
\norm{\mathcal W(\hat u,\hat \psi_{0})}_{L^{1}(\RZ)}=\norm{\mathcal W(u,\psi_{0})}_{L^{1}(\RZ)}.
$$
As a consequence we find
$$
\mathcal F\bigl(M^{1}(\R^{n})\bigr)\subset M^{1}(\R^{n})=\mathcal F^{2}\mathcal C (M^{1}(\R^{n}))
=\mathcal F^{2}(M^{1}(\R^{n}))\subset \mathcal F\bigl(M^{1}(\R^{n})\bigr),
$$ 
and consequently
\begin{equation}\label{}
M^{1}(\R^{n})=
\mathcal F(M^{1}(\R^{n}))\subset \mathcal F(L^{1}(\R^{n}))\subset C_{(0)}(\R^{n}),
\end{equation}
where the latter inclusion is due to the Riemann-Lebesgue Lemma
with 
$C_{(0)}(\R^{n})$ standing for space of 
continuous functions  
with limit $0$ at infinity.
Moreover, for $u\in M^{1}(\R^{n})$ and $\psi_{0}$ given by \eqref{psi000}, we get from 
\eqref{623new},
\begin{equation}
u(x_{1}) \bar \psi_{0}(x_{2})
=\int \mathcal W(u,\psi_{0})(\frac{x_{1}+x_{2}}2,\xi) e^{2i\pi (x_{1}-x_{2})\cdot\xi} d\xi,
\end{equation}
so that 
$$
u(x_{1}) =\iint \mathcal W(u,\psi_{0})(y, \eta) e^{4i\pi (x_{1}-y)\cdot\eta} \bar \psi_{0}(2y-x_{1}) dy d\eta 2^{n},
$$
implying
\begin{equation}\label{6414mm}
\norm{u}_{L^{1}(\R^{n})}\le \norm{\mathcal W(u,\psi_{0})}_{L^{1}(\RZ)} 2^{\frac{5n}4},
\end{equation}
and similarly for $p\in [1,+\io]$,
\begin{equation}\label{}
\norm{u}_{L^{p}(\R^{n})}\le \norm{\mathcal W(u,\psi_{0})}_{L^{1}(\RZ)} 2^{\frac{5n}4} p^{-\frac{n}{2p}},
\end{equation}
yielding the continuous injection of $M^{1}(\R^{n})$ into $L^{p}(\R^{n})$.
\end{rem}
\begin{theorem}
 The space $M^{1}(\R^{n})$ is a Banach algebra for convolution and  for pointwise multiplication.
\end{theorem}
\begin{proof}
Let $u, v\in M^{1}(\R^{n})$; then the convolution $u\ast v$ makes sense and belongs to all $L^{p}(\R^{n})$
for $p\in [1,+\io]$, since we have $u\in L^{1}(\R^{n})$. We calculate
$$
\mathcal W (u\ast v, \psi_{0})(x,\xi)=\int_{\R^{n}} u(y) \mathcal W(\tau_{y}v, \psi_{0})(x,\xi) dy, \quad (\tau_{y}v)(x)=v(x-y),
$$
so that 
$
\norm{\mathcal W (u\ast v, \psi_{0})}_{L^{1}(\RZ)}\le \int_{\R^{n}}\val{u(y)}
\norm{\mathcal W (\tau_{y}v, \psi_{0})}_{L^{1}(\RZ)} dy,
$
and since we have 
$$
\mathcal W (\tau_{y}v, \psi_{0})(x,\xi)=
\mathcal W (v, \tau_{y}\psi_{0})(x,\xi) e^{-4i\pi y\cdot \xi},
$$
we get
\begin{equation}\label{}
\norm{\mathcal W (u\ast v, \psi_{0})}_{L^{1}(\RZ)}\le \int_{\R^{n}}\val{u(y)}
\norm{\mathcal W (v, \tau_{y}\psi_{0})}_{L^{1}(\RZ)} dy,
\end{equation}
so that using \eqref{645lkj},
we obtain
$$
\norm{\mathcal W (u\ast v, \psi_{0})}_{L^{1}(\RZ)}\le \int_{\R^{n}}\val{u(y)}
2^{n} 
\norm{\mathcal W(\psi_{0}, \tau_{y}\psi_{0})}_{L^{1}(\RZ)}
dy
\norm{\mathcal W(v, \psi_{0})}_{L^{1}(\RZ)}.
$$
 We can check now that 
 $$
 \mathcal W(\psi_{0}, \tau_{y}\psi_{0})(x,\xi)= 2^{n} e^{-2\pi\left(\xi^{2}+(x-\frac y2)^{2}\right)} e^{2i\pi \xi y},
 $$
 so that 
 \begin{multline}\label{6417jj}
\norm{\mathcal W (u\ast v, \psi_{0})}_{L^{1}(\RZ)}\le 2^{n}\norm{u}_{L^{1}(\R^{n})}
\norm{\mathcal W (v, \psi_{0})}_{L^{1}(\RZ)}
\\
\underbrace{\le }_{\eqref{6414mm}}
2^{\frac{9n}4}
\norm{\mathcal W (u, \psi_{0})}_{L^{1}(\RZ)}
\norm{\mathcal W (v, \psi_{0})}_{L^{1}(\RZ)},
\end{multline}
proving that $M^{1}(\R^{n})$ is a Banach algebra for convolution
when equipped with the norm
\begin{equation}\label{norm49}
N(u)=2^{\frac{9n}4}\norm{\mathcal W(u,\psi_{0})}_{L^{1}(\RZ)}.
\end{equation}
On the other hand, for $u,v\in M^{1}(\R^{n})$, the pointwise product $u\cdot v$ makes sense and belongs to $L^{1}(\R^{n})$ (since both functions are in $L^{2}(\R^{n})$) and we have 
$$
u\cdot v=\mathcal C\mathcal F(\hat u\ast \hat v),
$$ 
so that 
$$
\mathcal W(u\cdot v,\psi_{0})(x,\xi)=\mathcal W\bigl(\mathcal C\mathcal F(\hat u\ast \hat v),\psi_{0}\bigr)(x,\xi)
=
\mathcal W\bigl(\mathcal F(\hat u\ast \hat v),\check\psi_{0}\bigr)(-x,-\xi),
$$
and since $\psi_{0}=\hat \psi_{0}$ is also even, we get 
\begin{multline*}
\norm{\mathcal W(u\cdot v,\psi_{0})}_{L^{1}(\RZ)}=\norm{\mathcal W\bigl(\mathcal F(\hat u\ast \hat v),\mathcal F\psi_{0}\bigr)}_{L^{1}(\RZ)}
\underbrace{=}_{\text{cf. }\eqref{segal+}}
\norm{\mathcal W\bigl(\hat u\ast \hat v,\psi_{0}\bigr)}_{L^{1}(\RZ)}
\\\underbrace{\le}_{\eqref{6417jj}} 
2^{\frac{9n}4}\norm{\mathcal W(\hat u,\hat\psi_{0})}_{L^{1}(\RZ)}
\norm{\mathcal W(\hat v,\hat \psi_{0})}_{L^{1}(\RZ)}
\\=2^{\frac{9n}4}\norm{\mathcal W(u,\psi_{0})}_{L^{1}(\RZ)}
\norm{\mathcal W(v,\psi_{0})}_{L^{1}(\RZ)},
\end{multline*}
proving as well that 
$M^{1}(\R^{n})$ is a Banach algebra for pointwise multiplication with the norm \eqref{norm49}.
\end{proof}
\subsection{Most pulses give rise to non-integrable Wigner distribution}\label{sec.meager}
In the sequel, $n$ is an integer $\ge 1$.
\begin{lem}\label{lem.616uyt}
We have  with $\psi_{0}$ given by \eqref{psi000},
 \begin{equation}\label{631gfd}
M^{1}(\R^{n})=\{u\in L^2(\R^n), \iint_{\RZ}\val{\mathcal W(u,\psi_{0})(x,\xi)} dxd\xi< +\io\}.
\end{equation}
Then $M^{1}(\R^{n})$
  is an $F_{\sigma}$ of $L^2(\R^n)$ with empty interior.
\end{lem}
\begin{proof}
\index{the space $M^{1}(\R^{n})$ is an $F_{\sigma}$ of $L^{2}(\R^{n})$}
 We have $M^{1}(\R^{n})=\cup_{N\in \N}\Phi_{N}$ with
 $$
  \Phi_{N}=\{u\in L^2(\R^n), \iint_{\RZ}\val{\mathcal W(u,\psi_{0})(x,\xi)} dxd\xi\le N\}.
 $$
 The set $\Phi_{N}$ is a closed subset of $L^2(\R^n)$ since if $(u_{k})_{k\ge 1}$ is a 
 sequence  in $\Phi_{N}$
 which 
 converges in $L^2(\R^n)$ with limit $u$, we get for $R\ge 0$,
 \begin{multline*}
  \iint_{\val{(x,\xi)}\le R}\val{\mathcal W(u,\psi_{0})(x,\xi)} dxd\xi
  \\\le  
   \iint_{\val{(x,\xi)}\le R}\val{\mathcal W(u-u_{k},\psi_{0})(x,\xi)} dxd\xi
   + \iint_{\val{(x,\xi)}\le R}\val{\mathcal W(u_{k},\psi_{0})(x,\xi)} dxd\xi
   \\
 \le \norm{u-u_{k}}_{L^2(\R^n)}(\val{\mathbb B^{2n}} R^{2n})^{1/2}+N,
\end{multline*}
implying
$
  \iint_{\val{(x,\xi)}\le R}\val{\mathcal W(u,\psi_{0})(x,\xi)} dxd\xi\le N,$ and this for any $R$, so that we obtain
$u\in \Phi_{N}$.
The interior of $\Phi_{N}$ is empty, since if it were not the case, as $\Phi_{N}$ is also convex and symmetric, $0$ would be an interior point of $\Phi_{N}$ in $L^{2}(\R^{n})$
and we would find $\rho_{0}>0$ such that
$$
\norm{u}_{L ^2(\R^n)}\le \rho_{0}\Longrightarrow  \iint_{\RZ}\val{\mathcal W(u,\psi_{0})(x,\xi)} dxd\xi\le N,
$$
and thus for any non-zero $u\in L ^2(\R^n)$, we would have 
$$
 \iint_{\RZ}\val{\mathcal W(u,\psi_{0})(x,\xi)} dxd\xi\norm{u}_{L ^2(\R^n)}^{-1}\rho_{0}\le N
 \quad
 \text{and thus}\quad
 \norm{u}_{M^1(\R^n)}\le N \rho_{0}^{-1}  \norm{u}_{L ^2(\R^n)},
$$
implying as well
$
{L ^2(\R^n)}={M ^1(\R^n)}
$
which is untrue,
thanks to the examples of Section \ref{sec.612tre},
e.g.
\eqref{gdfr44},
and this 
proves that the interior of $\Phi_{N}$ is actually empty.
Now the Baire Category Theorem implies that the $F_{\sigma}$ set $M^{1}(\R^{n})$ is a subset of $L^2(\R^n)$ with empty interior.
\end{proof}
Let us give another decomposition of the space $M^{1}(\R^{n})$.
\begin{lem}\label{lem.66hhfg}
According to Lemma \ref{lem.879kol},
we have 
 $$
M^{1}(\R^{n})=\{u\in L^2(\R^n), \iint_{\R^n\times \R^n}
 \val{\mathcal W(u,u)(x,\xi)}
  dxd\xi< +\io\}.
 $$
 Then defining
 \begin{equation}\label{lkqsaz--}
 \mathcal F_{N}=\{u\in L^2(\R^n), \iint_{\R^n\times \R^n}
 \val{\mathcal W(u,u)(x,\xi)}
  dxd\xi\le N\},
\end{equation}
each $\mathcal F_{N}$ is a closed subset of $L^{2}(\R^{n})$ with empty interior.
\end{lem}
\begin{proof}
 We have 
 $\mathscr F=M^{1}(\R^{n})=\cup_{N\in \N}\mathcal F_{N}$.
  The set $\mathcal F_{N}$ is a closed subset of $L^2(\R^n)$ since if $(u_{k})_{k\ge 1}$ is a sequence in $\mathcal F_{N}$ which converges in $L^2(\R^n)$ with limit $u$,  we have 
  $$\forall k\ge 1, \quad
   \iint_{\R^n\times \R^n}
 \val{\mathcal W(u_{k},u_{k})(x,\xi)}
  dxd\xi\le N,
  $$
  so that  we may apply Lemma \ref{lem.62llkk} with $C_{0}=N$,
  and readily get that $u$ belongs to $\mathcal F_{N}$.
  We have also that $
  \text{interior}_{L^{2}(\R^{n})}
  \left(\mathcal F_{N}\right)
  \subset \text{interior}_{L^{2}(\R^{n})}\left(M^{1}(\R^{n})\right)=\emptyset$.
\end{proof}
\begin{theorem}\label{thm.67kjhg}
Defining \begin{equation}\label{67kjhg}
\mathscr G=
 \{u\in L^2(\R^n), \iint_{\R^n\times \R^n}
 \val{\mathcal W(u,u)(x,\xi)}
  dxd\xi= +\io\}=L^{2}(\R^{n})\backslash M^{1}(\R^{n}),
\end{equation}
we obtain that the set $\mathscr G$ is a dense $G_{\delta}$ subset of $L^2(\R^n)$.
\end{theorem}
\begin{proof}
 It follows immediately from Lemma \ref{lem.66hhfg}
 and formula
 $
 \bigl\{\mathring A\bigr\}^{c}=\overline{A^{c}},
 $
 yielding for $\mathcal F_{N}$ defined in \eqref{lkqsaz--},
 $
 L^{2}(\R^{n})=\bigl\{\text{interior} ({\cup_{\N}\mathcal F_{N})}\bigr\}^{c}=\overline{\cap_{\N} \mathcal F_{N}^{c}}.
 $
\end{proof}
\index{the space $M^{1}(\R^{n})$ is a dual space}
\begin{rem}\rm
 It is interesting to note that the space $M^{1}(\R^{n})$ is not reflexive,
 as it can be identified to $\ell^{1}$ via Wilson bases, but it is a dual space.
 It turns out that both properties are linked to the fact that $M^{1}(\R^{n})$ is an $F_{\sigma}$ of $L^{2}(\R^{n})$
 as proven by Lemmas \ref{lem.616uyt} and \ref{lem.66hhfg}:
 if $\mathbb X$ is a reflexive Banach space continuously included in a Hilbert space $\mathbb H$,
 it is always an $F_{\sigma}$ of $\mathbb H$, since we may write 
 $$
 \mathbb X=\cup_{N\in \N} N\mathbf B_{\mathbb X},
 $$
 where $\mathbf B_{\mathbb X}
$
is the closed unit ball of $\mathbb X$ and $N\mathbf B_{\mathbb X}$ is $\mathbb H$-closed
since it is weakly compact (for the topology $\sigma(\mathbb H, \mathbb H)$);
we cannot use that abstract argument in the case of the non-reflexive 
$M^{1}(\R^{n})$,
so we produced a direct elementary proof above.
Also it can be proven that  
 if $\mathbb X$ is a Banach space continuously included in a Hilbert space $\mathbb H$,
 so that $\mathbb X$ is an $F_{\sigma}$ of $\mathbb H$, then $\mathbb X$ must have a predual.
As a result, the fact that $M^{1}(\R^{n})$ has a predual
appears as a consequence of the fact that $M^{1}(\R^{n})$
is an $F_{\sigma}$ of $L^{2}(\R^{n})$.
\end{rem}
\subsection{Consequences on integrals of the Wigner distribution}\label{sec.jhae22}
\begin{lem}\label{lem.hfgap9}
 Let $\mathscr G$ be defined in \eqref{67kjhg} and let $u\in \mathscr G$. Then the positive and negative part of the real-valued $\mathcal W(u,u)$ are such that
 \begin{equation}\label{}
\iint \mathcal W(u,u)_{+}(x,\xi) dx d\xi=\iint \mathcal W(u,u)_{-}(x,\xi) dx d\xi=+\io.
\end{equation}
\end{lem}
\begin{proof}
 For $h\in (0,1]$, we define the symbol 
 \begin{equation}\label{}
a(x,\xi, h)= e^{-h(x^{2}+\xi^{2})},
\end{equation}
and we see that it is a semi-classical symbol in the sense \eqref{semi00++}.
Let us start a {\it reductio ad absurdum} and assume that 
$\iint \mathcal W(u,u)_{-}(x,\xi) dx d\xi<+\io,$
(which implies since $u\in \mathcal G$,
$\iint \mathcal W(u,u)_{+}(x,\xi) dx d\xi=+\io$).
We note that 
$$
\poscal{\opw{a(x,\xi, h)} u}{u}_{L^{2}(\R^{n})}=\iint \underbrace{a(x,\xi, h)}_{\in L^{2}(\RZ)} 
\underbrace{\mathcal W(u,u)(x,\xi)}_{\in L^{2}(\RZ)}  dx d\xi,
$$
and thanks to Theorem \ref{thm.117jhg} we have also 
$$
\sup_{h\in (0,1]}\val{\poscal{\opw{a(x,\xi, h)} u}{u}_{L^{2}(\R^{n})}}\le \sigma_{n}\norm{u}^{2}_{L^{2}(\R^{n})},
$$
so that 
\begin{multline*}
\iint {e^{-h(x^{2}+\xi^{2})}}
{\mathcal W(u,u)(x,\xi)}dx d\xi
+\iint {e^{-h(x^{2}+\xi^{2})}}
{\mathcal W(u,u)_{-}(x,\xi)}dx d\xi
\\
=\iint {e^{-h(x^{2}+\xi^{2})}}
{\mathcal W(u,u)_{+}(x,\xi)}dx d\xi,
\end{multline*}
and thus with $\theta_{h}\in [-1,1]$, we have 
\begin{multline}\label{ta8chq}
\theta_{h} \sigma_{n}\norm{u}^{2}_{L^{2}(\R^{n})}
+\iint {e^{-h(x^{2}+\xi^{2})}}
{\mathcal W(u,u)_{-}(x,\xi)}dx d\xi
\\=\iint {e^{-h(x^{2}+\xi^{2})}}
{\mathcal W(u,u)_{+}(x,\xi)}dx d\xi.
\end{multline}
Choosing $h=1/m, m\in \N^{*}$, we note that 
$$
 {e^{-\frac 1m(x^{2}+\xi^{2})}}
{\mathcal W(u,u)_{+}(x,\xi)}\le {e^{-\frac 1{m+1}(x^{2}+\xi^{2})}}
{\mathcal W(u,u)_{+}(x,\xi)}.
$$
From the Beppo-Levi Theorem (see e.g. Theorem 1.6.1 in \cite{MR3309446}) we get that 
$$
\lim_{m\rightarrow+\io}
\iint {e^{-\frac1m(x^{2}+\xi^{2})}}
{\mathcal W(u,u)_{+}(x,\xi)}dx d\xi=\iint
{\mathcal W(u,u)_{+}(x,\xi)}dx d\xi=+\io.
$$
However the left-hand-side of  \eqref{ta8chq} is bounded above by
$$
\sigma_{n}\norm{u}^{2}_{L^{2}(\R^{n})}+\iint 
{\mathcal W(u,u)_{-}(x,\xi)}dx d\xi, \quad \text{which is finite,}
$$
triggering a contradiction.
We may now study the case where 
$$\iint \mathcal W(u,u)_{+}(x,\xi) dx d\xi<+\io,\quad
\iint \mathcal W(u,u)_{-}(x,\xi) dx d\xi=+\io.
$$
The identity \eqref{ta8chq}  still holds true 
with a left-hand-side going to $+\io$ when $h$ goes to 0 whereas the right-hand side is bounded.
This concludes the proof of the lemma.
\end{proof}
\begin{nb}
 \rm A shorter \emph{heuristic} argument would be that the identity  $\iint \mathcal W(u,u)(x,\xi) dx d\xi=\norm{u}_{L^{2}(\R^{n})}^{2}$ and 
 $\iint \val{\mathcal W(u,u)(x,\xi) }dx d\xi=+\io$ should imply the lemma, but the former integral is not absolutely converging, so that argument fails to be completely convincing
 since we need to give a meaning to the first integral.
\end{nb}
\index{the set $L^{2}(\R^{n})\backslash M^{1}(\R^{n})$ is a dense $G_{\delta}$ set of $L^{2}(\R^{n})$}
\begin{theorem}\label{thm.69gapp}
Defining
$
\mathscr G=L^{2}(\R^{n})\backslash M^{1}(\R^{n})
$
 (cf. \eqref{67kjhg})
  we find that the set $\mathscr G$ is a dense $G_{\delta}$ set in $L^{2}(\R^{n})$ and for all 
  $u\in \mathscr G$, we have\footnote{Note that $\mathcal W(u,u)$ is real-valued.} 
   \begin{align}\label{fapo66}
&\iint \mathcal W(u,u)_{+}(x,\xi) dx d\xi=\iint \mathcal W(u,u)_{-}(x,\xi) dx d\xi=+\io,
\end{align}
Defining\footnote{
Thanks to Theorem \ref{thm776655}, the function $\mathcal W(u,u)$ is a continuous function, so it makes sense to consider its pointwise values. } 
\begin{equation}\label{645yrp}
E_{\pm}(u)=\{(x,\xi)\in \RZ, \pm\mathcal W(u,u)(x,\xi)>0  \},
\end{equation}
we have for all $u\in \mathscr G$,
\begin{equation}\label{ats646} 
\iint_{E_{\pm}(u)} \mathcal W(u,u)(x,\xi) dx d\xi=\pm\io,
\end{equation}
and both sets $E_{\pm}(u)$ are open subsets of $\RZ$ with infinite Lebesgue measure.
\end{theorem}
\begin{proof}
 The first statements  follow from Theorem
 \ref{thm.67kjhg} and Lemma \ref{lem.hfgap9}. As far as \eqref{ats646} is concerned, we note that $\mathcal W(u,u) >0$
 (resp. $<0$) on $E_{+}(u)$ (resp. $E_{-}(u)$), so that 
 Theorem \ref{thm.67kjhg} implies   \eqref{ats646}.
 Moreover
  $E_{\pm}(u)$ are open subsets of $\RZ$ since, thanks to Theorem \ref{thm776655},
the function $\mathcal W(u,u)$ is continuous; also,  both subsets have infinite Lebesgue measure from 
\eqref{fapo66} since $\mathcal W(u,u)$ belongs to $L^{2}(\RZ)$.
\end{proof}
\begin{rem}\rm 
 There are many other interesting properties and generalizations of the space $M^{1}$
 and in particular a close link between the Bargmann transform, the Fock spaces and modulation spaces:
 we refer the reader to Remark 5 on page 243 in Section 11.4 of \cite{MR1843717},
 to our Section \ref{sec.nonneg} in this article and to Section 2.4 of 
 \cite{MR2599384}.
\end{rem}
 \begin{rem}\rm 
 As a consequence of the previous theorem, we could say that for any \emph{generic} $u$ in $L^{2}(\R^{n})$
 (i.e. any  $u\in \mathscr G=L^{2}(\R^{n})\backslash M^{1}(\R^{n})$), we can find  open sets $E_{+}, E_{-}$ such that  the real-valued $\pm\mathcal W(u,u)$ is positive on $E_{\pm}$ and 
 $$
 \iint_{E_{\pm}}\mathcal W(u,u)(x,\xi) dx d\xi=\pm\io.
 $$
 We shall see in the next section some results on polygons in the plane and for instance, we shall be able to prove that there exists a ``universal number''
 $\mu_{3}^{+}>1$ such that 
 for any triangle\footnote{We define a triangle as the intersection of three half-planes, which includes  of course the convex envelope of three points, but also the  set with infinite area
 $\{(x,\xi)\in \R^{2},  x\ge 0, \xi\ge 0, x+\xi\ge \lambda\}$ for some $\lambda>0$. \label{fn.triangle}} $\mathcal T$ in the plane, we have 
 \begin{equation}\label{ga51fg}\forall u\in L^{2}(\R),\quad
\iint_{\mathcal T} \mathcal W(u,u)(x,\xi) dx d\xi\le \mu_{3}^{+}\norm{u}_{L^{2}(\R)}^{2}.
\end{equation}
Note in particular  that we will show that \eqref{ga51fg} holds true regardless of the area of the triangle (which could be infinite according to our definition of a triangle in our footnote \ref{fn.triangle}).
 Although that type of result may look pretty weak, it gets enhanced by Theorem \ref{thm.69gapp}
 which proves that no triangle in the plane could be a set $E_{+}(u)$ (cf. \eqref{645yrp}) for a generic $u$ in $L^{2}(\R)$.
\end{rem}
\section{Convex polygons of the plane}
\subsection{Convex Cones}\label{sec.cvx345}
\index{convex cones}
We have seen  in Proposition \ref{pro.521kjs} and Theorem  \ref{thm.59}
that the self-adjoint bounded operator with Weyl symbol
$H(x) H(\xi)$ does satisfy
\begin{multline}\label{711aze}
\mu_{2}^{-}=m_{0}= \lambda_{\text{min}}\bigl(
\opw{H(x) H(\xi)}\bigr)
\le \opw{H(x) H(\xi)}
\\\le \lambda_{\text{max}}\bigl(\opw{H(x) H(\xi)}\bigr)= M_{0}=\mu_{2}^{+},
\end{multline}
\begin{gather}
[\mu_{2}^{-}, \mu_{2}^{+}]=\text{spectrum}\left(\opw{H(x) H(\xi)}\right),
\end{gather}
with
\begin{equation}\label{713ezj}
\mu_{2}^{-}\approx -0.155939843191243,\qquad \mu_{2}^{+}\approx1.00767997007003.
\end{equation}
This result is true as well for the  characteristic function of any convex cone (which is not a half-plane nor the full plane) in the plane since we can map it to the quarter plane by a transformation in $\text{Sl}(2,\R)=\text{Sp}(1,\R)$.
On the other hand a concave cone is the complement of a convex cone and the diagonalization offered by Theorem \ref{thm.calculs} proves that the spectrum of the Weyl quantization of the indicatrix of a concave cone is $$
1-\text{Spectrum}\left( \opw{H(x) H(\xi)}\right).
$$
We may sum-up the situation by the following theorem.
\begin{theorem}\label{thm.71aqsd}
 Let $\Sigma_{\theta}$ be a convex cone in $\R^{2}$ with aperture $\theta\in [0,2\pi]$ (cf. \eqref{539iua})
 and let $\mathcal A_{\theta}$ be the self-adjoint bounded operator with the indicator function of $\Sigma_{\theta}$ as a Weyl symbol.
 \begin{enumerate}
 \item If $\theta=0$, we have $\mathcal A_{\theta}=0$.
 \item If $\theta\in (0,\pi)$, the operator $\mathcal A_{\theta}$ is unitarily equivalent to $
 \opw{H(x) H(\xi)}$, thus with spectrum
 $[\mu_{2}^{-}, \mu_{2}^{+}]$ with $\mu_{2}^{-}<0<1<\mu_{2}^{+}$
 as given in Theorem \ref{thm.59}. 
  \item If $\theta=\pi$, $\Sigma_{\pi}$ is a half-space and $\mathcal A_{\pi}$ is a proper orthogonal projection, thus with spectrum $\{0,1\}$.
   \item If $\theta\in(\pi, 2\pi)$, $\Sigma_{\theta}$ is a concave cone  and the operator $\mathcal A_{\theta}$ is unitarily equivalent to $\Id-\opw{H(x) H(\xi)}$, thus with spectrum $[1-\mu_{2}^{+}, 1-\mu_{2}^{-}]$.\footnote{So that  we have in particular,
   from (2), the inequalities 
   $1-\mu_{2}^{+}<0<1<1-\mu_{2}^{-}$.}
  \item If $\theta=2\pi$, we have $\mathcal A_{2\pi}=\Id$.
  \end{enumerate}
\end{theorem}
\begin{rem}\label{rem.72poiu}
It is only in   the trivial cases $\theta\in \{0,\pi, 2\pi \}$ that
$A_{\theta}$ is an orthogonal projection. These cases are also characterized (among cones) by the fact that the spectrum of $\mathcal A_{\theta}$ is included in $[0,1]$.
\end{rem}
\begin{rem}\rm 
 It is interesting to remark that 
 all operators $\mathcal A_{\theta}$ for $\theta\in (0,\pi)$ are unitarily equivalent and thus with constant spectrum $[\mu_{2}^{-}, \mu_{2}^{+}]$ as given in Theorem \ref{thm.59}. Nevertheless the sequence $(\mathcal A_{\theta})_{0<\theta<\pi}$ is weakly converging to the orthogonal projection $\mathcal A_{\pi}$ whose spectrum is $\{0,1\}$:
 indeed for $\phi \in \mathscr S(\R)$, $\psi\in \mathscr S(\R)$, we have 
 $$
 \poscal{\mathcal A_{\theta}\phi}{\psi}_{L^{2}(\R)}=\iint_{\Sigma_{\theta}}\underbrace{\mathcal W(\phi, \psi)}_{\in \mathscr S(\R^{2})}(x,\xi) dx d\xi,
 $$
 and thus the Lebesgue Dominated Convergence Theorem implies that 
 \begin{equation}\label{714erz}
 \lim_{\theta\rightarrow \pi_{-}} \poscal{\mathcal A_{\theta}\phi}{\psi}_{L^{2}(\R)}= \poscal{\mathcal A_{\pi}\phi}{\psi}_{L^{2}(\R)}.
\end{equation}
 On the other hand for $u, v\in L^{2}(\R)$ and sequences  $(\phi_{k})_{k\ge 1}, (\psi_{k})_{k\ge 1}$
  in $\mathscr S(\R)$ with respective limits $u, v$
  in $L^{2}(\R)$,
  we have 
$$
 \poscal{\mathcal A_{\theta} u}{v}_{L^{2}(\R)}=
 \poscal{\mathcal A_{\theta} (u-\phi_{k})}{v}_{L^{2}(\R)}
 +\poscal{\mathcal A_{\theta} \phi_{k}}{v-\psi_{k}}_{L^{2}(\R)}
 +\poscal{\mathcal A_{\theta} \phi_{k}}{\psi_{k}}_{L^{2}(\R)},
$$
so that 
\begin{multline*}
\poscal{\mathcal A_{\theta} u}{v}_{L^{2}(\R)}-\poscal{\mathcal A_{\pi} u}{v}_{L^{2}(\R)}
\\=
 \poscal{\mathcal A_{\theta} (u-\phi_{k})}{v}_{L^{2}(\R)}
 +\poscal{\mathcal A_{\theta} \phi_{k}}{v-\psi_{k}}_{L^{2}(\R)}
 +\poscal{\mathcal A_{\theta} \phi_{k}}{\psi_{k}}_{L^{2}(\R)},
 \\
 -\poscal{\mathcal A_{\pi} (u-\phi_{k})}{v}_{L^{2}(\R)}
 -\poscal{\mathcal A_{\pi} \phi_{k}}{v-\psi_{k}}_{L^{2}(\R)}
 -\poscal{\mathcal A_{\pi} \phi_{k}}{\psi_{k}}_{L^{2}(\R)},
\end{multline*}
implying
\begin{multline*}\label{}
\val{\poscal{\mathcal A_{\theta} u}{v}_{L^{2}(\R)}-\poscal{\mathcal A_{\pi} u}{v}_{L^{2}(\R)}}
\\
\le (\mu_{2}^{+}+1)\bigl(\norm{u-\phi_{k}}_{L^{2}(\R)}\norm{v}_{L^{2}(\R)}
+\norm{v-\psi_{k}}_{L^{2}(\R)}\norm{\phi_{k}}_{L^{2}(\R)}
\bigr)
\\
+\val{\poscal{\mathcal A_{\theta} \phi_{k}}{\psi_{k}}_{L^{2}(\R)}-\poscal{\mathcal A_{\pi} \phi_{k}}{\psi_{k}}_{L^{2}(\R)}},
\end{multline*}
and thus, using \eqref{714erz}, we get 
\begin{multline*}
\limsup_{\theta\rightarrow 0_{+}}\val{\poscal{\mathcal A_{\theta} u}{v}_{L^{2}(\R)}-\poscal{\mathcal A_{\pi} u}{v}_{L^{2}(\R)}}
\\\le
 (\mu_{2}^{+}+1)\bigl(\norm{u-\phi_{k}}_{L^{2}(\R)}\norm{v}_{L^{2}(\R)}
+\norm{v-\psi_{k}}_{L^{2}(\R)}\norm{\phi_{k}}_{L^{2}(\R)}
\bigr).
\end{multline*}
Taking now the infimum with respect to $k$ of the right-hand-side in  the above inequality,
we obtain indeed the weak convergence
\begin{equation}
\lim_{\theta\rightarrow 0_{+}}
\poscal{\mathcal A_{\theta} u}{v}_{L^{2}(\R)}=\poscal{\mathcal A_{\pi} u}{v}_{L^{2}(\R)}.
\end{equation}
Of course we cannot have strong convergence of the bounded self-adjoint $\mathcal A_{\theta}$ towards (the bounded self-adjoint) $A_{\pi}$ because of their respective spectra
and the same lines can be written on the weak limit $0$ when $\theta\rightarrow 0_{+}$ of 
$\mathcal A_{\theta}$.
 \end{rem}
\subsection{Triangles}
\index{triangles}
We may consider general ``triangles'' in the plane that we define as
\begin{equation}
\mathcal T_{L_{1},L_{2},L_{3}}^{c_{1},c_{2},c_{3}}=
\bigl\{(x,\xi)\in \R^{2}, L_{j}(x,\xi)\ge c_{j}, j\in\{1,2,3\}\bigr\},
\end{equation}
$c_{j}$ are real numbers and $L_{j}$ are linear forms. To avoid degenerate situations, we shall assume that 
\begin{equation}\label{tri45-}
\text{for $j\not=k$}, \quad dL_{j}\wedge dL_{k}\not=0,\quad
\val{\mathcal T_{L_{1},L_{2},L_{3}}^{c_{1},c_{2},c_{3}}}>0
\text{ \quad and $\mathcal T_{L_{1},L_{2},L_{3}}^{c_{1},c_{2},c_{3}}$ is not a cone}.
\end{equation}
Note that this includes standard triangles (convex envelope of three non-colinear points) but also sets with infinite area such as 
\begin{equation}\label{areai}
\{(x,\xi)\in \R^{2}, x\ge 0, \xi\ge 0, x+\xi\ge \lambda\}, \quad\text{where $\lambda$ is a positive parameter.}
\end{equation}
Without loss of generality, we may assume that $L_{1}(x,\xi)-c_{1}=x, L_{2}(x,\xi) -c_{2}=\xi$,
so that 
$$
\mathcal T_{L_{1},L_{2},L_{3}}^{c_{1},c_{2},c_{3}}=\{(x,\xi)\in \R^{2}, x\ge 0, \xi\ge 0, ax +b \xi\ge \nu\},
$$
where $a,b,\lambda$ are real parameters with $a\not=0, b\not =0$ from the assumption \eqref{tri45-}; using the symplectic mapping
$(x,\xi)\mapsto (\mu x, \xi/\mu)$ with $\mu=\sqrt{\val{b/a}}$, we see that the condition $ax+b\xi\ge \nu$ becomes
$$
x\sign a +\xi\sign b \ge \lambda=\nu/\sqrt{\val{ab}}\quad\text{i.e\quad}
\begin{cases}
x+\xi&\ge  \tilde \nu,
 \\
x-\xi&\ge  \tilde \nu,
 \\
-x+\xi &\ge  \tilde \nu,
 \\
-x-\xi&\ge\tilde \nu.
\end{cases}
$$
The first case requires $\tilde \nu >0$ and the other cases
$\tilde \nu<0$. The only case with finite area is the fourth case
\begin{equation}\label{trifin}\mathcal T_{4,\lambda}=
\{(x,\xi)\in \R^{2}, x\ge 0, \xi\ge 0, x+\xi\le \lambda\} \ \text{ triangle with area $\lambda^{2}/2$, $\lambda>0$.}
\end{equation}
The second case is 
\begin{equation}\label{cas222}
\mathcal T_{2,\lambda}=
\{(x,\xi)\in \R^{2}, x\ge 0, \xi\ge 0, x-\xi\ge -\lambda\},\quad \lambda>0,
\end{equation}
The third case is 
\begin{equation}\label{cas333}\mathcal T_{3,\lambda}=
\{(x,\xi)\in \R^{2}, x\ge 0, \xi\ge 0, \xi-x\ge -\lambda\},\quad \lambda>0,
\end{equation}
and the first case is 
\begin{equation}\label{726+++}\mathcal T_{1,\lambda}=
\{(x,\xi)\in \R^{2}, x\ge 0, \xi\ge 0, \xi+x\ge \lambda\},\quad \lambda>0.
\end{equation}
\begin{pro}\label{pro.jkgh++}
Let  $\mathcal T_{4,\lambda}$ be a triangle with finite non-zero area in the plane given by \eqref{trifin},
where $\lambda$ is a positive parameter.
Then the operator $\opw{\mathbf 1_{\mathcal T_{4,\lambda}}}$ is unitarily equivalent to the operator with kernel 
\begin{equation}\label{727hgf}
\tilde k_{4,\lambda}(x,y)=\mathbf 1_{[0,\lambda]}\bigl(\frac{x+y}2\bigr)
\frac{\sin\bigl(\pi(x-y)(\lambda-\frac{x+y}2)\bigr)}{\pi (x-y)}.
\end{equation}
The operator $\opw{\mathbf 1_{\mathcal T_{4,\lambda}}}$ is self-adjoint and bounded on $L^{2}(\R)$ so that 
\begin{equation}\label{728mlk}
\norm{\opw{\mathbf 1_{\mathcal T_{4,\lambda}}}}_{\mathcal B(L^{2}(\R))}\le \frac12\left(
\mu_{2}^{+}+\sqrt{1+(\mu_{2}^{+})^{2}}\right):=\tilde \mu_{3},
\end{equation}
where $\mu_{2}^{+}$ is given in \eqref{711aze}.
\end{pro}

\begin{proof}
 The kernel $k_{4,\lambda}$ of $\opw{\mathbf 1_{\mathcal T_{4,\lambda}}}$  is such that 
 \begin{multline*}
 k_{4,\lambda}(x,y)= \mathbf 1_{[0,\lambda]}\bigl(\frac{x+y}2\bigr)\int_{0}^{\lambda-\frac {x+y}2}
 e^{2i\pi (x-y)\xi}d\xi
\\ =\mathbf 1_{[0,\lambda]}\bigl(\frac{x+y}2\bigr)
 \frac{\bigl(e^{2i\pi (x-y)(\lambda-\frac{x+y}2) }-1\bigr)}{2i\pi (x-y)}
 \\
 =
 e^{i\pi (\lambda x-\frac{x^{2}}2)}\mathbf 1_{[0,\lambda]}\bigl(\frac{x+y}2\bigr)
 \frac{\sin(\pi (x-y)(\lambda-\frac{x+y}2))}{\pi (x-y)}
 e^{-i\pi( \lambda y-\frac{y^{2}}2)},
\end{multline*}
proving \eqref{727hgf}.
 We note now that the kernel of the operator with Weyl symbol $H(\xi) H(\lambda-\xi-x)$
 is 
 \begin{equation}
 \ell_{\lambda}(x,y)=
 e^{i\pi (\lambda x-\frac{x^{2}}2)}
H(\lambda-\frac{x+y}2)\frac{\sin(\pi (x-y)(\lambda-\frac{x+y}2))}{\pi (x-y)}
e^{-i\pi (\lambda y-\frac{y^{2}}2)},
\end{equation}
and that 
$$
\opw{H(\xi) H(\lambda-\xi-x)}
$$
is unitarily equivalent to  the operator $\opw{H(x) H(\xi)}$ as given by Theorem \ref{thm.71aqsd}.
We get then 
 \begin{multline}
k_{4,\lambda}(x,y)
= H(x+y)\ell_{\lambda}(x,y)
=H(x) \ell_{\lambda}(x,y)H(y)
\\+H(x+y)\bigl(H(x) \check H(y)+\check H(x) H(y)\bigr)
H(\lambda-\frac{x+y}2)\\
\times\frac{\sin(\pi (x-y)(\lambda-\frac{x+y}2))}{\pi (x-y)}
\times
e^{i\pi (\lambda x-\frac{x^{2}}2)}
e^{-i\pi (\lambda y-\frac{y^{2}}2)},
\end{multline}
and we have thus
\begin{equation}
\opw{\mathbf 1_{\mathcal T_{4,\lambda}}}=
H\OPW{H(\xi) H(\lambda-\xi-x)} H +\Omega_{\lambda},
\end{equation}
where the kernel $\omega_{\lambda}(x,y)$ of the operator $\Omega_{\lambda}$ verifies
\begin{multline*}
\val{\omega_{\lambda}(x,y)}\le \frac{H(x+y)\bigl(H(x) \check H(y)+\check H(x) H(y)\bigr)}
{\pi\val{x-y}}
\\=
 \frac{H(x+y)\bigl(H(x) \check H(y)+\check H(x) H(y)\bigr)}
{\pi(\val{x}+\val {y})}.
\end{multline*}
 We obtain thanks to Proposition
\ref{pro.hardy} [2] that 
\begin{equation}\label{}
\iint \val{\omega_{\lambda}(x,y)}
\val{u(y)}\val{u(x)} dy dx\le  \norm{ \check H u}_{L^{2}(\R)}\norm{ H u}_{L^{2}(\R)}.
\end{equation}
 As a result, we find that 
 $$
\val{\poscal{ \opw{\mathbf 1_{\mathcal T_{4,\lambda}}} u}{u}_{L^{2}(\R)}}\le \mu_{2}^{+}\norm{Hu}^{2}_{L^{2}(\R)}
+\norm{ \check H u}_{L^{2}(\R)}\norm{ H u}_{L^{2}(\R)},
 $$
 proving \eqref{728mlk}.
 \end{proof}
 \begin{pro}\label{pro.jkg+++}
Let  $\mathcal T_{1,\lambda}$ be a triangle with infinite area in the plane given by \eqref{726+++},
where $\lambda$ is a positive parameter.
Then the operator $\opw{\mathbf 1_{\mathcal T_{1,\lambda}}}$ is unitarily equivalent to the operator with kernel 
\begin{equation}\label{727h++}
\tilde k_{1,\lambda}(x,y)=\mathbf 1_{[0,\lambda]}\bigl(\frac{x+y}2\bigr)
\frac{\sin\bigl(\pi(x-y)(\lambda-\frac{x+y}2)\bigr)}{\pi (x-y)}.
\end{equation}
The operator $\opw{\mathbf 1_{\mathcal T_{1,\lambda}}}$ is self-adjoint and bounded on $L^{2}(\R)$ so that 
\begin{equation}\label{728m++}
\norm{\opw{\mathbf 1_{\mathcal T_{1,\lambda}}}}_{\mathcal B(L^{2}(\R))}\le \frac12\left(
\mu_{2}^{+}+\sqrt{\frac14+(\mu_{2}^{+})^{2}}\right)\approx
1.066294188078,
\end{equation}
where $\mu_{2}^{+}$ is given in \eqref{711aze}.
\end{pro}
\begin{proof}
The kernel $k_{1,\lambda}$
of $\opw{\mathbf 1_{\mathcal T_{1,\lambda}}}$  is such that 
\begin{align*}
k_{1,\lambda}(x,y)&=H(x+y)
e^{2i\pi(x-y)\max(0, \lambda-\frac{x+y}2)}
\frac12\Bigl(\delta_{0}(y-x)+\frac{1}{i\pi(y-x)}\Bigr)
\\
&=
H(x)\frac12\delta_{0}(x-y) H(y)
+H(x)\frac{e^{2i\pi(x-y)\max(0, \lambda-\frac{x+y}2)}}{2i\pi(y-x)}H(y)
\\
&\hskip33pt+H(x+y)\bigl(H(x) \check H(y)+
\check H(x)  H(y)\bigr)\frac{e^{2i\pi(x-y)\max(0, \lambda-\frac{x+y}2)}}{2i\pi(y-x)}.
\end{align*}
 We note that the kernel of the operator
 $\opw{H(x+\xi-\lambda)H(\xi)}$ is 
 $$
 \ell_{1}(x,y)=e^{2i\pi(x-y)\max(0, \lambda-\frac{x+y}2)}
\frac12\Bigl(\delta_{0}(y-x)+\frac{1}{i\pi(y-x)}\Bigr),
 $$
 so that 
 \begin{equation}\label{765rez} 
 \opw{\mathbf 1_{\mathcal T_{1,\lambda}}}=
 H\underbrace{\opw{H(x+\xi-\lambda)H(\xi)}}_{
 \substack{
 \text{unitarily equivalent to}\\\opw{H(x)H(\xi)}
 }}H +\Omega_{1,\lambda},
\end{equation}
 where the kernel $\omega_{1,\lambda}$
 of the operator $\Omega_{1,\lambda}$ is such that
 $$
 \val{\omega_{1,\lambda}(x,y)}\le H(x+y)\frac{\bigl(H(x) \check H(y)+
\check H(x)  H(y)\bigr)}{2\pi(\val{x}+\val{y})},
 $$
 and, thanks to Proposition
\ref{pro.hardy} [2], we get from \eqref{765rez}  that 
 $$
\val{\poscal{ \opw{\mathbf 1_{\mathcal T_{1,\lambda}}} u}{u}_{L^{2}(\R)}}\le \mu_{2}^{+}\norm{Hu}^{2}_{L^{2}(\R)}
+\frac12\norm{ \check H u}_{L^{2}(\R)}\norm{ H u}_{L^{2}(\R)},
 $$
which gives \eqref{728m++}.
\end{proof}
We leave for the reader to check the two other cases \eqref{cas222}, \eqref{cas333},
which are very similar as well as the degenerate cases excluded by \eqref{tri45-}, which are in fact easier to tackle. 
\begin{theorem}\label{thm.75fdgy}
Let $\mathscr T=\bigl\{\mathcal T_{L_{1},L_{2},L_{3}}^{c_{1},c_{2},c_{3}}\bigr\}_{\substack{c_{j}\in \R,\ L_{j}\\\text{linear form on $\R^{2}$}
}}$ be the set of triangles of $\R^{2}$.
For all $\mathcal T\in \mathscr T$,
the operator $\opw{\mathbf 1_{\mathcal T}}$ is bounded on $L^{2}(\R)$, self-adjoint and 
we have
\begin{multline}
1.007680\approx\mu_{2}^{+}
=\sup_{\mathcal C\text{ cone}}\norm{\opw{\mathbf 1_{\mathcal C}}}_{\mathcal B(L^{2}(\R))}
\\\le \mu_{3}^{+}=
\sup_{\mathcal T\text{ triangle}}\norm{\opw{\mathbf 1_{\mathcal T}}}_{\mathcal B(L^{2}(\R))}
\le \tilde \mu_{3}\approx
1.213668.
\end{multline}
\end{theorem}
\begin{nb}\rm
 The $L^{2}$ boundedness is easy to prove  since it is obvious for triangles with finite areas and in the  case of triangles with infinite area, we may note that in the case \eqref{726+++} (resp. \eqref{cas222},
 \eqref{cas333}) they are the union of two cones (resp. one cone) with a strip $[0,1]\times \R_{+}$.
 What matters most in the above statement is the effective explicit bound.
Our result does not give an explicit value for $\mu_{3}^{+}$ and it is quite likely that the bound given by $\tilde \mu_{3}$ is way too large.
\end{nb}
\begin{proof}
 The second inequality is proven in Propositions \ref{pro.jkgh++} \& \ref{pro.jkg+++} and the first inequality is a consequence of Theorem \ref{thm.517ytu}.
\end{proof}
\begin{rem}
 This implies that  for any $u\in L^{2}(\R)$ and any $\mathcal T\in \mathscr T$, we have 
 \begin{equation}
\Bigl\vert{\iint_{\mathcal T}\mathcal W(u,u)(x,\xi) dx d\xi}\Bigr\vert \le \tilde \mu_{3}\norm{u}_{L^{2}(\R)}^{2},\qquad
\text{with
$ \tilde \mu_{3}\approx
1.213668$}.
\end{equation}
\end{rem}
 \subsection{Convex Polygons}
 \index{convex polygons}
We want to tackle now the general case of a convex polygon in the plane. We consider
$$
L_{1}, \dots, L_{N},
$$
to be $N$  linear forms of $x,\xi$ ($L_{j}(x,\xi)=a_{j} \xi-\alpha_{j}x=[(x,\xi);(a_{j}, \alpha_{j})]$) and $c_{1}, \dots, c_{N}$ some real constants.
We consider the convex polygon
\begin{equation}\label{731uyt}
\mathcal P=\{(x,\xi) \in \R^{2}, \ \forall j\in \{1,\dots, N\}, L_{j}(x,\xi)-c_{j}\ge 0\},
\end{equation}
so that 
$$
\mathbf 1_{\mathcal P}(x,\xi)=\prod_{1\le j\le N}H\bigl(L_{j}(x,\xi)-c_{j}\bigr).
$$
\begin{defi}\label{def.77hgfd}
Let $N\in \N^{*}$, let $L_{1}, \dots L_{N}$ be linear forms on $\R^{2}$ and let $c_{1},\dots, c_{N}$ be real numbers. The polygon with $N$ sides $\mathcal P^{c_{1},\dots, c_{N}}_{L_{1}, \dots, L_{N}}$ is defined by \eqref{731uyt}. We shall denote by $\mathscr P_{N}$ the set of all polygons with $N$ sides.
\end{defi}
\begin{nb}
 Since we may take some $L_{j}=0$ in \eqref{731uyt}, we see that $\mathscr P_{N}\subset \mathscr P_{N+1}$.
\end{nb}
Note as above that it includes some convex subsets of the plane with infinite area such as \eqref{areai}.
\begin{theorem}\label{thm.77ytre}
Let $\mathscr P_{N}$ be  the set  of convex polygons with $N$ sides of the plane $\R^{2}$. We define 
\begin{equation}
\mu_{N}^{+}=
\sup_{\mathcal P\in \mathscr P_{N}}\norm{\opw{\mathbf 1_{\mathcal P}}}_{\mathcal B(L^{2}(\R))}
\end{equation}
Then $\mu_{2}^{+}$ is given by Theorem \ref{thm.59} and 
\begin{equation}
\forall N\ge 3, \quad \mu_{N}^{+}\le \sqrt{N/2}.
\end{equation}
\end{theorem}
\begin{proof}
 Using an affine  symplectic transformation, we may assume that $L_{N}(x,\xi)-c_{N}=x$,
so that 
$$
\mathbf 1_{\mathcal P}(x,\xi)=H(x)\prod_{1\le j\le N-1}H\bigl(a_{j}\xi-\alpha_{j} x-c_{j}\bigr).
$$
and the kernel of the operator $\opw{\mathbf 1_{\mathcal P}}$ is 
$$
k_{N}(x,y)=H(x+y)
\int e^{2i\pi (x-y)\xi}\prod_{1\le j\le N-1}H\bigl(a_{j}\xi-\alpha_{j}(\frac{x+y}2)-c_{j}\bigr) d\xi.
$$
As a result, we have 
$$
k_{N}(x,y)=H(x+y) k_{N-1}(x,y),
$$
where $k_{N-1}$ is the kernel of $\opw{\mathbf 1_{\mathcal P_{N-1}} }$,
where 
$$
\mathcal P_{N-1}=\{(x,\xi) \in \R^{2}, \quad \forall j\in \{1,\dots, N-1\}, L_{j}(x,\xi)-c_{j}\ge 0\}.
$$
We may assume inductively that for any convex polygon $\mathcal P_{k}$ with $k\le N-1$ sides,
there exist $\mu_{k}^{+}$ such that 
\begin{equation}
\opw{\mathbf 1_{\mathcal P_{k}}}\le \mu_{k}^{+},
\end{equation}
where $\mu_{k}^{+}$ depends only on $k$ and not on the area of the polygon, a fact already proven 
for $k=1,2,3$.
We note that with $A_{N}=\opw{\mathbf 1_{\mathcal P_{N}}}$,
we have with $H$ standing for the operator of multiplication by $H(x)$,
$$
HA_{N}H=HA_{N-1}H, \quad A_{N-1}=\opw{\mathbf 1_{\mathcal P_{N-1}} },
$$
since the kernel of $HA_{N}H$ is
$$
H(x) H(y)k_{N}(x,y)=H(x+y) H(x) H(y)k_{N-1}(x,y)=H(x) H(y)k_{N-1}(x,y).
$$
Also we have, with $\check H(x)=H(-x)$, that 
$$
\check HA_{N}\check H=0,
$$
since the kernel of that operator is 
$
\check H(x) \check H(y) H(x+y) k_{N-1}(x,y)=0.
$
We have thus
\begin{equation}\label{531}
A_{N}=HA_{N-1}H+2\re \check HA_{N}H,
\end{equation}
and the kernel of $2\re \check HA_{N}H$ is
\begin{equation}\label{}
\omega_{N}(x,y)= H(x+y)\bigl(\check H(x) H(y)+\check H(y) H(x)\bigr)k_{N-1}(x,y).
\end{equation}
We calculate now
\begin{equation}\label{}
k_{N-1}(x,y)=
\int e^{2i\pi (x-y)\xi}\prod_{1\le j\le N-1}H\bigl(a_{j}\xi-\alpha_{j}(\frac{x+y}2)-c_{j}\bigr) d\xi.
\end{equation}
We check first the $j$
 such that $a_{j}=0$ (and thus $\alpha_{j}\not=0$)\footnote{In this induction proof, we may assume that all the linear forms $L_{j}$, $1\le j\le N$ are  different from 0, otherwise we may use the induction hypothesis.}. Without loss of generality, we may assume that this happens for 
 $1\le j<N_{0}$ so that with some interval $J$ of the real line,
 $\tilde \alpha_{j}=\alpha_{j}/a_{j}, \tilde c_{j}=c_{j}/a_{j}$,
 \begin{multline*}
k_{N-1}(x,y)=\mathbf 1_{J}(\frac{x+y}2)
 \int e^{2i\pi (x-y)\xi}\prod_{\substack{N_{0}\le j\le N-1\\a_{j}>0}}H\bigl(\xi-\tilde \alpha_{j}(\frac{x+y}2)-\tilde c_{j}\bigr)
\\ \times \prod_{\substack{N_{0}\le j\le N-1\\a_{j}<0}}\check H\bigl(\xi-\tilde \alpha_{j}(\frac{x+y}2)-\tilde c_{j}\bigr)  d\xi.
\end{multline*}
 We note that the integration domain is 
 $$
 \psi(\frac{x+y}2)=\max_{\substack{N_{0}\le j\le N-1\\a_{j}>0}}\bigl(\tilde \alpha_{j}(\frac{x+y}2)+\tilde c_{j}\bigr)\le \xi\le \min_{\substack{N_{0}\le j\le N-1\\a_{j}<0}}\tilde \alpha_{j}(\frac{x+y}2)+\tilde c_{j}=-\phi(\frac{x+y}2),
 $$
 with $\phi, \psi$ convex piecewise affine functions;
 since $\phi+\psi$ is also a convex function,
 we get the -- convex -- constraint $(\phi+\psi)((x+y)/2)\le 0$, so that $(x+y)/2$ must belong to a subinterval $\tilde J$ of the interval $J$.
 As a result we get that 
 \begin{align*}
k_{N-1}(x,y)&=\mathbf 1_{\tilde J}(\frac{x+y}2)\frac{e^{-2i\pi(x-y)\phi(\frac{x+y}2)}-e^{2i\pi(x-y)\psi(\frac{x+y}2)}}{2i\pi (x-y)}
\\
&=\mathbf 1_{\tilde J}(\frac{x+y}2) e^{{-i\pi (x-y)(\phi-\psi)(\frac{x+y}2)}}\frac{e^{-i\pi(x-y)(\phi+\psi)(\frac{x+y}2)}-e^{i\pi(x-y)(\phi+\psi)(\frac{x+y}2)}}{2i\pi (x-y)}
\\
&=\mathbf 1_{\tilde J}(\frac{x+y}2) e^{{-i\pi (x-y)(\phi-\psi)(\frac{x+y}2)}}\frac{\sin\bigl(
\pi(x-y)(\phi+\psi)(\frac{x+y}2)
\bigr)}{\pi(y-x)},
\end{align*}
and thus the kernel of $2\re \check HA_{N}H$ is
\begin{multline*}
\omega_{N}(x,y)=H(x+y)\bigl(\check H(x) H(y)+\check H(y) H(x)\bigr)
\mathbf 1_{\tilde J}(\frac{x+y}2) \\\times
e^{{-i\pi (x-y)(\phi-\psi)(\frac{x+y}2)}}\frac{\sin\bigl(
\pi(x-y)(\phi+\psi)(\frac{x+y}2)
\bigr)}{\pi(y-x)},
\end{multline*}
so that, thanks to Proposition \ref{pro.hardy} [2],
\begin{equation}
2\re \poscal{\check HA_{N}H u}{u}\le \norm{Hu}\norm{\check H u},
\end{equation}
and  with \eqref{531},
$$
\poscal{A_{N}u}{u}\le \mu_{N-1}^{+}\norm{Hu}^{2}+\norm{Hu}\norm{\check H u},
$$
we get
\begin{equation}
\mu_{N}^{+}\le  \frac{\mu_{N-1}^{+}+\sqrt{(\mu_{N-1}^{+})^{2}+1}}{2},
\end{equation}
which implies that
\begin{equation}\label{534}\forall N\ge 3,\quad
\mu_{N}^{+}\le \sqrt{N/2},
\end{equation}
since it is true for $N=3$ and\footnote{Indeed we have 
$
\mu_{3}^{+}\le \tilde \mu_{3}< 1.2137<1.2247\approx\sqrt{3/2}
$.}  if we assume that it is true for some $N\ge 3$,
we get 
$$
\mu_{N+1}^{+}\le\frac{\mu_{N}^{+}+\sqrt{(\mu_{N}^{+})^{2}+1}}{2}\le
\frac12\bigl(\sqrt{\frac{N}2}+\sqrt{\frac{N+2}2}\bigr)\le \sqrt{\frac{N+1}2},$$
where the latter inequality follows from the concavity of the square-root function since we have for a concave function $F$,  
$$
\frac12\frac{N}2+\frac12\frac{N+2}{2}=\frac{N+1}2\text{\quad and thus \quad}
\frac12F\bigl(\frac{N}2\bigr)+\frac12F\bigl(\frac{N+2}{2}\bigr)\le F\bigl(\frac{N+1}2\bigr).
$$
The proof of Theorem \ref{thm.77ytre} is complete.
\end{proof}
\begin{rem}\rm
 The above result is weak by its dependence on the number of sides, but it should be pointed out that it is independent of the area of the polygon (which could be infinite). Another general comment is concerned with convexity: although Flandrin's conjecture is not true,
 there is still something special about convex subsets of the phase space and it is in particular interesting that an essentially explicit calculation of the kernel of the operator $\opw{\mathbf 1_{\mathcal P}}$ 
 is tractable when $\mathcal P$ is a polygon with $N$ sides of $\R^{2}$.
\end{rem}
\subsection{Symbols supported in a half-space}
\index{symbols supported in a half-space}
\begin{theorem}\label{thm.78ytre}~
\par\no
$\mathbf{ [1]}$ Let $A$ be a bounded self-adjoint operator on $L^{2}(\R^{n})$
such that its Weyl symbol $a(x,\xi)$ is supported in $\R_{+}\times \R^{2n-1}$.
Then with $\check H$ standing for the orthogonal projection onto
\begin{equation}
\{u\in L^{2}(\R^{n}),\ \supp u\subset \R_{-}\times \R^{n-1}\},
\end{equation}
we have 
$
\check H A\check H=0.
$
\par\no
$\mathbf{ [2]}$ Let $A$ be as above;  if $A$ is a non-negative operator, then with $H=I-\check H$, we have  $$\check H A=A\check H=0, \quad
A=HAH,$$
\end{theorem}
 \begin{nb}
 We have seen explicit examples of bounded self-adjoint operators such that the Weyl symbol is supported in $x\ge 0$ but for which $\check H A H\not =0$: the quarter-plane operator (see Section
 \ref{sec.51})
has the Weyl symbol $H(x) H(\xi)$, the kernel of 
$$\check H \opw{H(x)H(\xi)}H
\quad\text{ is }\quad
 \check H(x) H(y) H(x+y)\frac{1}{2i\pi}\text{pv}\frac{1}{y-x},
 $$
 which is not the zero distribution and, according to the above result, this alone implies that $\opw{H(x)H(\xi)}$ cannot be non-negative.
\end{nb}
\begin{proof}
Let us prove first that 
$\check H A\check H=0$;
let $\phi, \psi\in \mooc(\R^{n})$ such that 
$$
\supp \phi\cup \supp \psi\subset (-\io,0)\times \R^{n-1}.
$$
Since the Wigner distribution $\mathcal W(\phi, \psi)$ 
belongs to $\mathscr S(\RZ)$
and is given by the integral
$$
\mathcal W(\phi, \psi)(x,\xi)=\int_{\R^{n}} \phi(x+\frac z2)\bar \psi(x-\frac z 2) e^{-2i\pi z\cdot \xi} dz,
$$
we infer right away\footnote{In the integrand, we must have,
$x_{1}+\frac {z_{1}}2\le -\epsilon_{0}<0, x_{1}-\frac {z_{1}}2\le -\epsilon_{1}<0$ and thus $x_{1}\le -(\epsilon_{0}+\epsilon_{1})/2$} that $\supp\mathcal W(\phi, \psi)\subset(-\io, 0)\times \R^{2n-1}$. We know also that 
$$
\poscal{A \phi}{\psi}_{L^{2}(\R^{n})}=
\poscal{A \phi}{\psi}_{\mathscr S'(\R^{n}),\mathscr S(\R^{n})}
=\poscal{a}{\mathcal W(\phi,\psi)}_
{\mathscr S'(\R^{2n}),\mathscr S(\R^{2n})}=0.
$$
As  a result, the $L^{2}(\R^{n})$ bounded operator $\check H A\check H$ is such that,
for $u,v\in L^{2}(\R^{n})$, $\phi, \psi$ as above,
\begin{multline*}
\poscal{\check H A\check Hu}{v}_{L^{2}(\R^{n})}=\poscal{\check H A\check H\check Hu}{\check H v}_{L^{2}(\R^{n})}
\\=
\poscal{\check H A\check H (\check Hu-\phi)}{\check H v}_{L^{2}(\R^{n})}+
\poscal{\check H A\check H \phi }{\check H v-\psi}_{L^{2}(\R^{n})}
+\underbrace{\poscal{\check H A\check H \phi }{\psi}_{L^{2}(\R^{n})}}_{
\poscal{A\phi }{\psi}_{L^{2}(\R^{n})}=0
},
\end{multline*}
so that
\begin{multline*}
\val{\poscal{\check H A\check Hu}{v}_{L^{2}(\R^{n})}}\le \norm{A}_{\mathcal B(L^{2}(\R^{n}))}
\Bigl(\norm{\check H u-\phi}_{L^{2}(\R^{n})}\norm{v}_{L^{2}(\R^{n})}
\\+
\norm{\check H v-\psi}_{L^{2}(\R^{n})}\norm{\phi}_{L^{2}(\R^{n})}
\Bigr).
\end{multline*}
Using now that the set $\{\phi\in \mooc(\R^{n}), \supp \phi\subset (-\io,0)\times \R^{n-1}\}$
is dense\footnote{Let $\chi_{0}$ be a function satisfying \eqref{521cut}
and let $w$ be in the set \eqref{742aze}. Let $(\phi_{k})_{k\ge 1}$ be a sequence in $\mooc(\R^{n})$
converging in $L^{2}(\R^{n})$
towards $w$;
the function  defined by 
$$
\tilde\phi_{k}(x)=\chi_{0}(-kx_{1})\phi_{k}(x),
$$
belongs to $\mooc(\R^{n})$,
is supported in $(-\io,-1/k]\times \R^{n-1}$, and that sequence converges in $L^{2}(\R^{n})$
towards $w$ since 
$$
\norm{\tilde \phi_{k}-w}_{L^{2}(\R^{n})}\le 
\underbrace{\norm{\chi_{0}(-kx_{1})\bigl(\phi_{k}(x)-w(x)\bigr)}_{L^{2}(\R^{n})}}_{\le 
\norm{\phi_{k}-w}_{L^{2}(\R^{n})}\rightarrow\ 0\ \text{when $k\rightarrow +\io$.}
}+
\norm{(\chi_{0}(-kx_{1})-1) w(x)}_{L^{2}(\R^{n})}
$$
and 
$
\norm{(\chi_{0}(-kx_{1})-1) w(x)}_{L^{2}(\R^{n})}^{2}\le \int\mathbf 1\bigl\{-\frac2k\le x_{1}\le 0\bigr\}\val{w(x)}^{2}dx
$
which has also limit 0 when $k$ goes to $+\infty$
by the Lebesgue Dominated Convergence Theorem.
} in 
\begin{equation}\label{742aze}
\{w\in L^{2}(\R^{n}),\ \supp w\subset(-\io, 0]\times \R^{n-1}\},
\end{equation}
we obtain that 
$\poscal{\check H A\check Hu}{v}_{L^{2}(\R^{n})}=0$ and the first result.
Let us assume that  the operator $A$ is  non-negative.
We have
 $$
 A=B^{2},\quad B=B^{*} \text{ bounded self-adjoint.}
 $$
 It implies  with $L^{2}(\R^{n})$ norms and dot-products,
 \begin{align*}
\poscal{Au}{u}&=\poscal{HAH u}{u}+2 \re \poscal{\check H A Hu}{\check H u}
 \\&=\poscal{HBBHu}{u}+
 2 \re \poscal{\check H BBHu}{\check H u}
 \\
 &=\norm{BH u}^{2}+2\re\poscal{BHu}{B\check H u}
 \\
 &=\norm{BH u+B\check H u}^{2}-\norm{B\check H u}^{2}
 \\
 &=\norm{Bu}^{2}-\norm{B\check H u}^{2}=\poscal{Au}{u}-\norm{B\check H u}^{2},
\end{align*} 
and thus 
$B\check H=0$, so that $\check H B=0$ and thus $\check H B^{2}=\check H A=0=A\check H$,
so that $\check H AH=0=HA\check H$,
and 
$
A=H A H,
$
concluding the proof of [2].
\end{proof}
\begin{cor}
Let $A$ be a bounded self-adjoint operator on $L^{2}(\R^{n})$ such that its Weyl symbol is supported in $\R_{+}\times \R^{2n-1}$ and 
 such that $\re(\check H A H)\not=0$, then the spectrum of $A$ intersects $(-\io, 0)$.
 \end{cor}
 \begin{proof}
 We have from $[1]$ in the previous
 theorem, $$A=(H+\check H)A(H+\check H)=HAH+ 2\re HA\check H,$$
 and from $[2]$, if $A$ were non-negative, we would have $A\check H =0$ and $\re HA\check H=0$,
 contradicting the assumption.
\end{proof}
\begin{rem}\rm
 If $\mathcal C$ is a compact convex body of $\R^{2n}$, we may use the fact (see e.g. \cite{MR0274683}) that 
 $$
 \mathcal C=\bigcap _{\substack{\mathfrak H_{j}\ \text{closed half-spaces}\\\text{containing $K$}}}\mathfrak H_{j}.
 $$
Then of course $\opw{\mathbf 1_{\mathcal C}}$ 
 is a bounded self-adjoint 
  operator on $L^{2}(\R^{n})$, and if $\mathfrak H_{j}$ is defined by
  $$
  \mathfrak H_{j}=\{(x,\xi)\in \R^{2}, L_{j}(x,\xi)\ge c_{j}\},
  $$
  where $L_{j}$ is a linear form on $\R^{2}$ and $c_{j}$ a real constant, we obtain with the symplectic covariance of the Weyl calculus, setting 
  $$
  \mathcal H_{j}(x,\xi)= H(L_{j}(x,\xi)-c_{j}),
  $$
  that for all $\mathfrak H_{j}$ closed half-spaces containing $K$, we have 
  \begin{equation}
\opw{\mathbf 1_{K}}=\opw{H_{j}}\opw{\mathbf 1_{K}}\opw{H_{j}}
+2\re \opw{\check H_{j}}\opw{\mathbf 1_{K}}\opw{H_{j}},
\end{equation}
where $\check H(x,\xi)= H(-L_{j}(x,\xi)+c_{j})$.
    \end{rem}
 \section{Open questions \& Conjectures}
In this section we review the rather long list of conjectures formulated in the text and we try to classify their statements by rating their respective interest, relevance and difficulty.
We should keep in mind that the study of $\opw{\mathbf 1_{E}}$ for a subset $E$ of the phase space  is highly correlated  to some particular set of special functions related to $E$:
Hermite functions and Laguerre polynomials for ellipses,
Airy functions for parabolas,
homogeneous distributions for hyperbolas and so on. It is quite likely
that the ``shape'' of $E$ will determine the type of special functions to be studied
to getting a diagonalization of the operator $\opw{\mathbf 1_{E}}$.
\subsection{Anisotropic Ellipsoids \& Paraboloids}
\index{conjecture on anisotropic ellipsoids}
\begin{conjecture}\label{cj8777}
 Let $E$ be an ellipsoid in $\RZ$ equipped with its canonical symplectic structure. Then the operator
 $\opw{\mathbf 1_{E}}$ is bounded on $L^{2}(\R^{n})$ (which is obvious from \eqref{norm01}) and we have 
\begin{equation}\label{fla897}
 \opw{\mathbf 1_{E}}\le \Id.
\end{equation}
\end{conjecture}
A sharp version of this result was proven for $n=1$
 in the 1988  P. Flandrin's article  \cite{conjecture}, and was improved to an isotropic higher dimensional setting  in the 
 paper \cite{MR2761287} by E.~Lieb
and Y.~Ostrover. Without isotropy, it remains a conjecture. As described in more details in
Section \ref{sec.oiu666}, it can be reformulated as a problem on Laguerre polynomials.
That conjecture is a very natural one and it would be quite surprising that a counterexample to \eqref{fla897} could occur from an anisotropic ellipsoid\footnote{We mean by anisotropic ellipsoid a set of type \eqref{ellip++} where  $0<a_{1}<a_{2}<\dots<a_{n}.$}.
We introduced in Section \ref{sec.parcj1} a conjecture on anisotropic paraboloids
directly related to Conjecture \ref{cj8777}.
\index{conjecture on anisotropic paraboloids}
\begin{conjecture}\label{cj877++}
 Let $E$ be an anisotropic paraboloid in $\RZ$ equipped with its canonical symplectic structure. Then the operator
 $\opw{\mathbf 1_{E}}$ is bounded on $L^{2}(\R^{n})$ 
  and we have 
\begin{equation}\label{fla89p}
 \opw{\mathbf 1_{E}}\le \Id.
\end{equation}
\end{conjecture}
In terms of special functions, it is related to a property of Airy-type functions. As a contrast with ellipses, we do not expect  \eqref{fla89p} to leave any room for improvement
whereas \eqref{fla897} can certainly be improved with its  right-hand-side replaced by a smaller operator as in \eqref{325tyu}.
\subsection{Balls for the \texorpdfstring{$\ell^{p}$}{lp} norm}
We have seen in Section \ref{sec.532zae} that the quantization of the indicatrix of a $\ell^{p}$ 
ball could have a spectrum intersecting $(1,+\io)$ when $p\not=2$.
More generally  one could raise the following question.
\begin{question}
 Let $p\in [1,+\io]$, $p\not=2$ and let $\mathbb B^{2n}_{p}$ be the unit $\ell^{p}$ ball in $\RZ$. For $\lambda>0$,
 we define the operator
 \begin{equation}
P_{n,p,\lambda}=\opw{\mathbf 1_{\lambda\mathbb B^{2n}_{p}}}.
\end{equation}
Is it possible to say something on the spectrum of the  operator $P_{n,p,\lambda}$,
even in a two-dimensional phase space ($n=1$)?
Is there an asymptotic behaviour for the upper bound of the spectrum of
$P_{n,p,\lambda}$
when $\lambda$ goes to $+\infty$?
\end{question}
\subsection{On generic pulses in  \texorpdfstring{$L^{2}(\R^{n})$}{l2}}
We have seen that the set $\mathcal G$ defined in \eqref{67kjhg}
is generic in the  Baire category sense, but our explicit examples were quite 
simplistic.
\begin{question}
Let $\mathcal G$ be defined in \eqref{67kjhg}. Does there exist $u\in \mathcal G$ such that
the set $E_{+}(u)$ 
(defined in \eqref{645yrp})
is connected?
\end{question}
\subsection{On convex bodies}
\begin{conjecture}\label{65kjhg}
For $N\ge 2$, we define
\begin{equation}
\mu_{N}^{+}=\sup_{\substack{\text{$\mathcal P$ convex bounded}\\\text{ polygon with $N$ sides}}}\text{\rm Spectrum}\left(\opw{\mathbf 1_{\mathcal P}}\right).
\end{equation}
Then the  sequence $(\mu_{N}^{+})_{N\ge 2}$ is increasing\footnote{According to our Definition \ref{def.77hgfd} of the set $\mathscr P_{N}$ of polygons with $N$ sides is increasing with respect to $N$. } 
 and there exists $\alpha>0$
such that 
\begin{equation}\label{}
\forall N\ge 2, \quad
\mu_{N}^{+}\le \alpha \ln N.
\end{equation}
\end{conjecture}
\begin{nb}
 Theorem \ref{thm.77ytre} is a small step in this direction. 
\end{nb}
A stronger version of Conjecture \ref{65kjhg} would be 
\begin{conjecture}\label{65kj++}
We define
\begin{equation}
\mu^{+}=\sup_{\substack{\text{$\mathcal C$ convex }\\\text{bounded}}}\text{\rm Spectrum}\left(\opw{\mathbf 1_{\mathcal C}}\right).
\end{equation}
Then we have $\mu^{+}<+\io$.
\end{conjecture}
The invalid Flandrin's conjecture was $\mu^{+}=1$ and we know now that $\mu_{+}\ge \mu_{2}^{+}>1$ as given 
by \eqref{713ezj}.
\begin{question}\label{que.88uytr}
There is a diagonalization  of the quantization of the indicator function of Ellipsoids, Paraboloids and Hyperbolic regions. Is there a non-quadratic example of diagonalization?
\end{question}
\begin{question}
The value of $\mu_{2}^{+}$ is known explicitly, but for $\mu_{3}^{+}$, we have only the upperbound $\tilde\mu_{3}$ as given by Theorem \ref{thm.75fdgy}. Is it possible to determine explicitly the value of $\mu_{3}^{+}$,
either by answering Question \ref{que.88uytr}, or via another argument?
\end{question}
\begin{conjecture}
 Let $\mathcal C$ be a proper closed convex subset of $\R^{2}$ with positive Lebesgue measure
 such that $\opw{\mathbf 1_{\mathcal C}}$ is bounded self-adjoint on $L^{2}(\R)$
 (that assumption is useless if Conjecture \ref{65kj++} is proven) with a spectrum included in $[0,1]$. Then $\mathcal C$ is the  strip $[0,1]\times \R$, up to an affine symplectic map. 
\end{conjecture}
All the explicitly avalaible examples are compatible with that conjecture (see also Remark \ref{rem.72poiu})
and  the second part of Theorem \ref{thm.78ytre} is also an indication in that direction. It would be nice in that instance to reach a spectral characterization of a subset modulo the affine symplectic group. 
\section{Appendix}
\subsection{Fourier transform, Weyl quantization, Harmonic Oscillator}\label{sec31}
\subsubsection{Fourier transform.}
\index{Fourier transform}
 We use in this paper  the following normalization for the Fourier transform and inversion formula: for $u\in \mathscr S(\R^{n})$,
 \begin{equation}\label{fourier}
\hat u(\xi)=\int_{\R^{n}} e^{-2i \pi x\cdot \xi} u(x) dx,\quad u(x)=\int_{\R^{n}} e^{2i \pi x\cdot \xi} \hat u(\xi) d\xi,
\end{equation}
a formula that can be extended to $u\in \mathscr S'(\R^{n})$,
with defining the distribution $\hat u$ by the duality bracket  
\begin{equation}\label{}
\poscal{\hat u}{\phi}_{\mathscr S'(\R^{n}), \mathscr S(\R^{n})}=\poscal{u}{\hat \phi}_{\mathscr S'(\R^{n}), \mathscr S(\R^{n})}.
\end{equation}
Checking \eqref{fourier} for $u\in \mathscr S'(\R^{n})$ is then easy, that  is
\begin{equation}\label{inve66}
\check{\hat{\hat u}}=u,
\end{equation}
where the distribution $\check u$ is defined by
\begin{equation}\label{}
\poscal{\check u}{\phi}_{\mathscr S'(\R^{n}), \mathscr S(\R^{n})}=\poscal{u}{\check \phi}_{\mathscr S'(\R^{n}), \mathscr S(\R^{n})}, \quad \text{with}\ \check \phi(x)=\phi(-x).
\end{equation}
It is useful to notice that for $u\in \mathscr S'(\R^{n})$,
\begin{equation}\label{inve55}
\check{\hat u}=\hat{\check u}.
\end{equation}
Using \eqref{phase} and denoting the Fourier transformation 
 by $\mathcal F$, \eqref{inve66} and 
\eqref{inve55} read
\begin{equation}\label{azqs22}
\sigma_{0}\mathcal F^{2}=\Id,\quad[\mathcal F,\sigma_{0}]=0, \quad{\text{so that}\quad \mathcal F^{*}=\mathcal F^{-1}=\sigma_{0}\mathcal F=\mathcal F\sigma_{0}.}
\end{equation}
This normalization yields simple formulas for the Fourier transform of Gaussian functions: for $A$ a real-valued symmetric  positive definite  $n\times n$ matrix, we define the function $v_{A}$ in the Schwartz space by
\begin{equation}\label{gau135}
v_{A}(x)= e^{-\pi\poscal{Ax}{x}},\quad \text{and we have }\quad \widehat{v_{A}}(\xi)=(\det A)^{-1/2}
 e^{-\pi\poscal{A^{-1}\xi}{\xi}}.
\end{equation}
Similarly when $B$ is 
a real-valued symmetric  non-singular  $n\times n$ matrix, the function $w_{B}$
defined by
$$
w_{B}(x)=e^{i\pi\poscal{Bx}{x}}
$$
is in $L^{\io}(\R^{n})$ and thus a tempered distribution and we have 
\index{signature}
\index{index}
\begin{equation}\label{foimga}
 \widehat{w_{B}}(\xi)=\val{\det B}^{-1/2} e^{\frac{i\pi}{4}\sign B}
 e^{-i\pi\poscal{B^{-1}\xi}{\xi}},
\end{equation}
where $\sign B$
stands for the {\it signature of $B$} that is, with $\mathtt E$ the set of eigenvalues of $B$
(which are real and non-zero),
\begin{equation}\label{ind001}
\sign B=\underbrace{\card (\mathtt E\cap \R_{+})}_{\nu_{+}(B)}-\underbrace{\card (\mathtt E\cap \R_{-})}_{\nu_{-}(B)=\inde{(B)}}.
\end{equation}
The integer $\nu_{-}(B)$ is called the \emph{index} of $B$, noted $\inde{(B)}$; Formula 
\eqref{foimga} can be written as 
\begin{equation}\label{foimga++}
e^{-i\pi n/4}\mathcal F\left(e^{i\pi\poscal{Bx}{x}}\right)=i^{-{\inde{B}}}\val{\det B}^{-1/2} 
e^{-i\pi\poscal{B^{-1}\xi}{\xi}},
\end{equation}
since $\nu_{+}+\nu_{-}=n$ (as $B$ is non-singular), 
$$
e^{\frac {i\pi n}4} e^{-\frac{i\pi\nu_{-}}2}=e^{\frac{i\pi}4(\nu_{+}+\nu_{-}-2\nu_{-})}
=e^{\frac{i\pi}4\sign (B)}.
$$
We note also that
\begin{equation}\label{ind002}
\sign(\det B)=(-1)^{\inde{B}},
\end{equation}
so that
$$
\left(i^{-{\inde{B}}}\val{\det B}^{-1/2}\right)^{2}=(-1)^{\nu_{-}}\val{\det B}^{-1}=\sign(\det B)
\val{\det B}^{-1}=(\det B)^{-1},
$$
and thus  the prefactor $i^{-{\inde{B}}}\val{\det B}^{-1/2} $ in the rhs of  \eqref{foimga++} is a square root of
$1/\det B$. 
\par
With
$H$ standing for the characteristic function of $\R_{+}$,
we have
\begin{align*}
&1=H+\check H,\quad \delta_{0}=\hat H+\hat{\check H},\\
&D\sign =\frac{\delta_{0}}{i\pi},\quad \widehat{D\sign }=\frac{1}{i\pi},\quad \xi\widehat{\sign}=\frac{1}{i\pi},\quad\widehat{\sign}=\frac{1}{i\pi}\text{pv}{\frac1{\xi}}, \quad\text{\footnotesize (principal value)}
\end{align*}
the latter formula following from the fact that
$$
\xi\bigl(\widehat{\sign}-\text{pv}\frac{1}{i\pi \xi} \bigr)=0,\text{\quad which implies\quad}
\widehat{\sign}-\text{pv}\frac{1}{i\pi \xi} =c\delta_{0}=0,
$$
since  $\widehat{\sign}-\frac{1}{i\pi \xi} $ is odd.
We infer from that 
\begin{align*}
&\hat H-\widehat{\check H}=\widehat{\sign}=\text{pv}\frac{1}{i\pi \xi},
\end{align*}
and
\begin{equation}\label{}
\hat H=\frac{\delta_{0}}{2}+\text{pv}\frac{1}{2i\pi \xi}.
\end{equation}
\begin{lem}\label{lem.9110}
 Let $T$ be a compactly supported distribution on $\R^{n}$ such that
 \begin{equation}\label{est2}
 \forall N\in \N, \quad \valjp \xi^{N}\hat T(\xi) \quad \text{is bounded, with $\valjp{\xi}=\sqrt{1+\val \xi^{2}}$.}
\end{equation}
Then $T$ is a $\moo$ function.
\end{lem}
\begin{proof}
Note that $\hat T$ is an entire function, as the Fourier transform of a compactly supported distribution. Moreover, from \eqref{est2} with $N=n+1$, we get that $\hat T$ belongs to $L^{1}(\R^{n})$ and thus $T$ is a continuous function. Moreover,
 we have for any $\alpha\in \N^{n}$, 
 $$
( D_{x}^{\alpha}T)(x)=\int e^{2i\pi x\cdot \xi}\underbrace{\xi^{\alpha}\hat T(\xi) }_{\in L^{1}(\R^{n})}d\xi,
 $$
 so that $T$ is a $\moo$ function.
\end{proof}
\begin{pro}\label{pro.92}
 Let $\rho>0$
 and let $f$ be an holomorphic function on a neighborhood of  $\{z\in \C, \val{\im z}\le \rho\}$
 such that
 \begin{align}
&\forall y\in [-\rho, \rho], \
 \int\val{f(x+iy)} dx<+\io, 
 \label{firstc}\\
 &
 \lim_{R\rightarrow+\io}\int_{\val y\le \rho}\val{f(\pm R+iy)} dy=0.
 \label{secondc}
\end{align}
Then we have
\begin{equation}\forall \xi\in \R, \quad 
\val{\hat f(\xi)}\le C e^{-2\pi\rho\val \xi},
\end{equation}
with 
$
C=\max(C_{+},C_{-}), \ C_{\pm}=
\int_{\R}\val{
f(x\pm i\rho)}
dx.
$
Conversely, if $f$ is a bounded measurable function such that $\hat f(\xi)$ is  $O(e^{-2\pi r\val \xi})$ for some $r>0$, then 
$f$ is holomorphic on $\{z\in \C, \val{\im z}<r\}$.
\end{pro}
\begin{proof}
 If $f$ is holomorphic
 near $\{z\in \C, \val{\im z}\le \rho\}$,
satisfies \eqref{firstc}
and
\eqref{secondc}, then Cauchy's formula
shows that for $\val y\le \rho$, 
\begin{align*}
\int_{\R} &e^{-2i\pi (x+iy) \xi} f(x+iy) dx
=e^{2\pi y\xi}\lim_{R\rightarrow+\io}\int_{-R}^{R}
e^{-2i\pi x \xi} f(x+iy) dx 
\\
&=\lim_{R\rightarrow+\io}\int_{[-R+iy, R+iy]}
e^{-2i\pi z\xi} f(z) dz 
\\
&=\lim_{R\rightarrow+\io}
\int_{[-R+iy, -R]\cup[-R,R]\cup[R,R+iy]}
e^{-2i\pi z\xi} f(z) dz
\\
&=\hat f(\xi) +
\lim_{R\rightarrow+\io}
\left(\int_{0}^{y}
e^{-2i\pi (R+it)\xi} f(R+it) idt-
\int_{0}^{y}
e^{-2i\pi (-R+it)\xi} f(-R+it) idt
\right).
\end{align*}
We have  for $\val y\le \rho$, 
$$
\Bigl\vert\int_{0}^{y}
e^{-2i\pi (\pm R+it)\xi} f(\pm R+it) idt\Bigr\vert\le 
\int_{\val t\le \rho}\val{f(\pm R+it)} dt\ 
e^{2\pi \rho\val \xi},
$$
which goes to 0 when $R$ goes to $+\io$, thanks to \eqref{secondc}, so that for all $y\in [-\rho,\rho]$, we have 
$$
\int_{\R} e^{-2i\pi (x+iy) \xi} f(x+iy) dx=\hat f(\xi),
$$
which implies for $y=-\rho \sign \xi$ (taken as 0, if $\xi=0$)
$$
\val{\hat f(\xi)}\le \int_{\R}\val{
f(x\mp i\rho)}
dx \ e^{-2\pi \rho\val \xi}\underbrace{\le}_{\text{from \eqref{firstc}}} Ce^{-2\pi \rho\val \xi},
$$
proving the first part of the proposition.
Let us consider  now a function $f$ in $L^{\io}(\R)$ such that 
$\hat f(\xi)$ is $O(e^{-2\pi r\val \xi})$ for some $r>0$, and let $\rho\in (0,r)$.
We have 
$
f(x)=\int e^{2i\pi x\xi}\hat f(\xi)d\xi
$
and for $\val y\le \rho$, we have 
$
\int_{\R} e^{2\pi \val y\val \xi}\val{\hat f(\xi)} d\xi<+\io,
$
so that 
$f$ is holomorphic on $\{z\in \C, \val {\im z}<r\}$ with
$$
f(x+iy)=\int_{\R} e^{2i\pi (x+iy) \xi}\hat f(\xi)d\xi,
$$
concluding the proof.
\end{proof}
\subsubsection{Weyl quantization}
\index{Weyl quantization}
Let $a\in \mathscr S'(\RZ)$. We define the operator $\opw{a}$, continuous from 
$\mathscr S(\R^{n})$
into $\mathscr S'(\R^{n})$, given by the formula
\begin{equation}\label{weylq1}
(\opw{a}u)(x)=\iint e^{2i\pi (x-y)\cdot \xi} a(\frac{x+y}2, \xi) u(y) dy d\xi,
\end{equation}
to be understood weakly as 
\begin{equation}\label{eza654+}
\poscal{\opw{a}u}{\bar v}_{\mathscr S'(\R^{n}), \mathscr S(\R^{n})}=\poscal{a}{\mathcal W(u,v)}_{\mathscr S'(\R^{2n}), \mathscr S(\R^{2n})},
\end{equation}
where the so-called Wigner function $\mathcal W(u,v)$ is defined for $u,v\in \mathscr S(\R^{n})$ by
\begin{equation}\label{wigner1}
\mathcal W(u,v)(x,\xi)=\int e^{-2i\pi z\cdot \xi} u(x+\frac z2) \bar v(x-\frac z2) dz.
\end{equation}
We note that the sesquilinear  mapping 
$\mathscr S(\R^{n})\times \mathscr S(\R^{n})\ni(u,v)\mapsto \mathcal W(u,v) \in \mathscr S(\RZ)$
is continuous so that the above bracket of duality 
$$
\poscal{a}{\mathcal W(u,v)}_{\mathscr S'(\R^{2n}), \mathscr S(\R^{2n})},
$$
makes sense.
We note as well that a temperate distribution $a\in\mathscr S'(\RZ)$ gets quantized
by a continuous  operator $a^{w}$ from 
$\mathscr S(\R^{n})$ 
into
$\mathscr S'(\R^{n})$.
\par
Also, we find that $\mathcal W(u,u)$ is real-valued since
\begin{multline*}
\overline{\mathcal W(u,u)(x,\xi)}
=\int e^{2i\pi z\cdot \xi} \bar u(x+\frac z2) u(x-\frac z2) dz
\\
=\int e^{-2i\pi z\cdot \xi} \bar u(x-\frac z2) u(x+\frac z2) dz
=\mathcal W(u,u)(x,\xi).
\end{multline*}
\begin{lem}
Let $a$ be a tempered distribution on $\RZ$ and let  $b$ be a  polynomial of degree $d$ on $\RZ$.
Then we have 
\begin{align}
a\sharp b&=\sum_{0\le k\le d} \omega_{k}(a,b),\quad \text{with}
\\
 \omega_{k}(a,b)&=\frac1{(4i\pi)^{k}}\sum_{\val \alpha+\val \beta=k}
\frac{(-1)^{\val \beta}}{\alpha!\beta!}(\p_{\xi}^{\alpha}\p_{x}^{\beta}a)(x,\xi)
(\p_{x}^{\alpha}\p_{\xi}^{\beta}b)(x,\xi),
\label{omegak}\\
 \omega_{k}(b,a)&=(-1)^k\omega_{k}(a,b).
 \label{9118}
\end{align}
The Weyl symbol of the commutator $[\opw{a}, \opw{b}]$ is
\begin{equation}\label{}
c(a,b)=2\sum_{\substack{0\le k\le d\\ k\ \text{\rm odd}}} \omega_{k}(a,b).
\end{equation}
If the degree of $b$ is smaller than 2, we have
\begin{equation}\label{}
c(a,b)=2\omega_{1}(a,b)=\frac{1}{2\pi i}\poi{a}{b},
\end{equation}
and if $a$ is a function of $b$, the commutator $[\opw{a},\opw{b}]=0$.
\end{lem}
\begin{rem}
 In particular if $q(x,\xi)$ is a quadratic polynomial and $$a(x,\xi)=H\bigl(1-q(x,\xi)\bigr),$$
is the characteristic function of the set $\{(x,\xi), q(x,\xi)\le 1\}$,
then we have
\begin{equation}\label{}
\bigl[\opw{a}, \opw{q}\bigr]=0.
\end{equation}
\end{rem}
\begin{proof}
 Applying \eqref{gfcd44}, \eqref{gfcd44+}, we obtain that this lemma follows from 
 \eqref{9118}, that we check now:
 \begin{align*}
(4i\pi)^k\omega_{k}(a,b)&=
 \sum_{\val \alpha+\val \beta=k}
\frac{(-1)^{\val \beta}}{\alpha!\beta!}(\p_{\xi}^{\alpha}\p_{x}^{\beta}a)(x,\xi)
(\p_{x}^{\alpha}\p_{\xi}^{\beta}b)(x,\xi)
\\
&= \sum_{\val \alpha+\val \beta=k}
\frac{(-1)^{\val \alpha}}{\alpha!\beta!}(\p_{\xi}^{\beta}\p_{x}^{\alpha}a)(x,\xi)
(\p_{x}^{\beta}\p_{\xi}^{\alpha}b)(x,\xi)
\\
&= \sum_{\val \alpha+\val \beta=k}
\frac{(-1)^{k-\val \beta}}{\alpha!\beta!}(\p_{\xi}^{\beta}\p_{x}^{\alpha}a)(x,\xi)
(\p_{x}^{\beta}\p_{\xi}^{\alpha}b)(x,\xi)
=(-1)^k(4i\pi)^k\omega_{k}(b,a),
\end{align*}
which is the sought result.\end{proof}
\begin{rem}\rm
 We can note that Formula \eqref{2.wecofo} is non-local in the sense that for $a,b\in \mathscr S(\RZ)$ with disjoint supports, although all $\omega_{k}(a,b)$ (given by \eqref{omegak}) are identically 0, the function $a\sharp b$ (which belongs to $\mathscr S(\RZ)$) is different from 0;
 let us give an example. Let $\chi_{0}\in \mooc(\R;[0,1])$ with support $[-1+\epsilon_{0},1-\epsilon_{0}]$ with $\epsilon_{0}\in (0,1)$
 and let us consider in $\R^{2}$,
 $$
 a (x,\xi)=\chi_{0}(x) e^{-\pi \xi^{2}}, \quad b(x,\xi)=\chi_{0}(x-2) e^{-\pi \xi^{2}},
 $$ 
 so that $a,b$ both belong to $\mathscr S(\R^{2})$ and 
 $$
\supp a=[-1+\epsilon_{0},1-\epsilon_{0}]\times \R, \quad \supp b=[1+\epsilon_{0},3-\epsilon_{0}]\times \R,
 $$
 so that the supports are disjoint and all $\omega_{k}(a,b)$ are identically vanishing.
 We check now
 \begin{align*}
 (a\sharp b)(x,\xi)&= 4\iint\hskip-5pt\iint\chi_{0}(y)e^{-\pi\eta^{2}}
 \chi_{0}(z-2)e^{-\pi\zeta^{2}}
 e^{-4i\pi(\xi-\eta)(x-z)} e^{4i\pi(x-y)(\xi-\zeta)} dy d\eta dzd\zeta
 \\
 &=4\iint \chi_{0}(y)\chi_{0}(z-2) e^{-4\pi(x-z)^{2}}e^{-4\pi(x-y)^{2}}
 e^{4i\pi\xi(z-x+x-y)} dy dz
 \\
&= 4\Bigl(\int \chi_{0}(y)
e^{-4i\pi\xi y} e^{-4\pi(x-y)^{2}} dy\Bigr)
\Bigl(\int\chi_{0}(z) e^{4i\pi\xi z}e^{-4\pi(x-2-z)^{2}}
dz\Bigr),
\end{align*}
 so that 
 $$
 (a\sharp b)(0,0)=
 4\underbracket[0.5pt][1pt]{\Bigl(\int \chi_{0}(y)
 e^{-4\pi y^{2}} dy\Bigr)}_{>0}
\underbracket[0.5pt][1pt]{\Bigl(\int\chi_{0}(z) e^{-4\pi(2+z)^{2}}
dz\Bigr)}_{>0}>0.
 $$
 \end{rem}
\subsubsection{Some explicit computations}
We may also calculate with 
\begin{equation}\label{gg}
u_{a}(x)=(2a)^{1/4}e^{- \pi a x^{2}},  a >0,
\end{equation}
\begin{multline}\label{gauss}
\mathcal W(u_{a}, u_{a})(x,\xi)
=(2a)^{1/2}\int e^{-2i\pi z\cdot \xi} e^{-\pi a\val{x-\frac z2}^{2}}e^{-\pi a\val{x+\frac z2}^{2}}dz
\\=(2a)^{1/2}\int e^{-2i\pi z\cdot \xi} e^{-2\pi a  x^{2}    } e^{-\pi a z^{2}/2}dz
=(2a)^{1/2}e^{-2\pi a  x^{2}    }  2^{1/2} a^{-1/2} e^{-\pi\frac2{a}\xi^{2}}
\\= 2 e^{-2\pi(a x^{2}+a^{-1}\xi^{2})},
\end{multline}
which is also a Gaussian function on the phase space (and positive function).
The calculation of $\mathcal W(u'_{a}, u'_{a})(x,\xi)$ is interesting since we have 
$$
4\pi^{2}\poscal{D_{x}b^{w}D_{x}u_{a}}{\bar u_{a}}_{\mathscr S'(\R^{n}), \mathscr S(\R^{n})}
$$
$$
=
\poscal{b^{w}u'_{a}}{\bar u'_{a}}_{\mathscr S'(\R^{n}), \mathscr S(\R^{n})}=\poscal{b}{\mathcal W(u'_{a},u'_{a})}_{\mathscr S'(\R^{2n}), \mathscr S(\R^{2n})},
$$
and for $b(x,\xi)$ real-valued we have
$$
\xi\sharp b\sharp \xi=\bigl(\xi b+\frac{b'_{x}}{4i\pi}\bigr)\sharp \xi=\xi^{2}b+\frac{b'_{x}\xi}{4i\pi}
-\frac{\p_{x}}{4i\pi}\bigl(\xi b+\frac{b'_{x}}{4i\pi}\bigr)=\xi^{2}b+\frac{b''_{xx}}{16\pi^{2}},
$$
so that
$$
4\pi^{2}\iint
2 e^{-2\pi(a x^{2}+a^{-1}\xi^{2})}\bigl(
\xi^{2}b+\frac{b''_{xx}}{16\pi^{2}}
\bigr)
dxd\xi
=\poscal{b}{\mathcal W(u'_{a},u'_{a})},$$
proving that
\begin{align*}
\mathcal W(u'_{a},u'_{a})(x,\xi)&=
2 e^{-2\pi(a x^{2}+a^{-1}\xi^{2})}4\pi^{2}\xi^{2}+\frac14
2 \p_{x}^{2}\bigl(e^{-2\pi(a x^{2}+a^{-1}\xi^{2})}\bigr)
\\&=2 e^{-2\pi(a x^{2}+a^{-1}\xi^{2})}
\bigl(4\pi^{2}\xi^{2}+\frac14
((-4\pi ax)^{2}-4\pi a)\bigl)
\\
&=8\pi^{2} e^{-2\pi(a x^{2}+a^{-1}\xi^{2})}a
\bigl(a^{-1}\xi^{2}+ax^{2}-\frac 1{4\pi}\bigr).
\end{align*}
We obtain that the function $\mathcal W(u'_{a},u'_{a})$ is negative on 
$$
a^{-1}\xi^{2}+ax^{2}<\frac1{4\pi},
$$
which has area  $1/4$.
We may note as well for consistency that for $u_{a}$ given by \eqref{gg},
we have
$$
u'_{a}=(2a)^{1/4}(-2\pi ax) e^{-\pi a x^{2}},\quad \norm{u'_{a}}_{L^{2}}^{2}=\pi a,
$$
and
\begin{align*}
\iint \mathcal W(u'_{a},u'_{a})(x,\xi) dx d\xi&=8\pi^{2}a\iint
e^{-2\pi(y^{2}+\eta^{2})}(y^{2}+\eta^{2}-\frac1{4\pi}) dyd\eta
\\
&=\frac{8\pi^{2}a}{8\pi}
=\pi a=\norm{u'_{a}}_{L^{2}}^{2}.
\end{align*}
For $\lambda>0$ and $a\in \mathscr S'(\RZ)$, we define
\begin{equation}\label{segal-}
a_{\lambda}(x,\xi)=a(\lambda^{-1}x, \lambda \xi),
\end{equation}
and we find that
\begin{align}\label{segal0}
&(a_{\lambda})^{w}=U_{\lambda}^{*} a^{w} U_{\lambda},\\
&\text{for $f\in \mathscr S(\R^{n})$, }
(U_{\lambda} f)(x)=f(\lambda x)\lambda^{n/2}, \quad U_{\lambda}^{*}=U_{\lambda^{-1}}=(U_{\lambda})^{-1}.
\end{align}
We note that the above formula   is a particular case of {\it Segal's Formula} (see e.g. Theorem 2.1.2 in \cite{MR2599384}).
\subsubsection{The Harmonic Oscillator}\label{sechar}
\index{harmonic oscillator}
The Harmonic oscillator $\mathcal H_{n}$ in $n$ dimensions is defined as the operator with Weyl symbol
$\pi(\val x^{2}+\val \xi^{2})$ and thus from \eqref{segal0}, we find that 
$$
\mathcal H= U_{\sqrt{2\pi}}\  \frac12\bigl(\val x^{2}+4\pi^{2}\val \xi^{2}\bigr)^{w}U_{\sqrt{2\pi}} ^{*}=
U_{\sqrt{2\pi}}\  \frac12\bigl(-\Delta+\val x^{2}\bigr)U_{\sqrt{2\pi}} ^{*}.
$$
We shall define in one dimension the Hermite function of level $k\in \N$, 
by
\index{Hermite functions}
\index{{~\bf Notations}!$\psi_{k},\Psi_{\alpha}$}
\begin{equation}\label{herm1}
\psi_{k}(x)=\frac{(-1)^{k}}{2^{k}\sqrt{k! }   }  2^{1/4} e^{\pi x^{2}}\left(\frac{d}{\sqrt \pi dx}\right)^{k}(e^{-2\pi x^{2}}),
\end{equation}
and we find that $(\psi_{k})_{k\in \N}$ is a Hilbertian orthonormal basis on $L^{2}(\R)$.
The one-dimensional harmonic oscillator can be written as
\begin{equation}\label{hardec}
\mathcal H_{1}=\sum_{k\ge 0}(\frac12+k) \mathbb P_{k},
\end{equation}
where $\mathbb P_{k}$ is the orthogonal projection onto $\psi_{k}$.
\vs
In $\did $ dimensions, we consider a multi-index $(\alpha_{1},\dots, \alpha_{\did})=\alpha\in \N^{\did}$
and we define on $\R^{\did}$, using the one-dimensional \eqref{herm1},
\begin{equation}\label{psidim}
\Psi_{\alpha}(x)=\prod_{1\le j\le \did}\psi_{\alpha_{j}}(x_{j}),\quad \mathcal E_{k}=\text{Vect}\bigl\{\Psi_{\alpha}\bigr\}_{\alpha\in \mathbb N^{\did}, \val \alpha=k},
\quad \val \alpha=\sum_{1\le j\le \did}\alpha_{j}.
\end{equation}
We note that 
\begin{equation}\label{}
\text{the  dimension of $\mathcal E_{k,n}$ is $\binom{k+\did-1}{\did-1}$}
\end{equation}
and that \eqref{hardec} holds with 
$\mathbb P_{k;n}$ standing for the orthogonal projection onto $\mathcal E_{k,n}$;
the lowest eigenvalue of $\mathcal H_{n}$ is $n/2$ and the corresponding eigenspace is one-dimensional
in all dimensions, although in two and more  dimensions, the eigenspaces corresponding to  the eigenvalue 
$\frac n2+k, k\ge 1$ are multi-dimensional with dimension $\binom{k+\did-1}{\did-1}$.
The $n$-dimensional harmonic oscillator can be written as
\index{{~\bf Notations}!$\mathbb P_{k;n}, \mathbb P_{\alpha}$}
\begin{equation}\label{hardecn}
\mathcal H_{n}=\sum_{k\ge 0}(\frac n2+k) \mathbb P_{k;n},
\end{equation}
where $\mathbb P_{k;n}$ stands for the orthogonal projection 
onto $\mathcal E_{k,n}$ defined above.
We have in particular
\begin{equation}\label{app.paln}
\mathbb P_{k;n}=\sum_{\alpha\in \N^{n}, \val \alpha=k}\mathbb P_{\alpha},\quad
\text{where $\mathbb P_{\alpha}$ is the orthogonal projection onto $\Psi_{\alpha}$.}
\end{equation}
\index{anisotropic harmonic oscillator}
\subsubsection{On the spectrum of the anisotropic harmonic oscillator}
The standard $n$-dimensional harmonic oscillator is the operator
$$
\mathcal H_{n}=\pi\sum_{1\le j\le n}(D_j^2+x_j^2), \text{ $D_j=\frac{1}{2\pi i}\partial_{x_j}$},
$$
and its spectral decomposition is
$$
\mathcal H=\sum_{k\ge 0}(\frac{n}{2}+k)\mathbb P_{k;n}, \quad \mathbb P_{k;n}=\sum_{
\alpha\in \mathbb N^n,\alpha_1+\dots+\alpha_n=k}
\mathbb P_{\alpha_1}\otimes\dots\otimes\mathbb P_{\alpha_n},
$$
where $\mathbb P_{\alpha_j}$ stands for the orthogonal projection onto the one-dimensional Hermite function with level $\alpha_j$. Now let us consider for ${\mu}=({\mu}_1,\dots, {\mu}_n)$ with $\mu_j>0$,
the operator
\begin{equation}\label{harani}
\mathcal H_{(\mu)}=\pi\sum_{1\le j\le n}{\mu}_j(D_j^2+x_j^2) =\pi
\opw{q_{\mu}(x,\xi)},
\end{equation}
with
\begin{equation}\label{}
q_{\mu}(x,\xi)=\sum_{1\le j\le n}\mu_{j}(x_{j}^2+\xi_{j}^2).
\end{equation}
With the notation $\vert {\mu}\vert=\sum_{1\le j\le n} {\mu}_j$ and  ${\mu}\cdot \alpha=\sum_{1\le j\le n}{\mu}_j\alpha_j$, we have 
\begin{equation}\label{}
\mathcal H_{(\mu)}=\sum_{\alpha\in \mathbb N^n}(\frac{\vert {\mu}\vert}2+{\mu}\cdot\alpha)
\underbrace{(\mathbb P_{\alpha_1}\otimes\dots\otimes\mathbb P_{\alpha_n})}_{\mathbb P_{\alpha}},
\end{equation}
so that the eigenspaces are the same as for $\mathcal H_{n}$ but the arithmetic properties of ${\mu}$ make possible that all eigenvalues $(\frac{\vert {\mu}\vert}2+{\mu}\cdot\alpha)$ are simple.  For instance for $$n=2, 0< \mu_{1}< \mu_{2}, \ \frac{\mu_{2}}{\mu_{1}}\notin \Q,$$
if $\beta\in \Z^2$ is such that 
$
\mu_{1}\beta_{1}+\mu_{2}\beta_{2}=0,
$
this implies that $\beta=0$ and thus that all the eigenvalues of $\mathcal H_{(\mu)}$ are simple.
\begin{rem}
 If $0<\mu_{1}\le \dots\le \mu_{n}$ and if for all $j\in [2,n]$ we have $\mu_{j}/\mu_{1}\in \N$, we then have for $\alpha\in \N^n$,
$$
\alpha\cdot \mu=\mu_{1}\underbrace{\Bigl(\alpha_{1}+\sum_{2\le j\le n}\frac{\alpha_{j}\mu_{j}}{\mu_{1}}\Bigr)}_{\beta_{1}}=\beta\cdot \mu, \quad \beta=(\beta_{1},0,\dots,0)\in \N^n.
$$
\end{rem}
\noindent
{\it Sinus cardinal}.
It is a classical result of Distribution Theory that the weak limit when $\lambda\rightarrow+\io$
of the Sinus Cardinal $\frac{\sin(\lambda x)}x$ is $\pi \delta_{0}$, where $\delta_{0}$ is the Dirac mass at 0, but we wish to extend that result to more general test functions.
\begin{lem}\label{lem93}
 Let $f$ be a function in $L^1_{\text{loc}}(\R)$ such that
 $$\int_{\val \tau\ge 1}\frac{\val{f(\tau)}}{\val\tau}d\tau<+\io
 \quad\text{and $\exists a\in \C$ so that\quad }
 \int_{\val \tau\le 1}
 \frac{\val{f(\tau)-a}}{\val \tau}d\tau<+\io.$$
 Then we have
 \begin{equation}\label{9227}
\lim_{\lambda\rightarrow+\io}\int_{\R}\frac{\sin(\lambda \tau)}{\pi \tau} f(\tau) d\tau=a.
\end{equation}
\end{lem}
\begin{nb}
 In particular if $f$ is an  H\"olderian function such that $f(\tau)/\tau\in L^1(\{\val \tau\ge 1\})$ we get that the left-hand-side of \eqref{9227} equals $f(0)$.
\end{nb}
\begin{proof}
 Let $\chi_{0}$ be a function in $\mooc(\R)$ equal to 1 near the origin and let us define $\chi_{1}=1-\chi_{0}$. We have
\begin{multline*}
 \int_{\R}\frac{\sin(\lambda \tau)}{\pi \tau} f(\tau) d\tau=
 \int_{\R}\frac{\sin(\lambda \tau)}{\pi } \underbrace{\frac{\bigl(f(\tau) -a\bigr)}{\tau}\chi_{0}(\tau)}_{\in L^1(\R)} d\tau+
 a \int_{\R}\frac{\sin(\lambda \tau)}{\pi \tau} \chi_{0}(\tau) d\tau
 \\+\int_{\R}\frac{\sin(\lambda \tau)}{\pi }\underbrace{ f(\tau) \tau^{-1} \chi_{1}(\tau)}_{\in L^1(\R)}d\tau,
\end{multline*}
so that the limit when $\lambda\rightarrow+\io$ of the first and the third integral is zero, thanks to the Riemann-Lebesgue Lemma. We note also that
$$
\frac{\sin(\lambda \tau)}{\pi \tau} =\widehat{\mathbf 1_{[-\frac{\lambda}{2\pi},\frac{\lambda}{2\pi}]}}(\tau),
$$
and applying Plancherel's Formula to the second integral yields
$$
 \int_{\R}\frac{\sin(\lambda \tau)}{\pi \tau} \chi_{0}(\tau) d\tau=\int_{\val t\le \lambda/(2\pi)}\widehat{\chi_{0}}(t) dt ,
$$
whose limit  when $\lambda\rightarrow+\io$ is
$\int_{\R}\widehat{\chi_{0}}(t) dt=\chi_{0}(0)=1$, thanks to the Lebesgue Dominated Convergence Theorem, completing the proof of the lemma.
\end{proof}
\subsection{Further properties of the metaplectic group}
\index{generators of the metaplectic group}
\index{{~\bf Notations}!$\mathcal M_{P,L,Q}$}
\subsubsection{Another set of generators for the metaplectic group}
\begin{definition}\label{def.lersym} Let $P, L, Q$ be $n\times n$ real matrices such that $P=P^{*}, Q=Q^{*}$ and $\det L\not=0$.
We define the operator $\mathcal M_{P,L,Q}$ by the formula
\begin{equation}\label{lersym}
(\mathcal M_{P,L,Q} u)(x)=e^{-i\pi n/4}(\det L)^{1/2}\int_{\R^{n}} e^{i\pi\{
\poscal{Px}{x}-2\poscal{Lx}{y}+\poscal{Q y}{y}\}}
u(y) dy.
\end{equation}
\end{definition}
\index{{~\bf Notations}!$\Mett{P}{L}{Q}{m(L)}$}
\begin{nb}\rm 
 In that definition, $(\det L)^{1/2}$ stands for a choice of a square root of the real number $\det L$, that is $\pm\sqrt{\det L}$ if $\det L>0$ and $\pm i\sqrt{-\det L}$ if $\det L<0$.
\end{nb}
With $m(L)\in \Z/4\Z$ defined by \eqref{maslov} we shall also define
\begin{equation}\label{}
(\Mett{P}{L}{Q}{m(L)}u)(x)=e^{-\frac{i\pi n}4} e^{\frac{i\pi m(L)}{2}}\val{\det L}^{1/2}\int_{\R^{n}} e^{i\pi\{
\poscal{Px}{x}-2\poscal{Lx}{y}+\poscal{Q y}{y}\}}
u(y) dy.
\end{equation}
\index{generators of the metaplectic group}
\begin{pro}\label{pro.989898}
 The operator $\mathcal M_{P,L,Q}$ given in Definition \ref{def.lersym} is an automorphism of $\mathscr S(\R^{n})$ and of $\mathscr S'(\R^{n})$ which is a unitary operator on $L^{2}(\R^{n})$ belonging to the metaplectic group
 (cf. Definition \ref{def.41fdhh}). Moreover the metaplectic group is generated by the set $\{\mathcal M_{P,L,Q}\}_{\substack{P=P^{*}, Q=Q^{*}\\
 \det L\not=0}}$.
\end{pro}
\begin{proof}
 Using the Notation \eqref{4.genmet} and \eqref{1254}, we see that\footnote{
 We note that $m(B)+n\in\{m(-B),m(-B)+2\}$ modulo 4:
 indeed we have modulo 4
 $$
\begin{cases} 
\text{for $n$ even,}&\underbracket[0.2pt]{\{0,2\}}_{\det B >0}+n=\underbracket[0.2pt]{\{0,2\}}_{\det(-B)>0},
\qquad \underbracket[0.2pt]{\{1,3\}}_{\det B<0}+n=\underbracket[0.2pt]{\{1,3\}}_{\det (-B)<0},
\\
\text{for $n$ odd,}&\underbracket[0.2pt]{\{0,2\}}_{\det B >0}+n=\underbracket[0.2pt]{\{1,3\}}_{\det(-B)<0},
\qquad \underbracket[0.2pt]{\{1,3\}}_{\det B<0}+n=\underbracket[0.2pt]{\{0,2\}}_{\det (-B)>0}.
 \end{cases}
 $$
 We have also 
 $
 m(L)-n\in \{m(-L), m(-L)+2\}
 $
 since we know already (from the above in that footnote)
 that  $m(L)-n\in\{m(-L),m(-L)+2\}-2n$,
which gives 
 $m(L)-n\in\{m(-L),m(-L)+2\}$ for $n$ even; for $n=2l+1$ odd we get
 the same result since 
 $$
 m(L)-n\in\{m(-L),m(-L)+2\}-4l-2=\{m(-L)-2,m(-L)\}=\{m(-L)+2,m(-L)\}.
 $$
 }
 \begin{equation}\label{linkme}
 M_{A,B,C}^{\scriptscriptstyle\{m(B)\}}=\mathcal M_{A,-B,C}^{\scriptscriptstyle\{m(B)+n\}}\mathcal F e^{-i\pi n/4},\quad \mathcal M_{P,L,Q}^{\scriptscriptstyle\{m(L)\}}=M_{P,-L,Q}^{\scriptscriptstyle\{m(L)-n\}}
 \left(\mathcal F e^{-i\pi n/4}\right)^{-1},
\end{equation}
and \eqref{778803} imply that  the set $\{\mathcal M_{P,L,Q}\}$ is included in $Mp(n)$
(second formula in \eqref{linkme})
whereas  the fact that
\begin{equation}\label{923923}
\mathcal F e^{-i\pi n/4}=\Mett{0}{I_{n}}{0}{0},
\end{equation}
the first formula in \eqref{linkme} and Definition \ref{def.41fdhh} imply that $Mp(n)$ is generated by the
 set $\{\mathcal M_{P,L,Q}\}$, proving the Proposition.
 \end{proof}
 \begin{rem}\rm
 From \eqref{linkme}, we deduce, noting that {\footnotesize$ m(I_{n})\in \{0,2\}, m(-I_{n})\in\{n, n+2\}$},
 \begin{align}\label{}
-\Id_{L^{2}(\R^{n})}&=M_{0,I_{n},0}^{\scriptscriptstyle\{2\}}=\Mett{0}{-I_{n}}{0}{n+2}\Mett{0}{I_{n}}{0}{0},
\qquad \text{so that}
\\
\Mett{P}{L}{Q}{m(L)+2}&=
-\Mett{P}{L}{Q}{m(L)}=\Mett{0}{-I_{n}}{0}{n+2}\Mett{0}{I_{n}}{0}{0}\Mett{P}{L}{Q}{m(L)}.
\end{align}
\end{rem}
 \index{{~\bf Notations}!$\Lambda_{P,L,Q}$}
\begin{lem} With the homomorphism $\Psi$ defined in \eqref{hqma28} and defining 
\begin{equation}\label{}
\Lambda_{P,L,Q}=\Psi(\mathcal M_{P,L,Q}),
\end{equation}
we find that 
\begin{equation}\label{}
\Lambda_{P,L,Q}=\mat22{L^{-1}Q}{L^{-1}}{PL^{-1}Q-L^{*}}{PL^{-1}}.
\end{equation}
\end{lem}
\begin{proof}
 Indeed, from the second formula in \eqref{linkme}, \eqref{kjh123}, \eqref{noti22} and \eqref{1270} we get that
 $$
 \Lambda_{P,L,Q}=\Xi_{P,-L,Q}\ \Xi_{-I_{n}, 2^{1/2}I_{n}, -I_{n}}^{-2}
 =\mat22{-L^{-1}}{L^{-1}Q}{-PL^{-1}}{-L^{*}+PL^{-1}Q}\mat22{0}{-I_{n}}{I_{n}}{0},
 $$
 providing the sought result.
\end{proof}
\begin{lem}\label{lem.910910}
 Let $P_{j}, L_{j}, Q_{j}, j=1,2$ be as in Definition \ref{def.lersym} and let us assume that 
 \begin{equation}\label{926926}
\mathcal M_{P_{1},L_{1}, Q_{1}}\mathcal M_{P_{2},L_{2}, Q_{2}}=e^{i\phi}\Id_{L^{2}(\R^{n})}, \quad \phi\in \R.
\end{equation}
Then we have
\begin{equation}\label{927927}
P_{1}+Q_{2}=Q_{1}+P_{2}=0,\quad L_{2}=-L_{1}^{*},  \qquad e^{i\phi}\in \{\pm 1\}.
\end{equation}
\end{lem}
\begin{proof}
The assumption \eqref{926926} implies that both sides of the equality belong to $Mp(n)$ and 
$$
\Lambda_{P_{1}, L_{1},Q_{1}}\Lambda_{P_{2}, L_{2},Q_{2}}=\Psi\bigl(e^{i\phi}\Id_{L^{2}(\R^{n})}\bigr)=I_{2n},
$$
where the last equality  follows from the fact that $e^{i\phi}\Id_{L^{2}(\R^{n})}$ commutes with every operator
$\opw{L_{Y}}$ given in Lemma \ref{lem.54dftz}.
We have thus
$$
\mat22{L_{1}^{-1}Q_{1}}{L_{1}^{-1}}{P_{1}L_{1}^{-1}Q_{1}-L_{1}^{*}}{P_{1}L_{1}^{-1}}\mat22{L_{2}^{-1}Q_{2}}{L_{2}^{-1}}{P_{2}L_{2}^{-1}Q_{2}-L_{2}^{*}}{P_{2}L_{2}^{-1}}=\mat22{I_{n}}{0}{0}{I_{n}},
$$ 
so that
\begin{align*}
\text{\scriptsize first line $\times$ second column:\hs}&L_{1}^{-1}Q_{1} L_{2}^{-1}+L_{1}^{-1}P_{2}L_{2}^{-1}=0\Longrightarrow Q_{1}+P_{2}=0,\\
\text{\scriptsize second line $\times$ first column:\hs}&\color{bured}(P_{1}L_{1}^{-1}Q_{1}-L_{1}^{*})L_{2}^{-1}Q_{2}+P_{1}L_{1}^{-1}(P_{2}L_{2}^{-1}Q_{2}-L_{2}^{*})=0,\\
\text{\scriptsize second line $\times$ second column:\hs}&\color{blue}(P_{1}L_{1}^{-1}Q_{1}-L_{1}^{*})L_{2}^{-1}+P_{1}L_{1}^{-1}P_{2}L_{2}^{-1}=I_{n},
\end{align*}
which gives
\begin{align*}
&\color{blue}(P_{1}L_{1}^{-1}Q_{1}-L_{1}^{*})L_{2}^{-1}+P_{1}L_{1}^{-1}\underbracket[0.3pt]{P_{2}}_{-Q_{1}}L_{2}^{-1}=I_{n}
\Longrightarrow -L_{1}^{*}L_{2}^{-1}=I_{n}\Longrightarrow L_{2}=-L_{1}^{*},\\
&\color{bured}P_{1}L_{1}^{-1}Q_{1}L_{2}^{-1}Q_{2}
\underbrace{-L_{1}^{*}L_{2}^{-1}}_{I_{n}}Q_{2}
+P_{1}L_{1}^{-1}\underbracket[0.3pt]{P_{2}}_{-Q_{1}}L_{2}^{-1}Q_{2}
-P_{1}\underbrace{L_{1}^{-1}L_{2}^{*}}_{-I_{n}}=0\Longrightarrow P_{1}+Q_{2}=0,
\end{align*}
providing the sought formulas in \eqref{927927}, except for the last one.
 Let $\kappa_{j}$ be the kernel of $\mathcal M_{P_{j},L_{j}, Q_{j}}$ and let $\kappa=\kappa_{1}\circ\kappa_{2}$
 be the kernel of the composition (in the lhs of \eqref{926926}).
 We have consequently 
\begin{align*}
 \kappa(x,y)&=(\det L_{1})^{1/2}(\det L_{2})^{1/2} e^{-i\pi n/2}
 \int e^{i\pi\{P_{1}x^{2}-2L_{1}x\cdot z+Q_{1}z^{2}
+P_{2}z^{2}-2L_{2}z\cdot y+Q_{2}y^{2}\}} dz
\\
&=(\det L_{1})^{1/2}(\det (-L_{1}^{*}))^{1/2} e^{-i\pi n/2}
e^{i\pi\{P_{1}x^{2}-P_{1}y^{2}\}}
\int e^{-2i\pi\{z\cdot(L_{1}x+L_{2}^{*}y)\}}dz
\\&=
(\det L_{1})^{1/2}(\det (-L_{1}^{*}))^{1/2} e^{-i\pi n/2}
e^{i\pi\{P_{1}x^{2}-P_{1}y^{2}\}} \delta_{0}\bigl(L_{1}x+L_{2}^{*}y\bigr)
\\&=
(\det L_{1})^{1/2}(\det (-L_{1}^{*}))^{1/2} e^{-i\pi n/2}
e^{i\pi\{P_{1}x^{2}-P_{1}y^{2}\}} \delta_{0}\bigl(x-y\bigr)\val{\det L_{1}}^{-1},
\end{align*}
 entailing
 $$e^{i\phi}\delta_{0}(x-y)\underbracket[0.3pt]{=}_{\eqref{926926}}
 \kappa(x,y)=e^{i\frac\pi 2 \left(m(L_{1})+m(L_{1}^{*})+n   \right)}\delta_{0}(x-y)e^{-i\pi n/2}= e^{i\pi m(L_{1})}\delta_{0}(x-y),
 $$
 proving that $e^{i\phi}=e^{i\pi m(L_{1})}\in \{\pm 1\}$. The proof of the lemma is complete.
\end{proof}
\begin{claim}\label{claim.911}
   Let $P, L, Q$ be as in Definition \ref{def.lersym}. Then we have 
   \begin{equation}\label{929929}
(\mathcal M_{P,L,Q}^{\scriptscriptstyle\{m(L)\}})^{-1}=\mathcal M_{-Q,-L^{*},-P}^{\scriptscriptstyle\{n-m(L)\}},
\end{equation}
and moreover $n-m(L)\in\{m(-L^{*}), m(-L^{*})+2\}$ modulo 4.
\end{claim}
\no
Indeed, calculating the kernel $\kappa$ of $\mathcal M_{P,L,Q}^{\scriptscriptstyle\{m(L)\}}
\mathcal M_{-Q,-L^{*},-P}^{\scriptscriptstyle\{n-m(L)\}}$,
we get
\begin{multline*}
\kappa(x,y)
=e^{\frac{i\pi}2(m(L)+n-m(L)-n)}
\val{\det L}
 \int e^{i\pi\{Px^{2}-2Lx\cdot z+Qz^{2}
-Qz^{2}+2L^{*}z\cdot y-Py^{2}\}} dz
\\
= \val{\det L}
e^{i\pi\{Px^{2}-P y^{2}\}} \delta_{0}(Lx-Ly)=
 \delta_{0}(x-y),
\end{multline*}
so that 
$\mathcal M_{P,L,Q}^{\scriptscriptstyle\{m(L)\}}
\mathcal M_{-Q,-L^{*},-P}^{\scriptscriptstyle\{n-m(L)\}}=\Id_{L^{2}(\R^{n})}$ and since 
$\mathcal M_{P,L,Q}$ is unitary, this proves \eqref{929929}.
The last assertion is equivalent to
$
m(L)\in\{n-m(-L^{*}),n-m(-L^{*})-2\}.
$
Since the latter
set is equal to $\{-m(L), -m(L)-2\}$ and the mapping 
$$\Z/4\Z\ni x\mapsto-x\in \Z/4\Z,$$
leaves invariant the sets $\{0,2\}, \{1,3\}$,
we obtain the sought result,
concluding the proof of the claim.\qed
\begin{pro}\label{pro.912912}
  Let $P_{j}, L_{j}, Q_{j}, j=1,2$ be as in Definition \ref{def.lersym} and let us assume that 
\begin{equation}\label{}
\det(Q_{1}+P_{2})\not=0.
\end{equation}
Then there exist $P,L,Q,$ as in Definition \ref{def.lersym}  such that 
\begin{equation}\label{}
\mathcal M_{P_{1},L_{1}, Q_{1}}^{\scriptscriptstyle{\{m(L_{1})\}}}
\mathcal M_{P_{2},L_{2}, Q_{2}}^{\scriptscriptstyle{\{m(L_{2})\}}}
=\mathcal M_{P,L, Q}^{\scriptscriptstyle{\{m(L_{1})+m(L_{2})-\inde{(Q_{1}+P_{2})\}}}}.
\end{equation}
More precisely, we have 
\begin{gather}
P=P_{1}-L_{1}^{*}(Q_{1}+P_{2})^{-1}L_{1}, \qquad Q=Q_{2}-L_{2}(Q_{1}+P_{2})^{-1}L_{2}^{*},
\label{9212-}\\
L=L_{2}(Q_{1}+P_{2})^{-1}L_{1}.\label{9213-}
\end{gather}
Moreover we have
\begin{equation}\label{indstu}
m(L_{1})+m(L_{2})-\inde{(Q_{1}+P_{2})}\in\{m(L), m(L)+2\}\ \mod 4.
\end{equation}

\end{pro}
\begin{proof}
 The kernel $\kappa$ of $\mathcal M_{P_{1},L_{1}, Q_{1}}\mathcal M_{P_{2},L_{2}, Q_{2}}$ is
 \begin{align*}
 \kappa(x,y)
 &=(\det L_{1})^{1/2}(\det L_{2})^{1/2} e^{-i\pi n/2}
 \int e^{i\pi\{P_{1}x^{2}-2L_{1}x\cdot z+Q_{1}z^{2}
+P_{2}z^{2}-2L_{2}z\cdot y+Q_{2}y^{2}\}} dz
 \\&=
(\det L_{1})^{1/2}(\det L_{2})^{1/2} e^{-i\pi n/2}
e^{i\pi\{P_{1}x^{2}+Q_{2}y^{2}\}}
\int e^{-2i\pi(L_{1}x+L_{2}^{*}y)\cdot z}
e^{i\pi(Q_{1}+P_{2})z^{2}} dz
\\
&=(\det L_{1})^{1/2}(\det L_{2})^{1/2} e^{-i\pi n/2}
e^{i\pi\{P_{1}x^{2}+Q_{2}y^{2}\}}
e^{-i\pi(Q_{1}+P_{2})^{-1}(L_{1}x+L_{2}^{*}y)^{2}}
\\
&\hskip195pt \times\val{\det(Q_{1}+P_{2})}^{-1/2}
e^{i\frac\pi 4\sign (Q_{1}+P_{2})},
\end{align*}
according to Formula \eqref{foimga} (see also \eqref{foimga++}), noting that the matrix $Q_{1}+P_{2}$ is real symmetric and non-singular.
As a result, we have 
\begin{multline*}
\kappa(x,y)=e^{i\pi\left\{\left(P_{1}-L_{1}^{*}(Q_{1}+P_{2})^{-1}L_{1}\right)x^{2}
+\left(Q_{2}-L_{2}(Q_{1}+P_{2})^{-1}L_{2}^{*}y^{2}\right)
\right\}}
e^{-2i\pi\left\{
L_{2}(Q_{1}+P_{2})^{-1}L_{1}x\cdot y
\right\}}
\\
\times (\det L_{1})^{1/2}(\det L_{2})^{1/2} e^{-i\pi n/2}
\val{\det(Q_{1}+P_{2})}^{-1/2}
e^{i\frac\pi 4\sign (Q_{1}+P_{2})}.
\end{multline*}
We note that, with $\mathtt {E_{12}}$ standing for the eigenvalues of $Q_{1}+P_{2}$,
\begin{align*}
\nu_{+}=\card(\mathtt{ E_{12}}\cap\R_{+}),\qquad
\nu_{-}=\card(\mathtt {E_{12}}\cap\R_{-})=\inde{(Q_{1}+P_{2})},
\end{align*}
implying that the kernel $\kappa$ is given by 
\begin{equation}\label{9211}
\kappa(x,y)=
e^{i\frac \pi 2\left(m(L_{1})+m(L_{2})-n+\frac12(\nu_{+}-\nu_{-})\right)}
\val{\det L}^{1/2}
e^{i\pi\{P x^{2}-2Lx\cdot y+Q y^{2}\}},
\end{equation}
with
\begin{gather}
P=P_{1}-L_{1}^{*}(Q_{1}+P_{2})^{-1}L_{1}, \qquad Q=Q_{2}-L_{2}(Q_{1}+P_{2})^{-1}L_{2}^{*},
\label{9212}\\
L=L_{2}(Q_{1}+P_{2})^{-1}L_{1}.\label{9213}
\end{gather}
Checking the unit factor in front of the rhs of \eqref{9211}, we note that $\nu_{+}+\nu_{-}=n$ since $Q_{1}+P_{2}$
is non-singular and  we get
\begin{multline*}
e^{i\frac \pi 2\left(m(L_{1})+m(L_{2})-n+\frac12(\nu_{+}-\nu_{-})\right)}=e^{-\frac{i\pi n}4}
e^{i\frac \pi 2\left(m(L_{1})+m(L_{2})-\frac n2+\frac12(\nu_{+}-\nu_{-})\right)}
\\=e^{-\frac{i\pi n}4}
e^{i\frac \pi 2\left(m(L_{1})+m(L_{2})-\nu_{-}\right)}.
\end{multline*}
We have also, since $\inde{(Q_{1}+P_{2})}=\inde{(Q_{1}+P_{2})^{-1}}$,
\begin{align*}
\left(e^{i\frac \pi 2\left(m(L_{1})+m(L_{2})-\nu_{-}\right)}\right)^{2}&=\sign(\det L_{1})
\sign(\det L_{2}) (-1)^{\nu_{-}}
\\
&=\sign(\det L_{1})
\sign(\det L_{2}) \sign(\det(Q_{1}+P_{2})^{-1})
\\&=\sign(\det L),
\end{align*}
entailing that
$
\kappa(x,y)=e^{-\frac{i\pi n}4}(\det L)^{1/2}e^{i\pi\{P x^{2}-2Lx\cdot y+Q y^{2}\}},
$
concluding the proof of the proposition.
\end{proof}
\begin{lem}\label{lem.913913}
  Let $P_{j}, L_{j}, Q_{j}, j=1,2,3$ be as in Definition \ref{def.lersym}. Then there exist
  $(P', L',Q'), (P'', L'', Q'')$ as in Definition \ref{def.lersym} such that
  \begin{equation}\label{91214}
\mathcal M_{P_{1}, L_{1},Q_{1}}\mathcal M_{P_{2}, L_{2},Q_{2}}\mathcal M_{P_{3}, L_{3},Q_{3}}
=
\mathcal M_{P', L',Q'}\mathcal M_{P'', L'',Q''}.
\end{equation}
\end{lem}
\begin{proof}
If $\det (Q_{1}+P_{2})\not=0$, Lemma \ref{pro.912912} implies that $\mathcal M_{P_{1}, L_{1},Q_{1}}\mathcal M_{P_{2}, L_{2},Q_{2}}=\mathcal M_{P', L',Q'}$ so that \eqref{91214} is satisfied with $(P'',L'',Q'')=(P_{3},L_{3},Q_{3}).$ We may thus assume in the sequel that $\det (Q_{1}+P_{2})=0$. Then the kernel of $Q_{1}+P_{2}$ is of dimension $r\in \llbracket 1,n\rrbracket$; let us define $J_{r}$ as the orthogonal projection onto 
$\ker (Q_{1}+P_{2})$.
\begin{claim}
 The matrix $J_{r}+ (Q_{1}+P_{2})^{2}$ is positive definite (thus invertible).
\end{claim}
\no
Indeed, if $J_{r}x+ (Q_{1}+P_{2})^{2}x=0$, we obtain by taking the dot-product with $x$  that 
$$\norm{J_{r}x}^{2}+\norm{(Q_{1}+P_{2})x}^{2}=0\Longrightarrow x\in \ker (Q_{1}+P_{2}), J_{r}x=0\Longrightarrow x=0.
$$
This matrix is also non-negative, proving the claim.\qed
\vs
Let us define the real $n\times n$ symmetric matrix
\begin{equation}\label{oi54zz}
P=\mu L_{2}\bigl[J_{r}+ (Q_{1}+P_{2})^{2}\bigr]^{-1}L_{2}^{*}-Q_{2},\end{equation}
where $\mu$ is a positive parameter to be chosen later;
we note that $P+Q_{2}$
is invertible.
Also 
we have
\begin{equation}\label{9216}
L_{2}^{*}\bigl(Q_{2}+P\bigr)^{-1}L_{2}
-(Q_{1}+P_{2})
=
\mu^{-1}\bigl[J_{r}+ (Q_{1}+P_{2})^{2}\bigr]
-(Q_{1}+P_{2}),
\end{equation}
which is invertible if $\mu$ (is different from 0 and) does not meet the spectrum of $Q_{1}+P_{2}$.\footnote{
The symmetric matrices $Q_{1}+P_{2}$ and $J_{r}$ can be diagonalized simultaneously so that the invertibility of
$\mu^{-1}\bigl[J_{r}+ (Q_{1}+P_{2})^{2}\bigr]
-(Q_{1}+P_{2})$ is equivalent to
$$
\mu\not=0, \quad \mu^{-1}\lambda_{j}^{2}\not=\lambda_{j}\text{  i.e.  } \mu\not=\lambda_{j},
\text{ where the $\lambda_{j}$ are the non-zero eigenvalues of $Q_{1}+P_{2}$.}
$$
}
We have also
\begin{multline*}
P-P_{3}=\mu L_{2}\bigl[J_{r}+ (Q_{1}+P_{2})^{2}\bigr]^{-1}L_{2}^{*}-(Q_{2}+P_{3})
\\
=L_{2}\Bigl\{\mu
\bigl[J_{r}+ (Q_{1}+P_{2})^{2}\bigr]^{-1}
-L_{2}^{-1}(Q_{2}+P_{3})L_{2}^{*-1}
\Bigr\}
L_{2}^{*},
\end{multline*}
which is invertible for $\mu$ large enough.\footnote{
Indeed the eigenvalues of 
$\bigl[J_{r}+ (Q_{1}+P_{2})^{2}\bigr]^{-1}$
are $1$ and $\lambda_{j}^{-2}$ where the $\lambda_{j}$ are the non-zero eigenvalues of $Q_{1}+P_{2}$.
To secure the invertibility of $P-P_{3}$, it is thus enough to have
$$
\min(\mu, \mu\lambda_{j}^{-2})> \norm{L_{2}^{-1}(Q_{2}+P_{3})L_{2}^{*-1}},
$$
\text{ where the $\lambda_{j}$ are the non-zero eigenvalues of $Q_{1}+P_{2}$.}
}
Eventually, defining 
\begin{equation}\label{}
\lambda_{0}=\max (\spectrum\val {Q_{2}+P_{1}}),
\end{equation}
the condition
\begin{equation}\label{}
\mu>\max\bigl\{
\lambda_{0}, 
\norm{L_{2}^{-1}(Q_{2}+P_{3})L_{2}^{*-1}}, \norm{L_{2}^{-1}(Q_{2}+P_{3})L_{2}^{*-1}}
\lambda_{0}^{2}\bigr\},
\end{equation}
implies that, with $P$ given by \eqref{oi54zz}, we obtain that  
\begin{equation}\label{9219}
\text{the matrices }P+Q_{2},\  Q_{1}+P_{2}-
L_{2}^{*}\bigl(Q_{2}+P\bigr)^{-1}L_{2},\  P-P_{3}\ \text{are invertible.} 
\end{equation}
 Using now Lemma \ref{pro.912912} and the first property in \eqref{9219}, we get that
 we can find $\tilde P, \tilde L, \tilde Q$ as in Definition \ref{def.lersym} such that 
 \begin{equation}\label{}
\mathcal M_{P_{2}, L_{2}, Q_{2}}\mathcal M_{P, I_{n}, 0}=\mathcal M_{\tilde P, \tilde L, \tilde Q},
\end{equation}
with (thanks to \eqref{9212}),
\begin{equation}\label{}
\tilde P=P_{2}-L_{2}^{*}(Q_{2}+P)^{-1}L_{2}.
\end{equation}
We check now 
\begin{equation}\label{9222}
\mathcal M_{P_{1}, L_{1}, Q_{1}}\mathcal M_{P_{2}, L_{2}, Q_{2}}\mathcal M_{P, I_{n}, 0}=
\mathcal M_{P_{1}, L_{1}, Q_{1}}\mathcal M_{\tilde P, \tilde L, \tilde Q},
\end{equation}
and we note that 
$$
Q_{1}+\tilde P=Q_{1}+P_{2}-L_{2}^{*}(Q_{2}+P)^{-1}L_{2}\quad\text{is invertible,}
$$
thanks to the second property in \eqref{9219} so that, from Lemma \ref{pro.912912}, we can find 
 $P', L', Q'$ as in Definition \ref{def.lersym} such that
 \begin{equation}\label{}
\mathcal M_{P_{1}, L_{1}, Q_{1}}\mathcal M_{\tilde P, \tilde L, \tilde Q}=\mathcal M_{P',L', Q'},
\end{equation}
and this yields
\begin{equation}\label{9224}
\mathcal M_{P_{1}, L_{1}, Q_{1}}\mathcal M_{P_{2}, L_{2}, Q_{2}}\mathcal M_{P, I_{n}, 0}=\mathcal M_{P',L', Q'}.
\end{equation}
Finally, we check
 \begin{equation*}\label{}
\underbracket[0.3pt]{ \mathcal M_{P,I_{n},0}^{-1}}_{
\substack{=\mathcal M_{0,-I_{n}, -P}\\ \text{cf. Claim \ref{claim.911}}}
}
\mathcal M_{P_{3}, L_{3}, Q_{3}}
=\mathcal M_{0,-I_{n}, -P}\mathcal M_{P_{3}, L_{3}, Q_{3}},
\end{equation*}
and since $-P+P_{3}$ is invertible
(thanks to the third property in \eqref{9219}), we obtain, using once again Lemma  \ref{pro.912912}, 
that
 we can find $P'', L'', Q''$ as in Definition \ref{def.lersym} such that 
 \begin{equation}\label{9225}
\mathcal M_{P,I_{n},0}^{-1}\mathcal M_{P_{3}, L_{3}, Q_{3}}=\mathcal M_{P'', L'', Q''}.
\end{equation}
Gathering the information above, we find that
\begin{multline}\label{}
\mathcal M_{P_{1}, L_{1}, Q_{1}}\mathcal M_{P_{2}, L_{2}, Q_{2}}
\mathcal M_{P_{3}, L_{3}, Q_{3}}
\\=
\underbracket[0.3pt]{\mathcal M_{P_{1}, L_{1}, Q_{1}}\mathcal M_{P_{2}, L_{2}, Q_{2}}\mathcal M_{P,I_{n},0}}_{\mathcal M_{ P',  L', Q'},\ \eqref{9224}}
\underbracket[0.3pt]
{\mathcal M_{P,I_{n},0}^{-1}
\mathcal M_{P_{3}, L_{3}, Q_{3}}}_{\mathcal M_{ P'',  L'', Q''},\ \eqref{9225}},
\end{multline}
which ends the proof of the lemma.
\end{proof}
\begin{pro}\label{pro.915915}
 The metaplectic group $Mp(n)$ is equal to the set
\begin{equation}\label{}
 \left\{\mathcal M_{P_{1}, L_{1}, Q_{1}}\mathcal M_{P_{2}, L_{2}, Q_{2}}\right\}_{
 \substack
 {P_{j}=P_{j}^{*}, Q_{j}=Q_{j}^{*}\\        \det L_{j}\not=0}
 }
\end{equation}
In other words, every metaplectic operator of $Mp(n)$ is the product of two operators of type 
$\mathcal M_{P, L, Q}$ as given by Definition \ref{def.lersym}.
\end{pro}
\begin{proof}
From Proposition \ref{pro.989898}, the metaplectic group is generated by the $\mathcal M_{P,L,Q}$ and since the inverse of $\mathcal M_{P,L,Q}$ is $\mathcal M_{-Q,-L^{*},-P}$, thanks to Claim \ref{claim.911}, it is enough to check the products  $\mathcal M_{P_{1}, L_{1}, Q_{1}}\dots \mathcal M_{P_{N}, L_{N}, Q_{N}}$ for $N\ge 3$.
Lemma \ref{lem.913913} is tackling the case $N=3$ and a trivial recurrence on $N$ provides the result of the proposition.
 \end{proof}
 \begin{theorem}\label{thm.phaseb}
 Let $M$ be an element of $Mp(n)$ such that $M=e^{i\phi}\Id_{L^{2}(\R^{n})}, \phi\in \R.$
 Then $e^{i\phi}$ belongs to the set $\{-1,1\}$. In other words, the intersection of the metaplectic group with the unit circle (identified to the unitary operators in $L^{2}(\R^{n})$ defined by the mappings $v\mapsto z v$ where $z\in\mathbb S^{1}\subset \C$)
 is reduced to the set $\{-1,1\}$.
 \end{theorem}
 \begin{proof}
 Using Proposition \ref{pro.915915}, the result follows from Lemma \ref{lem.910910}.
\end{proof}
We may go back to the description given by Proposition \ref{pro.kj77qq} and Definition \ref{def.41fdhh}.
\begin{pro}\label{pro.917917}
 The metaplectic group $Mp(n)$ is equal to the set
\begin{equation}\label{}
 \left\{M_{A_{1}, B_{1}, C_{1}}M_{A_{2}, B_{2}, C_{2}}\right\}_{
 \substack
 {A_{j}=A_{j}^{*}, C_{j}=C_{j}^{*}\\        \det B_{j}\not=0}
 },
\end{equation}
where the operators $M_{A,B,C}$ are defined in 
Proposition \ref{pro.kj77qq}.
\end{pro}
\begin{proof} Let $M$ be in $Mp(n)$. We have
\begin{multline*}
M=(M_{A_{1},B_{1}, C_{1}})^{\pm 1}\dots(M_{A_{N},B_{N}, C_{N}})^{\pm 1}
\\
\underbracket[0.3pt]{=}_{\eqref{linkme}}\left(\mathcal M_{A_{1},-B_{1}, C_{1}} e^{-i\pi n/4}\mathcal F\right)^{\pm 1}
\dots
\left(\mathcal M_{A_{N},-B_{N}, C_{N}} e^{-i\pi n/4}\mathcal F\right)^{\pm 1}
\\
=
\left(\mathcal M_{A_{1},-B_{1}, C_{1}} \mathcal M_{0,I_{n},0}\right)^{\pm 1}
\dots
\left(\mathcal M_{A_{N},-B_{N}, C_{N}} \mathcal M_{0,I_{n},0}\right)^{\pm 1},
\end{multline*}
 and since from Claim \ref{claim.911}, we have $\mathcal M_{A,B,C}^{-1}=\mathcal M_{-C, -B^{*}, -A}$,
 we find that $M$ is in fact a product of $2N$ terms of type $\mathcal M_{P,L,Q}$,
 and thanks to Proposition \ref{pro.915915}, we get
 \begin{align*}
 M&=\mathcal M_{P_{1},L_{1},Q_{1}}\mathcal M_{P_{2},L_{2},Q_{2}}
 =\underbrace{\mathcal M_{P_{1},L_{1},Q_{1}}e^{-i\pi n/4}\mathcal F}_{M_{P_{1},-L_{1}, Q_{1}}}
 \underbrace{\left(e^{-i\pi n/4}\mathcal F\right)^{-1}\mathcal M_{P_{2},L_{2},Q_{2}}}_{
 \left(\mathcal M_{-Q_{2},-L_{2}^{*},-P_{2}}e^{-i\pi n/4}\mathcal F\right)^{-1}
 }
 \\
 &=M_{P_{1},-L_{1}, Q_{1}}\left(M_{-Q_{2},-L_{2}^{*},-P_{2}}\right)^{-1}
 =M_{P_{1}, -L_{1}, 0}M_{0,I_{n},Q_{1}}
 \left(M_{-Q_{2}, -L_{2}^{*}, 0}
 M_{0, I_{n}, -P_{2}}\right)^{-1}
 \\
 &=M_{P_{1}, -L_{1}, 0}M_{0,I_{n},Q_{1}} M_{0, I_{n}, P_{2}}
 \left(M_{-Q_{2}, -L_{2}^{*}, 0}\right)^{-1}
\\& =M_{P_{1}, -L_{1}, 0}M_{0,I_{n},Q_{1}+P_{2}}
 \left(M_{-Q_{2}, -L_{2}^{*}, 0}\right)^{-1}\qquad\text{\tiny (cf. Formula \eqref{54gvrs})}
 \\
 &=M_{P_{1}, -L_{1}, Q_{1}+P_{2}}M_{A'', B'', 0}\quad\text{\tiny (cf. Lemma \ref{lem.784512} below in the next subsection)},
\end{align*}
proving the Proposition.
\end{proof}
\subsubsection{On some subgroups of the metaplectic group}\label{sec.metapl}
We have seen in \eqref{4.eqsym'}, \eqref{4.eqsyma} some equivalent conditions for a matrix
\begin{equation}\label{fgw123}
\Xi=\begin{pmatrix}
P&Q\\
R&S
\end{pmatrix}
\quad\text{where $P,Q,R, S$ are  $n\times n$ real matrices,}
\end{equation}
to be symplectic. 
We note here that when $\Xi\in Sp(n,\R)$, we have 
\begin{equation}\label{arfd58}
\Xi^{-1}=\begin{pmatrix}
S^{*}&-Q^{*}\\
-R^{*}&P^{*}
\end{pmatrix},
\end{equation}
as it is easily checked from \eqref{4.eqsym'}, \eqref{4.eqsyma}. 
When $\det P\not=0$, we proved  that $\Xi=\Xi_{A,B,C}$ as defined in  \eqref{gensym}.
Also from  \eqref{arfd58}, we get 
that if $\det S\not =0$ we have 
$$
\Xi^{-1}=\Xi_{A,B,C},
$$
so that 
\begin{equation}\label{}
\Xi=\begin{pmatrix}
I_{n}&C\\
0&I_{n}
\end{pmatrix}
\begin{pmatrix}
B&0\\
0&B^{*-1}
\end{pmatrix}
\begin{pmatrix}
I_{n}&0\\
-A&I_{n}
\end{pmatrix}.
\end{equation}
Some other properties of the same type  are available when $\det Q$ or $\det R$ are different from $0$. Indeed we have for $\Xi\in Sp(n,\R)$ and $\sigma$ given by \eqref{ffqq99},
\begin{equation}\label{singsing}
\Xi \sigma=
\begin{pmatrix}
P&Q\\
R&S
\end{pmatrix} \sigma
=\begin{pmatrix}
-Q&P\\
-S&R
\end{pmatrix}\underbracket[0.3pt]{=}_{\text{if $\det Q\not=0$}}\Xi_{A,B,C},
\end{equation}
so that 
\begin{equation}\label{}
\Xi =-\Xi_{A,B,C} \sigma=
\begin{pmatrix}
I_{n}&0\\
A&I_{n}
\end{pmatrix}
\begin{pmatrix}
B^{-1}&0\\
0&B^{*}
\end{pmatrix}
\begin{pmatrix}
I_{n}&-C\\
0&I_{n}
\end{pmatrix}
\begin{pmatrix}
0&-I_{n}\\
I_{n}&0
\end{pmatrix}.
\end{equation}
If we have $\det R\not=0$, using the two first equalities in \eqref{singsing}, we get that 
$
(\Xi \sigma)^{-1}=\Xi_{A,B,C} 
$
which gives
\begin{equation}\label{}
\Xi=
\begin{pmatrix}
I_{n}&C\\
0&I_{n}
\end{pmatrix}
\begin{pmatrix}
B&0\\
0&B^{*-1}
\end{pmatrix}
\begin{pmatrix}
I_{n}&0\\
-A&I_{n}
\end{pmatrix}
\begin{pmatrix}
0&-I_{n}\\
I_{n}&0
\end{pmatrix}.
\end{equation}
However, it is indeed possible when $n\ge 2$ to have a symplectic matrix in $Sp(n,\R)$ in the form \eqref{fgw123} such that all the blocks are singular,
as shown in the following remark.
\begin{rem}
 \rm The $4\times 4$ matrix
 $$
 \begin{pmatrix}
\begin{matrix}
0&0\\
0&1
\end{matrix}
&\begin{matrix}
1&0\\
0&0
\end{matrix}\\
\begin{matrix}
-1&0\\
0&0
\end{matrix}&
\begin{matrix}
0&0\\
0&1
\end{matrix}
\end{pmatrix}
=\begin{pmatrix}
P&Q
\\
R&S
\end{pmatrix}
 $$ 
 belongs to $Sp(2, \R)$
 although all the block $2\times 2$
 matrices $P,Q,R, S,$
 are singular (with rank 1).
\end{rem}
\begin{lem}\label{lem.784512} With $M_{A,B,C}$ defined in Proposition \ref{pro.kj77qq}, the sets 
\begin{equation}\label{921456}
\mathcal L=\{M_{A, B, 0}\}_{\substack{
A=A^{*}\\
\det B\not=0}
},
\qquad
\mathcal R=\{M_{0, B, C}\}_{\substack{
C=C^{*}\\
\det B\not=0}
},
\end{equation}
are subgroups of the metaplectic group (cf. Definition \ref{def.41fdhh}).
 \end{lem}
 \begin{proof}
 Indeed $\mathcal L$ contains the identity of $L^{2}(\R^{n})$ and we have for $v\in L^{2}(\R^{n})$, 
 \begin{align*}
&M_{A_{1}, B_{1}, 0}M_{A_{2}, B_{2}, 0}^{-1} v
 = M_{A_{1}, B_{1}, 0}\bigl\{
 M_{0, B_{2}^{-1}, 0}\{e^{-i\pi A_{2}x^{2}} v(x)\}
 \bigr\}
 \\
 &\hs=M_{A_{1}, B_{1}, 0}\bigl\{
 e^{-i\pi B_{2}^{*-1}A_{2}B_{2}^{-1}x^{2}} v(B_{2}^{-1} x)
 \bigr\}(\det{B_{2}})^{-1/2}
 \\
 &\hs
 =e^{i\pi A_{1}x^{2}}
e^{-i\pi B_{1}^{*}B_{2}^{*-1}A_{2}B_{2}^{-1}B_{1}x^{2}}
v(B_{2}^{-1}B_{1}x)
(\det{B_{1}})^{1/2} (\det{B_{2}})^{-1/2}
\\
&\hs
=e^{i\pi (A_{1} -  B_{1}^{*}B_{2}^{*-1}A_{2}B_{2}^{-1}B_{1} ) x^{2}
}
v(B_{2}^{-1}B_{1}x)
(\det{B_{1}})^{1/2} (\det{B_{2}})^{-1/2}
\\
&\hs
= (M_{A_{1} -  B_{1}^{*}B_{2}^{*-1}A_{2}B_{2}^{-1}B_{1}, B_{2}^{-1}B_{1},0}v)(x),
\end{align*}
so that $M_{A_{1}, B_{1}, 0}M_{A_{2}, B_{2}, 0}^{-1}$ belongs to the set $\mathcal L$ in \eqref{921456}, proving that 
$\mathcal L$
is indeed a subgroup of the metaplectic group.
We note also that the bijective mapping 
\begin{equation}\label{gf54ze}
\mathcal L\ni M\mapsto F^{*}MF\in \mathcal R,
\end{equation}
($F$ stands for the Fourier transformation) sends $\mathcal L$ onto $\mathcal R$
since we have  
\begin{multline}\label{9axgrt}
F^{*}M_{A,B,0} F=F^{*}M_{A,I_{n},0}F F^{*}M_{0,B,0} F=
M_{0,I_{n},A}M_{0,B^{*-1},0}
\\
=M_{0,B^{*-1},B^{*-1}AB^{-1}}.
\end{multline}
Moreover the mapping \eqref{gf54ze} is obviously one-to-one and is also onto
since, given $B_{1}\in Gl(n, \R)$ and $C_{1}$ a symmetric $n\times n$ matrix,
we see from \eqref{9axgrt} that
$$
F^{*}M_{B_{1}^{-1}C_{1}B_{1}^{*-1},B_{1}^{*-1}, 0}F=M_{0,B_{1}, C_{1}}.
$$
The mapping \eqref{gf54ze} also extends to a group isomorphism of $Mp(n)$, proving the lemma.
\end{proof}
\begin{rem}\rm
We may note that 
\begin{multline*}
(M_{A_{1}, B_{1}, 0}M_{A_{2}, B_{2}, 0} v)(x)=e^{i\pi A_{1}x^{2}}
(M_{A_{2}, B_{2}, 0} v)(B_{1}x)(\det B_{1})^{1/2}
\\
=e^{i\pi (A_{1}+B^{*}_{1}A_{2}B_{1})
x^{2}} v(B_{2}B_{1}x)(\det B_{1})^{1/2}(\det B_{2})^{1/2}
=(M_{A_{1}+B_{1}^{*}A_{2}B_{1}, B_{2}B_{1}, 0}v)(x),
\end{multline*}
so that the internal binary operation $\star$ can be defined on the set $\{(A,B)\}_{\substack{
A=A^{*}\\
\det B\not=0}
}
$
as 
\begin{equation}\label{}
(A_{1},B_{1})\star(A_{2},B_{2})=(A_{1}+B_{1}^{*}A_{2}B_{1}, B_{2}B_{1}),
\end{equation}
for which the identity is $(0,I_{n})$ and the inverse
\begin{equation}\label{}
(A,B)^{-1}=(-B^{*-1}AB^{-1},B^{-1}).
\end{equation}
\end{rem}
\begin{rem}\rm
A consequence of Lemma \ref{lem.784512}
is,
with $\Psi$ defined in \eqref{hqma28}, that 
$$
\{\Psi(M_{A, B, 0})\}_{\substack{
A=A^{*}\\
\det B\not=0}
}=\{\Xi_{A, B, 0}\}_{\substack{
A=A^{*}\\
\det B\not=0}
},
\quad
\{\Psi(M_{0, B, C})\}_{\substack{
C=C^{*}\\
\det B\not=0}
}=\{\Xi_{0, B, C}\}_{\substack{
C=C^{*}\\
\det B\not=0}
},
$$
are  subgroups of the symplectic group $Sp(n,\R)$.\end{rem}
\begin{pro}
 The metaplectic group $Mp(n)$ is equal to the set
\begin{equation}\label{}
 \left\{M_{A_{1},B_{1},C_{1}}
 M_{A_{2},B_{2},C_{2}}\right\}_{
 \substack
 {A_{j}=A_{j}^{*}, C_{j}=C_{j}^{*}\\        \det B_{j}\not=0}
 }
\end{equation}
In other words, every metaplectic operator of $Mp(n)$ is the product of two operators of type 
$M_{A, B, C}$ as given by Proposition \ref{pro.kj77qq}.
\end{pro}
\begin{proof}
Let $M\in Mp(n)$;
using Proposition \ref{pro.915915},
we may assume that 
\begin{align*}
M&=\mathcal M_{P_{1}, L_{1}, Q_{1}}\mathcal M_{P_{2}, L_{2}, Q_{2}}
\\
&=
\mathcal M_{P_{1}, L_{1}, Q_{1}}\mathcal F e^{-i\pi n/4}
\left(\mathcal F e^{-i\pi n/4}\right)^{-1}\mathcal M_{P_{2}, L_{2}, Q_{2}}
\\
\text{\scriptsize \eqref{linkme}}&=M_{P_{1}, -L_{1}, Q_{1}}
\left(\mathcal M_{P_{2}, L_{2}, Q_{2}}^{-1}
\mathcal F e^{-i\pi n/4}\right)^{-1}
\\
\text{\scriptsize (Claim \ref{claim.911})}&=
M_{P_{1}, -L_{1}, Q_{1}}
\left(\mathcal M_{-Q_{2}, -L_{2}^{*}, -P_{2}}
\mathcal F e^{-i\pi n/4}\right)^{-1}
\\
\text{\scriptsize \eqref{linkme}, \eqref{54gvrs}}&=M_{P_{1}, -L_{1}, Q_{1}}
M_{-Q_{2}, L_{2}^{*}, -P_{2}}^{-1}
\\&=M_{P_{1}, -L_{1},0}
M_{0, I_{n},Q_{1}}
\bigl(M_{-Q_{2}, L_{2}^{*}, 0}
M_{0,I_{n}, -P_{2}}\bigr)^{-1}
\\
&=M_{P_{1}, -L_{1},0}
M_{0, I_{n},Q_{1}} M_{0,I_{n},P_{2}}M_{-Q_{2}, L_{2}^{*}, 0}^{-1}
\\&=
M_{P_{1}, -L_{1},0}M_{0, I_{n},Q_{1}+P_{2}} M_{-Q_{2}, L_{2}^{*}, 0}^{-1}
=M_{P_{1}, -L_{1},Q_{1}+P_{2}}M_{-Q_{2}, L_{2}^{*}, 0}^{-1}
\\\text{\scriptsize  (using Lemma \ref{lem.784512})}&=
M_{P_{1}, -L_{1},Q_{1}+P_{2}} M_{A',B',0},
\end{align*}
proving the sought result.
\end{proof}
\begin{rem}\rm
 We have used two different sets of generators of the metaplectic group. First the set $\mathscr G_{1}=\big\{\mett{A}{B}{C}{m(B)}\bigr\}$ given by \eqref{1256hh} which is somewhat natural, also allowing us to recover the operator
 $e^{-i\pi n/4}\mathcal F$ where the phase factor appears via Formula \eqref{kjh123}.
 The Identity appears clearly as $\mett{0}{I_{n}}{0}{0}$, but the inverse of 
 $\mett{A}{B}{C}{m(B)}$ cannot always be expressed within $\mathscr G_{1}$.
\par
Also we have the set
$\mathscr G_{2}=\big\{\Mett{A}{B}{C}{m(B)}\bigr\}$ given in Definition \ref{def.lersym},
which incorporates a phase prefactor $e^{-i\pi n/4}$, looking a priori rather arbitrary but of course necessary for the sequel (this prefactor is also suggested by  \eqref{kjh123});
here to express the identity, we need to write it as 
$\Mett{0}{I_{n}}{0}{0}\Mett{0}{-I_{n}}{0}{n}$, but the inverse of  
$\Mett{A}{B}{C}{m(B)}$
is easily obtained by Claim \ref{claim.911} within 
$\mathscr G_{2}$.
Certainly the description given by $\mathscr G_{2}$
is much better, in particular because the calculations leading to 
Lemma \ref{lem.910910} and Proposition \ref{pro.912912}
are rather easy as well as the proof of Lemma \ref{lem.913913};
a statement analogous to Proposition \ref{pro.915915}  for $\mathscr G_{1}$ is true (cf. Proposition \ref{pro.917917}),
but its proof is quite indirect and relies heavily on the results for $\mathscr G_{2}$.
\end{rem}
\subsection{Mehler's formula}\label{sec.mehler}
\index{Mehler formula}
We provide here a couple of statements related to the so-called Mehler's formula, appearing as particular cases of L.~H\"ormander's study in \cite{MR1339714}
(see also the more recent K. Pravda-Starov' article  \cite{MR3880300}).
 In the general framework, we consider a complex-valued quadratic form $Q$
on the phase space $\RZ$ such that $\re Q\le 0$: we want to quantize the Gaussian 
function (here $X$ stands for $(x,\xi)$) 
$
\mathbf a (X)=e^{\poscal{Q X}{X}},
$
and to relate the operator with Weyl symbol $\mathbf a$ to the operator
$$
\exp\left\{
\opw{\poscal{Q X}{X}}\right\}.
$$
\begin{lem}
For $\re t \ge 0$, $t\notin i\pi(2\Z+1)$, we have in $n$  dimensions, 
\begin{equation}\label{mehler}
\bigl(\cosh (t/2)\bigr)^{n}\exp-t\pi
\opw{\val x^{2}+\val \xi^{2}}
=\OPW{e^{-2\tanh (\frac t2) \pi(x^{2}+\xi^{2})}}.
\end{equation}
\end{lem}
In particular, for $t=-2is,s\in \R$, $s\notin \frac{\pi}2(1+2\Z)$, we have  in $n$ dimensions
\begin{equation}\label{mehler+n}
(\cos  s)^{n} \exp \bigl(2 i \pi s
\opw{\val x^{2}+\val\xi^{2}}
\bigr)=
\OPW{e^{2i \pi\tan s(\val x^{2}+\val \xi^{2})}}.
\end{equation}
\begin{lem}
For any $z\in \C,\ \re z \ge 0$, we have  in $n$ dimensions
 \begin{equation}\label{6.knb44}
\OPWM{\exp-\bigl(2z\pi \bigl(\val \xi^2+\val x^{2}\bigr)\bigr)}
=\frac{1}{(1+z)^\did}\sum_{k\ge 0}\Big(\frac{1-z}{1+z}\Big)^k\mathbb P_{k;n},
\end{equation}
where $\mathbb P_{k;n}$ is defined in Section \ref{sechar} 	and the equality holds between
 $L^2(\R^\did)$-bounded operators.
\end{lem}
We provide first a proof of a particular case of the results of \cite{MR1339714}.
\begin{lem}
For $\re t \ge 0$, $t\notin i\pi(2\Z+1)$, we have in $n$  dimensions, 
\begin{equation}\label{mehler+++}
\bigl(\cosh (t/2)\bigr)^{n}\exp-t\pi
\opw{\val x^{2}+\val \xi^{2}}=\OPW{e^{-2\tanh (\frac t2) \pi(x^{2}+\xi^{2})}}.
\end{equation}
\end{lem}
\begin{proof}
By tensorisation, it is enough to prove that formula for $n=1$, which we assume from now on.
 We define
$$
L=\xi+ix,\quad \bar L=\xi-i x,\quad M(t)=\beta(t) 
\opw{e^{-\alpha(t) \pi L\bar L}}
,
$$
where $\alpha, \beta$ are smooth functions of $t$ to be chosen below.
Assuming $\beta(0)=1, \alpha(0)=0,$ we find that $M(0)=\Id$ and
$$
\dot M+\pi 
\opw{\val L^{2}}
M=
\OPWM
{\dot \beta e^{-\alpha \pi\val L^{2}}-\beta \dot \alpha \pi \val L^{2} e^{-\alpha \pi\val L^{2}}
+\pi (\val L^{2})\sharp \beta e^{-\alpha \pi\val L^{2}}}.
$$
We have from \eqref{gfcd44+}, since $\p_{x}\p_{\xi}\val{L}^{2}=0$,
\begin{align*}
\val L^{2}\sharp e^{-\alpha \pi \val L^{2}}&=\val L^{2} e^{-\alpha \pi \val L^{2}}+
\frac{1}{4i\pi}\overbrace{\poi {\val L^{2}}{e^{-\alpha \pi \val L^{2}}}}^{=0}
\\&\hskip25pt+\frac{1}{(4i \pi)^{2}}
\frac12
\Bigl(\p_{\xi}^{2}(\val L^{2})\p_{x}^{2}e^{-\alpha \pi \val L^{2}}
+\p_{x}^{2}(\val L^{2})\p_{\xi}^{2}e^{-\alpha \pi \val L^{2}}\Bigr)
\\&=\val L^{2} e^{-\alpha \pi \val L^{2}}
\\&\hskip25pt+
\frac{1}{(4i \pi)^{2}}
\frac12e^{-\alpha \pi \val L^{2}}
\Bigl(2\bigl((-2\alpha \pi x)^{2}-2\alpha \pi\bigr)
+2\bigl((-2\alpha \pi \xi)^{2}-2\alpha \pi\bigr)\Bigr)
\\&=
\val L^{2} e^{-\alpha \pi \val L^{2}}
\Bigl(1-\frac{4\alpha^{2}\pi^{2}}{16\pi^{2}}\Bigr)
+
\frac{\alpha \pi}{4\pi^{2}}e^{-\alpha \pi \val L^{2}},
\end{align*}
so that
\begin{multline*}
\dot M+\pi 
\opw{\val L^{2}}
M\\
=
\OPWM{\dot \beta e^{-\alpha \pi\val L^{2}}-\beta \dot \alpha \pi \val L^{2} e^{-\alpha \pi\val L^{2}}
+\pi  \beta 
\val L^{2} e^{-\alpha \pi \val L^{2}}
\Bigl(1-\frac{4\alpha^{2}\pi^{2}}{16\pi^{2}}\Bigr)+\frac{\alpha \pi \beta}{4\pi}e^{-\alpha \pi \val L^{2}}}
\\=\OPWM
{e^{-\alpha \pi \val L^{2}}
\Bigl\{\val L^{2}\bigl(-\pi \dot \alpha \beta+\pi \beta(1-\frac{\alpha^{2}}4)\bigr)
+\dot \beta+\frac{\alpha\beta}{4}\Bigr\}
}.
\end{multline*}
We solve now
$$
\dot \alpha=1-\frac{\alpha^{2}}4,\ \alpha(0)=0\Longleftrightarrow \alpha(t)= 2\tanh(t/2),
$$
and 
$$
4\dot \beta+\alpha \beta=0, \beta(0)=1\Longleftrightarrow \beta(t)=\frac1{\cosh (t/2)}.
$$
We obtain that
$
\dot M+\pi \opw{\val L^{2}}
M=0,\ M(0)=\Id,
$
and this implies 
$$
\beta(t) 
\opw{e^{-\alpha(t) \pi L\bar L}}
=M(t)=\exp{-t\pi (\val L^{2})^{w}},
$$
which proves \eqref{mehler+++}.
\end{proof}
In particular, for $t=-2is,s\in \R$, $s\notin \frac{\pi}2(1+2\Z)$, we have  in $n$ dimensions
\begin{equation}\label{mehler+n++}
(\cos  s)^{n} \exp \bigl(2 i \pi s
\opw{\val x^{2}+\val\xi^{2}}
\bigr)=
\OPW{e^{2i \pi\tan s(\val x^{2}+\val \xi^{2})}}
.
\end{equation}
\begin{lem}
For any $z\in \C,\ \re z \ge 0$, we have  in $n$ dimensions
 \begin{equation}\label{6.knb44++}
\OPW{\exp-\bigl(2z\pi(\val \xi^2+\val x^{2})\bigr)}
=\frac{1}{(1+z)^\did}\sum_{k\ge 0}\Big(\frac{1-z}{1+z}\Big)^k\mathbb P_{k;n},
\end{equation}
where $\mathbb P_{k;n}$ is defined in Section \ref{sechar} 	and the equality holds between
 $L^2(\R^\did)$-bounded operators.
\end{lem}
\begin{proof}
 Starting from \eqref{mehler+n++}, we get for $\tau\in \R$, in $n$ dimensions,
 $$
 (\cos  (\arctan \tau))^{n} \exp \bigl(2 i \pi \arctan \tau
 \opw{\val x^{2}+\val\xi^{2}}
 \bigr)=
 \OPW{e^{2i \pi\tau(\val x^{2}+\val \xi^{2})}}
 ,
 $$
 so that using the spectral decomposition of the ($n$-dimensional)
 Harmonic Oscillator and \eqref{arctan}, we get
 $$
 (1+\tau^{2})^{-n/2}\sum_{k\ge 0}e^{2i (\arctan \tau)(k+\frac n2)}\mathbb P_{k;n}=
 \OPW{e^{2i \pi\tau(\val x^{2}+\val \xi^{2})}},
 $$
 which implies
 $$
 (1+\tau^{2})^{-n/2}\sum_{k\ge 0}
\frac{(1+i\tau)^{2k+n}}{(1+\tau^{2})^{k+\frac n2}}
 \mathbb P_{k;n}=
 \OPW{e^{2i \pi\tau(\val x^{2}+\val \xi^{2})}}
,
 $$
 entailing
 $$
\sum_{k\ge 0}
\frac{(1+i\tau)^{k}}{(1-i\tau)^{k+n}}
 \mathbb P_{k;n}=
\OPW{e^{2i \pi\tau(\val x^{2}+\val \xi^{2})}}
,
 $$
 proving the lemma by analytic continuation
 (we may refer the reader as well to \cite{MR552965} (pp. 204-205)
 and note  that for any $z\in \C,\ \re z \ge 0$, we have $\val{\frac{1-z}{1+z}}\le 1$
).
\end{proof}
\subsection{Laguerre polynomials}\label{sec.laguerre}
\index{Laguerre polynomials}
\index{{~\bf Notations}!$L_{k}$}
\subsubsection{Classical Laguerre polynomials}
The Laguerre polynomials  $\{L_{k}\}_{k\in \N}$
are defined by 
\begin{equation}\label{laguerre}
L_{k}(x)=\sum_{0\le l\le k}\frac{(-1)^{l}}{l!}\binom{k}{l} x^{l}=e^{x}\frac{1}{k!}\left(\frac {d}{dx}\right)^{k}\bigl\{x^{k} e^{-x}\bigr\}=\left(\frac {d}{dx}-1\right)^{k}\bigl\{\frac{x^{k}}{k!}\bigr\},
\end{equation}
and we have
{\small\begin{align*}
L_{0}&=  1,  \\
L_{1}&=  -X+1 , \\
L_{2}&= \frac12(X^{2}-4X+2),   \\
L_{3}&= \frac16(-X^{3}+9X^{2}-18 X+6),   \\
L_{4}&= \frac{1}{24}(X^{4}-16 X^{3}+72 X^{2}-96 X+24) ,  \\
L_{5}&=  \frac{1}{120}(-X^{5}+25 X^{4}-200 X^{3}+600 X^{2}-600 X+120),\\
L_{6} &=\frac{1}{720} \left(X^6-36 X^5+450 X^4-2400 X^3+5400 X^2-4320 X+720\right),\\
L_{7}&=\frac{-X^7+49 X^6-882 X^5+7350 X^4-29400 X^3+52920
X^2-35280 X+5040}{5040}.
\end{align*}
}
We get also easily from the above definition that 
\begin{equation}\label{654poi}
L'_{k+1}=L'_{k}-L_{k},
\end{equation}
since with $T=d/dX-1$
$$
L'_{k}-L_{k}=T L_{k}= T^{k+1}(\frac{X^{k}}{k!})= T^{k+1}(\frac{d}{dX}\frac{X^{k+1}}{(k+1)!})=\frac{d}{dX} L_{k+1}.
$$
\index{Feldheim inequality}
Formula (6.8) and Theorem 12 in the R.~Askey \& G.~Gasper's article \cite{MR0430358} provide
the inequalities
\begin{equation}\label{inelag}
\forall k\in \N, \forall x\ge 0,\quad 
\sum_{0\le l\le k}(-1)^{l}L_{l}(x)\ge 0.
\end{equation}
This result follows as well from  Formula (73) in the 1940 paper \cite{MR0001401} by E.~Feldheim.
\index{Fourier tr. of Laguerre polynomials}
Let us calculate the Fourier transform of the Laguerre polynomials:
we have 
$$
L_{k}(x)=
\left(\frac {d}{dx}-1\right)^{k}\bigl\{\frac{x^{k}}{k!}\bigr\},
$$
so that 
$\dis
\widehat{L_{k}}(\xi)=(2i\pi \xi-1)^{k}\left(\frac{-1}{2i\pi}\right)^{k}\frac{\delta_{0}^{(k)}}{k!}=\frac{(-1)^{k}}{k!}(\xi-\frac1{2i\pi})^{k}
\delta_{0}^{(k)}(\xi).
$
As a result, defining for $k\in \N, t\in \R$, 
\begin{equation}\label{foulag-}
M_{k}(t)=(-1)^{k}H(t)e^{-t}L_{k}(2t),\quad H=\mathbf 1_{\R_{+}},
\end{equation}
we find, using the homogeneity of degree $-k-1$ of $\delta_{0}^{(k)}$, 
\begin{align*}
\widehat{M_{k}}(\tau)
&=\frac12
\frac{(-1)^{k}}{k!}\bigl(\frac\tau 2-\frac1{2i\pi}\bigl)^{k}
\delta_{0}^{(k)}(\frac\tau 2)\ast \frac{(-1)^{k}}{1+2i\pi \tau}
\\&={(-1)^{k}}{}(\frac{d}{d\sigma})^{k}\left\{
\frac{(\sigma-\frac 1{i\pi})^{k}/k!}{1+2i\pi (\tau-\sigma)}
\right\}_{\vert \sigma=0} 
\end{align*}
\begin{align*}
\widehat{M_{k}}(\tau)&=\sum_{l}{(-1)^{k}}{}\binom{k}{l}
\frac{(\sigma-\frac 1{i\pi})^{k-l}}{(k-l)!}\frac{(k-l)!(2i\pi)^{k-l}}{\bigl(1+2i\pi (\tau-\sigma)\bigr)^{1+k-l}}_{\vert \sigma=0}
\\&
=\sum_{l}{(-1)^{k}}{}\binom{k}{l}
{}\frac{(-2)^{k-l}}{\bigl(1+2i\pi \tau\bigr)^{1+k-l}}
\\&
=\frac{(-1)^{k}}{(1+2i\pi \tau)}\sum_{l}\binom{k}{l}
{}\frac{(-2)^{k-l}}{\bigl(1+2i\pi \tau\bigr)^{k-l}}
\\&=\frac{(-1)^{k}}{(1+2i\pi \tau)}
\left(1-\frac{2}{\bigl(1+2i\pi \tau\bigr)}\right)^{k}
\\&=\frac{(-1)^{k}}{(1+2i\pi \tau)}
\left(\frac{-1+2i\pi \tau}{1+2i\pi \tau}\right)^{k}
\\&=
\frac{1}{(1+2i\pi \tau)}
\left(\frac{1-2i\pi \tau}{1+2i\pi \tau}\right)^{k}
\end{align*}
so that 
\begin{equation}\label{foulag+}
\widehat{M_{k}}(\tau)=
\frac{(1-2i\pi \tau)^{k}}{(1+2i\pi \tau)^{k+1}}=\frac{(1-2i\pi \tau)^{2k+1}}{(1+4\pi^{2}\tau^{2})^{k+1}}.
\end{equation}
\subsubsection{Generalized Laguerre polynomials}
\index{generalized Laguerre polynomials}
\index{{~\bf Notations}!$L_{k}^{\alpha}$}
Let $\alpha$ be a complex number  
and let $k$ be a non-negative integer such that $\alpha+k\notin (-\N^*)$. We define the generalized Laguerre polynomial $L_{k}^\alpha$ by
\begin{equation}\label{laggen}
L_{k}^\alpha(x)=x^{-\alpha} e^x\left(\frac{d}{dx}\right)^k\bigl\{e^{-x}\frac{x^{k+\alpha}}{k!}\bigr\}
=x^{-\alpha} \left(\frac{d}{dx}-1\right)^k\bigl\{\frac{x^{k+\alpha}}{k!}\bigr\}.
\end{equation}
We note that  $L_{k}^\alpha$ is indeed a polynomial with degree $k$ with the formula
\begin{align}
L_{k}^\alpha(x)
&=\sum_{k_{1}+k_{2}=k}\frac{1}{k!}\binom{k}{k_{1}}(-1)^{k_{2}}
\Gamma(k+\alpha+1)
\frac{x^{k-k_{1}}}{\Gamma(k+\alpha+1-k_{1})}
\notag\\
&=\sum_{0\le k_{1}\le k}\frac{(-1)^{k_{2}}}{k_{1}!(k-k_{1})!}
\Gamma(k+\alpha+1)
\frac{x^{k-k_{1}}}{\Gamma(k+\alpha+1-k_{1})}
\notag\\
&=\sum_{0\le l\le k}
\binom{k+\alpha}{k-l}
\frac{(-1)^{l}x^{l}}{l!}.
\label{957}
\end{align}
\begin{nb}\rm
 We recall that the function $1/\Gamma$ is an entire function with simple zeroes at $-\N$. As a result to make sense for the binomial coefficient
 $$
 \binom{k+\alpha}{k-l}=\frac{\Gamma(k+\alpha+1)}{(k-l)!\Gamma(l+\alpha+1)},
 $$
 we need to make sure that $k+\alpha+1\notin -\N$, i.e. $\alpha\notin -\N^*-k$.\end{nb}
\begin{lem}
Let $\alpha\in \C\backslash(-\N^*)$ 
and let $k$ be a non-negative integer. For $\alpha=0$, we have $L_{k}^\alpha=L_{k}$, where $L_{k}$ is the classical Laguerre polynomial defined in \eqref{laguerre}. Moreover we have  for $l\le k$,
\begin{equation}\label{958}
\left(\frac{d}{dX}\right)^l L_{k}^\alpha=(-1)^lL_{k-l}^{\alpha+l}.
\end{equation}
\end{lem}
\begin{proof}Indeed,
 we have from \eqref{957}
\begin{multline*}
 \left(\frac{d}{dX}\right)^l L_{k}^\alpha
 =(-1)^l\sum_{l\le m\le k}
\binom{k+\alpha}{k-m}
\frac{(-1)^{m-l}X^{m-l}}{(m-l)!}
 \\=
 (-1)^l\sum_{0\le r\le k-l}
\binom{k-l+\alpha+l}{k-r-l}
\frac{(-1)^{r}X^{r}}{r!}=(-1)^lL_{k-l}^{\alpha+l},
\end{multline*}
proving the sought formula.
\end{proof}
\subsection{Singular integrals}\label{secsing}
\index{singular integrals}
\index{Hardy operator}
\begin{pro}\label{pro.hardy}~
\par\no
{$\mathbf{[1]}$\ }The (Hardy) operator with  distribution kernel $$\frac{H(x)H(y)}{\pi(x+y)}$$ is self-adjoint bounded on $L^2(\R)$ with spectrum $[0,1]$ and thus norm 1.
\par\no
{$\mathbf{[2]}$\ }The (modified Hardy) operators with respective distribution kernels $$H(x-y)\frac{H(x)H(y)}{\pi(x+y)},\quad
H(y-x)\frac{H(x)H(y)}{\pi(x+y)},$$ are  bounded on $L^2(\R)$ with norm $1/2$.
\end{pro}
\begin{proof}Let us prove $\mathbf{[1]}$: for $\phi\in L^2(\R_{+})$, we define for $t\in \R$, $\tilde\phi(t)=\phi(e^t) e^{t/2}$, and 
 we have to check the kernel 
 $$
 \frac{e^{t/2}e^{s/2}}{\pi(e^t+e^s)}=\frac{1}{\pi(e^{(t-s)/2}+e^{-(t-s)/2})}=\frac{1}{2\pi}
\sech\bigl(\frac{t-s}{2}\bigr),
 $$
 which is a convolution kernel.
 Using now the classical formula
 \begin{equation}\label{}
\int e^{-2i\pi x \xi}\sech x dx=\pi\sech (\pi^2 \xi),
\end{equation}
we get that 
$
\frac{1}{2\pi}\int
\sech(\frac{t}{2}) e^{-2i\pi t\tau} dt=\sech (\pi^2 2 \tau),
$
a smooth function whose range is $(0,1]$, proving the first part of the proposition.
To obtain $\mathbf{[2]}$,
we observe with the notations $\phi(t)=u(e^{t}) e^{t/2}$, $\psi(s)=v(e^{s}) e^{s/2}$  that we have to check 
\begin{multline*}
 \iint H(s-t)\frac{e^{t/2}e^{s/2}}{\pi(e^t+e^s)}\phi(t)\bar \psi(s) dtds
 \\=\iint\frac{H(s-t)}{\pi(e^{(t-s)/2}+e^{-(t-s)/2})}
 \phi(t)\bar \psi(s) dtds
 =\poscal{R\ast \phi}{\psi}_{L^{2}(\R)},
\end{multline*}
with 
\begin{equation}
R(t)=\frac{H(t)}{2\pi \cosh(t/2)},\quad \hat R(\tau)=\frac{1}{2\pi}\int_{0}^{+\io} \sech(t/2)
e^{-2i\pi t \tau} dt,
\end{equation}
so that\footnote{We recall  that $\frac{d}{ds}\arctan(\sinh s)=\sech s$.}
\begin{equation}
\val{ \hat R(\tau)}\le \hat R(0)=\frac{1}{2\pi}\int_{0}^{+\io} \sech(t/2) dt=\frac12,
\end{equation}
yielding the sought result.
\end{proof}
\subsection{On some auxiliary functions}
\subsubsection{A preliminary quadrature}
\begin{lem}\label{lem.913}
 We have 
 \begin{equation}
\int_{0}^{\pi/2}(\cosec s-\cosech s)ds=\int_{\pi/2}^{+\io}\cosech s ds=\Lg(\coth \frac\pi4),
\end{equation}
with $\csc s=1/\sin s, \cosech s=1/\sinh s$. 
\end{lem}
\begin{proof}
Note that the function $\dis
[0,\pi/2]\ni s\mapsto \frac{\sinh s-\sin s}{\sinh s \sin s},$ is continuous.
Moreover, we have 
$$\int \frac{ds}{\sin s}=\frac12\Lg\bigl(\frac{1-\cos s}{1+\cos s}\bigr) \
\text{\quad and \ }
 \int \frac{ds}{\sinh s}=\frac12\Lg\bigl(\frac{\cosh s-1}{\cosh s+1}\bigr),
 $$ so that 
 \begin{multline*}
 \int_{\epsilon}^{\pi/2}(\cosec s-\cosech s)ds=
 \frac12\left[\Lg\bigl(\frac{1-\cos s}{1+\cos s}\bigr)\right]^{\pi/2}_{\epsilon}
 -\left[\frac12\Lg\bigl(\frac{\cosh s-1}{\cosh s+1}\bigr)\right]^{\pi/2}_{\epsilon}
 \\
 =\frac12\Lg\Bigl(
 \underbrace{\bigl(\frac{1+\cos \epsilon}{1-\cos \epsilon}\bigr)
 \bigl(\frac{\cosh \epsilon-1}{\cosh \epsilon+1}\bigr)}_{=\frac{(2+O(\epsilon^{2}))(\frac{\epsilon^{2}}{2}+O(\epsilon^{4}))}{
 (\frac{\epsilon^{2}}{2}+O(\epsilon^{4}))
 (2+O(\epsilon^{2}))
 }
 \rightarrow 1\text{ for $\epsilon\rightarrow 0$}}
 \Bigr)
 +
 \frac12\Lg\bigl(\frac{\cosh \frac\pi2+1}{\cosh \frac\pi2-1}\bigr),
\end{multline*}
so that we obtain
$$
\int_{0}^{\pi/2}(\cosec s-\cosech s)ds
= \frac12\Lg\bigl(\frac{e^{\pi/2}+e^{-\pi/2}+2}{e^{\pi/2}+e^{-\pi/2}-2}\bigr)
=\Lg\frac{\cosh(\pi/4)}{\sinh(\pi/4)},
$$
which is the first result.
Also we have 
$
\int_{\pi/2}^{+\io}\cosech s ds=
\frac12\Lg\bigl(\frac{\cosh(\pi/2)+1}{\cosh(\pi/2)-1}\bigr),
$
yielding the second result.
\end{proof}
\subsubsection{Study of the function \texorpdfstring{$\rho_{\sigma}$}{rhos}}
We study in this section the real-valued Schwartz function $\rho_{\sigma}$ given in \eqref{526-}. Using the notations 
\begin{equation}\label{note}
\omega=2\pi \tau,\quad \kappa=2\pi \sigma,\quad \nu=\sqrt{\kappa/\omega},
\end{equation}
we have 
\begin{equation}
\rho_{\sigma}(\tau)=\int_{\R}\frac{s}{\sinh s} e^{2i\omega(s-\nu^{2}\tanh s)} ds
=\int_{\R}\frac{s}{\sinh s} \cos\bigl(2\omega(s-\nu^{2}\tanh s)\bigr) ds.
\end{equation}
Defining the holomorphic function $F$ by
\begin{equation}\label{953}
F(z)=\frac{z}{\sinh z} e^{2i\omega(z-\nu^{2}\tanh z)},
\end{equation}
we see that $F$ has simple poles at $i\pi \Z^{*}$ and essential singularities at
$i\pi(\frac12+\Z)$.
We already know that the function $\rho_{\sigma}$ belongs to the Schwartz space, but we want to prove a more precise exponential decay.
We start with the calculation of 
\begin{multline}\label{964}
\int_{\R+i\frac\pi 4} F(z) dz=\int_{\R}
\frac{t+i\frac\pi 4}{\sinh (t+i\frac\pi 4)} e^{2i\omega(t+i\frac\pi 4-\nu^{2}\tanh (t+i\frac\pi 4))}
dt 
\\= e^{-\pi\omega/2} 2\sqrt 2
\int_{\R}
\frac{t+i\frac\pi 4}{(1+i) e^{t}-(1-i) e^{-t}} 
e^{2i\omega t}e^{-2i\omega\nu^{2}\frac{e^{t}(1+i)-e^{-t}(1-i)}{e^{t}(1+i)+e^{-t}(1-i)}}
dt 
\\
=
e^{-\pi\omega/2} \sqrt 2
\int_{\R}
\frac{t+i\frac\pi 4}{\sinh t+i\cosh t} 
e^{2i\omega t}e^{-2i\omega\nu^{2}\frac{e^{t}(1+i)-e^{-t}(1-i)}{e^{t}(1+i)+e^{-t}(1-i)}}
dt. 
\end{multline}
We have 
$$
\im\left(\frac{e^{t}(1+i)-e^{-t}(1-i)}{e^{t}(1+i)+e^{-t}(1-i)}\right)=
\im\left(\frac{\sinh t+i\cosh t}{\cosh t+i\sinh t}\right)=\frac{1}{\cosh^{2} t+\sinh^{2} t},
$$
so that
\begin{multline}\label{955}
\Bigl\vert{\int_{\R+i\frac\pi 4} F(z) dz}\Bigr\vert\le 
e^{-\frac{\pi\omega}2} \sqrt 2
\int_{\R}
\frac{\sqrt{t^{2}+(\frac\pi 4)^{2}}}{\sqrt{\sinh^{2} t+\cosh^{2} t}} 
\ e^{\frac{2\omega\nu^{2}}{\sinh^{2} t+\cosh^{2} t}}
dt
\\= 
e^{-\frac{\pi\omega}2} \sqrt 2 e^{2\kappa}
\int_{\R}
\frac{\sqrt{t^{2}+(\frac\pi 4)^{2}}}{\sqrt{\sinh^{2} t+\cosh^{2} t}} 
dt
\le 6 e^{-\frac{\pi\omega}2} e^{2\kappa}.
\end{multline}
\begin{claim}\label{claim3}
 We have  
 $$
 \lim_{R\rightarrow+\io}\oint_{[R, R+i\pi/4]} F(z) dz=
  \lim_{R\rightarrow+\io}\oint_{[-R, -R+i\pi/4]} F(z) dz=0.
 $$
\end{claim}
\begin{proof}[Proof of the Claim]
We note first that 
$$
\oint_{[-R, -R+i\pi/4]} F(z) dz=-\overline{\oint_{[R, R+i\pi/4]} F(z) dz},
$$
so that it is enough to prove one equality.
Indeed for $R>0$, we have 
$$
\oint_{[R, R+i\pi/4]} F(z) dz=\int_{0}^{\pi/4}
\frac{R+it}{\sinh (R+it)} e^{2i\omega(R+it-\nu^{2}\tanh (R+it))} i dt,
$$
so that 
\begin{multline*}
\Bigl\vert{\oint_{[R, R+i\pi/4]} F(z) dz}\Bigr\vert
\le \int_{0}^{\pi/4}
\frac{2\sqrt{R^{2}+t^{2}}}{
\val{e^{R+it}}\val{1-e^{-2R-2it}}
} 
e^{-2\omega t}
e^{2\kappa\im(\tanh (R+it))}  dt
\\\le e^{-R}
\frac{\sqrt{4R^{2}+\pi^{2}/4}}
{
{1-e^{-2R}}
} 
 \int_{0}^{\pi/4}
e^{2\kappa\val{
\frac{1-e^{-2R-2it}}{1+e^{-2R-2it}}
}}  dt
\le 
e^{-R}
\frac{\sqrt{4R^{2}+\pi^{2}/4}}
{
{1-e^{-2R}}
} 
\frac \pi4
e^{
\frac{4\kappa}{(1-e^{-2R})}
},
\end{multline*}
proving the claim.
\end{proof}
\begin{lem}\label{lem.912}
 We have for $\tau>0, \sigma\ge 0$, $\rho_{\sigma}$ given in \eqref{526-},
 \begin{equation}\label{956}
\val{\rho_{\sigma}(\tau)}\le 
6 e^{-{\pi^2\tau}} e^{4\pi \sigma}.
\end{equation}
\end{lem}
\begin{proof}
 We have, with the notations \eqref{note},  $F$ given in \eqref{953} and 
 $\gamma_{R}=[-R,-R+i\frac\pi 4]\cup[-R+i\frac\pi 4,R+i\frac\pi 4]\cup
[R+i\frac\pi 4,R]$,
$$
\rho_{\sigma}(\tau)=\lim_{R\rightarrow+\io}\int_{[-R,R]} F(s) ds
=
\lim_{R\rightarrow+\io}\left(
\oint_{\gamma_{R}
}
F(z) dz\right)
\underbrace{=}_{\text{Claim \eqref{claim3}}}
\oint_{\R+\frac{i\pi}4}
F(z) dz,
$$
so that \eqref{955} implies the lemma.
\end{proof}
\subsubsection{On the function \texorpdfstring{$\psi_{\nu}$}{psis}}
Let $\nu\in(0,1)$ be given. We study first the function $\phi_{\nu}$ defined on $[0,\pi/2)$ by
\begin{equation}\label{}
\phi_{\nu}(s)=s-\nu^2\tan s,\quad\text{so that \quad}
\phi'_{\nu}(s)=1-\nu^2(1+\tan^2 s)=\frac{\cos^2 s-\nu^2}{\cos^2 s},
\end{equation}
so that
\begin{equation}\label{963}
\renewcommand{\arraystretch}{1.8}
\begin{array}{|c|lcccccc|}
\hline
s & 0&\hskip 33pt& s_{\nu}&\hskip 33pt&t_{\nu}&\hskip 33pt&\frac\pi 2\\
\hline
\phi'_{\nu}(s)&\text{\footnotesize{$1-\nu^2$}}&+&0&-&&-&\\
\hline
\phi_{\nu}(s) & 0 & \nearrow&\phi_{\nu}(s_{\nu})&\searrow&0&\searrow&-\io\\
\hline
\end{array}
\end{equation}
We have 
\begin{equation}\label{964+}
\begin{cases}
&\hskip-10pt s_{\nu}=\arccos  \nu=\frac\pi 2-\nu+O(\nu^{3}),
 \\
&\hskip-10pt \phi_{\nu}(s_{\nu})=\arccos \nu-\nu\sqrt{1-\nu^{2}}=\frac\pi 2-2\nu+O(\nu^{3}),
\end{cases}
 \qquad \text{for}\ \nu\rightarrow0.
\end{equation}
The function $\phi_{\nu}$ is concave  on $(0,\pi/2)$
since we have there
$$
\phi_{\nu}''(s)=-\nu^{2}(-2)(\cos s)^{-3} (-\sin s)=-\nu^2 2(\cos s)^{-3}\sin s\le 0.
$$
We have defined in \eqref{5254}
\begin{equation}\label{psinunu}
\psi_{\nu}(\omega)
=
\frac{e^{-\pi \omega}}{2\pi}
\int_{0}^{\pi/2}
\frac{e^{2\omega \phi_{\nu}(s)}-1}{\sin s}
ds.
\end{equation}
Let us start with an elementary lemma.
\begin{lem}\label{lem.915}
 Let $\lambda>0$ be given. Defining
 \begin{equation}
J(\lambda)=e^{-\lambda}\int_{0}^{\lambda}\frac{e^{\sigma}-1}{\sigma} d\sigma,
\qquad \text{we have}
\end{equation}
\begin{align}
&J(\lambda)=\lambda^{-1}+O(\lambda^{-2}), \quad \lambda\rightarrow+\io,
\label{966966}
\\
&\forall \lambda>0, \quad
J(\lambda)
\ge \lambda^{-1}-\lambda^{-2}.
\label{9610}
\end{align}
\end{lem}
\begin{proof}
 Indeed we have  for $\lambda>0$,
 \begin{multline}\label{968}
\lambda J(\lambda)=\lambda e^{-\lambda}\sum_{k\ge 1}\int_{0}^{\lambda}\frac{\sigma^{k-1}}{k!}
d\sigma =\lambda e^{-\lambda}\sum_{k\ge 1}\frac{\lambda^{k}}{k!k}
=e^{-\lambda}\sum_{k\ge 1}\frac{\lambda^{k+1}}{(k+1)!}\frac{k+1}{k}
\\
=e^{-\lambda}\sum_{k\ge 1}\frac{\lambda^{k+1}}{(k+1)!}
+
e^{-\lambda}\sum_{k\ge 1}\frac{\lambda^{k+1}}{(k+1)!}\frac{1}{k}
\\=
e^{-\lambda}(e^{\lambda}-1-\lambda)
+
\lambda^{-1} \underbrace{\left(e^{-\lambda}\sum_{k\ge 1}\frac{\lambda^{k+2}}{(k+1)!}\frac{1}{k}\right)}_{R(\lambda)},
\end{multline}
with 
\begin{align}\label{969}
0\le R(\lambda)\le e^{-\lambda}
\sum_{k\ge 1}\frac{\lambda^{k+2}}{(k+2)!}\frac{k+2}{k}\le e^{-\lambda}(e^{\lambda}-1-\lambda-\frac{\lambda^{2}}{2})\times 3=O(1),
\end{align}
so that 
$$
\lambda J(\lambda)
=
e^{-\lambda}(e^{\lambda}-1-\lambda)
+
\lambda^{-1}O(1)=1+\lambda^{-1}O(1)-(1+\lambda)e^{-\lambda}=1+\lambda^{-1}O(1),
$$
proving \eqref{966966}.
Note also that \eqref{968}, \eqref{969} imply, since $R(\lambda)\ge 0$,
$$
\lambda J(\lambda)\ge  1-e^{-\lambda}(1+\lambda),
$$
so that $J(\lambda)\ge \lambda^{-1}-e^{-\lambda}(1+\lambda^{-1}),$
and thus\footnote{We leave for the reader to check that for $\lambda>0$, $e^{-\lambda}(1+\lambda^{-1})\le \lambda^{-2}$, which boils dow to study $q(\lambda)=e^{-\lambda}(\lambda^{2}+\lambda)$ reaching its maximum for $\lambda\in \R_{+}$, at $\lambda_{0}=(1+\sqrt{5})/{2}$
with $q(\lambda_{0})\approx0.84<1.$
} the sought result 
\eqref{9610}.
\end{proof}
\begin{rem}\rm
 Considering now the function $ \varphi_{0}$ defined by 
 \begin{equation}
 \varphi_{0}(\omega)=
\frac{e^{-\pi \omega}}{2\pi}
\int_{0}^{\pi/2}
\frac{e^{2\omega s}-1}{\sin s}
ds,
\end{equation}
we find that, for $\omega\ge 0$, using Lemma \ref{lem.915},
$$
\varphi_{0}(\omega)\ge \frac{e^{-\pi \omega}}{2\pi}
\int_{0}^{\pi/2}
\frac{e^{2\omega s}-1}{s}
ds=
 \frac{e^{-\pi \omega}}{2\pi}
\int_{0}^{\pi \omega}
\frac{e^{\sigma}-1}{\sigma}
d\sigma=\frac1{2\pi} J(\pi \omega),
$$
so that 
\begin{equation}
\varphi_{0}(\omega)\ge \frac{1}{2\pi^{2}\omega}-\frac{1}{2\pi^{3}\omega^{2}}.
\end{equation}
It is our goal now to prove a minoration of the same flavour for the function 
\eqref{psinunu} defined above.
\end{rem}
Assuming $\nu\in (0,1/2)$, 
we have
$\frac\pi 3<s_{\nu}<t_{\nu}<\frac\pi 2$ ($s_\nu, t_\nu$ are defined in \eqref{963}, $\psi_{\nu}$ in \eqref{psinunu}),
\begin{align}\label{9611}
2\pi &e^{\pi\omega}\psi_{\nu}(\omega)
=
\int_{0}^{t_{\nu}}
\frac{e^{2\omega \phi_{\nu}(s)}-1}{\sin s}
ds+\int_{t_{\nu}}^{\pi/2}
\frac{e^{2\omega \phi_{\nu}(s)}-1}{\sin s}
ds
\\
&\ge 
\underbrace{\int_{0}^{t_{\nu}}
\frac{e^{2\omega \phi_{\nu}(s)}-1}{\sin s}
ds}_{\text{on $(0,t_{\nu}), \ \phi_{\nu}(s)\ge 0$}}
-\int_{t_{\nu}}^{\pi/2}
\frac{ds}{\sin s}
\ge \int_{0}^{s_{\nu}}
\frac{e^{2\omega \phi_{\nu}(s)}-1}{\sin s}
ds-\int_{\pi/3}^{\pi/2}
\frac{ds}{\sin s}
\notag\\
&
\hskip175pt =
\underbrace{\int_{0}^{s_{\nu}}
\frac{e^{2\omega \phi_{\nu}(s)}-1}{\sin s}
ds}_{\substack{
\text{on $(0,s_{\nu})$}\\ 
\phi_{\nu}(s) >0\ \text{and\ } \phi'_{\nu}(s)> 0
}}-\frac{\ln 3}{2}.
\notag
\end{align}
\begin{claim}\label{claim1}
 For $s\in (0,\pi/2)$, we have $\phi_{\nu}(s)\ge \phi'_{\nu}(s)\sin s$.
 Moreover,
 for $s\in (0,s_{\nu})$, we have 
$\frac{1}{\sin s}\ge \frac{\phi'_{\nu}(s)}{\phi_{\nu}(s)}$.
\end{claim}
\begin{proof}[Proof of the Claim]
Indeed, we have 
\begin{align}
\phi_{\nu}(s)- &\phi'_{\nu}(s)\sin s= s-\nu^{2}\tan s- \sin s+\nu^{2}(1+\tan^{2}s)\sin s
\notag\\
&=\nu^{2}\bigl(\sin s+\sin s \tan^{2} s-\tan s\bigr)+s-\sin s
\notag
\\&=\nu^{2}\bigl(\frac{\sin s}{\cos^{2}s}-\frac{\sin s}{\cos s}\bigr)+s-\sin s
\notag
\\
&=
\frac{\nu^{2}\sin s}{\cos^{2}s}\bigl(1-\cos s\bigr)+s-\sin s\ge 0, \quad \text{for $s\in(0,\pi/2)$}.
\end{align}
The last part of the claim follows from the first part and  the fact that $\sin s$ 
and $\phi_{\nu}(s)$ are both positive on $(0,s_{\nu})$.
 \end{proof}
Going back now to  \eqref{9611}, we obtain that for $\nu\in (0,1/2)$ and $\omega>0$, we have 
\begin{multline}
2\pi e^{\pi\omega}\psi_{\nu}(\omega)
\ge
\int_{0}^{s_{\nu}}
\frac{e^{2\omega \phi_{\nu}(s)}-1}{\phi_{\nu}(s)}
\phi'_{\nu}(s)ds-\frac{\ln 3}{2}
\\
=\int_{0}^{2\omega\phi_{\nu}(s_{\nu})}
\frac{e^{\sigma}-1}{\sigma}
d\sigma-\frac{\ln 3}{2}=e^{2\omega \phi_{\nu}(s_{\nu})} J(2\omega \phi_{\nu}(s_{\nu}))
-\frac{\ln 3}{2},
\end{multline}
so that, using \eqref{9610}, we get 
$$
\psi_{\nu}(\omega)\ge\frac{1}{2\pi} e^{-\pi \omega}
e^{2\omega \phi_{\nu}(s_{\nu})}\Bigl(\frac{1}{2\omega \phi_{\nu}(s_{\nu})}
-\frac{1}{(2\omega \phi_{\nu}(s_{\nu}))^{2}}\Bigr)
-\frac{\ln 3}{2}\frac{1}{2\pi} e^{-\pi \omega},
$$
and since $\phi_{\nu}(s_{\nu})=\frac\pi 2-\epsilon_{\nu}$, with $\epsilon_{\nu}
\in (0,\pi/2)$,
we find also that $\epsilon_{\nu}$ is a concave function\footnote{We have from \eqref{964+},
$$
\epsilon_{\nu}=\frac\pi 2-\arccos \nu+\nu\sqrt{1-\nu^{2}}, \
\frac{d\epsilon_{\nu}}{d\nu}=2\sqrt{1-\nu^{2}},\
\frac{d^{2}\epsilon_{\nu}}{d\nu^{2}}=-2\nu/\sqrt{1-\nu^{2}}<0,
$$
so that the concavity gives 
$
\frac{\pi}2 \nu\le \epsilon_{\nu}\le 2\nu.
$} 
of $\nu\in (0,1)$ and 
$$
\frac{\pi \nu}{2}\le \epsilon_{\nu}\le 2\nu\quad \text{so that } \quad
2\phi_{\nu}(s_{\nu})=\pi-2\epsilon_{\nu}\in [\pi-{4\nu},\pi-\pi\nu],
\quad
$$
so that for $\nu\in (0,1/2]$, we have\footnote{We know that
$
\omega(\pi-2\epsilon_{\nu})\ge \omega(\pi-4\nu)\ge \omega(\pi-2)
$
so that to ensure $\omega(\pi-2\epsilon_{\nu})\ge4$, it suffices to assume
$\omega\ge 4/(\pi-2)$.}
(assuming $\omega>0$), 
\begin{multline*}
\psi_{\nu}(\omega)\ge
\frac{1}{2\pi} e^{-\pi \omega}
e^{\omega(\pi-2\epsilon_{\nu})}\Bigl(\frac{1}{\omega (\pi-2\epsilon_{\nu})}
-\frac{1}{(\omega
(\pi-2\epsilon_{\nu})
)^{2}}\Bigr)
-\frac{\ln 3}{2}\frac{1}{2\pi} e^{-\pi \omega},
\\
\ge \frac{1}{2\pi} e^{-4\nu \omega}\Bigl(\frac{1}{\omega \pi}
-\frac{1}{\omega^{2}(\pi-2)^{2}
}\Bigr)
-\frac{\ln 3}{2}\frac{1}{2\pi} e^{-\pi \omega},
\end{multline*}
We recall the notations \eqref{note},
so that $\nu=\sqrt{\kappa/\omega}$ i.e.
$
\nu \omega= \sqrt{\kappa \omega}
$
and we get
  \begin{equation}\label{9521}
\forall \omega>0,\quad
\psi_{\nu}(\omega)\ge
\frac{1}{2\pi} e^{-4\sqrt{\kappa\omega}}
 \Bigl(\frac{1}{\pi \omega}
-\frac{1}{\omega^{2}}\Bigr)
-\frac{\ln 3}{2}\frac{1}{2\pi} e^{-\pi \omega},\ 
\nu=\sqrt{\kappa/\omega}.
\end{equation}
\subsubsection{An explicit expresssion for $a_{11}$}
According to \eqref{2525}, we have
\begin{equation}\label{9522}
a_{11}(\tau, \sigma)=
\frac12+\frac 1{2\pi}\int_{0}^{+\io}
\frac{
\sin(2\pi t\tau-4\pi \sigma\tanh(t/2))
}{\sinh(t/2)}
dt.
\end{equation}
We have used in Section \ref{sec.52} the equivalent expression
$a_{11}(\tau,\sigma)=\frac12+\hat T_{\sigma}(\tau)$, where $T_{\sigma}$ is defined in 
\eqref{525+} and we were able to prove the estimate in Lemma \ref{lem.57}. It turns out that \eqref{956} is not optimal, and it is interesting to give an ``explicit'' expression for $a_{11}$ as displayed in \cite{MR2131219}.
\par
Using the notations \eqref{note}, we can write \eqref{9522} as
\begin{multline}\label{9523++}
a_{11}(\tau, \sigma)=
\frac12+\frac 1{4\pi}\int_{\R}
\im\frac{
\exp {i(\omega t-2\kappa\tanh(t/2))}
}{\sinh(t/2)}
dt
\\
=\frac12+\im\lim_{R\rightarrow+\io}
\frac 1{2\pi}\int_{[-R,R]}
\frac{
\exp {2i(\omega s-\kappa\tanh s)}
}{\sinh s}
ds.
\end{multline}
Defining the holomorphic function $G$ by
\begin{equation}\label{953++}
G(z)=\frac{
\exp {2i(\omega z-\kappa\tanh z)}
}{2\pi \sinh z},
\end{equation}
we see that $G$ has simple poles at $i\pi \Z$ and essential singularities at
$i\pi(\frac12+\Z)$.
For $R\in \R_{+}\backslash \frac\pi2\Z, \epsilon\in (0,\pi/2)$, we have
\begin{multline}\label{9525+}
\oint_{[-R,-\epsilon]\cup[\epsilon, R]} G(z) dz+
\oint_{\substack{\gamma_{\epsilon}^{-}\\\gamma_{\epsilon}^{-}(\theta)=\epsilon e^{i\theta}\\
-\pi\le t\le 0}} G(z) dz
+\oint_{\substack{\gamma_{R}^{+}\\\gamma_{R}^{+}(\theta)
=R e^{i\theta}\\
0\le t\le \pi}} G(z) dz
\\=2i\pi \sum_{\substack{k\in \N\\
k\pi<2R}}\text{\tt Res}(G,ik\pi/2).
\end{multline}
\begin{claim}\label{claim917}
We have
 $\lim_{\epsilon\rightarrow 0}\oint_{\gamma_{\epsilon}^{-}} G(z) dz=\frac{i}{2}$.
\end{claim}
\begin{proof}
Indeed we have 
\begin{multline*}
\int_{-\pi}^{0} \frac{
\exp {2i(\omega \epsilon e^{i\theta}-\kappa\tanh(\epsilon e^{i\theta}))}
}{2\pi \sinh (\epsilon e^{i\theta})
} i\epsilon e^{i\theta}d\theta\\=\frac{i}{2\pi}
\int_{-\pi}^{0} 
\frac{e^{2i\omega\epsilon e^{i\theta}}\epsilon e^{i\theta}}{ \sinh (\epsilon e^{i\theta})}
\exp {(-2i\kappa\tanh(\epsilon e^{i\theta}))}
d\theta,
\end{multline*}
and since the function 
$z\mapsto \frac{ze^{2i\omega z}}{\sinh z} e^{-2i\kappa \tanh z}$ is holomorphic  near 0 with value 1 at 0, we get the result of the claim.
 \end{proof}
\begin{lem}\label{lem.918}We have
 $\lim_{\N\ni m\rightarrow +\io}\im\left(\oint_{\gamma_{\frac\pi4+m\frac\pi2}^{+}} G(z) dz\right)=0.$
\end{lem}
\begin{proof}
Indeed we have  with $R=\frac\pi4+m\frac\pi2$,
\begin{multline*}
\im \int_{0}^{\pi} \frac{
\exp {2i(\omega R e^{i\theta}-\kappa\tanh(R e^{i\theta}))}
}{2\pi \sinh (R e^{i\theta})
} iR e^{i\theta}d\theta
\\=\frac{R}{\pi}\re 
\int_{0}^{\pi} 
\frac{e^{2i\omega R \cos{\theta}} 
e^{-2R\omega\sin \theta}
e^{i\theta}}{ 
1-e^{-2Re^{i\theta}}
}
e^{-Re^{i\theta}}
\exp {(-2i\kappa\tanh(R e^{i\theta}))}
d\theta
\\
=\frac{2R}{\pi}
\int_{0}^{\pi/2} 
\re\left\{
\frac{e^{2i\omega R \cos{\theta}} 
e^{-2R\omega\sin \theta}
e^{i\theta}}{ 
1-e^{-2Re^{i\theta}}
}
e^{-Re^{i\theta}}
\exp {(-2i\kappa\tanh(R e^{i\theta}))}
\right\}
d\theta,
\end{multline*}
so that 
\begin{multline}\label{9526}
\im\Bigl(\oint_{\gamma_{\frac\pi4+m\frac\pi2}^{+}} G(z) dz\Bigr)
\\=
\frac{2R}{\pi}
\int_{0}^{\pi/2} 
e^{-R\cos \theta}
e^{-2R\omega\sin \theta}
\re\left\{
\frac{e^{2i\omega R \cos{\theta}} 
e^{i\theta}}{ 
1-e^{-2Re^{i\theta}}
}
e^{-iR\sin\theta}
\exp {(-2i\kappa\tanh(R e^{i\theta}))}
\right\}
d\theta.
\end{multline}
We have also
\begin{align}\label{78899}
\tanh (Re^{i\theta})&=\frac{1-e^{-2Re^{i\theta}}}{1+e^{-2Re^{i\theta}}}.
\end{align}
\begin{claim}\label{claim919}
Defining for $m\in \N, \theta\in [0,\pi]$, 
$
g_{m}(\theta)=1-e^{-(\frac\pi2+m\pi)e^{i\theta}},
$
we find that 
\begin{equation}
\inf_{\substack{\theta\in [0,\pi]\\m\in \N}}\val{g_{m}(\theta)}=\beta_{0}>0,\quad
\inf_{\substack{\theta\in [0,\pi]\\m\in \N}}\val{2-g_{m}(\theta)}=\beta_{1}>0.
\end{equation}
\end{claim}
\begin{proof}[Proof of the claim]
If it were not the case, we could find sequences $\theta_{l}\in[0,\pi],  m_{l}\in \N$ such that
\begin{equation}\label{7699}
\lim_{l\rightarrow +\io}e^{-(\frac\pi2+m_{l}\pi)e^{i\theta_{l}}}=1.
\end{equation}
Taking the logarithm of the modulus of both sides, we would get  
$$\lim_{l\rightarrow +\io}(\frac\pi2+m_{l}\pi)\cos \theta_{l}=0,$$ i.e.
$
\cos \theta_{l}=\frac{\epsilon_{l}}{\frac\pi2+m_{l}\pi}$,  
$\lim_{l\rightarrow +\io}\epsilon_{l}=0
$.
Going back to \eqref{7699}, we find then
$$\lim_{l\rightarrow +\io}e^{-i(\frac\pi2+m_{l}\pi)
\sin\theta_{l}}=1,$$
i.e. since $\sin \theta_l\ge 0$,
$$
\lim_{l\rightarrow +\io}\exp -i\Bigl\{(\frac\pi2+m_{l}\pi)\Bigl(1-
\frac{\epsilon_{l}^{2}}{(\frac\pi2+m_{l}\pi)^{2}}
\Bigr)^{1/2}\Bigr\}=1,
$$
implying
$
\lim_{l\rightarrow +\io}e^{-i(\frac\pi2+m_{l}\pi)}=1,
$
which is not possible since$$
e^{-i(\frac\pi2+m_{l}\pi)}=-i(-1)^{m_{l}}\in \{\pm i\},
$$
proving the first inequality of the claim.
The second inequality follows from the same {\sl reductio ad absurdum},
starting with
\begin{equation}\label{7699++}
\lim_{l\rightarrow +\io}e^{-(\frac\pi2+m_{l}\pi)e^{i\theta_{l}}}=-1,
\end{equation}
ending-up with an impossibility since $-1\notin\{\pm i\}$.
 \end{proof}
 As a consequence of Claim \ref{claim919} and \eqref{78899}, we obtain
 for $R=\frac\pi 4+m\frac\pi2$, $\theta\in (0,\pi)$,
 \begin{equation}
\val{\tanh (Re^{i\theta})}\le \frac{2}{\beta_{1}}.
\end{equation}
Formula \eqref{9526} gives then 
\begin{multline}\label{9526++}
\Val{\im\left(\oint_{\gamma_{\frac\pi4+m\frac\pi2}^{+}} G(z) dz\right)}
\le 
\frac{2R}{\pi}
\int_{0}^{\pi/2} 
e^{-R\cos \theta}
e^{-2R\omega\sin \theta}
\frac{1}{ 
\beta_{0}}
\exp {(4\kappa/\beta_{1})}
d\theta,
\end{multline}
where, for $\omega>0$, the right-hand-side goes to zero when $R$ goes to $+\io$,
completing the proof of Lemma \ref{lem.918}.
\end{proof}
\begin{lem}\label{lem.920}With $G$ defined in \eqref{953++},
 we have
 \begin{equation}\label{9533}
2\pi \sum_{\substack{k\in \N}}\res{G,ik\pi/2}
=\frac{1}{1+e^{-2\pi \omega}}
+
\frac{e^{-\pi \omega}}{i(1+e^{-2\pi \omega})}
\res
{\frac{e^{2i\omega z -2i\kappa\coth z}}{\cosh z},0}.
\end{equation}
\end{lem}
\begin{proof}
We have 
$
\text{\tt Res}(G,ik\pi/2)=\text{\tt Res}(G_{k},0)
$
and
with $k=2l$,
\begin{equation*}
G_{k}(z)=\frac{
\exp {2i(\omega (z+\frac{ik\pi}2)-\kappa\tanh (z+\frac{ik\pi}2))}
}{2\pi \sinh (z+\frac{ik\pi}2)}
=\frac{e^{-2l\pi \omega} e^{2i \omega z}
e^{-2i\kappa
\tanh z}}
{2\pi (-1)^{l}\sinh z},
\end{equation*}
so that 
\begin{equation}
\text{\tt Res}(G_{2l},0)=
\frac{(-1)^{l}e^{-2l\pi \omega} }
{2\pi },
\end{equation}
whereas for $k=2l+1$, we have 
\begin{multline*}
G_{2l+1}(z)=\frac{
\exp {2i(\omega (z+il\pi+\frac{i\pi}2)-\kappa\tanh (z+il\pi+\frac{i\pi}2))}
}{2\pi \sinh (z+il\pi+\frac{i\pi}2)}
\\
=\frac{e^{-(2l+1)\pi \omega} e^{2i \omega z}
e^{-2i\kappa
\coth z}}
{2\pi (-1)^{l}i\cosh z},
\end{multline*}
so that 
\begin{equation}
\text{\tt Res}(G_{2l+1},0)=
\frac{(-1)^{l}e^{-(2l+1)\pi \omega} }
{2\pi  i}\text{\tt Res}\Bigl(
\frac{e^{2i\omega z -2i\kappa\coth z}}{\cosh z},0\Bigr),
\end{equation}
yielding
\begin{multline*}
2\pi\sum_{\substack{k\in \N}}
\text{\tt Res}(G,ik\pi/2)
\\=
\sum_{l\in \N}
{(-1)^{l}e^{-2l\pi \omega} }
+
\sum_{l\in \N}
\frac{(-1)^{l}e^{-(2l+1)\pi \omega} }
{i }\text{\tt Res}\Bigl(
\frac{e^{2i\omega z -2i\kappa\coth z}}{\cosh z},0\Bigr),
\\
=\frac{1}{1+e^{-2\pi \omega}}
+
\frac{e^{-\pi \omega}}{i(1+e^{-2\pi \omega})}
\text{\tt Res}\Bigl(
\frac{e^{2i\omega z -2i\kappa\coth z}}{\cosh z},0\Big),
\end{multline*}
concluding the proof of the lemma.
\end{proof}
\begin{pro}\label{pro.921}
 Using the notations \eqref{note}, with $a_{11}$ defined in  \eqref{9522}
 (see also \eqref{9523++}), we have for $\tau>0, \sigma\ge 0$,
\begin{equation}\label{9536}
 a_{11}(\tau, \sigma)
 =\frac{1}{1+e^{-2\pi \omega}}
+
\frac{e^{-\pi \omega}}{1+e^{-2\pi \omega}}
\im\left\{
\res{
\frac{e^{2i(\omega z -\kappa\coth z)}}{\cosh z},0}\right\}.
\end{equation}
\end{pro}
\begin{proof}
 Taking  the imaginary part of both sides in \eqref{9525+}, and letting 
 $R\rightarrow +\io$, 
 $\epsilon\rightarrow 0_{+}$, we get, using \eqref{9533},
 \eqref{9523++}, Claim \ref{claim917},
 $$
a_{11}-\frac12+\im \frac i2=\im i\Bigl(
\frac{1}{1+e^{-2\pi \omega}}
+
\frac{e^{-\pi \omega}}{i(1+e^{-2\pi \omega})}
\text{\tt Res}\bigl(
\frac{e^{2i\omega z -2i\kappa\coth z}}{\cosh z},0\bigr)
\Bigr),
 $$
 which is \eqref{9536}.
\end{proof}
\begin{rem}\rm
 In particular, when $\sigma=0$, we find for $\tau>0$
 \begin{equation}
1- a_{11}(\tau, 0)=\frac{e^{-4\pi^{2} \tau}}{1+e^{-4\pi^{2} \tau}},
\end{equation}
and since \eqref{5229} implies that
\begin{align*}
2\pi\re a_{12}(\tau, 0)&=
\int_{0}^{+\io}
\frac{\sin(4\pi t\tau)}{\cosh t} dt
=\im \poscal{e^{i4\pi \tau t} H(t)}{\sech t}_{\mathscr S'(\R_{t}), \mathscr S(\R_{t})}
\\
&=\im \frac{1}{4i\pi\tau}\poscal{\frac{d}{dt}\bigl\{e^{i4\pi \tau t}\bigr\} H(t)}{\sech t}
\\
&=\im \frac{1}{4i\pi\tau}
\left(\poscal{\frac{d}{dt}\bigl\{e^{i4\pi \tau t} H(t)\bigr\}}{\sech t}
-\poscal{\delta_{0}}{\sech}
\right)
\\
&=\frac{1}{4\pi\tau}-\im\frac{1}{4i\pi\tau}\poscal{e^{i4\pi \tau t} H(t)}{\sech' (t)}
=\frac{1}{4\pi\tau}+O(\tau^{-3}),\quad \tau\rightarrow+\io,
\end{align*}
we readily find that 
$$
\re a_{12}(\tau,0)\gg 1-a_{11}(\tau,0), \quad \tau\rightarrow+\io,
$$
providing another proof of Theorem \ref{thm.59} in the case $\sigma=0$.
\end{rem}
\begin{rem}\label{rem.923}\rm The equation \eqref{zeze} gives also 
$
 \im a_{12}(\tau,\sigma)
=
\frac{e^{-2\pi^2\tau}}{2} a_{11}(\tau,\sigma),
$
where \eqref{2525} gives,  using the notations \eqref{note},
\begin{align}
\im a_{12}(\tau,\sigma)&=\frac{1}{4\pi}
\int_{0}^{+\io}
\frac{\cos(t\omega-{2\kappa}{\coth(t/2)})}{\cosh({t}/2)} dt 
\\&=\frac{1}{2\pi}
\int_{0}^{+\io}
\frac{\cos \bigl(2(t\omega-{\kappa}{\coth t})\bigr)}{\cosh {t}} dt
\notag
\\&=\frac{1}{4\pi}
\int_{\R}
\frac{\cos \bigl(2(t\omega-{\kappa}{\coth t})\bigr)}{\cosh {t}} dt.
\notag
\end{align}
With $G$ given by \eqref{953++}, we note that 
\begin{equation}
\tilde G(z)=\frac {ie^{\pi \omega}}2G(z+\frac{i\pi}2)=\frac{
\exp {2i(\omega z-\kappa\coth z)}
}{4\pi \cosh z},
\end{equation}
an holomorphic function with simple poles at ${i\pi}(\frac12+\Z)$
and essential singularities  at ${i\pi}\Z$.
Following now for $\tilde G$
the track of $G$ in  Claim \ref{claim917},
Lemmas \ref{lem.918}, \ref{lem.920}  and Proposition \ref{pro.921},
we get
\begin{equation}\label{1111++}
\im a_{12}(\tau,\sigma)=\lim_{\substack{m\rightarrow+\io\\\epsilon\rightarrow 0_{+}}}
\re 
\oint_{[-R_{m},-\epsilon]\cup[\epsilon,R_{m}]}\tilde G(z) dz,\quad R_{m}=\frac\pi4+m\frac\pi2,
\end{equation}
and we have also  
\begin{multline}\label{9525+++}
\oint_{[-R_{m},-\epsilon]\cup[\epsilon, R_{m}]} \tilde G(z) dz-
\oint_{\substack{\gamma_{\epsilon}^{+}\\\gamma_{\epsilon}^{+}(\theta)=\epsilon e^{i\theta}\\
0\le t\le \pi}} \tilde G(z) dz
+\oint_{\substack{\gamma_{R_{m}}^{+}\\\gamma_{R_{m}}^{+}(\theta)
=R_{m} e^{i\theta}\\
0\le t\le \pi}} \tilde G(z) dz
\\=2i\pi \sum_{1\le k\le m}\text{\tt Res}(\tilde G,ik\pi/2)
=-\pi 
{e^{\pi \omega}}
\sum_{1\le k\le m}\res
{G(\zeta+\frac{ik\pi}2+\frac{i\pi}2)
,0}
\\
=-\pi 
{e^{\pi \omega}}
\sum_{2\le l\le m+1}\res
{G(\zeta+\frac{il\pi}2)
,0}.
\end{multline}
\begin{claim}\label{claim.924-}
 We have 
 $\lim_{\epsilon\rightarrow 0}\oint_{\gamma_{\epsilon}^{+}} \tilde G(z) dz=0.$
\end{claim}
\begin{proof}
Indeed, we have
$
-2i\kappa\coth \epsilon e^{i\theta}={-2i\kappa\frac{1+e^{-2\epsilon e^{i\theta}}}{
1-e^{-2\epsilon e^{i\theta}}
}}
$
and  for $\theta\in (0,\pi)$,
\begin{multline*}
\im\bigl(
\frac{1+e^{-2\epsilon e^{i\theta}}}{
1-e^{-2\epsilon e^{i\theta}}
}
\bigr)=\im\frac{
(1+e^{-2\epsilon e^{i\theta}})(1-e^{-2\epsilon e^{-i\theta}})
}{\val{1-e^{-2\epsilon e^{i\theta}}}^{2}}
=\im \frac{
e^{-2\epsilon e^{i\theta}}
-e^{-2\epsilon e^{-i\theta}}
}
{\val{1-e^{-2\epsilon e^{i\theta}}}^{2}}
\\
=e^{-2\epsilon \cos \theta}
\im \frac{
e^{-2\epsilon i\sin \theta}
-e^{2\epsilon i\sin \theta}
}
{\val{1-e^{-2\epsilon e^{i\theta}}}^{2}}
=e^{-2\epsilon \cos \theta}
\im \frac{
-2i \sin(2\epsilon \sin \theta)
}
{\val{1-e^{-2\epsilon e^{i\theta}}}^{2}}
\\
=-2 e^{-2\epsilon \cos \theta}
\frac{
\sin(2\epsilon \sin \theta)
}
{\val{1-e^{-2\epsilon e^{i\theta}}}^{2}}\le 0,\quad \text{if $\epsilon\le \pi/4$,}
\end{multline*}
so that 
$
\val{e^{-2i\kappa\coth \epsilon e^{i\theta}}}\le 1,
$
implying 
$$
4\pi\Val{\oint_{\gamma_{\epsilon}^{+}} \tilde G(z) dz}\le \int_{0}^{\pi}
\frac{\val{e^{i\omega\epsilon e^{i\theta}}}}{\val{\cosh \epsilon e^{i\theta}}} \epsilon \val{i e^{i\theta}} d\theta=\epsilon\int_{0}^{\pi}
\frac{e^{-\omega\epsilon\sin \theta}}
{\val{\cosh \epsilon e^{i\theta}}} d\theta,
$$
which goes to zero when $\epsilon\rightarrow0_{+}$, concluding the proof of Claim \ref{claim.924-}.
\end{proof}
\begin{claim}\label{claim.924}
 We have 
 $\lim_{\N\ni m\rightarrow +\io}\oint_{\gamma_{\frac\pi4+m\frac\pi2}^{+}} \tilde G(z) dz=0.$
\end{claim}
\begin{proof}
Indeed, we have, using Claim \ref{claim919}, 
$$
\val{\coth (R_{m} e^{i\theta})}=\Bigl\vert\frac{1+e^{-2R_{m}e^{i\theta}}}{1-e^{-2R_{m}e^{i\theta}}}
\Bigr\vert\le
\begin{cases}
 \frac{1+e^{-2R_{m}\cos \theta}}{\beta_{0}}\le \frac2{\beta_{0},}&\text{for $\theta\in [0,\pi/2],$}\\
 \\
 \Val{\frac{
 1+e^{2R_{m}e^{i\theta}}
 }{
 1-e^{2R_{m}e^{i\theta}}
 }}\le \frac{2}{\beta_{0}},&\text{for $\theta\in [\frac\pi 2,\pi],$}
 \end{cases}
 $$
 so that
 \begin{multline*}
\val {\tilde G(R_{m} e^{i\theta}) iR_{m}e^{i\theta}}\le R_{m} e^{4\kappa/\beta_{0}} e^{-2\omega R_{m}\sin \theta}
\begin{cases}
\Val{\frac{2 e^{-R_{m}e^{i\theta}}}{1+e^{-2R_{m}e^{i\theta}}}}
\le \frac{2 e^{-R_{m}\cos \theta}}{\beta_{1}}&\text{for $\theta\in [0,\frac\pi 2]$,}
 \\
\Val{\frac{2 e^{R_{m}e^{i\theta}}}{1+e^{2R_{m}e^{i\theta}}}}
\le \frac{2 e^{R_{m}\cos \theta}}{\beta_{1}}&\text{for $\theta\in [\frac\pi 2,\pi]$,}
\end{cases}
\\
\le  \frac{2R_{m}}{\beta_{1}}e^{4\kappa/\beta_{0}} e^{-2\omega R_{m}\sin \theta-R_{m}\val{\cos \theta}},
\end{multline*}
which goes to 0 when $m$ goes to $+\io$, proving the claim.
\end{proof}
Using \eqref{9533},
we calculate now
\begin{align*}
2\pi\sum_{l\ge 2}&\res
{G(\zeta+\frac{il\pi}2)
,0}
\\&=\frac{1}{1+e^{-2\pi \omega}}
+
\frac{e^{-\pi \omega}}{i(1+e^{-2\pi \omega})}
\res
{\frac{e^{2i\omega z -2i\kappa\coth z}}{\cosh z},0}
\\&\hskip199pt-2\pi\Bigl(\res{G,i\pi/2}+\res{G,0}\Bigr)
\\
&=\frac{1}{1+e^{-2\pi \omega}}
+
\frac{e^{-\pi \omega}}{i(1+e^{-2\pi \omega})}
\res
{\frac{e^{2i\omega z -2i\kappa\coth z}}{\cosh z},0}
\\
&\hskip199pt
+ie^{-\pi \omega}\res
{\frac{e^{2i\omega z -2i\kappa\coth z}}{\cosh z},0}
-1
\\
&=-\frac{e^{-2\pi \omega}}{1+e^{-2\pi \omega}}
-i\left(\frac{e^{-\pi \omega}}{1+e^{-2\pi \omega}}-e^{-\pi \omega}\right)
\res
{\frac{e^{2i\omega z -2i\kappa\coth z}}{\cosh z},0}
\\
&=-\frac{e^{-2\pi \omega}}{1+e^{-2\pi \omega}}
+ie^{-\pi \omega}\left(\frac{e^{-2\pi \omega}}{1+e^{-2\pi \omega}}\right)
\res
{\frac{e^{2i\omega z -2i\kappa\coth z}}{\cosh z},0},
\end{align*}
so that 
from \eqref{1111++}, \eqref{9525+++},  Claims \ref{claim.924-} \& \ref{claim.924},
we obtain
\begin{align*}
&\im a_{12}(\tau,\sigma)
\\
&=-\pi e^{\pi \omega}\frac1{2\pi}
\left(
-\frac{e^{-2\pi \omega}}{1+e^{-2\pi \omega}}
-e^{-\pi \omega}\left(\frac{e^{-2\pi \omega}}{1+e^{-2\pi \omega}}\right)
\im\Bigl\{
\text{\tt Res}\Bigl(
{\frac{e^{2i\omega z -2i\kappa\coth z}}{\cosh z},0}
\Bigr)
\Bigr\}
\right)
\\
&=e^{\pi \omega}\frac1{2}
\left(
\frac{e^{-2\pi \omega}}{1+e^{-2\pi \omega}}
+e^{-\pi \omega}\left(\frac{e^{-2\pi \omega}}{1+e^{-2\pi \omega}}\right)
\im\Bigl\{
\text{\tt Res}\Bigl(
{\frac{e^{2i\omega z -2i\kappa\coth z}}{\cosh z},0}
\Bigr)
\Bigr\}
\right),
\end{align*}
so that 
\begin{multline}\label{9542}
\im a_{12}(\tau,\sigma)
\\=
\frac{e^{-\pi \omega}}{2(1+e^{-2\pi \omega})}
+\frac{e^{-2\pi \omega}}{2(1+e^{-2\pi \omega})}
\im\left\{\res
{\frac{e^{2i\omega z -2i\kappa\coth z}}{\cosh z},0}\right\},
\end{multline}
recovering \eqref{9536} from \eqref{zeze}.
\end{rem}
\begin{nb}\rm
 We note that
 \begin{align}\label{}
\res
{\frac{e^{2i\omega z -2i\kappa\coth z}}{\cosh z},0}=\frac12
\res
{\frac{e^{i(\omega z -2\kappa\coth (z/2))}}{\cosh(z/2)},0},
\end{align}
so that \eqref{9542} corroborates  $(\text{A}14)$ in \cite{MR2131219};
however, we were not able to understand formulas 
 $(\text{A}10)$, $(\text{A}11)$ and $(20)$ in \cite{MR2131219}.
\end{nb}
\subsection{Airy function}\label{secairy} 
\index{Airy function}
\index{{~\bf Notations}!$\ai(x)$}
\subsubsection{Standard results on the Airy function}
We collect in this section a couple of classical results on the Airy function
(see e.g. Definition 7.6.8 in Section 7.6 of \cite{MR1996773} or  the references \cite{MR2722693},
\cite{MR0136773}, \cite{MR0174795}). For all the statements of this section whose proofs are not included, we refer the reader to Chapter 9 of \cite{special}.
\begin{defi}
 The Airy function $\ai$ is defined as the inverse Fourier transform of $\xi\mapsto e^{i(2\pi \xi)^3/3}$.
\end{defi}
\begin{pro}\label{pro910}
 For any $h>0$ and all $x\in \C$, we have 
 \begin{equation}\label{}
\ai(x)=\frac1{2\pi}\int e^{\frac i3(\xi+ih)^{3}} e^{ix(\xi+ih)} d\xi =e^{-xh} e^{\frac{h^{3}}3}
\frac1{2\pi}\int e^{-h \xi^{2}} e^{i(\frac{\xi^{3}}{3}-\xi h^{2})}e^{ix\xi} d\xi.
\end{equation}
We note that the function $\R\ni \xi\mapsto e^{\frac i3(\xi+ih)^{3}}$ belongs to the Schwartz space 
 for any $h>0$ since 
 $$
 \frac i3(\xi+ih)^{3}=-h \xi^{2}+\frac{h^{3}}{3} +i\bigl(\frac{\xi^{3}}{3}-\xi h^{2}\bigr),
 $$
 so that 
 $$
 e^{\frac i3(\xi+ih)^{3}}=e^{-h \xi^{2}} e^{i(\frac{\xi^{3}}{3}-\xi h^{2})} e^{h^{3}/3}.
 $$
\end{pro}
\begin{theorem}
 The Airy function $\ai$ is an entire function on $\C$, real-valued on the real line, which is the unique solution of the initial value problem for the Airy equation
 \begin{equation}\label{992}
\ai''(x)-x\ai(x)=0, \quad \ai(0)=\frac{3^{-1/6}\Gamma(1/3)}{2\pi}, \quad \ai'(0)=-
\frac{3^{1/6}\Gamma(2/3)}{2\pi}.
\end{equation}
We have also, for any $x\in \C$,
\begin{equation}\label{992+}
\ai(x)=\frac1\pi\int_{0}^{+\io} e^{-\xi^3/3} e^{-x\xi/2}\cos\bigl(\frac{x\xi\sqrt 3}{2}+\frac\pi 6\bigr)d\xi,
\end{equation}
and the power series expansion of the Airy function is
\begin{equation}\label{}
\ai(x)=\frac{1}{\pi 3^{2/3}}\sum_{k\ge 0}\frac{(3^{1/3}x)^k}{k!}\Gamma\bigl(\frac{k+1}{3}\bigr)
\sin\bigl(2(k+1)\frac \pi 3\bigr).
\end{equation}
\end{theorem}
\begin{lem}
 For $x\in\C\backslash\R_{-}$, we have 
 \begin{equation}\label{995+}
\ai(x)=\frac{1}{2\pi} e^{-\frac23 x^{3/2}}\int_{\R} e^{-x^{1/2}\xi^{2}} e^{i\xi^{3}/3} d\xi.
\end{equation}
\end{lem}
\begin{proof}
 Using Proposition \ref{pro910}, we get \eqref{995+} for $x>0$ (choosing $h=x^{1/2}$), and then we may use 
 an analytic continuation argument.
\end{proof}
\begin{theorem}\label{thm913}
For all $M\in \N$, for all
  $x\in \C\backslash \R_{-}$, we have 
 \begin{multline}\label{995}
\ai(x)=\frac{1}{2\pi}e^{-\frac{2x^{3/2}}{3}} x^{-1/4}
\left\{
\sum_{0\le l\le M}\frac{(-1)^l}{3^{2l}(2l)!}\Gamma\bigl(3l+\frac 12\bigr) x^{-3l/2}
+R_{M}(x)
\right\},
\\\text{with}\quad
\val{R_{M}(x)}\le \frac{\Gamma\bigl(3M+3+\frac 12\bigr)}{3^{2M+2}(2M+2)!} \val{x}^{-\frac{3(M+1)}2)}
\left(
\cos(\frac{\arg x}2)
\right)^{-3(M+1)-\frac{1}{2}}.
\end{multline}
For $x<0$, we have 
\begin{align}
\ai(x)&=\frac{1}{\val x^{1/4}\sqrt \pi}\Bigl(\sin\bigl(\frac \pi 4+\frac 23\val x^{3/2}\bigr)
+O(\val x^{-3/2})\Bigr),
\label{aineg}
\\
\ai'(x)&=-\frac{\val x^{1/4}}{\sqrt \pi}\Bigl(\cos\bigl(\frac \pi 4+\frac 23\val x^{3/2}\bigr)
+O(\val x^{-3/2})\Bigr).
\label{aineg+}
\end{align}
\end{theorem}
\begin{lem}
 With $j=e^{2i\pi/3}$ we have for all $x\in \C$,
 \begin{equation}\label{}
\ai(x)+j\ai(jx)+j^{2}\ai(j^{2}x)=0.
\end{equation}
In particular for $r\ge 0$, we have
\begin{equation}\label{929l}
\ai(-r)=2\re\bigl(e^{\frac{i\pi}3}\ai(re^{\frac{i\pi}3})\bigr).
\end{equation}
\end{lem}
\begin{lem}\label{lem914}
 The zeroes of the Airy function are simple and located on $(-\io, 0)$. We shall use the notation
 \begin{equation}\label{ai000}
 \ai^{-1}(\{0\})=\{\eta_{k}\}_{k\ge 0}, \quad \eta_{k+1}<\eta_{k}<0, \quad \lim_{k\rightarrow+\io} \eta_{k}=-\io.
\end{equation}
 The largest zero of $\ai$ is $\eta_{0}\approx -2.338107410$ and $\ai(\eta)$ is positive for $\eta>\eta_{0}$. We have also for all $k\ge 0$, 
 \begin{align}
&\ai(\eta_{2k+1})=0,\ \ai'(\eta_{2k+1}) <0
,\hskip5pt 
\ai(\eta_{2k})=0,\ \ai'(\eta_{2k}) >0,
\\
&\text{$\ai(\eta)<0$ \ for $\eta\in(\eta_{2k+1},\eta_{2k})$},\ 
\ai(\eta)>0\text{\ for $\eta\in (\eta_{2k+2},\eta_{2k+1}),$}\label{c11}
\\
&\text{$\ai''(\eta)>0$ for $\eta\in(\eta_{2k+1},\eta_{2k})$}, \ 
\text{$\ai''(\eta)<0$ for $\eta\in(\eta_{2k+2},\eta_{2k+1})$}.\label{convex}
\end{align}
\end{lem}
\begin{nb}\rm
The simplicity of the zeroes of the Airy function holds true for any non-zero solution of the Airy differential equation
$
y''=xy.
$
The solutions of this ODE are analytic functions and if $a$ is a double zero, we have
$y(a)=y'(a)=0$ and thus from the Airy equation, we get $y''(a)=0$; we may then prove by induction on $k\ge 1$ that $y^{(l)}(a)=0$ for $0\le l\le k+1$: it is proven for $k=1$, and if true for some $k\ge 1$, we get
$$
y^{(k+2)}(x)=\bigl(x y(x)\bigr)^{(k)}\Longrightarrow y^{(k+2)}(a)=0,
$$
proving the final step in the induction; as a consequence, the function has a zero of infinite order, which is impossible for a non-zero analytic function.
 Assertion \eqref{convex} follows from the Airy differential equation
 \eqref{992}, from \eqref{c11} and $\eta_{2k}<0$.
\end{nb}
\begin{rem}\rm
 For $M=0$, $\val{\arg x}\le \pi/3$, we have
\begin{equation*}
\val{R_{0}(x)}\le \frac{\Gamma\bigl(3+\frac 12\bigr)}{3^{2}(2)!} \val{x}^{-\frac{3}2}
\Bigl(
\frac{\sqrt 3}2
\Bigr)^{-\frac{7}{2}}
=\val{x}^{-\frac{3}2}\sqrt\pi 
\frac{5}{3^{11/4}\sqrt 2}\le\val{x}^{-\frac{3}2}\times 0.305455,
\end{equation*}
so that
\begin{align}
&\val{R_{0}(x)}\le 0.305455 \val x^{-3/2}\quad \text{if $\val{\arg x}\le \pi/3$,}\label{nest1}
\\
&\text{and  for $\val x\ge 12$, $\val{\arg x}\le \pi/3$ we have}\quad 
 \val{R_{0}(x)}\le  0.007349.\label{nest2}
\end{align}
We get then for $\lambda>0$, using \eqref{929l}
\begin{align*}
\ai(-\lambda)&=\frac{1}{\pi}\re\Bigl(e^{i\pi/3} \lambda^{-1/4}
e^{-i\frac23 \lambda^{3/2}}\bigl(\sqrt \pi e^{-i\pi/12}+R_{0}(\lambda e^{i\pi/3})\bigr)
\Bigr)
\\&\hs=\frac{1}{\sqrt\pi} \lambda^{-1/4}\cos(\frac \pi 4-\frac23 \lambda^{3/2})+\frac1\pi\re
\Bigl\{  \lambda^{-1/4}R_{0}(re^{i\pi/3})
e^{i\pi/4}e^{-i\frac23 \lambda^{3/2}}\Bigr\}
\\
&\hs\hs=\frac{1}{\sqrt\pi}\lambda^{-1/4}\Bigl(\sin(\frac \pi 4+\frac23 \lambda^{3/2})
+
\frac{1}{\sqrt \pi}
\re  \bigl\{R_{0}(\lambda e^{i\pi/3})
e^{i\pi/4}e^{-i\frac23 \lambda^{3/2}}\bigr\}
\Bigr),
\end{align*}
so that 
\begin{align}
\text{for $\lambda>0$,\quad}&\ai(-\lambda)=\frac{1}{\sqrt\pi}\lambda^{-1/4}\Bigl(\sin(\frac \pi 4+\frac23 \lambda^{3/2})
+
\tilde R_{0}(\lambda)\Bigr),\label{n01}
\\
\text{with}\quad&\val{\tilde R_{0}(\lambda)}\le \lambda^{-3/2} \times 0.172335,
\label{n02}
\\
\text{and for $\lambda\ge 12$,}\quad&
\val{\tilde R_{0}(\lambda)}\le 0.004146.\label{n03}
\end{align}
\end{rem}
\begin{rem}\rm
 For $M=1$, $\val{\arg x}\le \pi/3$, we have 
 \begin{multline}
 \val{R_{1}(x)}\le \frac{\Gamma\bigl(6+\frac 12\bigr)}{3^{4}(4)!} \val{x}^{-3}
\Bigl(
\frac{\sqrt{3}}{2}
\Bigr)^{-6-\frac{1}{2}}
\\
=
\val{x}^{-3}\sqrt \pi\frac{11!}{2^{21/2}\times 3^{37/4}\times 5}
\le \val{x}^{-3}\times 0.377203,
\end{multline}
and 
\begin{equation}\label{}
\text{for $\val x\ge 12$,\quad}
\val{R_{1}(x)}\le 0.000219,
\end{equation}
so that
\begin{align*}
\ai(-r)&=
\frac{1}{\sqrt\pi}r^{-1/4}\Bigl(\sin(\frac \pi 4+\frac23 r^{3/2})
+
\frac{\Gamma(7/2)}{18\sqrt \pi}\sin(\frac23 r^{3/2}-\frac \pi 4) r^{-3/2}
\\&\hskip184pt +
\frac{1}{\sqrt \pi}
\re  \bigl\{R_{1}(re^{i\pi/3})
e^{i\pi/4}e^{-i\frac23 r^{3/2}}\bigr\}
\Bigr)
\\
&=\frac{1}{\sqrt\pi}r^{-1/4}\Bigl(\sin(\frac \pi 4+\frac23 r^{3/2})
+
\frac{\Gamma(7/2)}{18\sqrt \pi}\sin(\frac23 r^{3/2}-\frac \pi 4) r^{-3/2}+
\frac{1}{\sqrt \pi}
\tilde R_{1}(r)\Bigr),
\end{align*}
with
\begin{align}
\text{for $r>0$,\quad }\val{\tilde R_{1}(r)}&\le 
r^{-3}\times 0.377203,\label{vest1}
\\
\text{for $r\ge 12$,\quad }\val{\tilde R_{1}(r)}&\le 
0.000219.\label{vest2}
\end{align}
We find  for $\lambda>0$,
\begin{multline}\label{9916+}
G(-\lambda)
\\=
 \int_{\lambda}^{+\io}\frac{1}{r^{1/4}\sqrt \pi}\Bigl( \sin\bigl(\frac \pi 4+\frac 23r^{3/2}\bigr)
 +\frac{\Gamma(7/2)}{18\sqrt\pi}
 r^{-3/2}\sin\bigl(\frac 23r^{3/2}-\frac{\pi}4\bigr)
+\frac1{\sqrt \pi}\tilde R_{1}(r)\Bigr) dr,
\end{multline}
and we  have 
\begin{multline*}
\int_{\lambda}^{+\io}\frac{1}{r^{3/4}\sqrt \pi} r^{1/2}\sin\bigl(\frac \pi 4+\frac 23r^{3/2}\bigr) dr
\\={\cos\bigl(\frac \pi 4+\frac 23\lambda^{3/2}\bigr)\frac{1}{\lambda^{3/4}\sqrt \pi}}
-\frac34
\int_{\lambda}^{+\io}\frac{1}{r^{7/4}\sqrt \pi} \cos\bigl(\frac \pi 4+\frac 23r^{3/2}\bigr) dr,
\end{multline*}
as well as 
\begin{multline*}
-\frac34
\int_{\lambda}^{+\io}\frac{1}{r^{7/4}\sqrt \pi} \cos\bigl(\frac \pi 4+\frac 23r^{3/2}\bigr) dr
=
-\frac34
\int_{\lambda}^{+\io}\frac{1}{r^{9/4}\sqrt \pi} r^{1/2}\cos\bigl(\frac \pi 4+\frac 23r^{3/2}\bigr) dr
\\
{=\frac3{4\sqrt \pi}\sin\bigl(\frac \pi 4+\frac 23\lambda^{3/2}\bigr) \lambda^{-9/4}
-\frac3{4\sqrt \pi}\frac94\int_{\lambda}^{+\io}r^{-13/4} \sin\bigl(\frac \pi 4+\frac 23r^{3/2}\bigr) dr,}
\end{multline*}
so that
\begin{multline}\label{9917}
\int_{\lambda}^{+\io}\frac{1}{r^{1/4}\sqrt \pi} \sin\bigl(\frac \pi 4+\frac 23r^{3/2}\bigr) dr
={\cos\bigl(\frac \pi 4+\frac 23\lambda^{3/2}\bigr)\frac{1}{\lambda^{3/4}\sqrt \pi}}
\\
+\frac3{4\sqrt \pi}\sin\bigl(\frac \pi 4+\frac 23\lambda^{3/2}\bigr) \lambda^{-9/4}
-\frac3{4\sqrt \pi}\frac94\int_{\lambda}^{+\io}r^{-13/4} \sin\bigl(\frac \pi 4+\frac 23r^{3/2}\bigr) dr.
\end{multline}
We have also 
\begin{multline}\label{9918+}
 \int_{\lambda}^{+\io}\frac{1}{r^{1/4}}
 \frac{\Gamma(7/2)}{18\pi}
 r^{-3/2}\sin\bigl(\frac 23r^{3/2}-\frac{\pi}4\bigr)
 dr
 =
  \frac{\Gamma(7/2)}{18\pi}\int_{\lambda}^{+\io}
r^{-7/4}\sin\bigl(\frac 23r^{3/2}-\frac{\pi}4\bigr)
 dr
 \\=-\frac{\Gamma(7/2)}{18\pi}
 \cos\bigl(\frac 23\lambda^{3/2}-\frac{\pi}4\bigr)\lambda^{-9/4}+
 \frac{\Gamma(7/2)}{18\pi}\frac94
 \int_{\lambda}^{+\io}
 \cos\bigl(\frac 23r^{3/2}-\frac{\pi}4\bigr) r^{-13/4}dr,
\end{multline}
so that 
\eqref{9917},
\eqref{9918+}
and \eqref{9916+} entail
\begin{align*}
G(-\lambda)
&={\cos\bigl(\frac \pi 4+\frac 23\lambda^{3/2}\bigr)\frac{1}{\lambda^{3/4}\sqrt \pi}}
+\frac3{4\sqrt \pi}\sin\bigl(\frac \pi 4+\frac 23\lambda^{3/2}\bigr) \lambda^{-9/4}
\\&\hskip20pt-\frac3{4\sqrt \pi}\frac94\int_{\lambda}^{+\io}r^{-13/4} \sin\bigl(\frac \pi 4+\frac 23r^{3/2}\bigr) dr
\\
&
\hskip20pt-\frac{\Gamma(7/2)}{18\pi}
 \cos\bigl(\frac 23\lambda^{3/2}-\frac{\pi}4\bigr)\lambda^{-9/4}+
 \frac{\Gamma(7/2)}{18\pi}\frac94
 \int_{\lambda}^{+\io}
 \cos\bigl(\frac 23r^{3/2}-\frac{\pi}4\bigr) r^{-13/4}dr
 \\
 & \hskip20pt +\frac1\pi\int_{\lambda}^{+\io} r^{-1/4 }\tilde R_{1}(r).
\end{align*}
We get then 
\begin{align*}
G(-\lambda)
&=
\frac{\lambda^{-3/4}}{\sqrt\pi}\Biggl(
\cos\bigl(\frac \pi 4+\frac 23\lambda^{3/2}\bigr)
+\frac3{4}\sin\bigl(\frac \pi 4+\frac 23\lambda^{3/2}\bigr) \lambda^{-6/4}
\notag
\\&\hskip70pt-\frac3{4}\times\frac94\lambda^{3/4}\int_{\lambda}^{+\io}r^{-13/4} \sin\bigl(\frac \pi 4+\frac 23r^{3/2}\bigr) dr
\notag
\\
&
\hskip80pt-\frac{\Gamma(7/2)}{18\sqrt\pi}
 \cos\bigl(\frac 23\lambda^{3/2}-\frac{\pi}4\bigr)\lambda^{-6/4}
 \notag
 \\&\hskip90pt+
 \frac{\Gamma(7/2)}{18\sqrt\pi}\frac94\lambda^{3/4}
 \int_{\lambda}^{+\io}
 \cos\bigl(\frac 23r^{3/2}-\frac{\pi}4\bigr) r^{-13/4}dr
 \notag
 \\
 & \hskip100pt +\frac{\lambda^{3/4}}{\sqrt\pi}\int_{\lambda}^{+\io} r^{-1/4 }\tilde R_{1}(r)\Biggr)
 \notag
 ,\end{align*}
 so that 
 \begin{equation}\label{n04}
G(-\lambda)=\frac{\lambda^{-3/4}}{\sqrt\pi}\Bigl(
 \cos\bigl(\frac \pi 4+\frac 23\lambda^{3/2}\bigr)
 +\lambda^{-3/2} S_{1}(\lambda)
 \Bigr),
\end{equation}
with
\begin{equation}\label{n05}
\val{S_{1}(\lambda)}\le \frac34+\frac{3}{4}+\frac{\Gamma(7/2)}{18\sqrt\pi}
+\frac{\Gamma(7/2)}{18\sqrt\pi}+\frac4{9\sqrt \pi}\times 0.377203\le 1.80293
\end{equation}
where we have used \eqref{vest1} for the bound of the last term above.
As a consequence, if $\lambda\ge 12$, we get that 
\begin{equation}\label{n06}
\val{\lambda^{-3/2} S_{1}(\lambda)}\le 0.0433716.
\end{equation}
This is allowing us to extend the proof of Lemma \ref{lem920} to all values.
Note that the first 10 values (and more) are accessible numerically.
\end{rem}
Since we have $\eta_{9}=-12.82877675<-12$,
Formulas \eqref{n01}, \eqref{n03}, \eqref{n04}, \eqref{n06} imply the following result.
\begin{lem}\label{lem918}
 With $\ai$ and $G$ defined above, we have for $-\lambda\le\eta_{9}$
 \begin{align}
\ai(-\lambda)&=\frac{1}{\sqrt\pi}\lambda^{-1/4}\Bigl(\sin(\frac \pi 4+\frac23 \lambda^{3/2})
+
\tilde R_{0}(\lambda)\Bigr),
\\
& \val{\tilde R_{0}(\lambda)}\le  \lambda^{-3/2} \times 0.172335\le  0.004146,
\label{9928}\\
G(-\lambda)&=\frac{\lambda^{-3/4}}{\sqrt\pi}\Bigl(
 \cos\bigl(\frac \pi 4+\frac 23\lambda^{3/2}\bigr)
 +\tilde S_{1}(\lambda)
 \Bigr),
 \\& \val{\tilde S_{1}(\lambda)}\le \lambda^{-3/2}\times 1.80293\le  0.0433716.
 \label{9929}
\end{align} 
\end{lem}
\subsubsection{More on the Airy function}
\begin{pro}\label{proai1}
 We have
 \begin{equation}\label{996}
\int_{0}^{+\io} \ai(x) dx =\frac 13.
\end{equation}
\end{pro}
\begin{proof}
 According to Theorem \ref{thm913}, the Airy function $\ai$ is rapidily decreasing on the positive half-line 
 and thus belongs to $L^{1}(\R_{+})$, so that the integral in \eqref{996} makes sense.
 Also we have  from Theorem \ref{thm913} and the Lebesgue Dominated Convergence Theorem that,
 \begin{equation}\label{998}
\int_{0}^{+\io} \ai(x) dx =\lim_{h\rightarrow 0_{+}}\int_{0}^{+\io} \ai(x) e^{xh} dx e^{-h^{3}/3},
\end{equation}
and we shall now calculate the right-hand-side of \eqref{998}.
We have for $h>0$,
$$
\int_{0}^{+\io} \ai(x) e^{xh} dx e^{-h^{3}/3}=\int_{0}^{+\io} 
\frac1{2\pi}\int e^{-h \xi^{2}} e^{i(\frac{\xi^{3}}{3}-\xi h^{2})}e^{ix\xi} d\xi dx
=\int_{0}^{+\io} \widehat{\psi_{h}}(- x) dx,
$$
with 
\begin{equation}\label{psih}
\psi_{h}(\xi)= e^{-h (2\pi\xi)^{2}} e^{i(\frac{(2\pi\xi)^{3}}{3}-(2\pi \xi) h^{2})},
\end{equation}
so that 
\begin{align*}
\int_{0}^{+\io} \ai(x) e^{xh} dx e^{-h^{3}/3}
&=\poscal{\frac{\delta_{0}}2
-\frac1{2\pi i}\text{pv}\frac1{\xi}}{\psi_{h}}_{\mathscr S', \mathscr S}
\\&\hs=\frac12
-\frac1{2\pi i}\langle\text{pv}\frac1{\xi},{e^{-h (2\pi\xi)^{2}} e^{i(\frac{(2\pi\xi)^{3}}{3}-(2\pi \xi) h^{2})}}
\rangle
\\&\hs\hs=
\frac12
-\frac1{2\pi}\langle\text{pv}\frac1{\xi},{e^{-h \xi^{2}} \sin(\frac {\xi^{3}}3-\xi h^{2})}\rangle.
\end{align*}
We note at this point that, according to \eqref{sin3}, the right-hand-side of the above equality
is for $h=0$ equal to 
$$
\frac12-\frac1{2\pi}\frac\pi3=\frac13,
$$
so that, with \eqref{998}, we are left to proving that 
\begin{equation}\label{999}
\lim_{h\rightarrow 0_{+}}\langle\text{pv}\frac1{\xi},{e^{-h \xi^{2}} \sin(\frac {\xi^{3}}3-\xi h^{2})}\rangle=\frac\pi3.
\end{equation}
We have 
\begin{multline*}
\int\frac{ \sin(\frac {\xi^{3}}3-\xi h^{2})}{\xi} e^{-h \xi^{2}} d\xi=\frac\pi3+
\int\frac{ \sin(\frac {\xi^{3}}3-\xi h^{2}) e^{-h \xi^{2}}-\sin(\frac{\xi^{3}}{3})}{\xi} d\xi
\\
=\frac\pi3+\underbrace{\int\frac{\sin(\frac{\xi^{3}}{3})}{\xi}\bigl(\cos(\xi h^{2})e^{-h\xi^{2}}-1\bigr)d\xi}_{I_{1}(h)}
-\underbrace{\int\frac{\sin(\xi h^{2})}{\xi}\cos(\frac{\xi^{3}}{3})e^{-h\xi^{2}}d\xi}_{I_{2}(h)}.
\end{multline*}
We have 
\begin{multline*}
I_{1,1}(h)=\int_{1}^{+\io}\frac{\xi^{2}\sin(\frac{\xi^{3}}{3})}{\xi^{3}}\bigl(\cos(\xi h^{2})e^{-h\xi^{2}}-1\bigr)d\xi
\\=\int_{1}^{+\io}\frac{\frac{d}{d\xi}(\cos(\frac{\xi^{3}}{3}))}{\xi^{3}}\bigl(\cos(\xi h^{2})e^{-h\xi^{2}}-1\bigr)d\xi,
\end{multline*}
and a simple integration by parts\footnote{The boundary term is easy to handle and for the derivative falling on $\xi^{-3}$, we use that
$\val{\cos(\xi h^{2})e^{-h\xi^{2}}-1}\le 2$; if the derivative falls on the other term we get
$$
\int_{1}^{+\io}\frac{\cos(\frac{\xi^{3}}{3})}{\xi^{3}}\bigl(2h \xi\cos(\xi h^{2})e^{-h\xi^{2}}
+e^{-h\xi^{2}}\sin(\xi h^2)h^2\bigr)d\xi,
$$
which goes trivially to 0 with $h$. } shows that 
$\lim_{h\rightarrow 0}I_{1,1}(h)=0$; we have also trivially that
$$0=\lim_{h\rightarrow 0}\int_{0}^{1}\frac{\xi^{2}\sin(\frac{\xi^{3}}{3})}{\xi^{3}}
\bigl(\cos(\xi h^{2})e^{-h\xi^{2}}-1\bigr)d\xi.$$
On the other hand, we have  
$$
\val{I_{2}(h)}\le \int h^{2}e^{-h\xi^{2}}d\xi=O(h^{3/2}),
$$
which completes the proof of \eqref{999} as well as the proof of Proposition \ref{proai1}. 
\end{proof}
\begin{lem}We have
 \begin{equation}\label{99o}
\lim_{R\rightarrow+\io}\int_{-R}^{0} \ai(x) dx=\frac23.
\end{equation}
\end{lem}
\begin{proof}
 Using \eqref{aineg}, we find for $R\ge 1$,
 \begin{multline*}
 \int_{-R}^{0} \ai(x) dx=\int_{0}^{R} \ai(-r) dr=\int_{0}^{1}\ai(-r) dr
 \\+\int_{1}^{R}
 \Bigl( \frac{1}{r^{1/4}\sqrt \pi}\sin\bigl(\frac \pi 4+\frac 23r^{3/2}\bigr)
+O(r^{-7/4})\Bigr)dr,
\end{multline*} 
 proving that the limit in the left-hand-side of \eqref{99o} is existing.
 \par\no
\begin{claim}
 $
\lim_{h\rightarrow 0_{+}}\int_{-\io}^{0} \ai(x) e^{xh}dx=
\int_{-\io}^{0} \ai(x) dx.
 $
\end{claim}
 \begin{proof}[Proof of the Claim]
 We have 
 $$
\int_{-\io}^{0} \ai(x) e^{xh}dx=\int_{-\io}^{-1} \ai(x) e^{xh}dx+\underbrace{\int_{-1}^{0} \ai(x) e^{xh}dx}_{\text{with limit $\int_{-1}^{0} \ai(x) dx$}}
 $$
 and using  \eqref{aineg}, we have only to check 
 \begin{multline*}
 \int_{-\io}^{-1} \val x^{-1/4} e^{xh+i\frac23\val x^{3/2}}dx=\int_{1}^{+\io} t^{-1/4} e^{-th+i\frac23 t^{3/2}} dt
 \\=-\int_{1}^{+\io}\frac{d}{dt}\left\{e^{-th+i\frac23 t^{3/2}} \right\}(h-i t^{1/2})^{-1} t^{-1/4} dt
 =e^{-h+i\frac23 }(h-i )^{-1} \\+\int_{1}^{+\io}e^{-th+i\frac23 t^{3/2}}
 \bigl(
 (h-i t^{1/2})^{-2}\frac i2 t^{-3/4}-(h-i t^{1/2})^{-1}\frac14 t^{-5/4}
\bigr) dt,
\end{multline*}
and since the absolute value of the  integrand in the last integral is bounded above by
$
\frac34 t^{-7/4}
$, we get the result of the Claim.
\end{proof}
With \eqref{998}, \eqref{psih},
this gives
$$
\int_{-\io}^{+\io} \ai(x) dx =\lim_{h\rightarrow 0_{+}}\int_{-\io}^{+\io} \ai(x) e^{xh} dx e^{-h^{3}/3}
=\lim_{h\rightarrow 0_{+}}\left(\int_{\R}\widehat{\psi_{h}}(-\xi) d\xi=\psi_{h}(0)\right)=1,
$$
and Proposition \ref{proai1} provides the result of the lemma.
\end{proof}
\subsubsection{Asymptotic expansion for the function $G$ defined in \eqref{427}}
\begin{lem}\label{lem917}
With  $G$ defined in \eqref{427}, we have
\begin{equation}\label{9918}
G(-\lambda)=\lambda^{-3/4}\pi^{-1/2}\sin(\frac{3\pi}4+\frac23 \lambda^{3/2})+ O(\lambda^{-9/4}), \quad \lambda\rightarrow+\io.
\end{equation}
\end{lem}
\begin{proof}
Property \eqref{99o} and \eqref{aineg} give  for $\eta=-\lambda <0$,
\begin{align*}
&G(\eta)=\frac23+\int_{0}^{\eta}\ai(\xi) d\xi=\int_{-\io}^{\eta} \ai(\xi) d\xi
=\int_{\lambda}^{+\io} \ai(-r) dr
\\&=\int_{\lambda}^{+\io} 2\re\bigl(e^{\frac{i\pi}3}\ai(e^{\frac{i\pi}3} r)\bigr) dr\ 
\text{\tiny(we have used \eqref{929l}); we use now \eqref{995} for $M=1, x\in e^{i\pi/3}\R_{+}$)}
\\&=
 \int_{\lambda}^{+\io}\Bigl( \frac{1}{r^{1/4}\sqrt \pi}\sin\bigl(\frac \pi 4+\frac 23r^{3/2}\bigr)
 +\frac{\Gamma(7/2)}{3^{2}2\pi}
 r^{-7/4}\sin\bigl(\frac 23r^{3/2}-\frac{\pi}4\bigr)
+O(r^{-13/4})\Bigr) dr
\\
&=(2/3)^{1/2}\pi^{-1/2}\int_{\frac23 \lambda^{3/2}}^{+\io}
 s^{-1/2}\sin\bigl(\frac \pi 4+s\bigr) ds
\\&\hskip75pt +\frac{(2/3)^{3/2}\Gamma(7/2)}{3^{2}2\pi}\int_{\frac23 \lambda^{3/2}}^{+\io}s^{-3/2}\sin\bigl(s-\frac {\pi} 4\bigr)
ds
+O(\lambda^{-9/4}).
\end{align*}
We integrate by parts in the first integral with
$$
\int_{\frac23 \lambda^{3/2}}^{+\io}
 s^{-1/2}\sin\bigl(\frac \pi 4+s\bigr) ds=-\int_{\frac23 \lambda^{3/2}}^{+\io}
 s^{-1/2}\frac{d}{ds}\left\{\cos\bigl(\frac \pi 4+s\bigr)\right\}ds
$$
$$
=(\frac23 \lambda^{3/2})^{-1/2}\cos(\frac\pi4+\frac23 \lambda^{3/2})+
\int_{\frac23 \lambda^{3/2}}^{+\io}(-1/2) s^{-3/2}\cos(\pi/4+s) ds.
$$
We have to deal with two integrals of type
$$
\int_{\lambda^{3/2}}^{+\io}s^{-3/2} \frac{d}{ids}e^{is} ds=i(\lambda^{3/2})^{(-3/2)} e^{i\lambda^{3/2}}
-\frac1i\int_{\lambda^{3/2}}^{+\io}(-3/2) s^{-5/2} e^{is}ds=
 O(\lambda^{-9/4}).
$$
Eventually we find
$
G(-\lambda)=\lambda^{-3/4}\pi^{-1/2}\cos(\frac\pi4+\frac23 \lambda^{3/2})+ O(\lambda^{-9/4}).
$
\end{proof}
With $(\eta_{k})_{k\ge 0}$ standing for the decreasing sequence of the zeroes of the Airy function (cf. Lemma \ref{lem914}), we have the following table of variation for the function $G$.
\vs\vs
\begin{center}
{\tiny\renewcommand{\arraystretch}{2}
\begin{tabular}{   |    l      | cc c c c c c c c c c c c  |    }
\hline
$\eta$         &$-\io$&\dots&$\eta_{2k+2}$ &         &$\eta_{2k+1}$   &    &$\eta_{2k}$        &\dots  &$\eta_{1}$   &&$\eta_{0}$&&$+\io$\\
\hline
 $G''(\eta)=\ai'(\eta)$&0&\dots&$+$ &       &$-$ &      &$+$         &\dots  &$- $  &              &$+$&&0\\
 \hline
$G'(\eta)=\ai(\eta)$&0&\dots&$0$  &  +                  &0  &$-$                          &$0$                       &\dots  &$0$&$-$                 &0&$+$&0\\
\hline
$G(\eta)$   &0&\dots&$G(\eta_{2k+2})$ &$\nearrow$ &G($\eta_{2k+1})$&$\searrow$  &G($\eta_{2k})$   &\dots   &$G(\eta_{1})$&$\searrow$    &
$G(\eta_{0})$
&$\nearrow$&1\\
  \hline
 \end{tabular}}
 \end{center}
 \vs
 
 \vs
\vs
\begin{center}
\renewcommand{\arraystretch}{2}
\begin{tabular}{|c|c|c|c|c|l|l|l|l|l|c|l|}
  \hline
  $\mathbf\eta$&\tiny${\eta_{4}}\!=\!-7.944133589$&\tiny${\eta_{3}}\!=\!-6.786708100$ & \tiny${\eta_{2}}\!=\! -5.520559828$ &\tiny${\eta_{1}}\!=\!-4.087949444$&\tiny${\eta_{0}}\!=\!-2.338107410$
 \\
  \hline
{$\mathbf{G(\eta)}$}&\tiny$-0.1187912133$&\tiny0.1333996865& \tiny$-0.1550343634$&\tiny 0.1917571397&\tiny$-0.2743520591$ \\
  \hline
\end{tabular}
\end{center}
\vs
\begin{center}
\renewcommand{\arraystretch}{2}
\begin{tabular}{|c|c|c|c|c|l|l|l|l|l|c|l|}
  \hline
  $\mathbf\eta$&\tiny$\eta_{9}\!=\!-12.82877675$&\tiny$\eta_{8}\!=\!-11.93601556$&
  \tiny$\eta_{7}\!=\!-11.00852430$&\tiny${\eta_{6}}\!=\!-10.04017434,$& \!\!\tiny${\eta_{5}}\!=\!-9.022650854$\!\!\!
 \\
  \hline
{$\mathbf{G(\eta)}$}&\tiny 0.08315615192&\tiny$ -0.08775971160$&\tiny 0.09322050200&\tiny $-0.09984115980$&\tiny 0.1080976882 \\
  \hline
\end{tabular}
\end{center}
\vs\vs\vs
\begin{lem}\label{lem920}
 The zeroes of the function $G$ on the real line are simple and make a decreasing sequence of negative numbers $(\xi_{l})_{l\le 0}$ such that
 \begin{equation}\label{stuff}
 \dots\eta_{2k+2}<\xi_{2k+2}<\eta_{2k+1}<\xi_{2k+1}<\eta_{2k}<\xi_{2k}\dots
 ,\quad \xi_{0}\approx-1.38418.
\end{equation}
The largest ten zeroes of $G$ are given by the following table
\begin{center}
\vs\vs
\begin{tabular}{|c|c|c|c|c|c|c|c|c|c|}
  \hline
$\xi_{0}$& $\xi_{1}$ & $\xi_{2}$&$\xi_{3}$& $\xi_{4}$\\
  \hline
{$-1.38418$} & $-3.33004$ & $-4.86074$&$-6.18885$ &$-7.39024$\\
  \hline
\end{tabular}
\vs\vs\vs
\begin{tabular}{|c|c|c|c|c|c|c|c|c|c|}
  \hline
 $\xi_{5}$&  $\xi_{6}$&  $\xi_{7}$&  $\xi_{8}$&  $\xi_{9}$ \\
  \hline
$-8.5022$&$-10.5366$&$-11.4826$&$-12.3913$&$-13.2679$\\
  \hline
\end{tabular}
\end{center}
\vskip15pt \noindent
 For all $k\in \N$, we have 
 \begin{equation}\label{stuff+}
G(\eta_{2k})<0<G(\eta_{2k+1}),
\end{equation} 
 and 
 $G(\eta_{2k})$ (resp. $G(\eta_{2k+1})$) is a local minimum (resp. maximum) of $G$  near $\eta_{2k}$ (resp. $\eta_{2k+1}$).
 Moreover, $G(\eta_{0})$ is an absolute minimum of the function $G$ on the real line.
 \end{lem}
 \begin{nb}
 We claim  also  that 
  \begin{equation}\label{stuff++}
 \val{G(\eta_{2k})}>G(\eta_{2k+1})>\val{G(\eta_{2k+2})},
\end{equation}
but shall not provide a complete proof for that statement, which is anyway not needed is our Section \ref{secmainpara}.
\end{nb}
\begin{proof}
In the first place, we know that $G(\eta_{0})<0$ and $G$ strictly increases on $[\eta_{0},+\io)$ so that $\xi_{0}\approx -1.38418$ is defined as the unique zero of $G$ on 
$(\eta_{0},0)$ since $G(0)=2/3$. We may note that we found in particular that 
\begin{equation}\label{}
\forall \eta>\eta_{0}, \quad 1>G(\eta)>G(\eta_{0}).
\end{equation}
Also, the first ten zeroes of $G$ are simple and satisfy \eqref{stuff}, \eqref{stuff+}  and  \eqref{stuff++}.
Moreover, using Lemma \ref{lem918}, we obtain that  for $\lambda\ge 12$,
 \begin{align*}
 G(-\lambda)=0&\Longrightarrow \bigl\vert\cos\bigl(\frac{3\pi}4+\frac23\lambda^{3/2}\bigr)\bigr\vert\le 0.0433716,
 \\
 \ai(-\lambda)=0&\Longrightarrow \bigr\vert\sin\bigl(\frac{\pi}4+\frac23\lambda^{3/2}\bigr)\bigr\vert\le 0.004146,
\end{align*}
As a result, if $-\lambda$ is a double zero of $G$ we must have
both inequalities above,
which is impossible.
As a result all zeroes of $G$ are simple\footnote{It is not hard to obtain an asymptotic version of this, namely the same result for $\lambda$
large enough. However, asymptotic methods provide asymptotic results and to get a result at a finite distance, we had
to use the numerical results of Lemma \ref{lem918}, grounded on a numerical estimate of the constants appearing in Theorem \ref{thm913}.
}
and located on $(-\io, 0)$.
Let us consider the interval $[\eta_{2k+1}, \eta_{2k}]$: 
we have 
$$
\ai(\eta_{2k+1})=\ai(\eta_{2k})=0, \quad 
\ai'(\eta_{2k+1})<0<\ai'(\eta_{2k}), \quad \ai'' >0\text{ on $(\eta_{2k+1}, \eta_{2k})$.}
$$
As a result, we obtain that $G$ has a local minimum at $\eta_{2k}$ and a local maximum at
$\eta_{2k+1}$ . Moreover we find from \eqref{9928} in  Lemma \ref{lem918} and $k\ge 5$ that 
$$
\max\Bigl(\val{\sin(\frac \pi 4+\frac23\val{\eta_{2k}}^{3/2})}, 
\val{\sin(\frac \pi 4+\frac23\val{\eta_{2k+1}}^{3/2})}\Bigr)
\le  0.004146
$$
{which implies that}
$$
\min\Bigl(\val{\cos(\frac \pi 4+\frac23\val{\eta_{2k}}^{3/2})}, 
\val{\cos(\frac \pi 4+\frac23\val{\eta_{2k+1}}^{3/2})}\Bigr)
\ge   0.99999.
$$ 
We know that $\ai'(\eta_{2k})>0$,
which implies, thanks\footnote{Here this is proven if $k$ is large enough from \eqref{aineg+}, and we leave to the reader the proof of a numerical estimate
analogous to Lemma \ref{lem918} for the derivative of the Airy function.
A direct estimate is possible, using \eqref{995+} and the identity (to be differentiated) for $\lambda>0$,
\begin{align}
\ai(-\lambda)&=\frac{\lambda^{-1/4}}{\sqrt\pi}\Bigl\{
\sin\bigl(\frac \pi4+\frac23\lambda^{3/2}\bigr)+a_{0}(\lambda)\lambda^{-3/2}
\Bigr\},
\\
a_{0}(\lambda)&=\frac{\lambda^{3/2}}\pi e^{i(\frac\pi3-\frac23\lambda^{3/2})}\int_{\R}
e^{-\xi^{2}\lambda^{1/2} e^{i\pi/6}}\bigl(\cos(\xi^{3}/3)-1\bigr) d\xi.
\end{align}
} 
to \eqref{aineg+}
$$
\cos\bigl(\frac \pi 4+\frac23\val{\eta_{2k}}^{3/2}\bigr)\le -0.99999,
\qquad
\cos\bigl(\frac \pi 4+\frac23\val{\eta_{2k+1}}^{3/2}\bigr)\ge 0.99999,
$$
and Lemma \ref{lem918} implies that
$G(\eta_{2k})<0<G(\eta_{2k+1})$, which is \eqref{stuff+}.
Since the function $G$ is strictly  monotone decreasing
on the interval 
$[\eta_{2k+1},\eta_{2k}]$, it has a unique simple zero $\xi_{2k+1}$ 
on the interior of this interval.
Analogously, we can prove that 
on the interval 
$[\eta_{2k+2},\eta_{2k+1}]$, it has a unique simple zero $\xi_{2k+2}$ 
on the interior of this interval,
proving that the sequence of zeroes of the function $G$ is decreasing strictly with
$$
\eta_{2k+2}<\xi_{2k+2}<\eta_{2k+1}<\xi_{2k+1}<\eta_{2k}<\xi_{2k}, \quad k\ge 0.
$$
We shall prove a weaker statement than  \eqref{stuff++}: we know that 
$\val{G(\eta_{l})}<\val{G(\eta_{0})})$ for $1\le l\le 9$ from the numerical values obtained above.
Moreover if $\lambda\ge 12$ we find
$$
\val{G(-\lambda)}\le \lambda^{-3/4}\pi^{-1/2}(1+0.0433716)\le 0.0913016<\val{G(\eta_{0})}=0.2743520591,
$$
proving indeed that $G(\eta_{0})$ is the absolute minimum of the function $G$ on the real line,
since the desired estimate is proven for $\eta>\eta_{0}$ and for $\eta<\eta_{0}$,
either $G(\eta)\ge 0$, or  $-0.0913016 \le G(\eta)< 0$ if $\eta \le -12$.
As said above, the values less than 12 are treated directly by a numerical calculation.
The proof of the lemma is complete.
\end{proof}
 \begin{figure}[ht]
\centering
\hskip-55pt\scalebox{0.7}{\includegraphics[angle=90,height=750pt,width=1.2\textwidth]{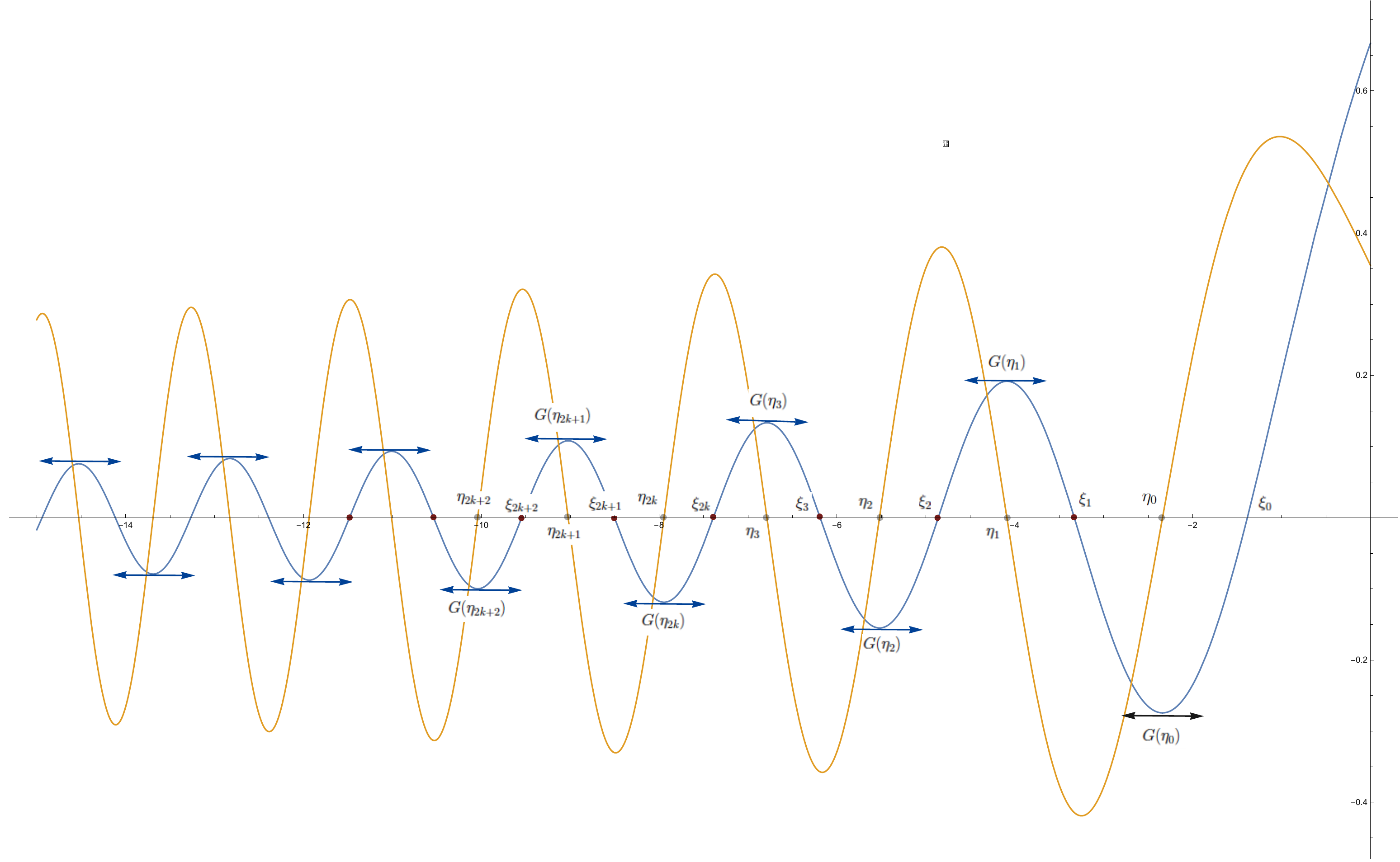}}\vs\vs
\caption{The function $G$ and its derivative the  Airy function,  on $\R_{-}$.}
\label{mmllkkjj}
\label{pic9}
\end{figure}
\subsection{Miscellaneous formulas}
\subsubsection{Some elementary formulas}\label{sec.arctan}
We define for $\tau\in \R$, 
\begin{equation}\label{}
\arctan \tau=\int_{0}^{\tau}\frac{dt}{1+t^{2}},
\end{equation}
and we note that $\arctan \tau\in (-\pi/2, \pi/2)$,
$$\forall \tau\in \R,\ 
\tan (\arctan \tau)=\tau,\qquad \forall\theta\in(-\pi/2, \pi/2),\ \arctan(\tan \theta)=\theta.
$$
Moreover we have for $\tau\in \R$,
\begin{equation}\label{arctan}
e^{i\arctan \tau}=\frac{1}{\sqrt{1+\tau^{2}}}(1+i\tau),
\end{equation}
since for $\theta\in (-\pi/2, \pi/2)$, $\tau=\tan \theta$, we have $1+\tau^{2}=\frac1{\cos^{2}\theta}$ and thus
$$
\cos \theta>0\Longrightarrow \cos \theta=\frac{1}{\sqrt{1+\tau^{2}}}
\Longrightarrow
-{\sin \theta}=-\frac12(1+\tau^{2})^{-3/2} 2\tau (1+\tau^{2}),
$$
so that 
$
e^{i\theta}=\frac{1}{\sqrt{1+\tau^{2}}}(1+i\tau).
$
\vs\vs\noindent
Let $a\in \R_{+}$ be given.
The Fourier transform of $\indic{[-a,a]}$ is
\begin{equation}\label{12345}
\int_{-a}^{a} e^{-2i\pi x\xi} dx=2\int_{0}^{a}\cos(2\pi x \xi) dx=\frac{2}{2\pi \xi}[\sin(2\pi x\xi)]^{x=a}_{x=0}
=\frac{\sin (2\pi a \xi)}{\pi \xi}.
\end{equation}
\vs\vs\vs\noindent
\subsubsection{Taking the derivative of $F_{k}$ on $\R_{+}$}\label{sub555+}
We have, using a parity argument, 
$$
F_{k}(a)=
\int_{\R}
\frac{\sin a\tau}{\pi \tau}
\frac{(1+i\tau)^{2k+1}}{(1+\tau^{2})^{k+1}}d\tau=\sum_{0\le 2l\le 2k}
\int_{\R}
\frac{\sin a\tau}{\pi \tau}
\frac{\binom{2k+1}{2l}(-1)^{l}\tau^{2l}}{(1+\tau^{2})^{k+1}}d\tau.
$$
We see also that
$
1+2k+2-2l=2k+3-2l\ge 3
$
so that we can take the derivative of $F_{k}$
and get
$$
F'_{k}(a)=
\sum_{0\le 2l\le 2k}
\int_{\R}
\frac{\cos a\tau}{\pi}
\frac{\binom{2k+1}{2l}(-1)^{l}\tau^{2l}}{(1+\tau^{2})^{k+1}}d\tau=
\frac1{\pi}\int_{\R}
({\cos a\tau})
\re\left(\frac{(1+i\tau)^{k}}{(1-i\tau)^{k+1}}\right)d\tau,
$$
with absolutely converging integrals.
For $a>0$, we have 
\begin{equation}\label{fprime}
F'_{k}(a)=\frac1{\pi}\int_{\R}
({\cos a\tau})
\frac{(1+i\tau)^{k}}{(1-i\tau)^{k+1}}
d\tau,
\end{equation}
since 
\begin{equation}\label{}
\lim_{\lambda\rightarrow+\io}\int_{-\lambda}^{\lambda}\frac{\tau^{j}\cos(a\tau)}{(1+\tau^{2})^{k+1}} d\tau
\text{\quad  makes sense for $j\le 2k+1$ (and vanishes for $j$ odd).}
\end{equation}
\subsubsection{A proof of the weak limit}\label{sub555}
We have for $u\in \mathscr S(\R^{n})$, according to \eqref{eza654},
$$
\poscal{\bigl(\mathbf 1\{2\pi (x^{2}+\xi^{2})\le a\}\bigr)^{w}u}{u}=\iint_{2\pi (x^{2}+\xi^{2})\le a}\mathcal W(u,u)(x,\xi) dx d\xi,
$$
so that 
 implies 
$$
\sum_{k\ge 0} F_{k}(a)\poscal{\mathbb P_{k} u}{u}_{L^{2}(\R^{n})}=\iint_{2\pi (x^{2}+\xi^{2})\le a}\mathcal W(u,u)(x,\xi) dx d\xi.
$$
Choosing now $u=u_{k}$ as a normalized  eigenfunction of the Harmonic Oscillator
with eigenvalue $k+1/2$,
we obtain
$$
F_{k}(a)=\iint_{2\pi (x^{2}+\xi^{2})\le a}\mathcal W(u_{k},u_{k})(x,\xi) dx d\xi.
$$
Since the function $(x,\xi)\mapsto \mathcal W(u_{k},u_{k})(x,\xi) $ belongs to the Schwartz class of $\RZ$,
we find that
$$
\lim_{a\rightarrow+\io} F_{k}(a)=\iint_{\RZ}\mathcal H(u_{k},u_{k})(x,\xi) dx d\xi=\norm{u_{k}}^{2}_{L^{2}(\R^{n})}=1,\quad\text{\tt qed.}
$$
\subsubsection{A different normalization for the Wigner function}
\index{a different normalization for the Wigner distribution}
The paper \cite{MR2761287} is using a different normalization for the Wigner distribution in $n$ dimensions with
\begin{equation}\label{}
\widetilde{\mathcal W}(u,v)(x,\xi)=(2\pi)^{-n}\int_{\R^{n}} u(x+\frac z2)\bar v(x-\frac z2) e^{-i z\cdot \xi} dz.
\end{equation}
The relationship with our definition \eqref{wigner} is 
\begin{equation}\label{wwwlll}
\widetilde{\mathcal W}(u,v)(x,\xi)=\mathcal W(u,v)(x,\frac{\xi}{2\pi})(2\pi)^{-n}.
\end{equation}
As a result, we find that
$$
\mathcal E_{lo}\bigl(\mathbb B^{2n}(R)\bigr)=\sup_{\norm{u}_{L^{2}(\R^{n})}=1}\iint_{\val x^{2}+\val \xi^{2}\le R^{2}}
\widetilde{\mathcal W}(u,u)(x,\xi) dx d\xi,
$$
is equal to 
\begin{multline*}
\sup_{\norm{u}_{L^{2}(\R^{n})}=1}\iint_{\val x^{2}+4\pi^{2}\val \xi^{2}\le R^{2}}
\mathcal W(u,u)(x,\xi) dx d\xi
\\=
\sup_{\norm{u}_{L^{2}(\R^{n})}=1}\iint_{2\pi(\val x^{2}+\val \xi^{2})\le R^{2}}
\mathcal W(u,u)(x,\xi) dx d\xi,
\end{multline*}
and we have proven here that for $u\in L^{2}(\R^{n})$ with norm 1
$$
\iint_{\val x^{2}+\val \xi^{2}\le \frac{a}{2\pi}=\frac{R^{2}}{2\pi}}
\mathcal W(u,u)(x,\xi) dx d\xi
\le 1-\frac1{(n-1)!}\int_{a}^{+\io} e^{-t}t^{n-1}dt=1-\frac{\Gamma(n, R^{2})}{\Gamma(n)},
$$
where the upper incomplete Gamma function $\Gamma(z,x)$
is given by
\index{incomplete Gamma function} 
\begin{equation}\label{incgam}
\Gamma(z,x)=\int_{x}^{+\io} t^{z-1}e^{-t}dt.
\end{equation}
This is indeed the result of Theorem 1 in \cite{MR2761287}.
\begin{nb}
 Let $x>0$ be given  and let $z\in \C$ with $\re z >0$. Then we have 
 $$
 \Gamma(z,x)=\int_{0}^{+\io} (s+x)^{z-1}e^{-s-x}ds = e^{-x}
 \int_{0}^{+\io} (s+x)^{z-1}e^{-s}ds,
 $$
 so that if $z=n+1$,  $n\in \N$, we find
\begin{multline*}
 \Gamma(n+1,x)= e^{-x}
 \int_{0}^{+\io} (s+x)^{n}e^{-s}ds=e^{-x}\sum_{0\le k\le n}\binom{n}{k} x^k
  \int_{0}^{+\io} s^{n-k}e^{-s}ds
  \\
  =e^{-x}\sum_{0\le k\le n}\binom{n}{k} x^k
\Gamma(n+1-k)
 =n!e^{-x}\sum_{0\le k\le n}\frac{x^k}{k!}.
\end{multline*}
\end{nb}
\bibliography{refsurvey}
\bibliographystyle{amsplain}
\providecommand{\bysame}{\leavevmode\hbox to3em{\hrulefill}\thinspace}
\providecommand{\MR}{\relax\ifhmode\unskip\space\fi MR }
\providecommand{\MRhref}[2]{%
  \href{http://www.ams.org/mathscinet-getitem?mr=#1}{#2}
}
\providecommand{\href}[2]{#2}

\printindex
\end{document}